\documentclass[11pt, leqno, amscd, psamsfonts]{amsart}
\usepackage{amsthm, amsmath}
\usepackage{amssymb} 
\usepackage{amscd}

\usepackage[curve,matrix,arrow,frame,tips]{xy}
\usepackage{verbatim}

\usepackage{amsxtra}
\usepackage{amscd}
\usepackage{longtable}
\usepackage{bigstrut}
\usepackage{color}

\usepackage[curve,matrix,arrow,frame,tips]{xy}
\usepackage{mathdots}
\usepackage{verbatim}

\usepackage{graphicx}
\usepackage{xspace}
\usepackage{ifthen}

\usepackage{verbatim}
\usepackage{amssymb}
\usepackage{amsfonts}
\usepackage{amsmath}
\usepackage{amsthm}
\usepackage{amscd}

\newtheorem{theorem}{Theorem}[section]

\newtheorem{corollary}[theorem]{Corollary}

\newtheorem{definition}[theorem]{Definition}

\newtheorem{lemma}[theorem]{Lemma}
\theoremstyle{definition}
\theoremstyle{plain}
\newtheorem{proposition}[theorem]{Proposition}
\newtheorem{remark}[theorem]{Remark}

\newcounter{ex}[section]

\newcommand{\cal}{\mathcal}

\newcommand{\g}{\gamma}

\newcommand{\E}{{\mathcal E}}

\newcommand{\rE}{{\rm E}}

\newcommand{\St}{{\rm St}}

\newcommand{\GL}{{\rm GL}}
\newcommand{\SL}{{\rm SL}}

\newcommand{\what}{\hat}

\newcommand{\Q}{{\mathbb Q}}

\newcommand{\Hh}{{\mathcal H}}

\newcommand{\Z}{{\mathbb Z}}

\newcommand{\F}{{\mathcal F}}
\newcommand{\ti}{\tilde}
 \renewcommand{\O}{{\mathcal O}}

\newcommand{\rK}{{\rm K}}
\newcommand{\rSK}{{\rm SK}}
 
\renewcommand{\ti}{\tilde}

\newcommand{\und}{\underline}

\newcommand{\M}{{\mathcal M}}
\renewcommand{\L}{{\mathcal L}}

\newcommand{\Kr}{{\rm K}}

\newcommand{\LL}{{\mathcal L}}

\newcommand{\PP}{{\Bbb P}}
\newcommand{\DD}{{\mathcal D}}

\newcommand{\V}{{\mathcal V}}

\newcommand{\Cl}{{\mathrm {Cl}}}

\renewcommand{\mod}{{\rm \ mod\ }}

\def\thfill{\null\nobreak\hfill}

\def\endproof{\thfill\vbox{\hrule
  \hbox{\vrule\hbox to 5pt{\vbox to 5pt{\vfil}\hfil}\vrule}\hrule}}

\newcommand{\lps}{[\![}
\newcommand{\rps}{]\!]}
\newcommand{\llps}{(\!(}
\newcommand{\lrps}{)\!)}
\newcommand{\ldb}{\{\!\{}
\newcommand{\rdb}{{\}\!\}}}

\def\arrow#1#2{\smash{\mathop{\longrightarrow}\limits^{#1}_{#2}}} 
\def\arrowdown#1#2{\Big\downarrow \rlap{$\vcenter{\hbox{$\scriptstyle#2$}}$}
{\hbox to -10pt{\hss{$\vcenter{\hbox{$\scriptstyle#1$}}$}}}}
\def\arrowup#1#2{\Big\uparrow \rlap{$\vcenter{\hbox{$\scriptstyle#2$}}$}
{\hbox to -10pt{\hss{$\vcenter{\hbox{$\scriptstyle#1$}}$}}}}

\newtheoremstyle{examplestyle}
                {3pt}
                {3pt}
                {}
                {}
                {\bf}
                {}
                {.5em}
                {\thmname{#1}.\ \thmnote{#3}.}
\theoremstyle{examplestyle}

\DeclareMathOperator{\Spec}{Spec}

\begin{document}

\title[Higher adeles and  non-abelian Riemann-Roch]{ Higher   adeles and   non-abelian Riemann-Roch   }

\author[T. Chinburg]{T. Chinburg}
\address{T. Chinburg, Dept. of Mathematics\\ Univ. of Pennsylvania \\ Phila. PA. 19104, U.S.A.}
\email{ted@math.upenn.edu}
\thanks{T. C.  is partially supported by  NSF Grant  DMS11-00355.}

\author[G. Pappas]{G. Pappas}
\thanks{G. P.  is partially supported by  NSF Grant DMS11-02208.}

\address{G. Pappas\\ Dept. of
Mathematics\\
Michigan State
University\\
E. Lansing\\
MI 48824\\
USA}
\email{pappas@math.msu.edu}

\author[M. J. Taylor]{M. J. Taylor}
\address{M. J. Taylor\\ School of Math. \\ Merton College, Univ. of  
Oxford\\ Oxford, OX1 4JD, U.K.}
\email{martin.taylor@merton.ox.ac.uk}

\date{\tomorrow}

\begin{abstract}
{ We show a Riemann-Roch theorem for group ring bundles over an arithmetic surface;
this is expressed using the higher adeles of Beilinson-Parshin and the tame symbol
via a theory of adelic equivariant Chow groups and Chern classes. The theorem is obtained by combining
a group ring coefficient version of the local Riemann-Roch formula as in Kapranov-Vasserot
with results on $\rm K$-groups of group rings and an explicit description of group ring bundles over
${\mathbb P}^1$. 
Our set-up provides an extension of several aspects of the classical Fr\"ohlich theory of 
the Galois module structure of 
rings of integers of number fields to arithmetic surfaces.}

 \end{abstract}
 \date{\today}

\maketitle





 \bigskip
\bigskip
\bigskip

\centerline{\sc Contents}
\medskip

\noindent  \ \ \ \ \ \ Introduction\\
\S 1. \ Beilinson-Parshin adeles on a surface\\
\S 2. \ Equivariant adelic Chow groups\\
\S 3. \ Lattices, determinant functors and determinant theories \\
\S 4. \ Pushdown maps and reciprocity laws\\
\S 5. \ Transition matrices and the first Chern class\\
\S 6. \ Elementary structures and the second Chern class\\
\S 7. \ Equivariant Euler characteristics and the Riemann-Roch theorem\\
\S 8. \ The proof of the theorem; reduction to the case of ${\Bbb P}^1_\Z$\\
\S 9. \ The proof of the theorem; bundles over ${\Bbb P}^1_\Z$\\
\S 10. Appendix: Adelic Riemann-Roch for general bundles when $G=1$\\
\phantom{a} \ \ \ References
\bigskip
\medskip

\bigskip


\vfill\eject

\section*{Introduction}
In this paper we initiate an adelic theory of Galois module structure for arithmetic surfaces 
which extends many aspects of the corresponding theory  
for ring of integers of finite Galois extensions of number fields. 
Adelic methods have played an important role in the classical theory of Galois module structure 
(\cite{FrohlichBook});
the starting point is Fr\"ohlich's adelic description of the class group of finitely generated locally free modules
for the integral group ring of a finite group. Here, we introduce into the picture the higher dimensional adeles of Beilinson and Parshin,
 certain ``adelic Chow groups" defined using these, and also a host of other constructions, some
 of which are inspired from the theory of loop groups.  Our main result is
an adelic Riemann-Roch theorem for group ring bundles over an arithmetic surface;
this can be used for the calculation of equivariant Euler characteristics of arithmetic surfaces with 
a finite group action.

To explain further  we need to introduce some notation.  Let $Y$ be a projective regular arithmetic surface over $\Z$;
\emph{i.e.}   the structure morphism $Y\to \Spec(\Z)$ is projective and flat of relative dimension $1$ and $Y$ is regular and irreducible. 
Suppose that $G$ is a finite group. By definition, an $\O_Y[G]$-bundle $\E$ of rank $n$ on $Y$ is a coherent sheaf of (left) $\O_Y[G]$-modules
which is locally  free on $Y$, \emph{i.e.} there is a finite   affine Zariski open  cover $Y=\cup_{i\in I} U_i$, $U_i=\Spec(A_i)$, of $Y$
such that $\E|_{U_i}$ is the sheaf that corresponds to a   free $A_i[G]$-module of rank $n$.  To such an $\E$ we
can associate a projective Euler characteristic $\chi^P(Y, \E)$ in the Grothendieck group $\rK_0(\Z[G])$ of finitely generated projective
$\Z[G]$-modules as follows (see \cite{ChinburgTameAnnals}). Consider the \v{C}ech complex $C^\bullet(\{U_i\}, \E)$ obtained from $\E$ and the cover $\{U_i\}$;
one can show that $C^\bullet(\{U_i\}, \E)$ is a ``perfect" complex of $\Z[G]$-modules, \emph{i.e.} that there is a bounded complex $(P^\bullet)$ of 
finitely generated projective $\Z[G]$-modules $P^j$ and a $\Z[G]$-map of complexes 
$
P^\bullet \to C^\bullet(\{U_i\}, \E)
$
 which induces an isomorphism on cohomology groups. Then we define 
 $$
 \chi^P(Y, \E)=\sum\nolimits_j (-1)^j[P^j]
 $$
 where $[P^j]$ stands for the class of the module $P^j$ in the Grothendieck group $\rK_0(\Z[G])$; this is independent of the choice of 
 the cover $\{U_i\}$ and of the complex $P^\bullet$.  Recall that by Swan \cite{SwanAnnals}
 all finitely generated projective $\Z[G]$-modules are locally free. This gives a rank homomorphism ${\rm rank}: \rK_0(\Z[G])\to \Z$
 whose kernel $\rK^{\rm red}_0(\Z[G])$ can be identified with the class group ${\rm Cl}(\Z[G])$ of finitely generated
 locally free $\Z[G]$-modules studied by Fr\"ohlich.
 
 If $G$ is abelian, we can consider $\E$ as a vector bundle over the scheme $ Y\times G^*$ with $G^*=\Spec(\Z[G])$
 the Cartier dual; the class group ${\rm Cl}(\Z[G])$ can be identified with the Picard group ${\rm Pic}(G^*)$.
 In this case, versions of the Riemann-Roch theorem for $Y\times G^*\to G^*$
 (such as the Deligne-Riemann-Roch theorem of \cite{DeligneDet}) can be used to calculate the element $\chi^P(Y, \E)-\chi^P(Y,\O_Y[G]^n)$
 in ${\rm Cl}(\Z[G])={\rm Pic}(G^*)$. This basic observation together with the theory of cubic structures
 eventually leads to a satisfactory theory in this case, especially in the crucial case when 
 the bundle $\E$ is obtained from a tame cover $X\to Y$ (\cite{PaCubeInvent}, \cite{CPTAnnals}; see also
 below). 
 When $G$ is not abelian, the above do not apply. A new method is developed in  \cite{PappasAdams}. However,  we will see that
 the adelic point of view also gives a   framework
 for developing a sufficiently fine  theory that  can be used to calculate 
 the classes $\chi^P(Y, \E)$.

Indeed, it is our point of view here that the bundle $\E$ can also be described by adelic transition matrices as follows, where ``adelic" is meant in the sense of 
 the higher dimensional adeles of Beilinson and Parshin. Recall that a  (non-degenerate) Parshin $m$-chain of $Y$ is an ordered $m$-tuple
 $\eta=(\eta_{i_1}, \ldots, \eta_{i_m})$ of points of $Y$ with $ i_1<\cdots <i_m$,
such that $\eta_{i_k}$ lies on the Zariski closure of the previous point $\eta_{i_{k-1}}$
 and with the codimension of the closure of $\eta_i$ in $Y$  equal to $i$.
Since $Y$ is of dimension $2$, we   have $m=1$, $2$ or $3$. For every such Parshin chain $\eta$, one can define the ``multicompletion"
$\hat\O_{Y,\eta}=\hat\O_{Y, \eta_{i_1}\ldots \eta_{i_{m}}}$ by successively taking   localizations and completions of $\O_Y$
starting from $\eta_{i_{m}}$  (see Proposition \ref{prop:multi}). For example, if $\eta$ is 
a $1$-chain and $\eta$ is a single point, $\hat\O_{Y, \eta}$ is the completion of the local ring of $Y$ at $\eta$. 
In particular, for the generic point $\eta_0$ of $Y$ we have $\hat\O_{Y, \eta_0}=K(Y)$, the function field of $Y$. 
If $\eta=(\eta_0,\eta_1,\eta_2)$
is a $3$-chain, then $\hat\O_{Y, \eta_0\eta_1\eta_2}$ is a finite direct sum of two-dimensional local fields. 
For each point $\xi$
of $Y$,  we can pick a $\hat\O_{Y, \xi}[G]$-basis $e_\xi=\{e_\xi^h\}_{h=1}^n$ of the completed stalk   $\hat\E_\xi$. If $(\eta_0,\eta_1, \eta_1)$ is a Parshin triple,  
and $0\leq i<j\leq 2$, then we can compare bases at $\eta_i$ and $\eta_j$ and write
\begin{equation*}
e_{\eta_i}=\lambda_{\eta_i\eta_j}\cdot e_{\eta_j}, \quad \hbox{\rm with}\ \ \lambda_{\eta_i\eta_j}\in \GL_n(\hat\O_{Y, \eta_i\eta_j}[G]).
\end{equation*}
The matrices $\lambda_{\eta_i\eta_j}$ are   ``adelic transition matrices" for the bundle $\E$. 
We say that $\E$ has elementary structure 
if we can choose bases   as above such that the corresponding transition matrices $\lambda_{\eta_i\eta_j}$, regarded in the infinite general
linear group $\GL (\hat\O_{Y, \eta_i\eta_j}[G])$, belong to the commutator subgroup $\rE(\hat\O_{Y, \eta_i\eta_j}[G])$
generated by elementary matrices.

By an ``adelic Riemann-Roch theorem" for $\E$,  we mean a formula that allows us to calculate 
the Euler characteristic $\chi^P(Y, \E)$ starting from the adelic transition matrices $\{\lambda_{\eta_i\eta_j}\}$ and which involves 
suitable ``adelic characteristic classes" of $\E$. 

Our main result gives
 an adelic Riemann-Roch theorem for 
bundles $\E$ that have elementary structure, under some technical assumptions on $Y$ and $G$.
(We give a more general result in the Appendix when $G$ is trivial.)
In particular, for this we will assume that the group algebra $\Q[G]$ splits in the sense that
we can write
\begin{equation}\label{splitgroupalgebraIntro}
\Q[G] = \prod\nolimits_i{{\rm Mat}}_{m_{i}\times m_i}( Z_{i}),
\end{equation}
where each $Z_{i}$ is a {\sl commutative} finite field extension of $\Q$, \emph{i.e.} a number field.
However, a number of the results in the paper are true for arbitrary finite groups $G$.
Also, in addition to our standing hypotheses on $Y$, we assume:
\medskip

\noindent {\sl (H) \ All the irreducible components of the fibers of the morphism
$Y\to \Spec(\Z)$ are smooth (therefore also reduced) and furthermore, the fibers at primes 
that divide the order of the group $G$ are irreducible.}

\medskip

To describe the Riemann-Roch theorem we need to explain  several important ingredients: 

We first have the adelic Chow groups ${\rm CH}^i_{\mathbb A}(Y[G])$ for $i=1$, $2$. 
We   define 
 \begin{equation*}\label{ch2def}
{\rm CH}^2_{\mathbb A}(Y[G]):=\frac{ \prod\nolimits_{(\eta_0,\eta_1, \eta_2)}'{\rm K}_2(\hat \O_{ Y, \eta_0\eta_1 \eta_2}[G])
\cdot \prod\nolimits_{(\zeta ,\xi)}'{\rm K}_2(\hat \O_{Y, \zeta \xi}[G])^\flat }{ \prod\nolimits_{(\zeta ,\xi)}'{\rm K}_2(\hat \O_{Y, \zeta \xi}[G])^\flat}
 \end{equation*}
as a  quotient of a suitably restricted (adelic) products  of ${\rm K}_2$-groups of  multicompletions, where
the indices range over all Parshin $3$-chains and all $2$-chains respectively.  (See \S \ref{s:Kadeles}, \ref{def:ECg} for 
details.) 
Similarly, we set
\begin{equation*}\label{ch1}
{\rm CH}^1_{\mathbb A}(Y[G]):=\frac{\prod\nolimits_{(\eta_0,\eta_1)}'{\rm K}_1(\hat \O_{Y, \eta_0\eta_1}[G])}{\rK_1(K(Y)[G])\cdot \prod\nolimits_{\eta_1}{\rm K}_1(\hat \O_{Y,\eta_1}[G])^\flat}.
 \end{equation*}
These definitions are interesting even when $G$ is the trivial group. If $G=\{1\}$,
 ${\rm CH}^1_{\mathbb A}(Y[G])\cong {\rm Pic}(Y)$. 
 In the case of the trivial group and when $Y$ is a projective smooth surface over a field  
this second adelic Chow group agrees with the one  considered by Osipov in \cite{OsipovAdelic}.
  Osipov shows  that, in this geometric
 non-equivariant case, his second adelic Chow group agrees with the classical Chow group ${\rm CH}^2(Y)$
 of codimension $2$ cycles up to rational equivalence.
On the other hand, recall that by Fr\"ohlich's classical results  we have a canonical isomorphism
$$
{\rm Cl}(\Z[G])\cong \frac{\prod\nolimits_{p}' {\rm K}_1(\Q_p[G])}{({\rm K}_1(\Q[G])\prod\nolimits_p{\rm K}_1(\Z_p[G]))^\flat}.
$$
This isomorphism allows us to identify
the class group ${\rm Cl}(\Z[G])$  with the first adelic Chow group ${\rm CH}^1_{\mathbb A}(\Spec(\Z)[G])$ of $\Spec(\Z)$.

The second ingredient of our Riemann-Roch theorem is  a pushdown (Gysin)
homomorphism along $f: Y\to \Spec(\Z)$ 
\[
f_*: {\rm CH}^2_{\mathbb A}(Y[G])\to {\rm CH}^1_{\mathbb A}(\Spec(\Z)[G])={\rm Cl}(\Z[G]).
\]
This is constructed by assembing homomorphisms 
 $$
 f_{*\eta_0\eta_1\eta_2}: \rK_2(\hat\O_{Y,\eta_0\eta_1\eta_2}[G])\to \rK_1(\Q_p[G]),
 $$
 where $p$ is the characteristic of the closed point $\eta_2$, which are obtained using
 either the classical tame symbol or Kato's residue symbol. Showing that these homomorphisms 
 produce a pushdown $f_*$ between the adelic Chow groups is a subtle affair that among other ingredients 
 involves using various reciprocity 
 laws. For the most part, the needed reciprocity laws can be shown via reduction to the case $G=\{1\}$ by using 
 the splitting (\ref{splitgroupalgebraIntro}) together with Morita equivalence. Then they 
  follow by work of Parshin, Kato or more recently Liu \cite{LiuKato}, see also Osipov-Zhu
  \cite{OsipovZhu}. However,  the hardest part of the argument 
 is  proving that the denominator in the definition of ${\rm CH}^2_{\mathbb A}(Y[G])$ maps to the denominator 
 in the Fr\"ohlich description of ${\rm Cl}(\Z[G])$. This involves checking that certain symbols are ``determinants",
 \emph{i.e.} that they belong to ${\rm Det}(\Z_p[G])$. This is a genuine integral group ring problem 
 and its proof uses in a crucial manner the central extension
 (\ref{Hintro}) which we will describe below (see also Remark \ref{NotMoritaYouIdiot}). By the way, extending the results of this paper to
 the case that the group algebra $\Q[G]$ does not split is a very interesting but highly non-trivial problem
 that would involve a  study of the groups $\rK_1(D)$, $\rK_2(D)$, for  certain division algebras $D$.
For example, we expect that this involves new corresponding reciprocity laws. 
 
Finally, the third ingredients are the adelic Chern classes $c_1(\E)$ and $c_2(\E)$ of $\E$.
The first Chern class $c_1(\E)$ is defined for an arbitrary $\O_Y[G]$-bundle $\E$:
It is given as the class  of $\prod_{(\eta_0, \eta_1)} {\rm Det}(\lambda_{\eta_0\eta_1})$ in 
 ${\rm CH}^1_{\mathbb A}(Y[G])$ where $\lambda_{\eta_0\eta_1}$ are   adelic transition matrices as 
 above and ${\rm Det}(\lambda)$   stands for the class of a matrix 
$\lambda\in \GL_n(\hat\O_{Y, \eta_0\eta_1}[G])$ in $\rK_1(\hat\O_{Y, \eta_0\eta_1}[G])$.
The second Chern class $c_2(\E)$ is only defined when $\E$ has an elementary structure.  Recall the Steinberg extension
  \begin{equation*}
  1\to \rK_2(\hat\O_{Y, \eta}[G])\to {\rm St}(\hat\O_{Y, \eta}[G])\to \rE(\hat\O_{Y, \eta}[G])\to 1
\end{equation*}
with $\rE(\hat\O_{Y, \eta}[G])$  the elementary subgroup of the infinite general linear group
  $\GL(\hat\O_{Y, \eta}[G])$.
To construct the second Chern class, we choose  lifts $\tilde\lambda_{\eta_i\eta_j}$ of the transition matrices
$\lambda_{\eta_i\eta_j}$ to the Steinberg group and consider 
$$
z(\tilde\lambda)_{(\eta_0, \eta_1, \eta_2)}:=\tilde\lambda_{\eta_0\eta_2}\cdot (\tilde\lambda_{\eta_0\eta_1})^{-1}\cdot (\tilde \lambda_{\eta_1\eta_2})^{-1}
$$
as an element in $\rK_2(\hat\O_{Y, \eta_0\eta_1\eta_2}[G])$. By choosing  the lifts $\tilde\lambda_{\eta_i\eta_j}$ carefully, 
we can guarantee
that the Steinberg cocycle $z(\tilde\lambda):=(z(\tilde\lambda)_{(\eta_0,\eta_1,\eta_2)})_{(\eta_0,\eta_1,\eta_2)}$ is ``adelic", \emph{i.e.} lies in the numerator of the right hand side in the
definition of ${\rm CH}^2_{\mathbb A}(Y[G])$ (See Proposition \ref{propae}).  
Then the Chern class $c_2(\E)$ is given as the class of the element $z(\tilde\lambda)$ in ${\rm CH}^2_{\mathbb A}(Y[G])$. (A similar construction was given by Budylin \cite{BudChern}.)

We are now ready to state our main result.

\begin{theorem}\label{ARRintro}
Assume that $\Q[G]$ splits and that $Y\to \Spec(\Z)$ is a regular arithmetic surface that 
satisfies (H). Then, if $\E$ is an $\O_Y[G]$-bundle of rank $n$ with elementary structure, we
have
\[
\chi^P(Y, \E)-\chi^P(Y, \O_Y[G]^n)=-f_*(c_2(\E)).
\]
\end{theorem}

Let us remark here that if $\E$ has an elementary structure, then $c_1(\E)$ is trivial; this then
explains the shape of the identity above. Indeed, in this case of relative dimension $1$, this agrees with the shape of
the classical Grothendieck-Riemann-Roch formula for vector bundles of rank $n$ with trivial determinant
(see for example \cite{DeligneDet}).   

Interesting examples of bundles $\E$ for which one can apply the Riemann-Roch formula
are provided as follows. Suppose that $q:X\to Y$ is a finite flat $G$-cover of the arithmetic surface
$Y$; one can see that if the ramification of $q$ is {\sl tame} and $\F$ is a $G$-equivariant bundle on $X$, then $\E=q_*(\F)$ gives a $\O_Y[G]$-bundle
on $Y$. Then $\chi^P(Y, \E)$ is equal to the equivariant projective Euler characteristic $\chi^P(X, \F)$
studied in \cite{ChinburgTameAnnals}, \cite{CPTAnnals}, and other articles. 
The ``cubic method" of \cite{CPTAnnals} provides 
a very effective way
of calculating such Euler characteristics but  with the crucial limitation that $G$ is abelian.
Here we are allowing more general finite groups and so Theorem \ref{ARRintro} adds 
to the tools currently available for the calculation of such Euler characteristics. 
In particular, it can complement the method developed in \cite{PappasAdams}.
For example, in \S \ref{exampleperf} we give conditions  on the group $G$, such that 
if $q$ is a $G$-torsor the $\O_Y[G]$-bundle $\E=q_*\O_X$ has always elementary structure and our result applies.
In particular, we show that this is the case when  $G$ is an alternating group $A_n$ with $n\geq 5$, or
 ${\rm PSL}(2, {\Bbb F}_p)$ with $p$ odd prime, or $\SL(2, {\Bbb F}_{2^n})$, $n>2$.
We would like to return to such applications in a future paper.

We will now give some more details about our techniques and discuss the proof of the Riemann-Roch theorem. 

Important input is provided by certain central extensions which are arithmetic versions of  a
standard construction  in the theory of loop 
groups and infinite dimensional Kac-Moody Lie algebras. Suppose that $R$ is a commutative ring and consider 
the formal power series ring $R{\lps t\rps}$ and the Laurent power series ring $R\llps t\lrps=R{\lps t\rps}[t^{-1}]$. We define a central extension 
\begin{equation}\label{Hintro}
1\to \rK_1(R[G])\to \Hh(R\llps t\lrps[G])\to \GL'(R\llps t\lrps[G])\to 1
\end{equation}
where $\GL'(R\llps t\lrps[G])$ is a subgroup of the infinite general linear group $\GL(R\llps t\lrps[G])$ that contains the commutator $\rE(R\llps t\lrps[G])$. This central extension is provided via the choice of a determinant theory
on $R{\lps t\rps}[G]$-lattices in $R\llps t\lrps[G]^n$ for $n\geq 1$. 
This notion has been introduced by Drinfeld and Kapranov. Set $L_0=R{\lps t\rps}[G]^n$. 
Recall that a $R{\lps t\rps}[G]$-lattice $L$ in $R\llps t\lrps[G]^n$
is a projective $R{\lps t\rps}[G]$-submodule of $R\llps t\lrps[G]^n$ such that $t^NL_0\subset L\subset t^{-N}L_0$
for some $N>0$. For us, a determinant theory is a suitable functor from a category of $R{\lps t\rps}[G]$-lattices
to the virtual category $V(R[G])$ of projective finitely generated $R[G]$-modules. We can construct a determinant theory as follows: 
First construct an $\O_{\PP^1_R}[G]$-bundle $\E(L)$ over the projective line $\PP^1_R$
by gluing the (trivial) bundles corresponding to the modules $R[t^{-1}][G]^n$ over ${\mathbb A}^1_R=\Spec(R[t^{-1}])$ 
and $L$ over 
${\mathbb A}^1_R=\Spec(R{\lps t\rps})$ using the identification $L\otimes_{R{\lps t\rps}}R\llps t\lrps\cong R[t^{-1}][G]^n\otimes_{R[t^{-1}]} R\llps t\lrps$ provided by the inclusion $L\subset R\llps t\lrps[G]^n$.
Now we can consider the determinant of the cohomology complex in the derived category
\[
\delta(L):={\rm det}({\rm R}\Gamma(\PP^1_R, \E(L))
\] 
as an object in the virtual category $V(R[G])$. The association $L\mapsto \delta(L)$ gives a determinant theory.
For $g\in \GL'_n(R\llps t\lrps[G])$, we consider $L=L_0\cdot g^{-1}$, so that $\E(L)$ has $g$ as a transition matrix along a formal neighborhood of $t=0$.  The central extension $\Hh_n(R\llps t\lrps[G])$   is a group with elements
pairs   $(g, \phi_g)$ with $g\in \GL'_n(R\llps t\lrps[G])$ and $\phi_g$ an isomorphism between $\delta(L_0)$ and $\delta(L_0\cdot g^{-1})$. (From the very definition of $\GL'_n(R\llps t\lrps[G])$ there exists such an isomorphism. See \S \ref{3d} for  details.)  Now consider
the direct limit as $n$ goes to infinity to obtain (\ref{Hintro}).

By the universality of the Steinberg extension we obtain a map of central extensions
\begin{equation}
\label{StoHmapInto}
\begin{matrix}
1&\to &\rK_2(R\llps t\lrps[G])&\to &{\rm St}(R\llps t\lrps[G])&\to &\rE(R\llps t\lrps[G])&\to &1\\
&&\partial\downarrow &&\partial\downarrow &&  \downarrow &\\
1&\to& \rK_1(R[G])&\to &\Hh(R\llps t\lrps[G])&\to &\GL'(R\llps t\lrps[G])&\to &1.
\end{matrix}
\end{equation}
A first incarnation of the Riemann-Roch theorem in this case is the fact that
$\partial$ can be calculated using the tame symbol, in fact, $\partial$  is equal to the inverse
of the tame symbol when $R$ is a field and $G$ is trivial (see Proposition \ref{tameprop}). In fact, this statement is often regarded
as a ``local" Riemann-Roch formula, see \cite{KapranovLocalRR}. 

We can use this to obtain an adelic
Riemann-Roch formula for bundles over $\PP^1=\PP^1_\Z$ as follows. First we show, by using an equivariant 
  version of an argument of Horrocks, that each  $\O_{\PP^1}[G]$-bundle $\E$ of rank $n$ over $\PP^1 $ which is trivial along the section $(1:1)$
  in homogeneous coordinates
  can be obtained by gluing
  trivial bundles over $\Spec(\Z[t])$ and $\Spec(\Z[t^{-1}])$ via a transition matrix $g\in \GL_n(\Z[t, t^{-1}][G])$.
If the bundle has, in addition, degree $0$, then the matrix $g$ regarded in $\GL_n(\Q[t, t^{-1}][G])$ and in $\GL_n(\Z_p[t, t^{-1}][G])$,
 for each prime $p$, lies in the subgroups $\GL'_n(\Q[t, t^{-1}][G])$ and $\GL'_n(\Z_p[t, t^{-1}][G])$ respectively. By definition,
 this means that the base changes $\delta(\E)_\Q$ and $\delta(\E)_{\Z_p}$ of the determinant of cohomology $\delta(\E)={\rm det}({\rm R}\Gamma(\PP^1, \E))$ are 
 isomorphic, as elements in the virtual categories $V(\Q[G])$ and $V(\Z_p[G])$, to the free rank $n$ elements $[\Q[G]^n]$ and 
 $[\Z_p[G]^n]$; suppose that $ \alpha_\Q$, $ \alpha_p $ are choices of corresponding isomorphisms. 
 The pairs $(g, \alpha_\Q)$ and $(g, \alpha_p)$ are then elements of $\Hh_n(\Q\llps t\lrps[G])$ and $\Hh_n(\Z_p\llps t\lrps[G])$; these elements lift $g$ considered in 
$\GL'_n(\Q[t, t^{-1}][G])$ and $\GL'_n(\Z_p[t, t^{-1}][G])$ respectively. 
 Both $\alpha_\Q$ and $\alpha_p$ induce isomorphisms between $\delta(\E)_{\Q_p}$ and $[\Q_p[G]]^n$;
 by comparing them we obtain an element $\alpha_p^{-1}\cdot \alpha_\Q$  of the automorphism
 group of $[\Q_p[G]^n]$, \emph{i.e.} an element of $\rK_1(\Q_p[G])$.
 The class $\chi^P(\PP^1, \E)-\chi^P(\PP^1, \O_{\PP^1}[G]^n)$ coincides with the class of $\delta(\E)$ in the class group $\Cl(\Z[G])$;
 by the above, this 
  can now be obtained as the class 
 of the element $\prod_p \alpha_p^{-1}\cdot \alpha_\Q\in \prod_p'\rK_1(\Q_p[G])$.  The local Riemann-Roch formula that relates the 
 central extensions above via the tame symbol will now eventually lead  to a proof of our main theorem for
 $\PP^1 $ but this still requires a fair amount of work. Indeed, first, we need to show that the bundles we are considering have,
 after suitable changes of basis, elementary  transition matrices
 and therefore also a well-defined second Chern class $c_2(\E)$. We also need to explain how  
 to express a Steinberg cocycle  that can be used to calculate $c_2(\E)$ in terms of the original transition 
 matrix $g$; notice that $g$ itself might not be elementary. 
 
 The notion of elementary structure is, as it turns out, quite subtle. 
 Observe that the transition matrix $\lambda_{\eta_i\eta_j}$
 is elementary when the class $[\lambda_{\eta_i\eta_j}]$ in $\rK_1(\hat\O_{Y, \eta_i\eta_j}[G])$ is trivial. 
Therefore, examining when adelic transition matrices
 are elementary involves the consideration of $\rK_1$-groups of group rings with coefficients 
in certain $p$-adically complete rings, as are some of the multicompletions 
considered above. For this we need to use results of Oliver on $\rSK_1$-groups
of group rings (e.g. \cite{OliverBAMS, OliverLNM}) and the extension of these results in \cite{CPTSK1b}.
In particular, we can see that our notion of elementary structure is appropriately restrictive;
for example, our considerations show that any $\O_{\PP^1 }[G]$-bundle which is trivial along $(1:1)$
and has zero degree has an elementary structure.
Considering these multicompletions also necessitates that we develop certain ``$p$-adically completed" variations of the central extension (\ref{Hintro}); for example, we need such extensions for group rings with coefficients in the $p$-adic completion $\Z_p{\ldb t\rdb}=\varprojlim_n \Z/p^n\llps t\lrps$
of $\Z_p\llps t\lrps$ or in the two-dimensional local field $\Q_p{\ldb t\rdb}=\Q_p\otimes_{\Z_p}\Z_p{\ldb t\rdb}$.

The above gives the rough idea of the proof of the Riemann-Roch theorem for $Y=\PP^1 $. To obtain the main result for an $\O_Y[G]$-bundle $\E$ on a more general (regular) arithmetic surface $Y\to \Spec(\Z)$ 
we argue as follows: By work of B. Green, there exists a finite flat morphism $\pi: Y\to \PP^1 $;  we use pushforward by $\pi$ to
reduce the proof to the case of $\PP^1 $. The fact that, as we assume, $\E$ has elementary structure does not imply that
this is also the case for $\pi_*(\E)$; this complicates the argument.
However, we can still find a simple bundle $\mathcal V$ with induced $G$-action so that the direct sum $\pi_*(\E)\oplus {\mathcal V}$
is  an $\O_{\PP^1 }[G]$-bundle with elementary structure on $\PP^1 $. We now explicitly relate 
Steinberg cocycles that compute the second Chern class of $\E$ with corresponding  Steinberg cocycles  that compute the second Chern class of   $\pi_*(\E)\oplus {\mathcal V}$ on $\PP^1 $ and there is a resulting identity (Proposition  
\ref{prop:state}) that
relates the second Chern classes of $\E$ and  of $\pi_*(\E)\oplus {\mathcal V}$.
This identity can be viewed as an adelic Riemann-Roch formula for the finite flat morphism $\pi$. These considerations  allow us to reduce the general case to the case $Y=\PP^1 $ which is handled  as explained above.

Our  definitions of the second adelic Chern class and of the Gysin map were initially inspired by the groundbreaking work of Parshin in \cite{ParshinCrelle}, \cite{Parshin2} and of Osipov in \cite{OsipovAdelic}. 
The reader can also find similar or related constructions in the work of 
H\"ubl-Yekutieli \cite{YekuChern}, Morrow \cite{MorrowTrace} and Budylin \cite{BudChern}. 
Let us remark here that
although our main interest in this paper is 
to the case of bundles for a group ring, our techniques 
can also provide interesting new results 
when the group $G$ is trivial and  even in the context of these references.
Indeed, the current paper also contributes to the theme of 
refined Riemann-Roch type theorems; examples of such theorems are 
Deligne's functorial Riemann-Roch theorem for relative curves 
\cite{DeligneDet}, or the second author's integral Grothendieck-Riemann-Roch theorem
\cite{PaIGRR}. 
For example, we can consider general vector bundles $\E$ 
over an arithmetic surface $Y\to \Spec(R)$ where $R$ is a
Dedekind ring with finite residue fields, as is the ring of integers
of a number field. We then show an adelic Riemann-Roch theorem for $f: Y\to \Spec(R)$ and $\E$
(see Theorem \ref{genARR})
by factoring $f$ as a composition of a finite flat morphism $\pi: Y\to \PP^1_R$ with 
the projection $h: \PP^1_R\to \Spec(R)$ and proving as above 
Riemann-Roch identities for $\pi$ and $h$. We can think of this as an alternative to   Grothendieck's strategy of 
proving Riemann-Roch by
factoring $f$ as a composition of a closed immersion followed by a projective bundle.
This is done in the Appendix of the paper.

We will now briefly describe the structure of the paper. In \S 1,  we explain
the theory of higher dimensional adeles of Beilinson and Parshin and
give examples of the corresponding multicompletions for the case
of arithmetic surfaces. In \S 2, we give the definitions of the adelic Chow
groups. The constructions of the central extensions (\ref{Hintro}) and of
its $p$-adically complete variants are given in \S 3. In the same paragraph, we also
show
that the corresponding maps $\partial$ (resp. $\hat\partial$ in the $p$-adic variant)
in (\ref{StoHmapInto}) are given via the inverse of the tame symbol (resp. of Kato's residue symbol).
In \S 4, we define the pushdown maps $f_{*\eta_0\eta_1\eta_2}$ and show that they induce a Gysin map
$f_*$ between the adelic Chow groups as above. In \S 5, we explain the formalism of 
adelic transition matrices and give the definition of the first adelic Chern class.
The notion of elementary structure and  the definition of the second adelic Chern class 
for bundles with elementary structure is given in \S 6. In \S 7 we state the main theorem 
and in \S 8 we explain the reduction of the proof to the case of bundles over $\PP^1 $
by working with pushdown along a finite flat morphism $Y\to \PP^1 $.
Finally, the proof of the adelic Riemann-Roch identity for bundles over $\PP^1 $ occupies \S 9.
In the Appendix, \S 10, we consider the case that the group $G$ is trivial;
then we can obtain a stronger result.
\bigskip

\noindent{\sl Acknowledgement:} The second author would like to thank X. Zhu for a useful conversation.
 
\bigskip

\noindent{\sc Notations:}  
\begin{itemize}
\item If $R$ is a ring, we denote by $R\lps t\rps$ the ring of formal
power series in the variable $t$ with coefficients in $R$ and by $R\llps t\lrps=R\lps t\rps [t^{-1}]$ the ring of formal Laurent power series. 

\item Suppose $p$ is a prime such that $pR$ is a proper ideal of
$R$. We denote by $\hat R$ the $p$-adic completion $\varprojlim_n R/(p)^n$ of $R$. 
We will denote by $\hat R\ldb t\rdb$ the $p$-adic completion $\varprojlim_n(R/(p)^n\llps t\lrps) $ of the Laurent power series $R\llps t\lrps$
and by $\hat R\langle\!\langle  t\rangle\!\rangle$ the $p$-adic completion $\varprojlim_n R/(p)^n[ t] $ of the polynomial ring $R[t]$.

\item Let $R$ be the ring of integers in a finite extension $F$ of the $p$-adic field $\Q_p$. We set 
\[
\ \ \ \ \ \ \ \ \ F\ldb t\rdb:=F\otimes_R R\ldb t\rdb=\left\{\sum_{i} a_i t^i\ |\ a_i\in F,\  \lim_{i\to -\infty} v(a_i)=+\infty, \ v(a_i)>>-\infty\right\}
\]
where $v$ denotes the $p$-adic valuation. Then $F\ldb t\rdb$ is the fraction field 
of the $p$-adically complete dvr $R\ldb t\rdb$. We also set $F\{ t\}:=F\otimes_R R\langle\!\langle  t\rangle\!\rangle$.
The ring 
$F\{ t\}$ is often referred to as the free Tate algebra in the variable $t$ over $F$.

\end{itemize}

\section{Beilinson-Parshin adeles on a surface}\label{section1}

\setcounter{equation}{0}

\subsection{Parshin tuples and multicompletions.}\label{multicomplete}

 In this section we will let $Y$ be an irreducible separated Noetherian scheme of dimension $d$. We will recall the theory of adeles for $Y$ developed by Parshin and Beilinson; 
see \cite{ParshinCrelle}, \cite{Parshin2}, \cite{BeilinsonAdeles}, \cite{Huber}, 
 \cite{YekuAsterisque} and the useful survey \cite{MorrowSurvey} for more
detailed accounts. 

Following \cite{BeilinsonAdeles}, let $P(Y)$ be the set of points of $Y$.  If $\eta,\eta' \in P(Y)$ we will say that $\eta \ge \eta'$ if $\eta'$ is a point on the closure  $\overline{\eta}$ of $\eta$.   Let $S(Y)$ be the simplicial set associated
to $P(Y)$ and this order relation.  Thus the $n$-simplex $S(Y)_n$ is the set of all Parshin $n+1$-tuples
$(\eta(0),\ldots,\eta(n))$ of points on $Y$, these being ordered sequences of $n+1$ points in $P(Y)$ such that $\eta(0) \ge \eta(1) \ge \cdots \ge \eta(n)$.  We will call such an $n+1$-tuple degenerate if $\eta(i) = \eta(i+1)$ for some $i$;  otherwise it is non-degenerate.
We will use the convention that a subscript on a point indicates its codimension on $Y$.  Thus
$\eta_0$ is the generic point.  The Parshin $1$-tuples thus have the form $(\eta_i)$ for some $0 \le i \le d$, the Parshin
$2$-tuples have the form $(\eta_i,\eta_j)$ for some $0 \le i \le j \le d$, 
and so on. 

 Suppose $K_n$ is a subset of $S(Y)_n$ and that $\eta$ is a point of $Y$.  Let $\O_\eta = \O_{Y,\eta}$
 be the local ring of $Y$ at $\eta$, with maximal ideal $m_{\eta} = m_{Y,\eta}$.  Let $j_\eta:\mathrm{Spec}(\O_\eta) \to Y$
 be the natural morphism of schemes.  If $M$ is a module for $\O_\eta$, we also use   $M$ to denote both
 the quasi-coherent sheaf $\mathcal{M}$ on $\mathrm{Spec}(\O_\eta)$ associated to $M$ and
 the quasi-coherent sheaf $(j_{\eta})_*(\mathcal{M})$ on $Y$.  In particular, the support of $M$
 as a sheaf on $Y$ is contained in the closure $\overline{\eta}$ of $\eta$.  Define
\begin{equation}
\label{eq:chump}
{}_{\eta}K_{n-1} = \{(\eta(1), \eta(2),\ldots,\eta(n)) \in S(Y)_{n-1}: (\eta,\eta(1),\ldots,\eta(n)) \in K_n\}.
\end{equation}

\begin{definition}
\label{def:adeldef}
As $n$ and $K_n$ vary, there is a unique family of functors $A(K_n,\bullet)$ from the category of
quasi-coherent $\O_Y$-modules to the category of abelian groups for which the following
is true:
\begin{enumerate}
\item[1.] $A(K_n,\bullet)$ commutes with direct limits.
\item[2.] Suppose $M$ is a coherent $\O_Y$-module.
\begin{enumerate}
\item[a.] If $n = 0$, then 
\begin{equation}
\label{eq:nzero}
A(K_n,M) = A(K_0, M)=\prod_{\eta \in K_0} \lim_{\leftarrow \atop \ell} (M \otimes_{\O_Y} (\O_\eta/m_\eta^\ell)).
\end{equation}
\item[b.] If $n > 0$, then 
\begin{equation}
\label{eq:nbig}
A(K_n,M) = \prod_{\eta \in P(Y)} \lim_{\leftarrow \atop \ell} A({}_{\eta}K_{n-1}, M \otimes_{\O_Y} (\O_\eta/m_\eta^\ell)).
\end{equation}
\end{enumerate}
\end{enumerate}
\end{definition}

A subtlety in this definition is that the sheaf
$M \otimes_{\O_Y} (\O_\eta/m_\eta^\ell)$ appearing on the right side of
(\ref{eq:nbig}) will not in general be coherent.  Thus one must calculate
the value of $A({}_{\eta}K_n,\bullet)$ on the latter sheaf by taking an
inductive limit. 

When $K_n = \{(\eta(0),\ldots,\eta(n))\}$ consists of a single non-degenerate
Parshin chain of length $n+1$ and $M = \O_Y$, we will denote by 
\begin{equation}
\label{eq:OY}
\hat{\O}_{Y,\eta(0)\eta(1) \cdots\eta(n)} =
A(K_n,\O_Y)
\end{equation} the corresponding multicompletion of $\O_Y$.

\subsection{Examples of multicompletions.}

Suppose here that $K_n = \{(\eta(0),\ldots,\eta(n))\}$ consists of a single non-degenerate
Parshin chain of length $n+1$.  Let $\mathrm{Spec}(R)$ be an open affine subset of $Y$ which contains
$\eta(0)$.  Then, for all $i$, $\eta(i)$ corresponds to a prime ideal of $R$. Suppose $\mathfrak{a}$ and $\mathfrak{p}$ are
ideals of $R$, $\mathfrak{p}$ is prime and that 
 $N$ is an $R$-module.  As in \cite[p. 250]{Huber}, let $S_{\mathfrak{p}}^{-1} N$ be the localization
of $N$ at $S_{\mathfrak{p}} = R - \mathfrak{p}$ and define $C_{\mathfrak{a}} N =\varprojlim_{n} N/\mathfrak{a}^n N$.  

The following result is shown by Huber in \cite[Prop. 3.2.1]{Huber}.

\begin{proposition}
\label{prop:multi}Let $M$ be a quasi-coherent $\O_Y$-module, and suppose that the restriction of $M$ to $\mathrm{Spec}(R)$
is the sheaf associated to the $R$-module $N$.  Then $C_{\eta(0)} S_{\eta(0)}^{-1} \ldots C_{\eta(n)} S_{\eta(n)}^{-1} R = B$
is a flat Noetherian $R$-algebra, and there is a natural isomorphism 
\begin{equation}
\label{eq:localdes}
A(K_n,M) \cong  B \otimes_R N.
\end{equation}
If $M$ is coherent, so that $N$ is Noetherian, one has 
\begin{equation}
\label{eq:localcodes}
A(K_n,M) \cong  C_{\eta(0)} S_{\eta(0)}^{-1} \ldots C_{\eta(n)} S_{\eta(n)}^{-1} N.
\end{equation}
\end{proposition}

We now specialize further to the case in which $M = \O_Y$ as in (\ref{eq:OY}).

\subsubsection{Some Parshin chains of length $1$.}
\label{s:parsh0}

We suppose in this subsection that $n = 0$  and $K_n = K_0 = \{(\eta(0))\}$ for a point $\eta(0) = \eta_i$ of codimension $i$ on $Y$.
Then (\ref{eq:localcodes}) shows that 
$$
\what{\O}_{Y,\eta(0)} = \what{\O}_{Y,\eta_i} = C_{\eta_i} S_{\eta_i}^{-1} R
$$
is the completion of the local ring $\O_{Y,\eta_i}$ at the powers of its maximal ideal.

We now suppose further that $Y$ is irreducible, normal and flat over $\mathbb{Z}$,
and that  $\eta_i = \eta_1$ has
codimension $1$. Then $\hat{\O}_{Y,\eta_1}$ is a complete discrete valuation ring (dvr)
of characteristic $0$ 
with residue field $k(\eta_1)$ given by the function field of the irreducible divisor $\overline{\eta}_1$.
Let $t$ be a uniformizer in $\hat{\O}_{Y,\eta_1}$.  

If $\eta_1$ is horizontal then $k(\eta_1)$
has characteristic $0$ and transcendence degree $\mathrm{dim}(Y) -2$ over $\mathbb{Q}$.  In this case Hensel's Lemma shows 
there is an algebra homomorphism $k(\eta_1) \to \hat{\O}_{Y,\eta_1}$ which is a section of the residue map
$\hat{\O}_{Y,\eta_1} \to k(\eta_1)$ and that $\hat{\O}_{Y,\eta_1}$ is isomorphic to the formal power series ring $k(\eta_1){\lps t\rps}$.

Suppose now that $\eta_1$ is vertical, and let $p$ be the prime of $\Z$ determined by $\eta_1$.  
Recall that if $A \to B$ is a local homomorphism between two local Noetherian
rings such that $B$ is  complete and flat over $A$ and $B/m_B$ is a separable
extension of $A/m_A$, then $B$ is called a Cohen algebra over $A$.
By \cite[Chap. $0_{IV}$, 19.7.2]{EGAIV}, $B$ is determined by its residue field $B/m_B$ if $m_B = Bm_A$.
We have assumed $Y$ is 
flat over $\mathbb{Z}$.  Hence if $pB = m_B$ then  $B = \hat{\O}_{Y,\eta_1}$ is the Cohen algebra over $A = \mathbb{Z}_p$
associated to $k(\eta_1)$.
The statement that $pB = m_B$ is equivalent to the statement that $\overline{\eta}_1$  has multiplicity $1$ in  the fiber of $Y$ over $p$.

\subsubsection{Some Parshin chains of length $2$.}
We suppose in this subsection that $n =1$ and that $Y$ is regular, quasi-projective and flat over $\mathbb{Z}$.  
As a result, all the local rings of $Y$ are excellent. Let $K_n  = \{(\eta(0),\eta(1))\}$
consist of a Parshin chain of length $2$.  
If $\eta(0)$ is the generic point $\eta_0$ of 
$Y$, then $\eta(1)$ may be a point $\eta_i$ of arbitrary codimension $i \ge 1$.  
The functor $C_{\eta_0}$ is the identity functor, so (\ref{eq:localcodes}) shows
$$
\hat{\O}_{Y,\eta(0) \eta(1)} = \hat{O}_{Y,\eta_0 \eta_i} =  (K(Y) - \{0\})^{-1} \hat{\O}_{Y,\eta_i}
$$
where $K(Y)$ is the function field of $Y$.   

The other case in which $n =1$ which will be relevant to us is when,
in addition to the above assumptions, $Y$ is of dimension $2$, $\eta(0)$ is a codimension $1$ point $\eta_1$ 
on $Y$ and $\eta(1)$ is a closed point $\eta_2$ on the closure of $\eta_1$.
The local ring $\O_{Y,\eta_2}$ and its completion $\hat{\O}_{Y,\eta_2}$
are then two-dimensional UFD's.  A local equation $\pi_1 \in \O_{Y,\eta_2}$ for $\eta_1$ factors
in $\hat{\O}_{Y,\eta_2}$ into the product
$\pi_1 = u \prod_{\alpha = 1}^m t_\alpha^{b_\alpha}$
of a unit $u \in \hat{\O}_{Y,\eta_2}^\times$ together with positive integral powers of non-associate irreducibles $t_\alpha \in \hat{\O}_{Y,\eta_2}$. 
These $t_\alpha$ define the analytic branches at $\eta_2$ of the closure of $\eta_1$.
Notice that since $\O_{Y,\eta_2}/(\pi_1)$ is a reduced excellent local ring, the same is true for its completion
which can be identified with $\what{\O}_{Y, \eta_2}/(\pi_1)=\what{\O}_{Y, \eta_2}/(\prod_{\alpha = 1}^m t_\alpha^{b_\alpha})$. This implies that $b_\alpha=1$,  for all $\alpha$, and we have
\begin{equation}
\label{eq:pilocal}
\pi_1 = u \prod_{\alpha = 1}^m t_\alpha .
\end{equation}

Let $B_\alpha$ be the discrete valuation ring which is the
completion of the localization of $\hat{\O}_{Y,\eta_2}$ at the codimension one prime ideal generated by $t_\alpha$.
Let $p > 0$ be the residue characteristic of $\eta_2$.  The residue ring $R_\alpha =  \hat{\O}_{Y,\eta_2}/(t_\alpha)$ is a complete local integral domain of dimension $1$ with finite residue field $k(\eta_2)$ of characteristic $p$.  The fraction field of $R_\alpha$
is the residue field $k(B_\alpha)$ of $B_\alpha$.  We will also use the notation $k(\eta_{1,\alpha})$ for $k(B_\alpha)$
in order to emphasize its dependence on  $\eta_1$.  
The integral closure $R'_\alpha$ of $R_\alpha $ in $k(\eta_{1,\alpha})$
is finite over $R_\alpha$,  Hence $tR'_\alpha \subset R_\alpha$ for some $0 \ne t \in R_\alpha$, so since $R_\alpha / R_\alpha t$ is a
finite ring, a power of 
the radical of $R'_\alpha$ lies in $R_\alpha$.  Thus $R'_\alpha$
is complete with respect to the powers of its radical because $R_\alpha$ is complete.  It follows that $R'_\alpha$ is local 
because it is an integral domain.  We conclude that $R'_\alpha$ is a complete discrete valuation ring with
finite residue field.  Thus $k(B_\alpha) = k(\eta_{1,\alpha})$ is a local field of dimension $1$ with finite
residue field.  We distinguish two cases:
\begin{itemize}
\item{\bf $\eta_1$ is horizontal:}
 Then  
$k(\eta_{1,\alpha})$ is isomorphic to a finite extension of $\mathbb{Q}_p$. By Hensel's Lemma,  $B_\alpha$ is isomorphic   
 to the formal power
series ring $\mathbb{Q}_p(\eta_{1,\alpha})\lps t_\alpha\rps$. 
\smallskip
 
\item {\bf $\eta_1$ is vertical:} 
Then 
$k(\eta_{1,\alpha})$ is the completion of a global function field at a closed point.  Let $t$
be an element of $B_\alpha$ which has image equal to a uniformizer $\overline{t}$ in the discretely valued field $k(\eta_{1,\alpha})$.
Then $k(\eta_{1,\alpha})$ is isomorphic to the Laurent formal power series field $k(\eta_2)\llps\overline{t}\lrps$.  
\end{itemize}

\begin{lemma}
\label{lem:cohen}  
Suppose $\eta_1$ is vertical.    The maximal ideal $B_\alpha t_\alpha$ of $B_\alpha$ equals 
$B_\alpha p$ if and only if   $\eta_1$ occurs with multiplicity $1$ in the fiber of $Y$ over $p$. In this case, $B_\alpha$ is isomorphic to the Cohen ring over $\mathbb{Z}_p$ having residue field 
$k(\eta_2)\llps\overline{t}\lrps$.  This is true, in particular, if the fiber of $Y$ over $p$ is smooth.
\end{lemma}

\begin{proof} In $\O_{Y,\eta_2}$ one has a factorization 
\begin{equation}
\label{eq:pfactor}
p = v\cdot \prod_{i = 1}^j \pi_i^{a_i}
\end{equation}
in which $v \in \O_{Y,\eta_2}^\times$ is a unit, $j \ge 1$, $\pi_1$ is our chosen local equation for $\overline{\eta_1}$
and the $\pi_i$ are non-associate irreducibles.  The multiplicity of $\eta_1$ in the fiber of $Y$ is $1$ if
and only if $p$ is a uniformizer in the local ring $\O_{Y,\eta_1} = (\O_{Y,\eta_2} - \O_{Y,\eta_2}\pi_1)^{-1} \O_{Y,\eta_2}$.
This is the case if and only if $a_1 = 1$.  If $2 \le i \le j$ then $\O_{Y,\eta_2}/(\pi_1,\pi_i)$ is a finite discrete
quotient of $\O_{Y,\eta_2}$, so $(\pi_1,\pi_i)$ contains a positive power of the maximal ideal of $\O_{Y,\eta_2}$.
Hence $\hat{\O}_{Y,\eta_2}/\hat{\O}_{Y,\eta_2}(\pi_1,\pi_i)$ is finite, so (\ref{eq:pilocal}) implies $\pi_i$
has valuation $0$ in $B_\alpha$ when $2 \le i \le j$ because $\hat{\O}_{Y,\eta_2}/\hat{\O}_{Y,\eta_2}t_\alpha$ is infinite.  Thus 
(\ref{eq:pilocal}) and (\ref{eq:pfactor}) show that $p$ has valuation $a_1 $ with respect to the
discrete valuation of $\hat{\O}_{Y,\eta_2}$ associated to $\alpha$.  Thus $B_\alpha t_\alpha = B_\alpha p$
if and only if $a_1   = 1$, and this proves the first assertion in Lemma \ref{lem:cohen}.  The second
is a consequence of the results about Cohen rings cited in \S \ref{s:parsh0}.  The last assertion is clear
from the first. 
\end{proof}
\smallskip

We can make the isomorphism in  Lemma \ref{lem:cohen} more explicit in the following way.  Define $W(k(\eta_2))$ to be the ring of infinite Witt vectors over the finite residue
field $k(\eta_2)$.  Recall $W(k(\eta_2)){\ldb t\rdb}$ is the ring of all doubly infinite formal power series
$\sum_{n = -\infty}^\infty a_n t^n$
in which $a_n \in W(k(\eta_2))$ and $\lim_{n \to -\infty} a_n = 0$ in the $p$-adic topology on $W(k(\eta_2))$.
Viewing $k(\eta_2)$ as a finite subfield of the residue field $k(\eta_2)\llps\overline{t}\lrps$ of $B_\alpha$,
we can take Teichmuller lifts of elements of $k(\eta_2)$ to $B_\alpha$ via the usual limit process.  This produces a canonical
algebra embedding of $W(k(\eta_2))$ into $B_\alpha$.  There is then a  unique topological ring isomorphism from $W(k(\eta_2)){\ldb t\rdb}$ to $B_\alpha$ which extends this embedding and sends $t$ to itself as an element of $B_\alpha$.

We now return to the more general case in which we assume only that $n = 2$, 
$Y$ is  regular, quasi-projective and flat over $\mathbb{Z}$ of dimension $2$, $\eta(0)$ is a codimension $1$ point $\eta_1$ 
on $Y$ and $\eta(1)$ is a closed point $\eta_2$ on the closure of $\eta_1$.

\begin{lemma}
\label{lem:ex12}
With the above notations, the ring homomorphism $\mu$
 \begin{equation}
 \label{eq:urpy}
 \hat{\O}_{Y,\eta(0) \eta(1)} =   \hat{\O}_{Y,\eta_1 \eta_2} = C_{\eta_1} S_{\eta_1}^{-1} \hat{\O}_{Y,\eta_2} \quad \arrow{\mu}{} \quad \prod_{\alpha = 1}^m B_\alpha  
 \end{equation}
 resulting from (\ref{eq:localcodes})  is an isomorphism.  If $\eta_1$ is horizontal we have
\begin{equation}
\label{eq:horiznice}
\prod_{\alpha = 1}^m B_\alpha  \cong 
\bigoplus _{\alpha =1}^m \mathbf{Q}_{p}( \eta _{1\alpha })\lps t_{\alpha }\rps. 
 \end{equation} 
 Suppose $\eta_1$ is vertical and has multiplicity one in the fiber of $Y$ over $p$. Then  
 \begin{equation}
 \label{eq:vertnice}
 \prod_{\alpha = 1}^m B_\alpha \cong \bigoplus _{\alpha =1}^m W(k(\eta_2))\ldb t_\alpha\rdb.
 \end{equation}
 \end{lemma}
 
\begin{proof} Statements (\ref{eq:horiznice}) and (\ref{eq:vertnice}) follow from (\ref{eq:urpy}) and the above
computations of the $B_\alpha$.  To show (\ref{eq:urpy}) it will suffice to prove the following.   
 Fix $\alpha$, and let $\tau_\alpha$ and $ \beta_\alpha$ be elements of $\hat{\O}_{Y,\eta_2}$ such 
that $\beta_\alpha \not \in B_\alpha t_\alpha$. Then $\tau_\alpha/\beta_\alpha$ defines an element $\overline{\tau_\alpha/\beta_\alpha}$ of the residue field $k(B_\alpha)$ of $B_\alpha$.  
It will suffice to show that there is an element of the image of $\mu$ whose component at $B_\alpha$  has  image $\overline{\tau_\alpha/\beta_\alpha}$ 
in $k(B_\alpha)$   and whose component at $B_k$ for $k \ne \alpha$ is a non-unit.  The element 
$$
z = \prod_{k \ne \alpha, k = 1}^m t_k
$$
of $\hat{\O}_{Y,\eta_2}$ has non-zero image $\overline{z}$ in the one-dimensional local ring $R_\alpha = \hat{\O}_{Y,\eta_2}/t_\alpha \hat{\O}_{Y,\eta_2}$.
Since the image of the  ring $\O_{Y,\eta_2}$ in the completion $\hat{\O}_{Y,\eta_2}$ is dense, there is an element $w \in
\O_{Y,\eta_2}$ such that  $w$ and $\beta_\alpha z$ generate the same ideal in $R_\alpha$.  Thus there is an element $u \in \hat{\O}_{Y,\eta_2}$
whose image in $R_\alpha$ is a unit such that $wu$ and $\beta_\alpha z$ have the same image in $R_\alpha$.  This $u$ must be a unit of 
$\hat{\O}_{Y,\eta_2}$.  Now $w$ has non-zero image in $R_\alpha$, so
 $w$ must be an element of $\O_{Y,\eta_1}$ which is not
in the maximal ideal of $\O_{Y,\eta_1}$. Thus $w^{-1} u^{-1} \tau_\alpha  z$ lies in $S_{\eta_1}^{-1}  \hat{\O}_{Y,\eta_2}$
and has image $\overline{\tau_\alpha/\beta_\alpha}$ in $k(B_\alpha) = \mathrm{Frac}(R_\alpha)$.  Now $ w^{-1} u^{-1} \tau_\alpha z \in  \hat{\O}_{Y,\eta_1 \eta_2}$ because of the second equality in (\ref{eq:urpy}), so
we have constructed the desired element. It follows that  (\ref{eq:urpy}) is an isomorphism.
\end{proof}

\subsubsection{Some Parshin chains of length $3$.}  \label{s:parsh1}
 
The last special case we will discuss is when $Y$ is regular and integral of dimension $n = 2$.
Let $K_2 = \{(\eta(0),\eta(1),\eta(2))\} = \{(\eta_0,\eta_1,\eta_2)\}$ with $\eta(0) = \eta_0$ the generic point of $Y$, $\eta(1) = \eta_1$
 a codimension $1$ point and $\eta(2) = \eta_2$ a closed point on the closure of $\eta_1$.
 We find from (\ref{eq:localcodes}) that
 \begin{equation}
 \label{eq:sounds}
 \hat{\O}_{Y,\eta(0) \eta(1)\eta(2) } =  \hat{\O}_{Y,\eta_0 \eta_1 \eta_2 }  = (K(Y) - \{0\})^{-1} \hat{\O}_{Y,\eta_1\eta_2}
 \end{equation}
 where $\hat{\O}_{Y,\eta_1\eta_2}$ is a product of discrete valuation rings $B_\alpha$ of
 the kind described above for the pair $(\eta_1,\eta_2)$. Since a uniformizer in $B_\alpha$ divides
 the image in $B_\alpha$ of an element of $K(Y)$, we find that 
 that   $\hat{\O}_{Y,\eta_0 \eta_1 \eta_2 }$ is the product of the fraction fields
 of the $B_\alpha$.  
 
 \subsubsection{Base extensions}
 
 In this section we suppose that $h: X \to Y$ is a finite flat morphism of regular projective connected flat schemes over $\mathbb{Z}$ of
 dimension $2$.  Let $[X:Y]$ be the degree of $h$.  Then $h$ induces a map of simplicial sets $h:S(X) \to S(Y)$. The following
 result will be used in  \S \ref{s:redP1}.  
 
 \begin{proposition}
 \label{prop:finitefree} For all Parshin chains $\eta$ in $S(Y)$, the homomorphism $\O_Y \to h_* \O_X$ of sheaves of rings gives an isomorphism
 \begin{equation}
 \label{eq:blap}
 \O_X \otimes_{\O_Y} \hat{\O}_{Y,\eta} \to \oplus_{\eta' \in h^{-1}(\eta)} \hat{\O}_{X,\eta'}
 \end{equation}
 of free $\hat{\O}_{Y,\eta}$-modules of rank $[X:Y]$.
 \end{proposition}
 
 \begin{proof}  Suppose first that $\eta = \eta(0)$ consists of a single point of $Y$.
 Then $\hat{\O}_{Y,\eta}$ is just the completion $\hat{\O}_{Y,\eta(0)}$ of $Y$ at $\eta(0)$,
 so (\ref{eq:blap}) is clear from the fact that $h: X \to Y$ is finite and flat of degree $[X:Y]$.
 
 Suppose next that (\ref{eq:blap}) holds for some $\eta = (\eta(0),\ldots,\eta(n))$ and that $\eta(0)$ is 
 not the generic point $\eta_Y$ of $Y$.  We now show (\ref{eq:blap}) holds when $\eta$ is replaced by
 $\eta^* = (\eta_Y,\eta(0),\ldots,\eta(n))$, where $h^{-1}(\eta^*) = \{(\eta_X,\eta'): \eta' \in h^{-1}(\eta)\}$.  By Proposition \ref{prop:multi}, 
 $$\hat{\O}_{Y,\eta^*} = K(Y) \otimes_{\O_Y} \hat{\O}_{Y,\eta}$$ when $K(Y) = \O_{Y,\eta_Y}$ is the function field of $Y$.
 Thus since we assumed (\ref{eq:blap})  is an isomorphism, 
  \begin{eqnarray}
 \O_X \otimes_{\O_Y} \hat{\O}_{Y,\eta^*} &=& K(Y) \otimes_{\O_Y} (\O_X \otimes_{\O_Y} \hat{\O}_{Y,\eta})\nonumber\\
 & =& \oplus_{\eta' \in h^{-1}(\eta)} \left ( K(X) \otimes_{\O_X} \hat{\O}_{X,\eta'}\right )\\
&=&\oplus_{\eta^{*'} \in h^{-1}(\eta^*)}  \hat{\O}_{X,\eta{*'}}\nonumber
 \end{eqnarray}
which proves (\ref{eq:blap}) for $\eta^*$.   

To complete the proof it will now be enough to consider the case in
 which $\eta = (\eta_1,\eta_2)$ for some codimension $i$ points $\eta_i$
 such that $\eta_2$ lies on the closure of $\eta_1$.  By (\ref{eq:urpy}), we have an isomorphism
 \begin{equation}
 \label{eq:urpy2}
 \hat{\O}_{Y,\eta} =    \prod_{\alpha} B_\alpha  
 \end{equation}
 where $\alpha$ runs over the irreducible factors in $\hat{\O}_{Y,\eta_2}$ of a  local equation
 $\pi_1 $ for $\eta_1$ in $\O_{Y,\eta_2}$ (as in \ref{eq:pilocal}), and $B_\alpha$ is the dvr which is the completion of the local ring of $\hat{\O}_{Y,\eta}$ at the valuation associated to $\alpha$.  

We obtain the Parshin chains $\eta' = (\eta'_1,\eta'_2)$ in $h^{-1}(\eta)$ by first taking
the points $\eta'_2 \in h^{-1}(\eta_2)$ and by then considering the factorization
of $\pi_1$ in the local ring $\O_{X,\eta'_2} \supset \O_{Y,\eta_2}$ in order to find
the $\eta'_1$ lying over $\eta_1$ which contain $\eta'_2$ in their closure.  Here 
 \begin{equation}
 \label{eq:summer}
 \O_X \otimes_{\O_Y} \hat{\O}_{Y,\eta_2} = \oplus_{\eta'_2 \in h^{-1}(\eta_2)} \hat{\O}_{X,\eta'_2}.
 \end{equation}
 
 For each irreducible factor $\alpha$ of $\pi_1$ in $\hat{\O}_{Y,\eta_2}$ we consider
 the factorization of $\alpha$ into a product of irreducibles in $\hat{\O}_{X,\eta'_2}$ 
for $\eta'_2 \in h^{-1}(\eta_2)$.  These irreducibles give via (\ref{eq:urpy}) with $Y$
replaced by $X$ the dvr summands of each ring $\hat{\O}_{X,\eta'}$ as $\eta' = (\eta'_1,\eta'_2)$
runs over the elements of $h^{-1}(\eta)$.  We see from this that the natural ring homomorphism
$$
\O_X \otimes_{\O_Y} \hat{\O}_{Y,\eta} = \O_X \otimes_{\O_Y} ( \prod_{\alpha} B_\alpha ) \to \oplus_{\eta' \in h^{-1}(\eta)} \hat{\O}_{X,\eta'}
$$
is the direct sum over $\alpha$ of the homomorphisms
\begin{equation}
\label{eq:valext}
\mu_\alpha: \O_X \otimes_{\O_Y} B_\alpha \to \oplus_{\alpha'} B'_{\alpha'}
\end{equation}
where $\alpha'$ ranges over the irreducible factors of $\alpha$ in $\hat{\O}_{X,\eta'_2}$ as $\eta'_2$
ranges over the elements of $h^{-1}(\eta_2)$, and where $B'_{\alpha'}$ is the completion of
$\hat{\O}_{X,\eta'_2}$ with respect to the discrete valuation associated to $\alpha'$.  To complete the proof of  
Proposition \ref{prop:finitefree}
it will suffice to show that (\ref{eq:valext}) is an isomorphism.

From (\ref{eq:summer}) and the fact that $X$ is flat and finite over $Y$ we conclude that
the sum
\begin{equation}
\label{eq:sumfrac}
\oplus_{\eta'_2 \in h^{-1}(\eta_2)} \mathrm{Frac}(\hat{\O}_{X,\eta'_2})
\end{equation}
of the fraction fields of the summands on the right side of (\ref{eq:summer}) is
an \'etale algebra of dimension $[X:Y]$ over the characteristic $0$ field
$\mathrm{Frac}(\hat{\O}_{Y,\eta_2})$. The ring $B_\alpha$ is the completion of
 the discrete valuation ring of $\mathrm{Frac}(\hat{\O}_{Y,\eta_2})$ associated to $\alpha$.
The rings $B'_{\alpha'}$ on the right side of (\ref{eq:valext}) are the completions
of the discrete valuation rings of the summands in (\ref{eq:sumfrac}) at extensions
of the valuation associated to $\alpha$. Thus by the theory of discrete valuations
in finite separable extensions of fields, we see that  the right hand side of (\ref{eq:valext})
is a free $B_\alpha$-module of rank $[X:Y]$.  The left hand side of   (\ref{eq:valext}) is
a $B_\alpha$-module which is generated by less than or equal to $[X:Y]$ elements since
$\O_X$ is a locally free $\O_Y$-module of rank $[X:Y]$. Thus to show that (\ref{eq:valext})
is an isomorphism 
it will suffice to show that if $w$ is an element of the residue field $k(\alpha')$ of a summand
$B'_{\alpha'}$ on the right side of (\ref{eq:valext}), then there is an element of 
the image of $\mu_{\alpha}$ whose component at $\alpha'$ is congruent to $w$
mod the maximal ideal of $B'_{\alpha'}$ and whose component in any other
summand $B'_{\alpha''}$ appearing on the right in (\ref{eq:valext}) is in the maximal ideal of $B'_{\alpha''}$.

We know that there is a closed point $\eta'_2$ lying over $\eta_2$ such that
${\alpha'}$ is an irreducible factor of $\pi_1$ in $\hat{\O}_{X,\eta'_2}$ and $B'_{\alpha'}$
is the completion of the localization of $\hat{\O}_{X,\eta'_2}$ at the discrete
valuation associated to $\alpha'$. Thus there are elements $t$, $s \in \hat{\O}_{X,\eta'_2}$
such that $s \not \in \hat{\O}_{X,\eta'_2} \cdot {\alpha'}$ and $w \equiv t/s$ in
$k(\alpha')$.  By multiplying both $t$ and $s$ by the product of a set of representatives
for the irreducible factors $\alpha''$ of $\pi_1$ in $\hat\O_{X, \eta_2'}$ which are 
not associate to $\alpha'$, we may assume that $t$ has image in the maximal ideal of $B'_{\alpha''}$
for all such $\alpha''$.

The factorization of $s$ in the UFD $\hat{\O}_{X,\eta'_2}$ 
does not involve ${\alpha'}$, but it might involve some other irreducibles ${\alpha''}$
which are irreducible factors of $\pi_1$.   However, if ${\alpha''}$ is such an irreducible,
then ${\alpha''} + {\alpha'}$ is congruent to ${\alpha''}\mod \hat{\O}_{X,\eta'_2} \cdot {\alpha'}$
but not congruent to $0\mod \hat{\O}_{X,\eta'_2} \cdot {\alpha''}$.  We can therefore
replace each appearance of an irreducible of the form ${\alpha''}$ in the factorization
of $s$ by ${\alpha''} + {\alpha'}$ so as to be able to assume that $s \not \in \hat{\O}_{X,\eta'_2} \cdot {\alpha''}$
for all irreducible factors ${\alpha''}$ of $\pi_1 $ in $\hat{\O}_{X,\eta'_2}$ (including $\alpha'' = \alpha'$).
Since these $\alpha''$ define all the discrete valuations of $\hat{\O}_{X,\eta'_2}$ which
lie over the discrete valuation of $\hat{\O}_{Y,\eta_2}$ associated to $\alpha$, we
conclude that
$$
g = \mathrm{Norm}_{\hat{\O}_{X,\eta'_2}/\hat{\O}_{Y,\eta_2}}(s)
$$
is an element of $\hat{\O}_{Y,\eta_2}$ which does not lie in $\hat{\O}_{Y,\eta_2} \cdot t_{\alpha}$.
When we view $g$ as an element of $\hat{\O}_{X,\eta'_2}$, it does not lie in $\hat{\O}_{X,\eta'_2}\cdot {\alpha'}$
and it is a multiple of $s$.  Thus $w \equiv t/s \equiv t'/g$ in $k(\alpha')$
where $t' = t (g/s) \in \hat{\O}_{X,\eta'_2}$.  In view of the isomorphism (\ref{eq:summer}), 
we can now find an element $q$ of $\O_X \otimes_{\O_Y} \hat{\O}_{Y,\eta_2} $ whose
image in $\hat{\O}_{X,\eta'_2}$ 
is equal to $t'$ and whose components in $\hat{\O}_{X,\eta''_2}$ is $0$ if $\eta''_2 \ne \eta'_2$ lies over $\eta_2$.  It
follows that the image of $q/g$ under the map $\mu_\alpha$ in (\ref{eq:valext})
has the prescribed image $w$ in the residue field $k(\alpha')$ of the summand
corresponding to $\alpha'$ and image in the maximal ideal 
 in all the other summands.  This
completes the proof. 
 \end{proof}

\subsection{Adeles and cosimplicial structure}
\label{s:adco}

The construction of the adeles associated to the structure sheaf $\O_{Y}$ does not play a major role in this paper.  However, 
we include this
subsection since it will pave the way for the crucial construction of the $\rK_{2}$-adeles associated to $Y$.

Recall that 
$S(Y)$ is the simplicial set associated
to the set $P(Y)$ of all point of $Y$ and the order relation defined by $\eta \ge \eta'$ if $\eta'$ is a point on the closure  $\overline{\eta}$ of $\eta$.  
The $n$-simplex $S(Y)_n$ is the set of all Parshin $n+1$-tuples
$(\eta(0),\ldots,\eta(n))$ of points on $Y$, these being $n+1$-tuples such that  $\eta(0) \ge \eta(1) \ge \cdots \ge \eta(n)$. 
We define the $n$-dimensional adele group of $Y$ to be 
$$
\mathbb{A}_{Y}^{\prime }( n) = A(S(Y)_n,\O_Y)
$$
in the notation of Definition \ref{def:adeldef}.  From this definition we see that there is a natural inclusion
\begin{equation}
\label{eq:includeadel}
\mathbb{A}_{Y}^{\prime }( n) \to \mathbb{A}_{Y}( n) =\prod \what{\mathcal{O}}_{Y,\eta _{I}}
\end{equation}
where the product extends over all Parshin $n+1$-tuples $\eta_I:=\{\eta
_{i_{0}},\ldots,\eta _{i_{n}}\}$ on $Y$.

Suppose  $I = (i_0,\ldots, i_n )$  is an ordered subset of $J=( j_{0},\ldots, j_{m})$.
From Proposition \ref{prop:multi} we have a natural map 
$$
\tau _{I}^{J}:
\what{\mathcal{O}}_{Y,\eta _{I}}\rightarrow \what{\mathcal{O}}_{Y,\eta
_{J}}.
$$ The maps $\tau _{I}^{J}$ may be used to endow the various
multicompletions of $Y$ with a cosimplicial structure.

If we now specify that $I= ( j_{0},\ldots,\widehat{j}_{i_{k}},\ldots, j_{m})$
(so that $n+1=m$), then we define the coboundary map
\begin{equation*}
\mathbb{A}_{Y}(m-1) \overset{\partial _{m-1}}{\longrightarrow }%
\mathbb{A}_{Y}( m)
\end{equation*}%
by stipulating that for $a\in \what{\mathcal{O}}_{Y,\eta _{I}}$, $\partial _{m-1}(a)_J = ( -1 )^{k} \tau_{I}^{J}(a)$. This then
gives us a complex  
\begin{equation*}
\mathbb{A}_{Y}^{\bullet }:\quad \mathbb{A}_{Y}( 0) \overset{%
\partial _{0}}{\longrightarrow }\mathbb{A}_{Y}( 1) \overset{\partial
_{1}}{\longrightarrow } \cdots \overset{\partial
_{d- 1}}{\longrightarrow } \mathbb{A}_{Y}( d) 
\end{equation*}%
when $d = \mathrm{dim}(Y)$.  
There are degeneracy maps induced by mapping the Parshin cycle $( \eta
_{i_{0}},\ldots,\eta_{i_{m-1}})$ of length $m$ to the Parshin cycle
$( \eta _{i_{0}},\ldots,\eta _{i_{k}},\eta _{i_{k}},\ldots,\eta _{i_{m-1}})$
 of length $m+1$. By \cite[\S 2]{Huber}, the inclusion (\ref{eq:includeadel}) gives a complex 
\begin{equation}
\label{eq:realadele}
\mathbb{A}_{Y}^{\prime \bullet }:\quad \mathbb{A}_{Y}^{\prime} ( 0) \overset{%
\partial _{0}}{\longrightarrow }\mathbb{A}_{Y}^{\prime}( 1) \overset{\partial
_{1}}{\longrightarrow } \cdots \overset{\partial
_{d- 1}}{\longrightarrow } \mathbb{A}_{Y}^{\prime}( d) 
\end{equation}
which we will call the adelic complex of $Y$ and which has degeneracy maps
defined in the above way.   

When there is no confusion, we will
also use the symbols $\mathbb{A}_{Y}^{ \bullet }$ and $\mathbb{A}_{Y}^{\prime \bullet }$
to denote the complexes defined as above but using only non-degenerate Parshin  cycles $\left( \eta _{i_{0}},\ldots ,\eta
_{i_{m-1}}\right)$ in which the $\eta _{i_{j}}$ are all distinct. 
Omitting such degenerate cycles does not effect the cohomology of the
complexes we consider -- see the remark after (1) on page 179 of \cite{ParshinCrelle}.

\subsubsection{}
We conclude this section by considering the case in which $Y$ is
a regular integral scheme of dimension $2$.  We will recall from \cite{ParshinCrelle}
the local conditions on elements $\mathbb{A}_{Y}(2)$ 
which are necessary and sufficient for these elements to lie in $\mathbb{A}_{Y}'(2)$.
This motivates the definition of $\rK_2$-adeles to be given in the next section.

Recall that $ K(Y)$ denotes the function field of $%
Y$, so that $K(Y)$ may be identified with the two sheaves of rings on a point $%
\mathcal{O}_{Y,\eta _{0}}$ and $\hat{\mathcal{O}}_{Y,\eta _{0}}$. We
start by considering a Parshin triple $( \eta _{0},\eta _{1},\eta
_{2})$ on $Y$ and we recall that $\hat{\mathcal{O}}_{Y,\eta
_{0}\eta _{2}}=K(Y)\cdot\hat{\mathcal{O}}_{Y,\eta _{2}}$. We write $\hat{\mathcal{O}}_{Y,\eta _{2}}[\eta_{1}^{-1}]$
 for the subring of  elements in $\hat{\mathcal{O}}_{Y,\eta_{0}\eta _{2}}$
 which are regular off the curve $\overline{\eta}_{1}$ and denote by $v_{\eta_1}$ the valuation of $K(Y)$
 that corresponds to $\overline{\eta}_1$.

We let $v_{\eta _{1}\eta_2}$ denote a discrete valuation on $\hat{\mathcal{O}}_{Y,\eta _{1}\eta _{2}}$  
corresponding to one of the components (branches) as in Lemma \ref{lem:ex12} and let $\mathfrak{p}_{\eta _{1}\eta_2}$ denote the corresponding prime ideal of $\hat{\mathcal{O}}_{Y,\eta _{1}\eta _{2}}$. 
We then identify $\mathbb{A}_{Y}^{\prime }( 2) $
with the restricted direct product 
\begin{equation*}
\mathbb{A}_{Y}^{\prime }( 2) =\sideset{}{^\prime}\prod_{(\eta_0,\eta_1,\eta_2)} \hat{\mathcal{O}}_{Y,\eta _{0}\eta _{1}\eta _{2}}
 \subset 
\mathbb{A}_{Y}(2) = \prod_{(\eta_0,\eta_1,\eta_2)}  \hat{\mathcal{O}}_{Y,\eta _{0}\eta _{1}\eta _{2}}
\end{equation*}
consisting of all elements of  $\mathbb{A}_{Y}(2)$ whose terms   $( f_{\eta_0\eta _{1}\eta _{2}})$ with 
$$
f_{\eta_0\eta _{1}\eta
_{2}}\in \hat{\mathcal{O}}_{Y,\eta _{0}\eta _{1}\eta _{2}} = \mathrm{Frac}(\hat{\mathcal{O}}_{Y,\eta
_{1}\eta _{2}})
$$
 satisfy the following two properties (cf. page 179
in \cite{ParshinCrelle}):
\begin{enumerate}
\item[P1.] ({Adelic Property 1})
There exists a divisor $D$ on $Y$ such that for each codimension one
point $\eta _{1}$ on $Y$, each $\eta_2\leq\eta_1$ (and each branch of $\bar\eta_1$ at $\eta_2$)
we have
\begin{equation*} 
v_{\eta _{1}\eta_2}  (f_{\eta_0\eta _{1}\eta _{2}} ) \geq v_{\eta _{1}} (
D ); 
\end{equation*}
\item[P2.] ({Adelic Property 2})
Suppose that $\eta_1$ is a codimension one point on $Y$. Then for any positive integer $k$, for all but a finite number of $\eta_{2}$ on  $\overline{\eta}_{1}$, we have
\begin{equation*}
f_{\eta_0\eta _{1}\eta _{2}}\in \hat{\mathcal{O}}_{Y,\eta _{2}}[\eta_{1}^{-1}]+
\mathfrak{p}_{\eta _{1}\eta_2}^{k}.   
\end{equation*}
\end{enumerate}

\bigskip
\bigskip

\section{Equivariant adelic Chow groups}\label{equivchow}

\setcounter{equation}{0}

\subsection{Generalities.}\label{equivchow-gen}

For $\ell\geq 0$ and for a ring $S$ we let $\Kr_{\ell}(S)$
denote the $\ell$-th $\rK$-group of the ring $S$. For a two-sided ideal ${\mathcal A}$
of $S$ we set $\overline{S}=S/\mathcal A$ for the quotient 
ring. Recall from  \cite[Theorem 6.2]{MilnorK} that we have the long
exact sequence 
\begin{equation}
\label{eq:Kseq}
\begin{split}
\rK_{2} ( S,{\mathcal A} )  \rightarrow \rK_{2} ( S ) &\rightarrow
\rK_{2} ( \overline{S} ) \rightarrow \rK_{1} ( S, {\mathcal A} )
\rightarrow \rK_{1}(S)\rightarrow   \\
&\rightarrow \rK_{1} ( \overline{S}) \rightarrow \rK_{0}(
S, {\mathcal A} ) \rightarrow \rK_{0} ( S ) \rightarrow \rK_{0} (
\overline{S} ).
\end{split}
\end{equation}%
Recall that $\rK_{1}( S, {\mathcal A}) $ may be described as the quotient
group
\begin{equation}
\rK_{1} ( S, {\mathcal A} ) =\frac{{\rm GL} ( S, {\mathcal A}) }{\rE ( S, {\mathcal A}) }
\end{equation}%
where ${\rm GL}( S, {\mathcal A})$ is the subgroup of elements in the full
general linear group ${\rm GL} ( S )$ which are congruent to the
identity $\mod\, {\mathcal A}$ and $\rE ( S, {\mathcal A} )$ is the smallest normal subgroup
of ${\rm GL}  (S)$ containing the elementary matrices $e_{ij} (
a )$ for all $a\in {\mathcal A}$. (See for instance page 93 in \cite{RosenbergBook}.)

\subsection{$\rK_{\ell}$-adeles of arithmetic surfaces}
\label{s:Kadeles}

We suppose in this section that $Y$ is an irreducible  regular flat projective scheme over $\Z$ and that $\ell$ is either $1$ or $2$. 
We now make the following important definitions (cf. Definition 10 on page 719 of \cite{OsipovAdelic}):

\begin{definition}\label{defnonrestrictedproducts}
\begin{enumerate}
\item[a)]   We define \begin{equation*}
 {\rK}_{\ell}( \mathbb{A}_{Y,012}[ G] )
=\prod_{\eta_0\eta_1\eta_2} {\rK}_{\ell} ( \hat{\mathcal{O}}_{Y,\eta _{0}\eta
_{1}\eta _{2}}[G])
\end{equation*}%
where the product is over all non-degenerate Parshin triples.
\smallskip
 \item[b)] For $0 \le i < j \le 2$ we define  $ {\rK}_{\ell}( \mathbb{A}_{Y,ij}[G])  = 
\prod_{\eta _{i}\eta_j} {\rK}_{\ell} ( \hat{\mathcal{O}}%
_{Y,\eta _{i}\eta _{j}} [ G]  ) $
where the product is over all non-degenerate Parshin pairs consisting of a 
codimension $i$ point $\eta_i$ and a codimension $j$ point $\eta_j<\eta_i$.
\smallskip
 \item[b)] For $0 \le i  \le 2$ we define  $ {\rK}_{\ell}( \mathbb{A}_{Y,i}[G])  = 
\prod_{\eta _{i}} {\rK}_{\ell} ( \hat{\mathcal{O}}%
_{Y,\eta _{i}} [ G]  ) $
where the product is over all points $\eta_i$ of
codimension $i$.
\end{enumerate}
\end{definition}

\begin{definition}\label{defrestrictedproducts}
\begin{enumerate}
\item[a)]   We define $ {\rK}_{\ell}^{\prime
}( \mathbb{A}_{Y,012}[ G] ) $ to be the restricted
product
\begin{equation*}
 {\rK}_{\ell}^{\prime }( \mathbb{A}_{Y,012}[ G] )
=\prod{}'\, {\rK}_{\ell} ( \hat{\mathcal{O}}_{Y,\eta _{0}\eta
_{1}\eta _{2}}[G])
\end{equation*}%
consisting of elements $(\kappa_{\eta_0\eta _{1}\eta _{2}})$ as $(\eta_0,\eta_1,\eta_2)$
ranges over all non-degenerate Parshin triples for which  $\kappa _{\eta_0\eta _{1}\eta _{2}}\in {\rK}_{\ell} (\hat{\mathcal{O}}%
_{Y,\eta _{0}\eta _{1}\eta _{2}}[G])$ satisfies  the
following two properties:
\smallskip
 
\begin{enumerate}

\item[(PK1)] Almost all $\eta _{1}$ have the property that
$
\kappa _{\eta_0\eta _{1}\eta _{2}}\in  {\rK}_{\ell}  ( \hat{%
\mathcal{O}}_{Y,\eta _{1}\eta _{2}} [ G ]  )^\flat
$
for all $\eta _{2}<  {\eta }_{1}$, where $ {\rK}_{\ell}( \hat{\mathcal{O}}_{Y,\eta _{1}\eta _{2}} [G] )^{\flat }$ denotes the image of ${\rK}_{\ell}( \hat{%
\mathcal{O}}_{Y,\eta _{1}\eta _{2}} [ G ] )$ in ${\rK}_{\ell} ( \hat{\mathcal{O}}_{Y,\eta _{0}\eta _{1}\eta _{2}} [ G] )$.

\item[(PK2)]
Given $\eta _{1}$ and a positive integer $k$ then for all but a finite
number of closed points $\eta _{2}$ on $\overline{\eta }_{1}$
\begin{equation*}
\kappa _{\eta_0\eta _{1}\eta _{2}}\in  {\rK}_{\ell} ( \hat{%
\mathcal{O}}_{Y,\eta _{1}\eta _{2}}[G] ,\mathfrak{p}_{\eta
_{1}\eta_2}^{k} )^{\flat} \cdot {\rK}_{\ell} ( \hat{\mathcal{O}}%
_{Y,\eta _{2}} [ \eta _{1}^{-1} ]  [ G ]  )^{\flat }
\end{equation*}
where $\hat{\mathcal{O}}_{Y,\eta _{2}}[ \eta _{1}^{-1}]$
denotes the subring of elements in $\hat{\mathcal{O}}_{Y,\eta_{0}\eta_{2}}$ which are regular off the curve $\overline{\eta }%
_{1}$.
\end{enumerate}
(Note that these properties parallel the restricted direct product conditions (P1) and (P2) at the end of \S \ref{s:adco}.  )
 \smallskip
 \smallskip
 
\item[b1)] We define  ${\rK}_{\ell}^{\prime
}( \mathbb{A}_{Y,01}[G]) $ to be the subgroup of
elements $(\kappa _{\eta _{0}\eta _{1}})\in
\prod_{\eta _{1}}{\rK}_{\ell} ( \hat{\mathcal{O}}%
_{Y,\eta _{0}\eta _{1}} [ G]  ) $ with the property that $ 
\kappa _{\eta _{0}\eta _{1}}\in {\rK}_{\ell} ( \hat{\O}_{Y,\eta
_{1}} [ G ]  )^\flat$ for almost all $\eta _{1}$.

\item[b2)]  We define  ${\rK}_{\ell}^{\prime } (
\mathbb{A}_{Y,12} [ G ]  )  =\prod_{\eta _{2}}{\rK}_{\ell} ( \hat{%
\mathcal{O}}_{Y,\eta _{1}\eta _{2}} [ G])$, \emph{i.e.} we impose
no restriction.

\item[b3)] We define ${\rK}_{\ell}' ( \mathbb{A}_{Y,02} [ G
 ]  )$ to be the subgroup of  $\prod_{\eta _{2}}{\rK}_{\ell} ( \widehat{%
\mathcal{O}}_{Y,\eta _{0}\eta _{2}} [ G])$ consisting of $x=(x_{\eta_0\eta_2})_{\eta_2}$
 with the following property: There is a divisor $D\subset Y$ (that could depend on $x$)
 such that: For all $\eta_2$,  $x_{\eta_0\eta_2}$ is in $\rK_\ell(\hat \O_{Y,  \eta_2}[D^{-1}][G])^\flat$
 where $\hat\O_{Y,\eta_2}[D^{-1} ]$ is the subring of  $\hat{\mathcal{O}}_{Y,\eta _{0}\eta _{2}}$
consisting of elements which are regular off $D$.
\smallskip

\item[c)] We define $ {\rK}_\ell^{\prime}(\mathbb{A}_{Y,i} [ G ]  ) =
\prod_{\eta_i} {\rK}_\ell(\hat{\O}_{Y,\eta_i}[G])$, \emph{i.e.} we impose no restriction.
\end{enumerate}
\end{definition}

\begin{remark}\label{ko2}
{\rm The group $ {\rK}_{\ell}' ( \mathbb{A}_{Y,ij} [ G ]  )$ maps diagonally to  $
 \prod {\rK}_{\ell} ( \hat{\mathcal{O}}_{Y,\eta _{0}\eta
_{1}\eta _{2}}[G])$.
We can see that the image $
 {\rK}_{\ell}' ( \mathbb{A}_{Y,02} [ G ]  )^\flat$ is actually a subgroup of the restricted product
$ {\rK}_{\ell}' ( \mathbb{A}_{Y,012} [ G ]  )$. This is not necessarily true
for the images of $
 {\rK}_{\ell}' ( \mathbb{A}_{Y,01} [ G ]  )$ and $
 {\rK}_{\ell}' ( \mathbb{A}_{Y,12} [ G ]  )$.}
 \end{remark}

\subsection{The adelic Chow groups }
\label{def:ECg}

\begin{definition}
\label{def:ith}
For $\ell\in \{1,2\}$, the $\ell$-th equivariant adelic Chow group is defined to be
\begin{equation}\label{23ch1}
{\rm CH}^1_{\Bbb A}(Y[G])=\frac{ {\rK}_{1}^{\prime } ( \mathbb{A}%
_{Y,01} [ G ] )  }{  \prod_{0\leq i\leq 1}   {\rK}_{1} ( \mathbb{A}_{Y, i}[ G ]  )^\flat  },
\end{equation}
\begin{equation}\label{24ch2}
{\rm CH}^2_{\Bbb A}(Y[G])=\frac{ {\rK}_{2}^{\prime } ( \mathbb{A}%
_{Y,012} [ G ] )\cdot \prod_{0\leq i<j\leq 2}  {\rK}_2^\prime(\mathbb{A}_{Y, ij}[G])^\flat }{\prod_{0\leq i<j\leq 2} {\rK}%
_{2}^{\prime } ( \mathbb{A}_{Y,ij} [ G ]  )^\flat }
\end{equation}
where again, the superscript $\flat$ denotes the image 
of the corresponding group in the unrestricted product $\prod_{\eta_0\eta_1} {\rK}_{1} ( \hat{\mathcal{O}}_{Y,\eta _{0}\eta
_{1} }[G])$, resp. 
$\prod_{\eta_0\eta_1\eta_2} {\rK}_{2} ( \hat{\mathcal{O}}_{Y,\eta _{0}\eta
_{1}\eta _{2}}[G])$.
\end{definition}

 In (\ref{23ch1}), both numerator and denominator are
subgroups of $\prod_{\eta_0\eta_1}{\rK}_{1} ( \hat{\mathcal{O}}_{Y,\eta _{0}\eta
_{1} }[G])$. In (\ref{24ch2}),  they are both
subgroups of $\prod_{\eta_0\eta_1\eta_2} {\rK}_{2} ( \hat{\mathcal{O}}_{Y,\eta _{0}\eta
_{1}\eta _{2}}[G])$.

\begin{remark}  \label{remarkOsipov}
{\rm   In general, there are several   notions of  $\rK_2$-adeles
that appear in the literature, see for example \cite{MorrowSurvey}, \cite{KerzIdeles}. Here, we essentially follow constructions of Parshin and Osipov. If $G=\{1\}$ and $Y$ is a smooth algebraic surface over a field, we 
 have an isomorphism ${\rm CH}^1_{\Bbb A}(Y)\xrightarrow{\sim} {\rm Pic}(Y)$.
 Similarly, there is a natural isomorphism ${\rm CH}^2(Y)\xrightarrow{\sim }{\rm CH}^2_{\Bbb A}(Y) $,
 where ${\rm CH}^2(Y)$ denotes the classical Chow group of algebraic cycles of codimension $2$ on $Y$
 up to rational equivalence. This second isomorphism is obtained using the Gersten resolution
 by arguments as in \cite{OsipovAdelic}. Since we are not going to use this, we only sketch the proof:  For $G=\{1\}$, we can see that the diagonal inclusion maps ${\rK}_{2}' ( \mathbb{A}_{Y,01}  )$ into ${\rK}_{2}' ( \mathbb{A}_{Y,012}  )$ and that 
 in fact ${\rK}_{2}' ( \mathbb{A}_{Y,012}  )\cap {\rK}_{2} ( \mathbb{A}_{Y,01}  )^\flat={\rK}_{2}' ( \mathbb{A}_{Y,01}  )^\flat$. To see the inclusion, we can represent elements of $\rK_2(\hat\O_{\eta_0\eta_1})$
 by Milnor symbols $\{ f, g\}$ with $f$, $g\in \hat\O_{\eta_0\eta_1}^\times$ and use this to quickly show
 that the adelic condition (PK2) is satisfied for all $(\kappa_{\eta_0\eta_1})\in {\rK}_{2}( \mathbb{A}_{Y,01}  )$. Now we can check that the adelic condition (PK1) for  the image of $(\kappa_{\eta_0\eta_1})$ in $
 {\rK}_{2}  ( \mathbb{A}%
_{Y,012}  )$
 is equivalent to $(\kappa_{\eta_0\eta_1})\in  {\rK}_{2}( \mathbb{A}_{Y,01}  )$ belonging to ${\rK}_{2}' ( \mathbb{A}_{Y,01}  )$; these two facts imply the statement.  Similarly, by the definitions, we  have ${\rK}_{2}' ( \mathbb{A}_{Y,02}  )^\flat\subset {\rK}_{2}' ( \mathbb{A}_{Y,012}  )$, while we can also check ${\rK}_{2}' ( \mathbb{A}_{Y,012}  )\cap {\rK}_{2} ( \mathbb{A}_{Y,02}  )^\flat={\rK}_{2}' ( \mathbb{A}_{Y,02}  )^\flat$.  This gives an isomorphism
 $$
 {\rm CH}^2_{\Bbb A}(Y) \cong \frac{ {\rK}_{2}^{\prime } ( \mathbb{A}%
_{Y,012}  ) }{ {\rK}_{2}^{\prime } ( \mathbb{A}_{Y,01}   )^\flat\cdot   {\rK}_{2}' ( \mathbb{A}_{Y,02}  )^\flat
\cdot ({\rK}_{2}^{\prime } ( \mathbb{A}_{Y,12})^\flat \cap {\rK}_{2}^{\prime } ( \mathbb{A}%
_{Y,012}  )    ) }.
$$ 
The quotient on the right hand side is, by definition, Osipov's second adelic Chow group.
This together with \cite[Theorem 3]{OsipovAdelic} (in which the main ingredient is the Gersten resolution) 
gives the isomorphism
${\rm CH}^2(Y)\xrightarrow{\sim }{\rm CH}^2_{\Bbb A}(Y) $.
}
\end{remark}

\begin{remark}
{\rm Our definition of the groups ${\rm CH}^\ell_{\Bbb A}(Y[G])$ is designed so
they can be used to compute Euler characteristics of bundles from (complete) local trivializations 
via Chern classes and the Riemann-Roch theorem.
Another reasonable approach  
would be to define the $\ell$-th equivariant Chow group of $Y$ to be 
${\rm CH}^\ell(Y[G]):={\rm H}^\ell_{\rm cd}(Y, {\cal K}_\ell[G])$ where the cohomology group  
is for the completely decomposed (Nisnevich) topology and ${\cal K}_\ell[G]$ is the Nisnevich sheaf associated to the presheaf $U\mapsto \rK_\ell(\O(U)[G])$. However, we do not know how to define directly
a push down map or how to prove a Riemann-Roch formula that would involve Chern classes in these groups. 
One might speculate that, when $Y$ is regular and flat projective over $\Z$,
and $\ell=1, 2$, we have  natural  isomorphisms ${\rm H}^\ell_{\rm cd}(Y, {\cal K}_\ell[G])\xrightarrow{\sim }{\rm CH}^\ell_{\Bbb A}(Y[G]) $.  
Such an isomorphism would provide a more intrinsic interpretation 
of the adelic equivariant Chow groups ${\rm CH}^\ell_{\Bbb A}(Y[G])$.

}
\end{remark}

\subsubsection{ }\label{Frodescription}

Here we recall Fr\"{o}hlich's adelic description of the class group of a
group ring; for details see \cite{MJTclassgroups} and \cite{FrohlichBook}. We define $\Cl ( \Z [
G ]  )$ to be the kernel of the extension of scalars map $\ker
 ( \rK_{0} ( \Z[ G ]  ) \rightarrow \rK_{0} (
\Q [ G ]  )  )$. By \cite{SwanAnnals}, this coincides with the subgroup $\rK_0^{\rm red}(\Z[G])$
of $\rK_0(\Z[G])$ generated by elements of the form $[M]-{\rm rank}(M)\cdot [\Z[G]]$. Then from Ch.~I Sect.~3 in \cite{MJTclassgroups}
and Ch.~II Sect.~1 in \cite{FrohlichBook} we know that,   there is a natural isomorphism
\begin{equation}\label{fro2.5}
\Cl ( \Z[G] ) \cong \frac{ \prod^\prime_p\rK_{1} (
\Q_{p}[G])) }{\rK_{1} ( \Q
 [ G ]  )^{\flat } \prod_p \rK_{1} ( \Z_{p} [ G]  )^{\flat } }.
\end{equation}
Here: $\rK_{1} ( \Z_{p}[G])^{\flat}$ denotes
the image of $\rK_{1} ( \Z_{p}[G])$ in $\rK_{1} ( \Q_{p}[G])$; the restricted product
$ \prod^\prime_p \rK_{1} ( \Q_{p}[G])$ in the
numerator consists of elements almost all of whose terms lie in 
the subgroup $\rK_{1} ( \Z_{p}[G])^{\flat}$; and 
$\rK_{1} ( \Q[G])^{\flat }$ denotes the image
of $\rK_{1} ( \Q[G])$ in $ \prod^\prime \rK_{1} ( \Q_{p} [ G ]  )$.
Now notice that we can interpret the right hand side 
of (\ref{fro2.5}) as ${\rm CH}^1_{\mathbb{A}}({\Spec(\Z)[G]})$ (see Definition \ref{def:ith}).
Hence, we obtain an isomorphism
$\Cl ( \Z [ G ]  )\cong {\rm CH}^1_{\mathbb{A}}({\Spec(\Z)[G]})$. 
See also \S \ref{ChernFrohlich}.

\subsection{${\rm SK}_1$ of $p$-adic group rings.}\label{s:SK_1}

\subsubsection{ }\label{ss:2.2.1}

Throughout this subsection $R$ will always denote a commutative ring which is an integral domain
with field of fractions $N$. We define the group ${\rm SK}_{1}(R) $ to be the kernel of the 
group homomorphism  ${\rm Det}: \Kr_{1}(R) \rightarrow
\rK_{1} ( N)= N  ^\times $ induced by ring extension. We recall from
\cite[45.12, p. 142]{CurtisReiner} that if $R$ is in addition local, then $\rSK_{1}(R) =\{1\}$.

\begin{lemma}\label{sk_1lemma1}
For any field $N$  and for an indeterminate $t$ we
have $$\rSK_{1}(N[ t,t^{-1}])=\{1\}.$$
\end{lemma}

\begin{proof}  In \S \ref{9b2}   we show that Det is injective on $\Kr_{1}(N[ t,t^{-1}])$.  
\end{proof}
\smallskip

\begin{proposition}\label{sk_1vanish}    Suppose that $R$ is a regular local Noetherian ring. Let $f_1, f_2,\ldots , f_n$ be a sequence of irreducible elements of $R$ such that,
 for all $i=1, \ldots, n$, 
 $R_{f_1\cdots f_{i-1}}/(f_i)$ is regular and satisfies $\rK_0(R_{f_1\cdots f_{i-1}}/(f_i))\simeq \Z $
 given by rank. Then 
 $\rK_0(R_{f_1\cdots f_n})\simeq  \Z$ and $\rSK_1(R_{f_1\cdots f_n})=\{1\}$.
\end{proposition} 
\begin{proof}
Apply induction on $n$. Since $R$ is local, $\rK_0(R)=\Z$, $\rSK_1(R)=\{1\}$, which is  the case $n=0$. Set $g=f_1\cdots f_{n-1}$ and $f=f_n$. We have $R_{f_1\cdots f_n}=(R_g)_f$
and by the induction hypothesis, $\rK_0(R_g)=\Z$ and $\rSK_1(R_g)=\{1\}$, hence $\rK_1(R_g)=R_g^\times$.
Also, $((R_g)_f)^\times=R_g^\times\times f^\Z$. Consider the localization exact sequence
\[
\to \rK_1(R_g)\to \rK_1(R_{gf})\to {\rm G}_0(R_g/(f))\xrightarrow{\phi} \rK_0(R_g)\to \rK_0(R_{gf})\to 0.
\]
Since $R_g$ is regular and multiplication by $f$ is injective on $R_g$, 
the class of $R_g/(f)$ is trivial in $\rK_0(R_g)={\rm G}_0(R_g)=\Z$. Therefore
$\phi$ is the zero map and $\rK_0(R_{gf})\simeq \Z$. By our assumption, 
we have $ {\rm G}_0(R_g/(f))\cong \rK_0(R_g/(f))\simeq \Z$, the isomorphism given by 
the rank. 
By comparing the above exact sequence with 
$$
R_g^\times \to ((R_g)_f)^\times=R_g^\times\times f^\Z \to \Z\to 0
$$
we now obtain that $\rK_1(R_{gf})\cong R_{gf}^\times$ which gives $\rSK_1(R_{f_1\cdots f_n})=\{1\}$.
\end{proof}

\begin{corollary} \label{cor:Blochcor}
Suppose that $R$ is  a regular local Noetherian ring
of Krull dimension $2$. If $S$ is a multiplicative closed
subset of $R-\{0\}$ that contains at least one irreducible element
$f_1$ with $R/(f_1)$ regular, then $\rSK_1(R_S)=\{1\}$.
In particular, this applies to $R=\hat\O_{Y, \eta_2}$ and $S=\O_{Y,\eta_2}-\{0\}$,
resp. $S=\{f^n | n\geq 1\}$, where $f$ is the local equation in $\O_{Y, \eta_2}$ 
of a divisor of $Y$ with an irreducible component which is regular at $\eta_2$.
Then $R_S=\hat\O_{Y,\eta_0\eta_2}$, resp. $R_S=\hat\O_{Y, \eta_2}[D^{-1}]$,
and hence $\rSK_1(\hat\O_{Y,\eta_0\eta_2})=\{1\}$, $\rSK_1(\hat\O_{Y, \eta_2}[D^{-1}])=\{1\}$.
\end{corollary}
\begin{proof}
By taking direct limits we see that it is enough to show that $\rSK_1(R_f)=\{1\}$,
where $f=f_1 f_2\cdots f_n$ with $f_1$ as above and $f_i$ 
irreducible. We can assume that no two distinct $f_i$'s are associates. Then the assumptions
of the proposition are satisfied. Indeed, the localizations $R_{f_1\cdots f_{i-1}}$ are all UFD's of Krull dimension $1$
and the ideals $(f_i)\subset R_{f_1\cdots f_{i-1}}$ are prime. 
\end{proof}

\subsubsection{ }\label{ss:2.2.1b}   From here and on, we suppose that  $N$ has characteristic zero.  
Let $N^{c}$ be a
chosen algebraic closure of $N$.
We now consider the case of group rings and we again denote by $\mathrm{Det}$ the
map
\begin{equation}
\mathrm{Det}:\Kr_{1}(R[G])\rightarrow \Kr_{1} (
N^{c} [ G ] ) =\oplus_\chi (N^{c})^\times
\end{equation}%
where the direct sum extends over the irreducible $N^{c}$-valued characters $\chi$
of $G$. We write $\rSK_{1}(R[G])=\ker  ( \mathrm{Det} )$, so that we have the exact sequence
\begin{equation}
1\rightarrow \rSK_{1}(R[G])\rightarrow \Kr_{1}(R[G]) \rightarrow \mathrm{Det} ( \rK_{1}(R[G])\rightarrow 1.
\end{equation}
We also define ${\rm SL}( R[ G]) $ 
to be the kernel of the composite homomorphism
\begin{equation}\label{eq70}
{\rm SL}( R[ G]) =\ker \left(\GL( R[ G]) \rightarrow \rK_{1}( R[ G] )\xrightarrow{{\rm Det}} 
\rK_{1}( N^{c}[ G] )\right).  
\end{equation}
Clearly $\rE( R[ G] ) \subset {\rm SL}( R[ G])$ and we have the equality $\rE( R[ G] )
={\rm SL}( R[ G])$ precisely when $\rSK_{1}(R[ G])  =(1)$. 
Recall that if $R$ is the ring of integers of a $p$-adic field, then $\rSK_{1} ( R[G]) $ is completely described in Oliver's
papers, see \cite{OliverBAMS}.

 \subsubsection{}\label{approximation} Suppose now in addition that $R$
is  a dvr with maximal ideal $\mathfrak{p}$ and uniformizer $\pi$.  Let
$\hat{N}$, resp. $\hat{R}$, denote the $\mathfrak{p}$-adic completion
of the fraction field $N$, resp. $R$. We denote by $\SL(\hat R[G], {\mathfrak p}^m)$
the subgroup of $\SL(\hat R[G])$ consisting of matrices which are congruent to the identity modulo ${\mathfrak p}^m$. 

Recall that we say that the group algebra $N[G]$ splits if we can write
\begin{equation}\label{splitgroupalgebra}
N[G] = \prod\nolimits_i{{\rm M}}_{m_{i}}( Z_{i}),
\end{equation}
where each $Z_{i}$ is a {\sl commutative} finite field extension of $N$.

\begin{lemma}\label{lemmaApprox}
Assume $N[G]$ splits as above. For $m\geq 0$ we have 
\begin{itemize}
\item[(a)]
 $\SL( \hat N[G]) =\SL( N[G])
  \cdot \SL( \hat R[G],\mathfrak{p}^{m})$;

\item[(b)]
  $\SL( \hat R[G]) =\SL( R[G])
  \cdot \SL( \hat R[G],\mathfrak{p}^{m})$.
\end{itemize}
\end{lemma}

\begin{proof} We prove (a), and note that (b) follows easily from (a). 
We let $\hat{\mathfrak{M}}_{R, G}$ denote a maximal $\hat{R}$-order
in $\hat{N}[G]$.
Clearly we can take $\hat{\mathfrak{M}}_{R,G}=\hat R[G]$
 unless the residue characteristic of $R$ divides the order of   $G$.
Under our assumption on $N[G]$ above we can take  
\begin{equation*}
 \hat{\mathfrak{M}}_{R,G}=\sideset{}{_i}\prod{\rm M}_{m_{i}}%
( \hat{\O}_{Z_{i}}) .
\end{equation*}
Write also ${\rm M}_{n}( \hat{N}[G]) = \prod_i{\rm M}_{n_{i}}(
\hat{Z}_{i})$.  We choose $r$ such that $\pi ^{r}\hat{\mathfrak{M}}_{R,G}\subset \hat{R}[G]$ and we set $a=r+m$. Note that, as $N$ is dense in $\hat{N}$, we know that
for any non-negative integer $a$ we have the equality
\begin{equation*}
\GL_{n}(\hat{N}[G]  ) =\GL_{n}( N[G] ) \cdot \GL( \hat{R}[G] , \mathfrak{p}^{a}) .
\end{equation*}%
Let $\hat{x}\in \SL_{n}(\hat{N}[G]  )$ and choose $y\in
\GL_{n}(N [ G ] )$ close to $\hat{x}$, so that $\hat{x}%
y^{-1}=1+\pi ^{a}\lambda $ with $\lambda \in {\rm M}_{n}(\hat{R}[G])$. Then
\begin{equation*}
\mathrm{Det}( 1+\pi^{a}\lambda ) =\mathrm{Det}( y)
^{-1}\in \mathrm{Det}( \GL_{n}( N[G]) )
\cap \mathrm{Det}( 1+\pi ^{a}{\rm M}_{n}( \hat R[G]
) ) .
\end{equation*}
We write $1+\pi ^{a}\lambda =\prod 1+\pi ^{a}\lambda _{i}$ with $\lambda
_{i}\in {\rm M}_{n_{i}}( \hat{\O}_{Z_{i}})$. As $N[G]$
is semi-local we can write $y=\prod\nolimits_{i}y_{i}=\prod 
\nolimits_{i}e_{i}\delta _{i}d_{i}$ where the $e_{i} $ and $d_{i}$ lie in
the group of elementary matrices $\rE( Z_{i})  $, and where $\delta
_{i}$ is diagonal matrix with all non-leading terms 1; so that the leading
diagonal term $\xi _{i}\ $must have $\det ( y_{i}) =\xi_{i}\in
Z_{i}^{\times }$. By Lemma 2.2.b in \cite{CPTDet} we have a similar decomposition
\begin{equation*}
1+\pi ^{a}\lambda =\prod\nolimits_{i}1+\pi ^{a}\lambda
_{i}=\prod\nolimits_{i}e_{i}^{\prime }\delta _{i}^{\prime }d_{i}^{\prime }
\end{equation*}%
where the $e_{i}^{\prime }$ and $d_{i}^{\prime }$ lie in the group of
elementary matrices $\rE( \O_{Z_{i}}, \mathfrak{p}^{a})$, and where $%
\delta _{i}^{\prime }$ is diagonal with all non-leading terms 1; so that the
leading diagonal term must be $\xi _{i}^{\prime }$ with
\begin{equation*}
\det ( y_{i}^{-1}) =\det ( 1+\pi ^{a}\lambda _{i})
=\xi _{i}^{\prime }\in 1+\pi ^{a}\hat{\O}_{Z_{i}}.
\end{equation*}%
Thus we have shown that
\begin{equation*}
\xi _{i}^{-1}=\det ( y_{i}^{-1}) =\det ( 1+\pi ^{a}\lambda
_{i}) =\xi _{i}^{\prime }\in Z_{i}\cap ( 1+\pi ^{a}\hat{\O}%
_{Z_{i}}) =1+\pi^{a}\O_{Z_{i}}.
\end{equation*}%
We set $\delta
=\prod\nolimits_{i}\delta _{i}$; we can then write 
\begin{equation*}
\widehat{x}=(y\delta^{-1} )\cdot  \delta (1+\pi^a\lambda  ) \in \SL_{n}( N[ G ] ) \cdot \SL_{n}( \hat{R}[G], \mathfrak{p}^m)  
\end{equation*}%
since $y\delta^{-1} \in \SL_{n}( N[G] ) $ and $\delta
(1+\pi^a\lambda  ) \in 1+\pi ^{a}\hat{\mathfrak{M}}_{R,G}\subset 1+\pi ^{m}%
\hat{R} [ G ]$.
\end{proof}

\subsubsection{} \label{ss:SK_1paper} 

In this paragraph, we recall some results from \cite{CPTSK1b} (see the introduction of loc. cit.).

\begin{theorem}\label{SK1theorem1} Suppose that $R$ is a Noetherian domain with fraction field
of characteristic zero. Assume that 
the natural map $R\rightarrow\varprojlim_n R/p^{n}R$ is an isomorphism, so that $R$ is $p$-adically
complete. Then for any integer $k\geq 2$, ${\rK}_1(R[G], (p)^k)$ is a
subgroup of $\rK_1(R[G])$ and we have $$ {\rK}_{1}(R[G],(p)^{k}) \cap  {\rSK}_{1}(R[G])=\{1\}.$$
\end{theorem}

\begin{proof} 
This follows from \cite{CPTSK1b} Theorems 1.3 and 1.4.
\end{proof}

\begin{corollary}\label{SK1density}
Let $R$ be a discrete valuation ring of mixed characteristic with fraction field $N$
and denote by $\hat R$ its $p$-adic completion. 
Assume that $N[G]$ splits as in (\ref{splitgroupalgebra}). 
Then the natural map
$\rSK_1(R[G])\to  {\rSK}_1(\hat R[G])$
is surjective.
\end{corollary}

\begin{proof} Define $\rSK_1(R/ (p)^m[G])$ to be the image of $\rSK_1(R[G])$
in $\rK_1(R/(p)^m[G])$. Theorem  \ref{SK1theorem1} implies  that, for $m\geq 2$, the map
\begin{equation*}
 {\rSK}_{1}( \hat R[G] ) \rightarrow
\rSK_{1}( \hat R/(p)^m[G])  
\end{equation*}%
is injective and hence an isomorphism. The result now follows from
 Lemma \ref{lemmaApprox}  (b).
\end{proof}
\medskip

In \cite{CPTSK1b}, we obtain more precise results about  $ {\rSK}_{1} $ when we assume that, among other additional 
 hypotheses, our coefficient rings afford a lift of Frobenius. We are going to use the following
corollaries of the main result of  \cite{CPTSK1b}. 
Here we will assume that $W=W(k)$ is the ring of integers in a 
 finite unramified extension of $\Q_p$ with residue field $k$.

\begin{corollary}\label{SK1cor1} 
  The inclusion $W\subset W{\lps t\rps}$ induces an isomorphism $$  {\rSK}_{1}(W[G])\xrightarrow{\sim} 
 {\rSK}_{1}( W{\lps t\rps}[G]). $$
\end{corollary}

\begin{corollary}\label{SK1cor2}

a) The  inclusion $W \langle \langle t^{-1} \rangle  \rangle \subset
W{\ldb t\rdb}$ induces an isomorphism 
$$
{\rSK}_{1} ( W\langle\! \langle t^{-1} \rangle\!  \rangle [ G ]  )\xrightarrow{\sim} {\rSK}_{1} ( W{\ldb t\rdb} [ G ]  ).
$$

b) The inclusion $W{\lps t\rps}\subset W{\ldb t\rdb}$
induces an injection 
$
{\rSK}_{1} ( W\lps t \rps [ G ]  )\hookrightarrow {\rSK}_{1} (W{\ldb t\rdb} [ G ]  ).
$
\end{corollary}

\subsubsection{ }\label{vanishing}

Let us also record:

\begin{lemma}\label{sk_1lemma2} Suppose that $\Q_p[G]$ splits. Then we have:

a) $ {\rSK}_{1} (\Q\otimes_{\Z_p}\Z_{p} \langle\!
 \langle t^{-1} \rangle\!  \rangle[G])=\{1\}$, 

b)  If also $p$ does not divide the order of $G$, we have $ {\rSK}_{1} ( \Z_{p} \langle\!  \langle
t^{-1} \rangle\!  \rangle[G]  ) =\{1\}$.
\end{lemma}

\begin{proof} Using Morita equivalence, we see that it is enough to show  that ${\rSK}_1(   R\langle\!\langle t^{-1}\rangle\!\rangle)=\{1\}$ and
${\rSK}_1( N\otimes_{R} R\langle\!\langle t^{-1}\rangle\!\rangle)={\rSK}_1(N\{t^{-1}\})=\{1\}$,
where $R$ are the integers in a finite extension  $N$ of $\Q_p$ and $N\{t^{-1}\}$ is the Tate algebra. 
This first statement follows from \cite[proof of IV, Prop. 4]{Gruson} applied to $A=R$, $B=R \langle\!  \langle
t^{-1} \rangle\! \rangle$: Indeed, Gruson's argument  implies that the natural map $\rSK_1(R[t^{-1}])\to \rSK_1(B)$
is surjective and the result follows since $\rSK_1(R[t^{-1}])=\{1\}$. The proof of ${\rSK}_1(N\{t^{-1}\})=\{1\}$
is similar. In fact, this is  a special case of  \cite[Theorem 1]{Gruson}.
  \end{proof}
\bigskip
\bigskip

\section{Lattices, determinant functors and determinant theories}\label{sLAT}

\setcounter{equation}{0}

In what follows, $R$ is a commutative Noetherian ring, $A$ is  a commutative Noetherian flat $R$-algebra
and $t$ a non-zero divisor in $A$ such that $A/tA$ is finitely generated
and flat over $R$. We also consider $A_t=A[t^{-1}]$. 
In the main examples we have in mind, $A=R[t]$, or $A=R{\lps t\rps}$.
Also   all modules over a group ring such as $A[G]$ are left modules.

\subsection{Some lemmas}\label{3aaa}

We start with:

\begin{lemma}\label{proj1}
Suppose that $S$ is a local Noetherian commutative ring with $1$
and residue field $k$ of characteristic $p$. Suppose that $P$ is $p$-Sylow
subgroup of $G$.  If $p=0$,   take $P=\{1\}$.  Let $M$ be a finitely generated  $S[G]$-module.
Then $M$ is $S[G]$-projective if and only if the $S[P]$-module $M$
obtained by restriction of operators from $G$ to $P$ is $S[P]$-projective.
\end{lemma}

\begin{proof}
Observe that since $[G:P]$
is invertible in $S$, the $S[G]$-module $S[G/P]$ admits the $S[G]$-module $S$
with trivial $G$-action as a direct summand.
By Frobenius reciprocity
$$
S[G/P]\otimes_S M\simeq S[G]\otimes_{S[P]} ({\rm Res}_{G\to P}(M)).
$$
Hence, $M$ is a direct summand of $S[G]\otimes_{S[P]} ({\rm Res}_{G\to P}(M))$.
The result follows from this.  
\end{proof}

\begin{lemma}\label{proj2}
Suppose that $S$ is a local Noetherian commutative ring
with residue field $k$ of characteristic $p$. Suppose that $G$
is a $p$-group. ($G=\{1\}$, if $p=0$.)
Let $M$ be a finitely
generated $S[G]$-module.
Let $J $ be the Jacobson radical
of $S[G]$. Then
the following are equivalent:

a) $M$ is $S[G]$-free,

b) $M$ is $S[G]$-projective,

c) $M$ is $S[G]$-flat,

d)  ${\rm Tor}^{S [G]}_1(S [G]/J, M )=(0)$.
\end{lemma}

\begin{proof}
Notice that since $S$ is Noetherian, $S[G]$ is also Noetherian.
Clearly (a) implies (b), (b) implies (c), (c) implies (d).
It remains to show that (d) implies (a).
Recall, $G$ is a $p$-group. In this case, $S[G]/J=k$.
Suppose that $\phi: k^n\xrightarrow{\sim} M/JM$.
Lift $\phi$ to an $S[G]$-module homomorphism
$$
0\to K\to  S[G]^n\xrightarrow{\Phi}  M\to 0
$$
with $K$ the (finitely generated) kernel. By the non-commutative version of Nakayama's
lemma, $\Phi$ is surjective.
By tensoring the exact sequence above with
$S[G]/J\otimes_{S[G]}\ -$ we obtain (using (d))
an exact sequence of $S[G]/J$-modules
$$
0\to K/JK\to (S[G]/J)^n\xrightarrow{\phi} M/JM\to 0.
$$
But $\phi$ is an isomorphism so $K/JK=(0)$.
Another application of Nakayama's lemma now gives $K=(0)$
and so $M$ is actually free.
\end{proof}

\subsection{Lattices} \label{latticessect}Suppose $M_0$ is a finitely generated projective $R[G]$-module.
We set $\M=M_0\otimes_{R}A_t$ and $L_0=M_0\otimes_{R}A$.

\begin{definition}
A finitely generated projective $A[G]$-submodule $L$ of $\M=M_0\otimes_{R}A_t$ with $\sum_{n\leq 0}L\cdot t^n=\M$, will be called an ``$A[G]$-lattice",
or simply a ``lattice". \end{definition}

Notice that for a lattice $L$ there  is $n\geq 0$ such that
$
t^nL_0\subset L\subset t^{-n}L_0
$
and we have a canonical $A_t[G]$-isomorphism
$
L\otimes_{A}A_t=\M.
$

\begin{proposition}\label{lattice}
Suppose that $L\subset \M=M_0\otimes_{R}A_t$ is a finitely generated $A[G]$-submodule of $\M$. Then $L$ is  a lattice
if and only if the following condition is satisfied:
There is $n\geq 0$, such that
$
t^nL_0\subset L\subset t^{-n}L_0
$,
and the quotients $L/t^nL_0$, $t^{-n}L_0/L$ are both $R[G]$-projective.
 \end{proposition}

\begin{proof}
First assume that $L$ is $A[G]$-projective.
Consider the exact sequence
\begin{equation}\label{exL0}
0\to  t^nL_0\to L \to L/t^nL_0\to 0.
\end{equation}
To show that $L/t^nL_0$, $t^{-n}L_0/L$,  are $R[G]$-projective it is enough to
reduce to the case that $R$ is local Noetherian with residue field of characteristic $p$
and by Lemma \ref{proj1}   suppose that $G$ is a $p$-group.
 Set $\bar L=L/t^nL_0$.
Since $A$ is flat over $R$, the $R[G]$-modules $L$ and $t^nL_0$ are also flat.
The exact sequence (\ref{exL0}) implies that $\bar L$ has $R[G]$-Tor dimension $\leq 1$.
Since $\bar L$ is finitely generated over $R[G]$, Lemma \ref{proj2} above
and a standard argument shows that $\bar L$ has $R[G]$-projective dimension $\leq 1$.
The same is true for the $R[G]$-module $t^{-n}L/L_0$. Assume that the  $R[G]$-projective dimension of
$t^{-n}L_0/L$ is $1$. The exact sequence
$$
0\to \bar L\to t^{-n}L_0/t^nL_0\to t^{-n}L_0/L\to 0
$$
would then give that the projective dimension of
$t^{-n}L/L_0$ is $>1$, a contradiction. Hence, $\bar L$
is $R[G]$-projective. The same argument now shows that $t^{-n}L_0/L_0$
is also $R[G]$-projective. We conclude that $L/t^nL_0$, $t^{-n}L_0/L$ are both $R[G]$-projective.

Conversely, assume
$$
t^nL_0\subset L\subset t^{-n}L_0
$$
and that the quotients $L/t^nL_0$, $t^{-n}L_0/L$ are both $R[G]$-projective.
We will show that $L$ is $A[G]$-projective.

We can assume that $A$ and $R$ are local and that $t$ is in the unique maximal ideal of $A$.
(Otherwise, $t$ is invertible and we get $L=L_0$ in the
corresponding localization.) In addition, by Lemma \ref{proj1},
we can suppose that $G$ is a $p$-group where $p$ is the characteristic of the residue field
of $A$. We first claim that it is enough to show that, under our assumptions,
$L/tL$ is $R[G]$-free. Indeed, we will first show that
if $L/tL$ is $R[G]$-free, then $L$ is $A[G]$-free. Consider a map $F\to L$ from a free $A[G]$-module which lifts
$F/tF\xrightarrow{\sim} L/tL$. By Nakayama's lemma,
$F\to L$ is surjective; let $K$ be its the kernel.
Now notice that since $L\subset \M$, $L$ is $t$-torsion free
and so
$$
0\to K/tK\to F/tF\to L/tL\to 0
$$
is exact. Hence, $K/tK=(0)$. Since $t$ is in the maximal ideal of the local ring $A$, by Nakayama's
lemma again, $K=(0)$. It now remains to show that $\bar L:=L/tL$ is $R[G]$-free.
For simplicity, set $L_n=t^{-n}L_0$ which is $A[G]$-free. By our assumption and Lemma \ref{proj2},
$L_n/L$ is $R[G]$-free. By enlarging $n$ if needed, we can assume that $L\subset tL_n$.
Now tensor the exact $A$-sequence
$$
0\to L\to L_n\to L_n/L\to 0
$$
with $-\otimes_A{A/tA}$. Since $t$ is not a zero-divisor
in $A$, we obtain
$$
0\to T(L_n/L)\to L/tL\to L_n/tL_n\to (L_n/L)/t(L_n/L)\to 0
$$
where $T(L_n/L):=\{x\in L_n/L\ |\ t\cdot x=0\}$ is an $R[G]$-module.
Since $L\subset tL_n$, the map $L_n/tL_n\to  (L_n/L)/t(L_n/L)$ is an isomorphism.
Hence,
 $$
 T(L_n/L)\simeq L/tL.
 $$
Notice that we have an exact sequence of $R[G]$-modules
$$
0\to T(L_n/L)\to L_n/L\xrightarrow {t} t(L_n/L)\to 0.
$$
Since $L\subset tL_n$, $t(L_n/L)=tL_n/L$. The module $tL_n/L$
is the kernel of the surjective map $L_n/L\to L_n/tL_n$ between $R[G]$-free
modules and so it is $R[G]$-free. Hence, $T(L_n/L)$ is also $R[G]$-free.
Therefore, $L/tL$ is also $R[G]$-free.
 \end{proof}

\begin{corollary}\label{L1L2}
If $L_1\subset L_2$ are two $A[G]$-lattices, then $L_1/L_2$ is a finitely generated
projective $R[G]$-module.
\end{corollary}

\begin{proof}
There is $n\geq 0$ such that
$
t^nL_0\subset L_1\subset L_2\subset t^{-n}L_0.
$
This gives an exact sequence
$$
0\to L_2/L_1\to t^{-n}L_0/L_1\to t^{-n}L_0/L_2\to 0
$$
with middle and right terms $R[G]$-projective. It follows
that $L_2/L_1$ is $R[G]$-projective.
\end{proof}

\subsubsection{}

Notice that if $\gamma$ is an $A_t[G]$-isomorphism of the  $A_t[G]$-module
$\M=M_0\otimes_R A_t$, then the image $\gamma(L_0)\subset \M$
is an $A[G]$-lattice. In particular, if $M_0=R[G]^n$ and
$\gamma$ is given by right multiplication by the element $g\in \GL_n(A_t[G])$, \emph{i.e.} by
$\gamma(m):= m\cdot g$, then $L_0\cdot g\simeq A[G]^n$ is an $A[G]$-lattice.

\subsection{Determinants}

We continue to assume that $R$ is a Noetherian commutative
ring.
Recall the definition of the virtual category $V(R[G])$
of finitely generated projective (left) $R[G]$-modules from \cite{DeligneDet} (see also \cite{BurnsFlachDoc}).
This is a commutative Picard category (\emph{i.e.} a symmetric
monoidal category in which all arrows are invertible
and all objects have inverses).
Any finitely generated projective $R[G]$-module
 $P$ gives an object in $V(R[G])$, which we will denote by $[P]$. 
 The inverse of $[P]$ is denoted by $-[P]$. As in \cite{DeligneDet} we will denote the monoidal structure
additively.
The set of isomorphism classes
of objects in $V(R[G])$ is a group which is identified with $\rK_0(R[G])$;
the group of automorphisms of the zero object $[0]$
is identified with $\rK_1(R[G])$. If $R=K$ is a field and
$G=\{1\}$, $V(R[G])=V(K)$ can be identified with the Picard category
${\rm Pic}^\Z_K$
of ``$\Z$-graded $K$-lines". Recall that the   objects
of ${\rm Pic}^\Z_K$ are pairs
$(L, n)$ of a $K$-line $L$ and an integer $n$
and the monoidal structure is given by
$$
(L, n)+(M, m)=(L\otimes_K M, n+m).
$$
The identification above is then given by sending $P$ to $(\det(P), {\rm rank}(P))$.

Consider the (full) subcategory ${\rm D}^p(R[G])$
of the derived
category ${\rm D}(R[G])$ of the homotopy category of complexes of
$R[G]$-modules whose objects are perfect complexes.
Recall that there is a  ``determinant" functor
$$
\det: {\rm D}^p(R[G])\to V(R[G])
$$
which takes the value $[P]$ on complexes $P[0]: \cdots \to 0\to P\to 0\to \cdots $
consisting of a finitely generated projective $R[G]$-module
placed in  degree $0$. The functor $\det$ satisfies  an additivity property
for ``true" exact triangles, and other properties which are listed in \cite{BurnsFlachDoc}.
To simplify our notations, we will sometimes write $[P^\bullet]$ instead of $\det(P^\bullet)$
for the virtual object in $ V(R[G])$ associated to the perfect complex $P^\bullet$.

\subsubsection{}\label{3c1}

By definition (cf. \cite{DrinfeldInfDimBun}, \S 5), a ``determinant theory" on $\M$ is a rule that associates to any $A[G]$-lattice $L$  as above,
an object $\delta(L)$ of $V(R[G])$ and to each pair $L_1\subset L_2$ of $A[G]$-lattices
an arrow in $V(R[G])$
\begin{equation}\label{detcomp}
\delta_{L_1,L_2}: \delta(L_1)+[L_2/L_1]\xrightarrow{\  }\delta(L_2)
\end{equation}
(with $[L_2/L_1]$ well-defined by Corollary \ref{L1L2}), such that:

If $L_1\subset L_2\subset L_3$, the obvious diagram
\begin{equation}
\begin{matrix}
 \delta(L_1)+[L_2/L_1]+[L_3/L_2]&\xrightarrow{ }&\delta(L_2)+[L_3/L_2]\\
 \downarrow &&\downarrow\\
 \delta(L_2)+[L_3/L_2]&\xrightarrow{} & \delta(L_3)
\end{matrix}
\end{equation}
obtained using $\delta_{L_1,L_2}$, $\delta_{L_2, L_3}$, $\delta_{L_1, L_3}$
commutes, and the diagonal morphism is obtained by combining
$\delta_{L_1, L_3}$ with the arrow $[L_2/L_1]+[L_3/L_2] \to [L_3/L_1]$
given by the exact sequence $0\to L_2/L_1\to L_3/L_1\to L_3/L_2\to 0$.

(In fact, we will often also find that our construction
satisfies additional compatibilities for suitable base changes
$R\to R'$ as in \cite{DrinfeldInfDimBun}, \S 5.)

We can see that the set of determinant theories is a torsor over the
commutative Picard category $V(R[G])$; in particular, if $\delta$, $\delta'$ are two
determinant theories, then there is an object $Q$ of $V(R[G])$ and arrows
\begin{equation}\label{compare}
\delta'(L)\xrightarrow{ \ }\delta(L)+Q\
\end{equation}
for each lattice $L$ which are functorial (for inclusion
of lattices).

Consider the group ${\rm Aut}(\M)$ of
$A_t[G]$-linear isomorphisms of $\M$.
If $L$ is an $A[G]$-lattice, so is its image $\g L$
under $\g$.
Notice that, for each pair of lattices $L_1\subset L_2$,
an element
$\g\in {\rm Aut}(\M) $ induces an arrow
$$
[L_2/L_1]\to [\g L_2/\g L_1]
$$
given by an actual $R[G]$-module isomorphism.
Hence, we can see that we can ``twist $\delta$ by $\g$"
to form a new determinant theory
given by $L\mapsto \delta(\g L)$.
By the above, the object
$$
\V_\g=\V_\g(L)=\delta(\g L)-\delta(L)
$$
does not depend on $L$.
This is meant in the sense that for any two
lattices $L\subset L'$ there is
a well-defined arrow
\begin{equation}\label{indep}
\V_\g(L)\to \V_\g(L')
\end{equation}
which respects compositions for
chains of inclusions.

\subsubsection{} \label{3c2} Now take $A=R[t]$. Suppose that $L$ is an $A[G]$-lattice in $\M=M_0\otimes_RR[t, t^{-1}]$.
To that, we can associate 
 a coherent locally projective $\O_{\PP^1_R}[G]$-module
$\E(L)$ on $\PP^1_R$ obtained by gluing the sheaves on
  ${\Bbb A}^1_\infty = \mathrm{Spec}(R[t^{-1}])$ and ${\Bbb A}^1_0 = \mathrm{Spec}(R[t])$
that  correspond  to $M_0\otimes_\Z R[t^{-1}]$ and $L$ respectively, along
 the identification
 $$
 L\otimes_{R[t]}R[t, t^{-1}]=\M=(M_0 \otimes_\Z R[t^{-1}])\otimes_{R[t^{-1}]}R[t, t^{-1}].
 $$
Now suppose that $\gamma$ is an $A_t[G]$-isomorphism of $M_0\otimes_R A_t$;
this gives the $A[G]$-lattice $L=\gamma(L_0)$, $L_0=M_0\otimes_R R[t]$.
 By Theorem \ref{thm:Horrocks}, when $R$ is a Dedekind ring with finite residue fields, all coherent 
 locally free $\O_{\PP^1_R}[G]$-modules
 can be obtained as $\E(L)=\E(\gamma(L_0))$ for a suitable  $M_0$ and  $\gamma$ as above.

 \subsubsection{}  \label{3c3} Denote by $R\Gamma(\PP^1_R, \E(L))$ the   complex
in the derived category ${\rm D}(R[G])$  that calculates the cohomology of $\E(L)$ over $\PP^1_R$.
This is quasi-isomorphic
to the  \v{C}ech complex
 \begin{equation}
 C^\bullet(L): (M_0 \otimes_\Z R[t^{-1}])\oplus L \xrightarrow{ \ } M_0\otimes_\Z R[t, t^{-1}]
 \end{equation}
 The standard argument shows that  $R\Gamma(\PP^1_R, \E(L))$ is ``perfect", \emph{i.e.} is in ${\rm D}^p(R[G])$.
 Hence,  we can set
$$
\delta(L):=\det(R\Gamma(\PP^1_R, \E(L)) \in V(R[G]).
$$
This  gives a determinant theory as above.
Indeed, an inclusion $i: L_1\hookrightarrow L_2$ of $A[G]$-lattices, gives
a corresponding homomorphism
of sheaves $i: \E(L_1)\to \E(L_2)$ and of \v{C}ech complexes
$i: C^\bullet(L_1)\to C^\bullet(L_2)$. Notice that there is a short exact sequence
of complexes
$$
0\to C^\bullet(L_1)\xrightarrow{i} C^\bullet(L_2)\to (L_2/L_1)[0]\to 0
$$
of $R[G]$-modules. Using this, we obtain a true triangle  in ${\rm D}^p(R[G])$
$$
R\Gamma(\PP^1_R, \E(L_1))\to R\Gamma(\PP^1_R, \E(L_2))\to (L_2/L_1)[0]\to R\Gamma(\PP^1_R, \E(L_1))[1].
$$
This induces the isomorphism
\begin{equation*}
\delta_{L_1,L_2}: \delta(L_1)+[L_2/L_1]\xrightarrow{\  }\delta(L_2).
\end{equation*}
as required. We can now see that the required properties of $\delta$
follow from the corresponding properties of $\det$.

\subsection{A central extension}\label{3d}

Consider the group ${\rm Aut}(\M)$ of
$R[G][t, t^{-1}]$-linear isomorphisms of $\M=M_0\otimes_R R[t, t^{-1}]$.
Following ideas in \cite{DrinfeldInfDimBun} or \cite{BeilinsonBlochEsnault} we construct the ``canonical" $V(R[G])$-extension ${\rm Aut}(\M)^\vee$ of ${\rm Aut}(\M) $
(in the sense of \cite[A2]{BeilinsonBlochEsnault}) associated to the determinant theory $\delta$.
Explicitly, ${\rm Aut}(\M)^\vee:={\rm Aut}_\delta(\M)$ is given as follows:

(i) To every
$\gamma: \M\to \M$ in ${\rm Aut}(\M)$ we associate
the object
$$
\V_\gamma=\delta(\gamma L_0)-\delta(L_0)
$$
of $V(R[G])$;

(ii)  To every pair
of elements $\gamma$, $\gamma'$ in ${\rm Aut}(\M)$,
we associate a ``composition" isomorphism
$$
c_{\gamma,\gamma'}: \V_{\gamma}+\V_{\gamma'}\xrightarrow {}\V_{\gamma \cdot\gamma'}
$$
which is given as follows:

By (\ref{indep}) applied to $L_0$ and $\g'L_0$, we have an arrow
$$
\V_\g+\V_{\g'}\to (\delta(\g\g' L_0)-\delta(\g'L_0))+(\delta(\g' L_0)-\delta(L_0)).
$$
This composed with the contraction
$$
 (\delta(\g\g' L_0)-\delta(\g'L_0))+(\delta(\g' L_0)-\delta(L_0))\to  \delta(\g\g' L_0)-\delta(L_0)=\V_{\g\g'}
 $$
defines $c_{\g, \g'}$.

We can see that the  arrows $c_{\g, \g'}$ satisfy
 associativity,
\emph{i.e.} that the obvious diagrams
\begin{equation*}
\begin{matrix}
(\V_\g+\V_{\g'})+\V_{\g''}&\xrightarrow{ \ \ \ \ \ \ \ \ \ \ \ \ \ \ }&\V_{\g\g'}+\V_{\g''}\\
\downarrow &&\downarrow \\
\V_\g+(\V_{\g'}+\V_{\g''})&\to \ \V_\g+\V_{\g'\g''}\ \to &\V_{\g\g'\g''}\\
\end{matrix}
\end{equation*}
formed using the $c$'s and the associativity constraint in $V(R[G])$ are commutative.

Finally, we can see, using (\ref{compare}), that
the $V(R[G])$-extension ${\rm Aut}(\M)^\vee:={\rm Aut}_\delta(\M)$ is independent
up to isomorphism (in the sense of \cite[A3]{BeilinsonBlochEsnault})
of the choice of determinant theory $\delta$.

\subsubsection{}\label{split}
Notice that if $\gamma$ belongs to the subgroup ${\rm Aut}(L_0)={\rm Aut}(M_0\otimes_RR[t])\subset {\rm Aut}(\M)$,
or to the subgroup ${\rm Aut}(M_0\otimes_RR[t^{-1}])\subset {\rm Aut}(\M)$,
 we have   $\E(\gamma L_0)=\E(L_0)$ as $\O_{\PP^1_R}[G]$-sheaves on $\PP^1_R$
 and hence  $\delta(\gamma L_0)=\delta(L_0)$; this gives a canonical arrow $[0]\to \V_\gamma$
and the central extension ${\rm Aut}(\M)^\flat$ splits over ${\rm Aut}(L_0)$ and also over
  ${\rm Aut}(M_0\otimes_RR[t^{-1}])$.

\subsubsection{}\label{332}

Taking   isomorphism classes $\gamma\mapsto [\V_\gamma]$
gives a group homomorphism
\begin{equation}
\chi: {\rm Aut}(\M)\to \rK_0(R[G]).
\end{equation}
Denote by ${\rm Aut}'(\M)$ the kernel of $\chi$.
Now for each $\gamma: \M\to \M$   in
${\rm Aut}'(\M)$,  choose an arrow
$\phi_\gamma: \V_1=[0]\xrightarrow{\sim} \V_\gamma$
in $V(R[G])$. If $\g$, $\g'$ are in ${\rm Aut}'(\M)$,
using the trivializations $\phi_\gamma$, $\phi_{\gamma'}$,
$\phi_{\gamma\gamma'}$ allows us to identify
the compositions $c_{\gamma,\gamma'}$
with  elements of  $\rK_1(R[G])$. We can check that
the associativity amounts to the fact that
$$
c: {\rm Aut}'(\M)\times {\rm Aut}'(\M)\to \rK_1(R[G]);\quad (\g, \g')\mapsto c_{\g, \g'},
$$
is a $2$-cocycle. There is a corresponding
 central extension
\begin{equation}\label{cextmaster}
1\to \rK_1(R[G])\to \Hh_\delta(\M)\to {\rm Aut}'(\M)\to 1
\end{equation}
which can be described more explicitly as follows:
\begin{equation}\label{isodescript}
\Hh_\delta(\M)=\{(\gamma, \phi_\gamma)\ |\ \gamma\in {\rm Aut}'(\M),
\phi_\gamma: \V_1=[0]\xrightarrow{ } \V_\gamma\}
\end{equation}
with multiplication defined using the cocycle $c$ above.
Again, up to isomorphism, the central extension
$\Hh_\delta(\M)$ is independent of the choice
of determinant theory. By \S \ref{split},
we see that the central extension $\Hh_\delta(\M)$ splits
over ${\rm Aut}(L_0)$ and also over
  ${\rm Aut}(M_0\otimes_RR[t^{-1}])$.
  (They are obviously both subgroups of ${\rm Aut}'(\M)$.)

\subsubsection{}
\label{s:categorical}

Now take $A=R[t]$, so that $A_t=R[t, t^{-1}]$ and take $M_0=R[G]^n$, $\M=A_t[G]$.
Using the isomorphism
$$
\GL_n(A_t[G])\xrightarrow{} {\rm Aut}(\M);\quad g\mapsto (m\mapsto m\cdot g^{-1})
$$
we pull-back ${\rm Aut}(\M)^\vee$ to a categorical $V(R[G])$-extension 
$\GL_n(A_t[G])^\vee$ of $\GL_n(A_t[G])$.
This in turn extends to a categorical extension $\GL(A_t[G])^\vee$
of the infinite linear group $\GL(A_t[G])=\varinjlim_{n} \GL_n(A_t[G])$.
Notice that the commutator subgroup $\rE(A_t[G])\subset \GL(A_t[G])$
is contained in $\varinjlim_n \GL_n'(A_t[G])$; this allows us to assemble
the extensions obtained as above from  (\ref{cextmaster})  for $n>>0$ and give a central extension
\begin{equation}\label{cextmaster2}
1\to \rK_1(R[G])\to \Hh(A_t[G])\to \rE(A_t[G])\to 1.
\end{equation}
Since $\rE(A_t[G])$ is a perfect group and the Steinberg
extension ${\rm St}(A_t[G])$ is its universal central extension (see \cite[Chapter 4.2]{RosenbergBook} or  
\cite[Section 5]{MilnorK}),
there is a (unique) group homomorphism
\begin{equation}\label{boundary}
\partial: \rK_2(A_t[G])\to \rK_1(R[G])
\end{equation}
that fits in a (unique) commutative diagram
\begin{equation}
\label{eq:StoHmap}
\begin{matrix}
1&\to &\rK_2(A_t[G])&\to &{\rm St}(A_t[G])&\to &\rE(A_t[G])&\to &1\\
&&\partial\downarrow &&\partial\downarrow &&  \downarrow &\\
1&\to& \rK_1(R[G])&\to &\Hh(A_t[G])&\to &\rE(A_t[G])&\to &1
\end{matrix}
\end{equation} with the right vertical map the identity.
Observe here that by \S \ref{split}, the extension
(\ref{cextmaster2}) splits over $\rE(A[G])$. Hence,
the homomorphism $\partial$ is trivial on the image of $\rK_2(A[G])$ in
$\rK_2(A_t[G])$, \emph{i.e.} the composition
\begin{equation}\label{vanish}
 \rK_2(A[G]))\to \rK_2(A_t[G])\xrightarrow{\partial} \rK_1(R[G])
\end{equation}
is trivial.

\begin{remark}\label{rem7}  {\rm  Notice that there is a $1$-$1$ correspondence
between $R{\lps t\rps}[G]$-lattices in $R\llps t\lrps[G]^n$
and $R[t][G]$-lattices in $R[t,t^{-1}][G]^n$;
indeed, by Proposition \ref{lattice},
both these sets are in $1$-$1$ correspondence
with the union over $n\geq 0$ of all $R[t][G]$-submodules
of $t^{-n}R[t][G]/t^nR[t][G]\simeq R[t][G]/(t^{2n})$ which are
$R[G]$-projective. Hence, our determinant theory for $R[t,t^{-1}][G]^n$
also gives a determinant theory for $R\llps t\lrps[G]^n$. Then the above results
also apply to $A=R{\lps t\rps}$.
The corresponding central extensions  (\ref{cextmaster}) are compatible
in the sense that the central extension
for $R\llps t\lrps[G]^n$ pulls back to the one for $R[t, t^{-1}][G]^n$
under $\GL_n'(R[t,t^{-1}][G])\hookrightarrow \GL_n'(R\llps t\lrps[G])$.
In particular, the same argument gives a boundary $\partial: \rK_2(R\llps t\lrps[G])\to \rK_1(R[G])$ that satisfies (\ref{vanish})
as above. }
 \end{remark}

\begin{remark}  {\rm The homomorphism $\partial$ is a refined version of the 
inverse of the tame symbol.
(See below.)
In a previous version of this paper,
a  homomorphism $\rK_2(A_t[G])\to \rK_1(R[G])$ was constructed
 as a boundary map
on a suitable localization sequence for $\rK$-groups
using work of Neeman-Ranicki \cite{NeemanRanI}, \cite{NeemanRanII}.
This should agree with the construction given above
but working out the details of this comparison
is a complicated affair.}
\end{remark}

\subsubsection{} \label{3d4} In   this paragraph, we will consider $R{\lps t\rps}[G]$-lattices but the construction works with $R[t][G]$-lattices too.
Let us fix a determinant theory $\delta$.
Suppose that $L_1$, $L_2$ are two lattices and find $N>>0$ such that $t^NL_0\subset L_1, L_2$.
Then we can see that a choice of an isomorphism $a: \delta(L_1)\xrightarrow{\sim} \delta(L_2)$ amounts to an isomorphism
$$
[a]: [L_1/t^NL_0]\xrightarrow{\sim} [L_2/t^NL_0]\ .
$$
Indeed, $a$ is the unique isomorphism for which the diagram
\begin{equation}\label{bra}
\begin{matrix}
\delta(t^NL_0)+[L_1/t^NL_0]&\xrightarrow{\delta_{L_1,t^NL_0}}& \delta(L_1)\\
{\rm id}+[a] \downarrow\ \ \  &&\downarrow a\\
\delta(t^NL_0)+[L_2/t^NL_0]&\xrightarrow{\delta_{L_2,t^NL_0}}& \delta(L_2)
\end{matrix}
\end{equation}
commutes.
The central extension 
\begin{equation}\label{anotherext??}
1\to \rK_1(R[G])\to \Hh_\delta(R\llps t\lrps[G]^n)\to \GL'_n(R\llps t\lrps[G])\to 1
\end{equation}
can now also be described as follows.  Recall
$$
\Hh_\delta(R\llps t\lrps[G]^n)=\{(g, \phi_g)\ |\ g\in \GL'_n(R\llps t\lrps[G]), \ \phi_g: \delta(L_0)\xrightarrow{\sim} \delta(L_0\cdot g^{-1}) \}\ .
$$
We define an operation on $\Hh_\delta(R\llps t\lrps[G]^n)$ by
$$
(g, \phi_g)\star (h, \phi_h)=(g\cdot h,  \phi_h^{g}\circ \phi_{g})
$$
where $\phi_h^{g}$ can be defined as follows: For $N>>0$, $[\phi_h^{g}]$ is given by the composition
$$
[L_0\cdot g^{-1}/t^NL_0]\xrightarrow{\ \cdot g } [L_0/t^NL_0\cdot g]\xrightarrow{[\phi_h]} [L_0\cdot h^{-1}/t^NL_0\cdot  g]
\xrightarrow{\ \cdot g^{-1}} [L_0\cdot h^{-1}g^{-1} /t^NL_0]
$$
where $[\phi_h]$ is induced by $\phi_h$ as above and $r(g)=\cdot g$, $r(g^{-1})=\cdot g^{-1}$ are given by right multiplication.
We can see that the corresponding $\phi^g_h: \delta(L_0\cdot g^{-1})\to \delta(L_0\cdot (gh)^{-1})$ is independent of the choice of $N$; we will abuse notation and write
$$
\phi^g_h=r(g^{-1})\circ \phi_h\circ r(g)\ .
$$
We can now see that $\star$ defines a group structure on $\Hh_\delta(R\llps t\lrps[G]^n)$.
The inverse of $(g, \phi_g)$ is given by $(g^{-1}, \psi_g)$
with
$$
\psi_g=\phi_g^{-g^{-1}}=r(g)\circ \phi^{-1}_g\circ r(g^{-1})
$$
(with the same abuse of notation as before).

\subsubsection{}\label{tame}

Here we explain how the homomorphism $\partial: \rK_2(A_t[G])\to \rK_1(R[G])$
of the previous paragraph
can often be calculated using the classical tame symbol.

Recall that for a field $E$ we know by Matsumoto's theorem that
$\rK_2(E)$ is generated by symbols $\{a,b\}=a\cup b $ for
$a,b\in E^{\times }$. Suppose now that $E$ supports a valuation $v$, with
valuation ring $\O$, maximal ideal ${\mathfrak m}$ and with residue field $k=\O/
{\mathfrak m}$. Then the tame symbol 
$$
\tau : \rK_2(E)\rightarrow
\rK_1(k)=k^{\times }
$$
 is defined by the rule that
\begin{equation}\label{explicitTame}
\tau (\{a,b\}) =( -1) ^{v(a)v(b)
}\frac{a^{v(b)}}{b^{v(a) }}\bmod {\mathfrak m}.
\end{equation}

\begin{proposition}\label{tameprop}
Assume   that $R=F$ is a field and
$A=F{\lps t\rps}$, so that $A_t=F\llps t\lrps$ is the field of Laurent power series,
and that we take $G=\{1\}$. Then
$$
\partial : \rK_2(F\llps t\lrps)\to \rK_1(F)=F^*
$$
constructed above is equal to the inverse of
the tame symbol, \emph{i.e.} $\partial=\tau^{-1}$.  
\end{proposition}

\begin{proof} See \cite{KapranovLocalRR} for a very similar statement. We set $E=F\llps t\lrps$, $\O=F{\lps t\rps}$ and suppose $m\geq 3$. Using the universality of the Steinberg central extension of the perfect group ${\rm SL}_m(E)$ we see that the conclusion is equivalent to the following statement:  
Our  central extension 
\begin{equation}\label{ext3}
1\to F^*\to \Hh(E^m)\to {\rm SL}_m(E)\to 1
\end{equation}
above is isomorphic to the central extension $\widetilde\Hh(E^m)$ obtained  by pushing out 
  the Steinberg extension by $\tau^{-1}: \rK_2(E)\to F^*$. It follows from \cite[Theorem (12.24)]{Garland}
and \cite[Lemma 8.4]{Moore} (this last reference shows that a 
``Steinberg $2$-cocycle" is determined by its restriction to 
the subgroup ${\rm diag}(a,1,\ldots ,1, a^{-1})$ of the  diagonal torus of 
${\rm SL}_m(E)$ that corresponds to the long root) that the extension 
$\widetilde\Hh(E^m)$ is isomorphic to the canonical central $F^*$-extension of the loop group 
${\rm SL}_m(F\llps t\lrps)$ which appears in the theory of affine Kac-Moody algebras. 
It is well-known (cf. \cite{ACKac}, \cite{FaltingsLoops}) that this ``Kac-Moody extension" is the central extension $\widetilde\Hh(E^m)$ given by pairs $(g, \alpha)$ with 
$g\in {\rm SL}_m(E)$ and $\alpha$ a generator of the 
(graded) line bundle $\langle g\cdot \O^m\, |\, \O^m\rangle$,
where for a $\O$-lattice $L\subset E^m$, we set
$$
\langle L\, |\, \O^m\rangle:=\det (L/t^N\O^m)\otimes \det (\O^m/t^N\O^m)^{-1}\  
$$
for $t^N\O^m\subset L$.  The group law is given by $(g, \alpha )\cdot (h, \beta)=(gh, g(\beta)\otimes\alpha )$ (compare to the previous paragraph).  
Our central extension $\Hh(E^m)$ is obtained
in a similar manner using $\langle  \O^m\cdot g^{-1}\, |\, \O^m\rangle$.
Both extensions  $\Hh(E^m)$ and $\widetilde\Hh(E^m)$ split over ${\rm SL}_m(F{\lps t\rps})$.
By the argument of \cite[Proof of (5.4.5)]{KapranovLocalRR} we see that
it is enough to show that the two extensions  agree on the diagonal  maximal torus $T(F\llps t\lrps)$
of ${\rm SL}_m(F\llps t\lrps)$. This fact now follows easily 
from $\tau(a^{-1}, b^{-1})=\tau(a, b)$ and the above.
\end{proof}
 
 \subsubsection{}\label{representations} Here we suppose that
 $ \rho : R[G]\to {\rm M}_{d}(F)$ is a  representation over $F$
 where $F$ is a field of characteristic zero.
 Base-changing by $\rho$ gives an exact functor $M\mapsto \rho (M)={\rm M}_{d}(F)\otimes_{\rho , R[G]}M$
from finitely generated projective $R[G]$-modules
to ${\rm M}_{d}(F)$-modules. Let $e$ be an indecomposable idempotent of ${\rm M}_{d}(F)$.
Multiplying by $e$ gives an equivalence of categories between ${\rm M}_{d}(F)$-modules
and $F$-vector spaces. This gives an equivalence of Picard categories
$V({\rm M}_{d}(F))\xrightarrow{\sim} V(F)={\rm Pic}^\Z(F)$.
Sending $M$ to $e \cdot \rho (M)$ gives an exact functor
to $F$-vector spaces which induces an additive functor
$$
\rho : V(R[G])\to V(F).
$$
This induces on automorphisms of the
identity object,  the  determinant  (norm)
 $$
 N(\rho): \rK_1(R[G])\to F^*.
 $$

Recall that we take $A=R[t]$, or $A=R{\lps t\rps}$.
For simplicity, let us discuss $A=R{\lps t\rps}$.
Notice that if $L$ is an $A[G]$-lattice in $A_t[G]^n$,
then $e \cdot ({\rm M}_{d}(F)\otimes_{\rho , R[G]}L))$ is an
$F{\lps t\rps}$-lattice in $F\llps t\lrps^{nd}$.
Our construction for $G=\{1\}$, gives
a determinant theory $\delta_F$ for $F{\lps t\rps}$-lattices
in $F\llps t\lrps^{nd}$. We can see that for each
$A[G]$-lattice $L$ we have
canonical isomorphisms
$$
\rho  ( \delta(L))\to \delta_{F}(e\cdot ({\rm M}_{d}(F)\otimes_{\rho , R[G]}L))
$$
in $V(F)$. We obtain a commutative diagram of central extensions
\begin{equation}\label{commext}
\begin{matrix}
1&\to & \rK_1(R[G])&\to &\Hh(R\llps t\lrps)[G]^n)&\to &\GL_n'(R\llps t\lrps[G])&\to &1&\ \ \ \ \ \ \\
&&\hbox{\footnotesize {$N(\rho)$}}\downarrow \ \ \    &&\downarrow & &  \downarrow\hbox{\footnotesize{$\rho$}} &&\ \ \ \ \ \ \\
1&\to &F^*&\to &\Hh(F\llps t\lrps^{nd})&\to &\GL'_{nd}(F\llps t\lrps)&\to &1&\ \ \ \ \ \
\end{matrix}
\end{equation}
where  the bottom row can also be identified with   the extension of \cite{ACKac} as in the proof
of Proposition \ref{tameprop}.

\subsection{Some $p$-adic limits}\label{3.5}

In this section, we assume that $R$ is a commutative ring
in which $p$ is a non-unit and set $R_m=R/p^mR$. We suppose that $R_m$ is, for each $m$, a finite ring.
Recall $\hat R=\varprojlim_mR_m$ is the $p$-adic completion of $R$.
Set
\begin{equation*}
\widehat{\rK}_{i} ( R[G]) =\varprojlim_m\rK_{i}(R_{m}[G]) \ \ \text{for }i=0, 1, 2.
\end{equation*}%
 Observe that
\begin{equation}
\widehat{\rK}_{i}(\hat R[G])
=\varprojlim_m\rK_{i}(\hat{R}_{m}[ G])
=\varprojlim_m\rK_{i} ( R_{m}[G]) =\widehat{\rK}
_{i}(R[G]).
\end{equation}
Notice that the natural maps $\rK_1(R_{m+1}[G])\to \rK_1(R_m[G])$
are surjective since $p^mR_{m+1}[G]$ is nilpotent in $R_{m+1}[G]$.
Since $R_m$ is finite for each $m$
the natural map
$$
\rK_1(\hat R[ G])\xrightarrow{\sim } \varprojlim_m \rK_1(R_m[G])=\widehat {\rK}_1(R[G])
$$
is an isomorphism (\cite{FukayaKatoConj}).
Similarly, we have  $\rK_0(R_{m+1}[G])\xrightarrow{\sim} \rK_0(R_m[G])$ and so
$$
\rK_0(\hat R[ G])\xrightarrow{\sim } \widehat {\rK}_0(R[G]).
$$
Tensoring $P\mapsto R_m\otimes_{R_{m+1} }P$  induces
an additive functor
\begin{equation}
V(R_{m+1}[G])\xrightarrow{\ r_m\ } V(R_m [G])\ .
\end{equation}

\subsubsection{} \label{3e1}

 Set $A=R{\lps t\rps}$, $M_0=R[G]^n$.
Recall
$
  \hat R{\ldb t\rdb}=\varprojlim_mR_m\llps t\lrps
$.
Set $\GL'_n(\hat R{\ldb t\rdb}[G]):=\varprojlim_m \GL_n'(R_m\llps t\lrps G)$,
where $\GL_n'(R_m\llps t\lrps G)={\rm ker}(\GL_n(R_m\llps t\lrps G)\to \rK_0(R_m[G]))$.
We can see that this limit is a subgroup of $\GL_n(\hat R{\ldb t\rdb}G)$ that contains
$\GL'_n(R\llps t\lrps[G])$.
By applying the above (and Mittag-Leffler) we see that, for each $n$, there is a  central extension
\begin{equation}\label{cextCompl}
1\to \rK_1(\hat R [G]) \to \hat\Hh(\hat R{\ldb t\rdb}[G]^n)\to \GL'_n(\hat R{\ldb t\rdb}G)\to 1
\end{equation}
where $\hat\Hh(\hat R{\ldb t\rdb}[G]^n)=\varprojlim_m\Hh(R_m\llps t\lrps[G]^n)$.
This extension restricts to (\ref{anotherext??}) after
pulling back via the inclusion $\GL'_n(R\llps t\lrps[G])\hookrightarrow \GL'_n(\hat R{\ldb t\rdb}[G])$.
The extensions (\ref{cextCompl}) are compatible for various $n$.  

By restricting   to
the commutator subgroup $\rE(\hat R{\ldb t\rdb}[G])\subset \GL'(\hat R{\ldb t\rdb}[G])$
we obtain a central extension
\begin{equation}\label{cextComplSt}
1\to \rK_1(\hat R [G]) \to \hat\Hh(\hat R{\ldb t\rdb}[G])_E \to \rE(\hat R{\ldb t\rdb}G)\to 1
\end{equation}
and the argument using the universality of the Steinberg extension
now gives
$$
\hat\partial_R: \rK_2(\hat R{\ldb t\rdb}[G])\to \rK_1(\hat R[G]).
$$

\subsubsection{}  When $R=\hat R$ is a $p$-adically complete discrete valuation ring, we can 
also construct the group $\hat\Hh_n(R{\ldb t\rdb}[G])$ as follows:
Given $ g=( g_m)_m \in \GL'_n(R{\ldb t\rdb}G)$ consider the complex of 
$R[G]$-modules
\begin{equation}
\hat C_R(g): (R\langle\!\langle t^{-1}\rangle\!\rangle G)^n\oplus (R{\lps t\rps}G)^n\cdot  g\xrightarrow{\ \ } 
(R{\ldb t\rdb}G)^n.
\end{equation}
Now for each $m$, we have the complex 
\begin{equation}
C_m( g): (R_m[t^{-1}]G)^n\oplus (R_m{\lps t\rps}G)^n\cdot g_m\xrightarrow{\ \ } 
(R_m\llps t\lrps G)^n
\end{equation}
and $C_m( g)=R_m\otimes_{R_{m+1}}C_{m+1}(g)$.
By the above, for each $m$, $C_m( g)$ is represented by a perfect complex $P_m(g)$
of $R_m[G]$-modules. Using a standard argument  (see \cite[Lemma VI.13.13]{MilneEtaBook}), we can   find 
such perfect 
complexes $P_m( g)$ which support quasi-isomorphisms  
$R_{m}\otimes_{R_{m+1}}P_{m+1}( g)\xrightarrow{\sim} P_m( g)$;
then
$$
\hat P( g):=\varprojlim_m P_m( g)
$$
is a perfect complex of $R[G]$-modules that represents  $\hat C_R( g)$.
(In fact, $\hat C_R(g)$ represents the cohomology of a locally free $\O_{\mathbb{P}^1}[G]$-module
over $\mathbb{P}^1_R$ obtained by patching as in \cite{HarbaterPatchBranch}.)
Therefore, $[\hat C( g))]=[\hat P( g)]$ makes sense in $V(R[ G])$.
For  $\GL'_n(R{\ldb t\rdb})=\varprojlim_m\GL_n'(R_m\llps t\lrps G)$  we can now
see that 
$$
 \hat\Hh( R{\ldb t\rdb}[G]^n)\simeq \{( g, \phi_g)\ |\ g\in \GL'_n(R{\ldb t\rdb}G), \phi_g: [R[G]]=
[\hat C_R(1)]\xrightarrow{\sim} [\hat C_R(g^{-1})]\}\ .
$$
By arguing as in \S \ref{split}, we can now see that the extension 
(\ref{cextCompl}) splits over the subgroup $\GL_n(R\langle\!\langle t^{-1}\rangle\!\rangle[G])$.

\subsubsection{} 
\label{3.5.2}

Suppose here that $R=\hat R$ is a $p$-adically complete discrete valuation ring with valuation $v$, 
fraction field $F$ of characteristic zero
and finite residue field $k$. We will assume that $F[G]$ is split
\[
F[G]\simeq \prod_i {\rm M}_{m_i}(Z_i).
\]
The ring $R{\ldb t\rdb}$ is also a $p$-adically complete discrete valuation ring with residue field $k\llps t\lrps$.
Recall $F{\ldb t\rdb}$ is the fraction field of
$R{\ldb t\rdb}$. Here we explain how we can also construct a central extension of
$\rE_n(F{\ldb t\rdb}[G])$ by
$$
\overline \rK_1(F[G])=\varprojlim_m\rK_1(F[G])/{\rm Im}(\rK_1(R[G], (p^m)).
$$
Notice here that under our assumptions we have
\begin{equation}
\rK_1(F[G])\xrightarrow{\sim}\overline \rK_1(F[G]).
\end{equation}
and we can identify $\rK_1(F[G])$ and $\overline \rK_1(F[G])$.

In what follows, we denote by $\GL^*_n(F{\ldb t\rdb}[G])$ the subgroup
of  $g\in \GL_n(F{\ldb t\rdb}[G])$ with constant  ${\rm Det}$, \emph{i.e.}    ${\rm Det}(g)\in \prod_i Z_i^\times$. 

\begin{lemma}\label{lemmaApprox*}
a) If $g\in  \GL_n(F{\ldb t\rdb}[G])$ and $m>>0$, we can find $g_m\in \GL_n(F\otimes_RR\llps t\lrps[G])$,
$u_m\in  \GL_n(R{\ldb t\rdb}[G], (p)^m)$, such that $g=u_mg_m$.

b) If in addition $g \in \GL^*_n(F{\ldb t\rdb}[G])$
then, for $m>>0$, we can write $g=u_m g_m $ 
with $g_m\in \GL^*_n(F\otimes_R R\llps t\lrps[G])\subset \GL_n'(F\llps t\lrps[G])$, $u_m\in \GL^*_n(R{\ldb t\rdb}[G], (p)^m)$. 
\end{lemma}

\begin{proof}
Let $\mathfrak M=\prod_i{\rm M}_{m_i}(\O_{Z_i})$ be a maximal order in $F[G]=\prod_i {\rm M}_{m_i}(Z_i)$;
Since we have $\GL_n({\mathfrak M}{\ldb t\rdb}, (p)^a)\subset \GL_n(R{\ldb t\rdb}[G], (p)^m)$
for $a>>m$ (see the  proof of Lemma \ref{lemmaApprox}) we can reduce, by Morita equivalence, 
the proof of (a) to the case $G=\{1\}$. Now use the Bruhat decomposition
\[
\GL_n(F{\ldb t\rdb})=\bigcup_{(s_1,\ldots, s_n)\in \Z^n} \GL_n(R{\ldb t\rdb})\cdot {\rm diag}(\pi^{s_1},\ldots , \pi^{s_n})\cdot 
 \GL_n(R{\ldb t\rdb})
\]
and the fact that $\GL_n(R\llps t\lrps)$ is dense in $\GL_n(R{\ldb t\rdb})$ to deduce that $\GL_n(F\otimes_R R\llps t\lrps)$
is dense in $\GL_n(F{\ldb t\rdb})$; part (a) then follows. Part (b) then follows by part (a) and an argument as 
in the  proof of Lemma \ref{lemmaApprox}. 
\end{proof}

Starting with $g \in \GL^*_n(F{\ldb t\rdb}[G])$ write $g=u_m g_m $ as in Lemma \ref{lemmaApprox*} (b) 
and consider $\V_m:=\V_{g_m^{-1}}=\delta(L_0\cdot g^{-1}_m)-\delta(L_0)$   in $V(F[G])$.
Consider the Picard category $V_m(F[G])$ which has the same objects
as $V(F[G])$ and morphisms equivalence classes of arrows in $V(F[G])$, where $a, a': A\to B$
are equivalent if $a'\cdot a^{-1}\in {\rm Im}(\rK_1(R[G], (p^m)))\subset \rK_1(F[G])$.
Similarly, we can define $V_m(R[G])$.
There are additive functors $q_m: V_{m+1}(F[G])\to V_m(F[G])$.
The object $\V_m$ can be made, up to unique isomorphism in $V_m(F[G])$,
to be independent of the choice of $g_m$; Suppose another choice $g'_m$
gives $\V'_m$. Now  the central extension structure gives a canonical arrow
$$
\V_m\to \V'_m+\V_u
$$
for some $u=g_m'g_m^{-1}\in \GL^*_n(R{\ldb t\rdb}[G], (p^m))$.
We can think of $\V_u$ as an object of $V(R[G])$;
it supports a unique trivialization in $V_m(R[G])$
and this provides us with a fixed choice of a trivialization of $\V_u$
in $V_m(F[G])$. Hence, we have given an arrow
$\V_m\to \V'_m$ in $V_m(F[G])$ between any two   choices $\V_m$, $\V_m'$.
These arrows satisfy the obvious composition
condition when we are dealing with three choices $\V_m$,
$\V'_m$, $\V''_m$.

We can now consider  the group of pairs
$$
 \Hh_n(F{\ldb t\rdb}[G])_m= \{(g, \phi_m)\ |\ g\in \rE_n(F{\ldb t\rdb}[G]),  \ \phi_m:   [0]\to \V_m \hbox{\rm \ in\ } V_m(F[G])\}
$$
with  group law as in \S \ref{332}.
This gives   central extensions
\begin{equation}\label{extm}
1\to   \rK_1(F[G])/{\rm Im}(\rK_1(R[G], (p^m)))\to \Hh_n(F{\ldb t\rdb}[G])_m\to \rE_n(F{\ldb t\rdb}[G]) \to 1.
\end{equation}
The inverse limit of these, consisting of  $\{(g, (\phi_m)_m)\}$
 such that  $q_m(\phi_{m+1})=\phi_m$, for all $m$,
 gives the desired central extension
\begin{equation}\label{exthat}
1\to  \rK_1(F[G])=\overline{ \rK}_1(F[G])\to \hat\Hh_n(F{\ldb t\rdb}[G])\to \rE_n(F{\ldb t\rdb}[G]) \to 1.
\end{equation}
This provides us with
\begin{equation}
\hat\partial_F: \rK_2(F{\ldb t\rdb}[G])\to  \rK_1(F[G]).
\end{equation}
 Our constructions show that the diagram
\begin{equation}
\begin{matrix}
\rK_2(R{\ldb t\rdb}[G])&\xrightarrow{\hat\partial_{R, m}}  & {\rK}_1(R[G])/{\rm Im}(\rK_1(R[G], (p^m)))\\
\downarrow && \downarrow\\
\rK_2(F{\ldb t\rdb}[G])&\xrightarrow{\hat\partial_{F, m}}  &   {\rK}_1(F[G])/{\rm Im}(\rK_1(R[G], (p^m))).\\
\end{matrix}
\end{equation} commutes. Here   the vertical arrows are given by base change, and $\hat\partial_{R, m}$
is obtained from $\hat\partial_R$ of the previous paragraph.
 Therefore, $\hat\partial_{F, m}$ vanishes on ${\rm Im}(\rK_2(R{\ldb t\rdb}[G], (p^m)))$.

\subsubsection{} \label{3e3} Suppose $g\in \GL_n(F{\ldb t\rdb}[G])$ and consider the complex
\begin{equation*}
\hat C_F(g): (F\{t^{-1}\}[G])^n\oplus (R{\lps t\rps}\otimes_R F[G])^n\cdot g\rightarrow (F{\ldb t\rdb}[G])^n
\end{equation*}
of $F[G]$-modules. If $g\in \GL_n(R{\ldb t\rdb}[G])$, then $\hat C_F(g)=F\otimes_R \hat C_R(g)$
with $\hat C_R(g)$ the perfect complex of $R[G]$-modules considered in \S \ref{3e1}.
In general, we have

\begin{proposition}\label{rigidperfect}
For $g\in \GL_n(F{\ldb t\rdb}[G])$,  $\hat C_F(g)$ is a perfect complex of $F[G]$-modules.
\end{proposition}

\begin{proof}
By Morita equivalence, it is enough to consider the case $G=\{1\}$.
For simplicity, set $A_+=R{\lps t\rps}\otimes_R F$, $A_-=F\{t^{-1}\}$, $A_0=R\llps t\lrps\otimes_RF$.
These are  $F$-subalgebras of the field $F{\ldb t\rdb}$. We have, $F{\ldb t\rdb}=A_-+A_+$, the algebras
$A_+$, $A_-$ are complete,  $A_0$  is dense in $F{\ldb t\rdb}$. Set $G_\pm=\GL_n(A_\pm)$
and $G_0=\GL_n(A_0)$. By Lemma \ref{lemmaApprox*} (a) and its proof, 
$G_0$ is dense in $\GL_n(F{\ldb t\rdb})$. Notice that we can write $a\in F{\ldb t\rdb}$ as a sum $a=a_-+a_+$ with $a_\pm \in A_\pm$ and $|a_+|$, $|a_-|\leq |a|$, where $|\ .|$ denotes the $p$-adic absolute value
on $F{\ldb t\rdb}$. It now follows by the ultrametric version of Cartan's Lemma (see for example \cite[III 6.3]{FvdPut}
III 6.3, or  \cite{HaranVol})
that we can write $g=g_0\cdot g_+\cdot g_-$.
Then $g=h_0\cdot g_-$
with $h_0=g_0\cdot g_+\in \GL_n(A_0)$ since $A_+\subset A_0$. 
We have $\hat C_F(g)=\hat C_F(h_0)$; this reduces 
to showing the proposition for $g\in \GL_n(A_0)$.
Now also observe that if $g'= g_+\cdot g$ with $g_+\in \GL_n(A_+)$, then 
$\hat C_F(g')\simeq \hat C_F(g)$; this allows us to 
restrict attention to the cosets $\GL_n(A_+)\backslash \GL_n(A_0)$;
we can see that these parametrize free rank $n$ $A_+$-submodules $Q$ of $A_0^n$ 
with $t^NA^n_+\subset Q\subset t^{-N}A^n_+$ for some $N$. Since $t^{-N}A_+/t^NA_+\simeq t^{-N}F[t]/t^NF[t]$,
the usual argument shows that each coset $ \GL_n(A_+)\cdot g$ has a representative
$h$ in $\GL_n(F[t, t^{-1}])$. Hence, $\hat C_F(g)$
is isomorphic to the $p$-adic completion of 
$$
F[t^{-1}]^n\oplus A_+^n\cdot h\to A_0^n
$$
and this is perfect: Indeed, we can compare this to the complex for $h=1$
which is quasi-isomorphic to $F^n$ and hence is perfect; since the quotient of any two $A_+$ lattices
is a finite $F$-vector space the result follows. 
\end{proof}
\smallskip

We can now see that $\hat C_F(g)$ gives the cohomology of a 
rank $n$-vector bundle over $\mathbb{P}^1_F$; this is the bundle obtained
by glueing by the element $h\in \GL_n(F[t, t^{-1}])$ in the
above proof. Now we can construct
\begin{equation}
1\to \rK_1(F[G])\to \hat\Hh_n(F{\ldb t\rdb}[G])\to \GL_n^*(F{\ldb t\rdb}[G])\to 1
\end{equation}
by setting 
\begin{equation}
\hat\Hh_n(F{\ldb t\rdb}[G]):=\{(g, \phi_g)\ |\ g\in \GL_n^*(F{\ldb t\rdb}[G]), \phi_g: [F[G]]\xrightarrow{\sim} [\hat C_F(g^{-1})]\}.
\end{equation}
Now we can check that the restriction of this extension over $E(F{\ldb t\rdb}[G])$ is isomorphic to the extension (\ref{exthat})
constructed in the previous paragraph.
We can see from the above that this extension splits over the subgroup $\GL_n(F\{t^{-1}\}[G])$ and
over $\GL_n(F\otimes_RR{\lps t\rps}[G])$.

\subsection{Kato's Residue map}\label{3.4.3}

 Recall that for $q\geq 1$ and for
a field $N$, $\rK_{q}( N) $ is generated by symbols $\{
a_{1},\ldots ,a_{q}\} $ with  $a_{i}\in N^{\times }$. Suppose that $N$ supports a normalized additive discrete valuation $v$.
Let $\rK_{\ast}(N) $ denote the graded ring $\oplus _{q\geq 0}\rK_{q}(N)$. For $n\geq 1$, 
we let ${\rm U}^{n}\rK_{\ast }( N) =\oplus
_{q\geq 0}{\rm U}^{n}\rK_{q}( N) $ denote the graded $\rK_{\ast }(
N) $-ideal generated by elements $a\in N^{\times }=\rK_{1}(
N) $ with $v( a-1) \geq n$. We then define
\begin{equation*}
\ti{\rK}_{q}( N) =\lim_{\leftarrow n}\frac{\rK_{q}(
N) }{{\rm U}^{n}\rK_{q}( N) }.
\end{equation*}

From \cite[Lemma 1]{KatoRes} we quote: 

\begin{lemma}\label{lemma12}
 Let $q\geq 0$, $n\geq 1$. If $\hat{N}$ denotes the completion of $N$ with
respect to the valuation $v$, then the natural map $N\rightarrow \hat{N}$
induces isomorphisms:
\begin{equation*}
\frac{\rK_{q}( N) }{{\rm U}^{n}\rK_{q}( N) }\xrightarrow{\sim} \frac{\rK_{q}( \hat{N}) }{{\rm U}^{n}\rK_{q}( \hat{N}) },\qquad
\ti{\rK}_{q}( N) \xrightarrow{\sim} \ti{\rK}_{q}( \hat{N}) .
\end{equation*}
 \end{lemma}

\subsubsection{}
We consider the case where $F$ is a finite extension of the $p$-adic field $\Q_{p}$ with valuation ring $R$ and valuation ideal  $P=\pi R$.  Let $N$ denote the field of fractions of the ring of power series $R{\lps t\rps} $ for an indeterminate $t$, endowed with the
discrete valuation associated to the height one prime ideal $\pi R{\lps t\rps}$. Then the completion $\hat{N}$ is the field $ 
F{\ldb t\rdb}$ described previously. In   \cite[Sec. 1]{KatoRes} Kato defines a map 
\begin{equation}
\mathrm{Res} :   \ti{\rK}_{2}( F{\ldb t\rdb}) \rightarrow
\ti{\rK}_{1}( F) =F^{\times}.
\end{equation}
 By composing with the
natural map $\rK_{2}( F{\ldb t\rdb}) \rightarrow \widetilde{\rK}
_{2}( F{\ldb t\rdb}) $ we obtain a map which we also call $\mathrm{
Res}$:  $\rK_{2}( F{\ldb t\rdb}) \rightarrow F^{\times }$.

From \cite[Theorem 1]{KatoRes} we know that $\mathrm{Res}$ is continuous with
respect to the topology given by the subgroups $\{ {\rm U}^{n}\rK_{2}( N) \} $.
In
fact $\mathrm{Res}( {\rm U}^{n}\rK_{2}( F{\ldb t\rdb}) ) \subset
1+P^{n}$. We also have 
\begin{equation*}
\mathrm{Res}(\{ a,t\}) =a, \quad \hbox{\rm for $a\in F^{\times }$},
\end{equation*}%
and $\mathrm{Res}$ annihilates the image of $\rK_{2}(F\otimes_R  R[[ t]]) $ in $\rK_{2}( F{\ldb t\rdb})$. 
Observe that $\partial_F$ on the image of $\rK_2 (F\otimes_R R \llps t\lrps)$ in $\rK_2 (F \llps t\lrps)$ is the negative of the tame symbol by Proposition \ref{tameprop}.
 We now explain why  $\rm Res$ on the image of $\rK_2 (F\otimes_R R \llps t\lrps)$ in 
$\rK_2 (F {\ldb t\rdb})$ agrees with the tame symbol: From the localization sequence (see \cite[p. 294]{RosenbergBook}), noting that an arbitrary non-zero element 
$x\in \rK_2(F\otimes_RR\llps t\lrps)$  can be written as $x = t^n y$ with  $y\in F\otimes_RR{\lps t\rps}$, we have an exact sequence
$$
0\to {\rm Im}(\rK_2(F\otimes_RR{\lps t\rps}))\to \rK_2(F\otimes_RR\llps t\lrps)\to \rK_1(\mathcal{C})
$$
 where $\mathcal{C}$ is the category of perfect $F\otimes_RR\llps t\lrps)$-complexes whose homology is killed by a power of $t$.  In the usual way we 
get  $\rK_1(\mathcal{C})\cong \rK_1(F)$ and the above sequence is split by mapping $a\mapsto a\cup t$.
 Therefore we get a decomposition
 $\rK_2(F\otimes_RR\llps t\lrps)  =  {\rm Im}(\rK_2(F\otimes_RR{\lps t\rps}))\oplus \rK_1(F)$
 and we can see that ${\rm Res}$ on $\rK_2(F\otimes_RR\llps t\lrps)$ is the tame symbol.
 
We conclude 
that $\mathrm{Res}$ is compatible with the tame symbol in the sense that we
have a commutative diagram
\begin{equation}\label{eq49}
\begin{array}{ccc}
\rK_{2}( F( ( t) )) & \xrightarrow{\tau_F=\partial_F^{-1}} &
\rK_{1}( F) \\
\uparrow &  & \ \ \ \uparrow {\rm id} \\
\rK_{2}( F\otimes_R R(( t))) & \xrightarrow{\ \ \ \ } &
\rK_{1}( F) \\
\downarrow &  & \downarrow \\
\rK_{2}( F{\ldb t\rdb}) & \xrightarrow{\rm Res}& \widehat{\rK}_{1}( F).
\end{array}
\end{equation}
induced by the natural inclusions.  

\subsubsection{Compatibility of $\hat\partial^{-1}$ with Res}\label{3.4.3b} 
Combining Proposition \ref{tameprop} with (\ref{eq49}) above we will show:

\begin{proposition}\label{prop13}
The maps $\hat\partial^{-1}_F $ and $\mathrm{Res}$ agree on $\rK_{2}(
F{\ldb t\rdb})$.
\end{proposition}

\begin{proof} Let $\hat x\in \rK_{2}( F{\ldb t\rdb})$. We omit the subscript $F$. We shall show
that for arbitrary $n>0$ we have
\begin{equation}\label{eq50}
\bar\partial ( \hat{x}) \equiv \mathrm{Res}( \hat{x})^{-1} {\rm \ mod\ }\pi^{n}R
\end{equation}
so that we can conclude
\begin{equation}\label{eq51}
\partial ( \hat{x}) =\mathrm{Res}( \hat{x})^{-1} .
\end{equation}
By Matsumoto's theorem, it will suffice to take $\hat{x}=\{\hat x_1, \hat x_2\}=\hat{x}
_{1}\cup \hat{x}_{2}$ with $\hat{x}_{i}\in F{\ldb t\rdb}^{\times }.$
Since $F{\ldb t\rdb}$ is the field of fractions of the discrete valuation ring $R{\ldb t\rdb}$
with valuation ideal $\pi R{\ldb t\rdb}$ we can write each
element $\hat{x}_{i}$ in the form
\begin{equation*}
\hat{x}_{i}=\pi ^{M_{i}}\hat{s}_{i}\ \text{with\ }\hat{s}_{i}\in
R{\ldb t\rdb}^{\times }.
\end{equation*}
Since $F\otimes_R R( ( t)) $ is dense in $F{\ldb t\rdb}$,
for any $n>0$, we  can find $x_{i}^{( n) }\in ( F\otimes_RR( ( t) ) ) ^{\times }$ and $r_{i}^{(
n) }\in R{\ldb t\rdb}$ so that
\begin{equation*}
\pi ^{M_{i}}\hat{s}_{i}=\hat{x}_{i}=x_{i}^{( n) }(
1+\pi^{n}r_{i}^{( n) }) ^{-1}.
\end{equation*}
By Proposition \ref{tameprop} and the above we know that both $\hat\partial^{-1}_F $ and $\mathrm{Res}$ 
restrict to the tame symbol on $\rK_{2}( F\otimes_R R((t) ) ) $; so, in order to prove (\ref{eq50}), it will suffice to
show for $n>0$ that
\begin{equation}\label{eq52}
\hat\partial ( \hat{x}) \equiv \partial ( x_{1}^{(
n) }\cup x_{2}^{( n) }) {\rm \ mod\ }\pi^{n}   
\end{equation}
\begin{equation}\label{eq53}
\mathrm{Res}( \hat{x}) \equiv \mathrm{Res}(
x_{1}^{( n) }\cup x_{2}^{( n)}) {\rm \ mod\ }\pi^{n};
\end{equation}%
indeed, then we can then conclude that the two right-hand terms in (\ref{eq52}) and
(\ref{eq53}) are equal and this will then give (\ref{eq50}).

We have the equalities in $\rK_{2}( F{\ldb t\rdb})$:
\begin{eqnarray*}
x_{1}^{( n) }\cup x_{2}^{( n) } &=&\hat{x}
_{1}( 1+\pi^{n}r_{1}^{( n) }) \cup \hat{x}_{2}(
1+\pi^{n}r_{2}^{( n) }) \\
&=&[\hat{x}_{1}\cup \hat{x}_{2}]\cdot [( 1+\pi^{n}r_{1}^{( n)
}) \cup \hat{x}_{2}] \cdot[\hat{x}_{1}\cup (
1+\pi^{n}r_{2}^{( n) })]\cdot [ ( 1+\pi^{n}r_{1}^{( n)
}) \cup ( 1+\pi^{n}r_{2}^{( n) })].
\end{eqnarray*}
Since the three last terms above
 belong to  ${\rm U}^{n}\rK_{2}(
F{\ldb t\rdb})$,
we conclude that
\begin{equation*}
{\rm Res}(\hat x)\equiv \mathrm{Res}(x_{1} \cup x_{2}  )\equiv \mathrm{Res}(x_{1}^{( n) }\cup x_{2}^{( n) } )  {\rm \ mod\ }\pi^{n}.
\end{equation*}
Now since $( 1+\pi^{n}r_{1}^{( n) }) \cup (
1+\pi^{n}r_{2}^{( n) }) \in {\rm Im}(\rK_{2} (
R{\ldb t\rdb}, \pi^{n}))$, from \S \ref{3.5.2} we know that
\begin{equation*}
\hat\partial ( ( 1+\pi^{n}r_{1}^{( n) }) \cup (
1+\pi^{n}r_{2}^{( n) }) ) \equiv 1\mbox{\rm \ mod\ }\pi^{n}.
\end{equation*}%
Next  we observe that
\begin{eqnarray*}
\hat{x}_{1}\cup ( 1+\pi^{n}r_{2}^{( n) }) &=&\pi
^{M_{1}}\hat{s}_{1}\cup ( 1+\pi^{n}r_{2}^{( n) }) \\
&=&(\pi ^{M_{1}}\cup ( 1+\pi^{n}r_{2}^{( n) })) +(\hat{s
}_{1}\cup ( 1+\pi^{n}r_{2}^{( n) })) .
\end{eqnarray*}
As previously, since $\hat{s
}_{1}\cup ( 1+\pi^{n}r_{2}^{( n) })\in {\rm Im}(\rK_{2}( R{\ldb t\rdb}, \pi^{n}))$, we know
that
\begin{equation*}
\hat\partial ( \hat{s
}_{1}\cup ( 1+\pi^{n}r_{2}^{( n) })) \equiv 1\mbox{\rm \ mod\ }\pi^{n}
\end{equation*}
and so, in order to prove the proposition, it will suffice to show:

\begin{lemma}\label{lemma14}
For $r\in R{\ldb t\rdb}$ we have $\hat\partial ( \pi \cup (
1+\pi r))=1 $.
\end{lemma}

\begin{proof} First we recall the construction of the cup product $\pi \cup u\in
\rK_{2}( F{\ldb t\rdb}) $ for $u\in F{\ldb t\rdb}^{\times }$. As in
Section 8 of \cite{MilnorK}, we form
\begin{equation*}
d=\left(
\begin{array}{ccc}
\pi & 0 & 0 \\
0 & \pi ^{-1} & 0 \\
0 & 0 & 1%
\end{array}%
\right) ,\quad \ e=\left(
\begin{array}{ccc}
u & 0 & 0 \\
0 & 1 & 0 \\
0 & 0 & u^{-1}%
\end{array}%
\right) \in {\rm SL}_{3}( F{\ldb t\rdb})
\end{equation*}%
and we note that these two elements commute in $\SL_{3}(
F{\ldb t\rdb})$; we then choose lifts $\tilde{d}$, $\tilde{e}$
in the Steinberg group of $F{\ldb t\rdb}$ and $\pi \cup u$ is defined to be the
commutator
\begin{equation*}
\pi \cup u=[\, \tilde{d}, \tilde{e}\,] \in \rK_{2}(
F{\ldb t\rdb}) .
\end{equation*}
From (\ref{exthat})  we have the exact sequence
\begin{equation}\label{eq54}
1\rightarrow \rK_{1}( F) \rightarrow  \hat{\Hh}(
F{\ldb t\rdb}) \rightarrow {\rm SL}( F{\ldb t\rdb}) \rightarrow 1
\end{equation} 
and we recall that elements of $ \hat{\Hh}( F{\ldb t\rdb}) $ are
coherent sequences of pairs $( g,\phi _{m}) $ with $g\in
{\rm SL} ( F{\ldb t\rdb})$, where we choose $g_{m}\in {\rm SL} (
F( ( t) ) ) $ with $gg_{m}^{-1}\in
1+\pi^{m}{\rm M}(R{\ldb t\rdb})$ and $\phi _{m}$ an isomorphism for $N>>0$
\begin{equation*}
\phi _{m}:\frac{L_{0}}{t^{N}L_{0}}\xrightarrow{\simeq} \frac{L_{0}g_{m}^{-1}}{t^{N}L_{0}}.
\end{equation*}
\smallskip
Recall that the elements of $\hat{\mathcal H}( F( ( t)) )$ multiply by the  rule
\begin{equation}\label{eq55}
\left( \gamma ,\phi _{\gamma }\right)\star  \left( \delta ,\phi _{\delta }\right)
=( \gamma \delta ,\phi _{\delta }^{\gamma }\circ \phi _{\gamma })
\end{equation}
where, writing $r( \gamma ) $ for right multiplication by $\gamma$,  $\phi _{\delta }^{\gamma }=r( \gamma ^{-1}) \circ \phi
_{\delta }\circ r( \gamma )$; thus, as seen previously in \S \ref{3d4}, we have
$
( \gamma ,\phi _{\gamma }) ^{-1}=( \gamma ^{-1},\phi
_{\gamma }^{-\gamma ^{-1}})
$.

We now set $u=1+\pi r$ for $r\in R{\ldb t\rdb}$; we fix a positive integer $m\geq 1$ and choose $r_{m}\in R( ( t)) $ with 
$r\equiv r_{m}{\rm \ mod\ }\pi ^{-1}\pi^{m}R$; we let $u_{m}=1+\pi r_{m}$ and
define $e_{m}$ to be the matrix
\begin{equation*}
e_{m}=\left(
\begin{array}{ccc}
u_{m} & 0 & 0 \\
0 & 1 & 0 \\
0 & 0 & u_{m}^{-1}%
\end{array}%
\right) .
\end{equation*}%
Then from the above we get
\begin{equation*}
\partial ( \pi \cup u_{m}) =( d,\phi _{d}) \star (
e_{m},\phi _{e_{m}}) \star ( d^{-1},\phi _{d}^{-d^{-1}}) \star (
e_{m}^{-1},\phi _{e_{m}}^{-e_{m}^{-1}})
\end{equation*}%
in $\rK_{1}( F)$.
However, from \S \ref{split} we know that the $\mathcal{H}$-sequence splits
on ${\rm SL}( F [ t] ] )$, and \textit{a fortiori} on ${\rm SL}( F)$, and so  we can take $\phi _{d}=1$.
 This
then gives
\begin{eqnarray*}
\bar\partial ( \pi \cup ( 1+\pi r) ) &=&\varprojlim_m \ ( d,1_{m}) \star ( e_{m},\phi _{e_{m}})\star  (
d^{-1},1_{m}) \star( e_{m}^{-1},\phi _{e_{m}}^{-e_{m}^{-1}}) \\
&=&\varprojlim_m\ ( de_{m},\phi^d_{e_{m}})\star  (
d^{-1}e_{m}^{-1},\phi _{e_{m}}^{-e_{m}^{-1}d^{-1}}) \\
&=&\varprojlim_m\ ( de_{m}d^{-1}e_{m}^{-1},  \phi _{e_{m}}^{-1}\phi
_{e_m}  ^{d}) \\
&=&\varprojlim_m\ ( 1,  \phi _{e_{m}}^{-1}\phi
_{e_m}  ^{d} )
\end{eqnarray*}%
and so it will suffice to show that for $m\geq 1$
\begin{equation}\label{eq56}
\phi _{e_{m}}^{-1}\phi _{e_{m}}^{d}\equiv 1{\rm \ mod\ }p^{m},\quad  \hbox{\rm
i.e that \ \ } \phi _{e_{m}}^{-1}\phi _{e_{m}}^{d}\in \rK_{1}(
R,\pi^{m}) .
\end{equation}

We set $\overline{R}=R/P$. Let $L_{1}( R) =R[ [ t] ] ^3$, $L_{1}( F) =F[ [ t] ] ^3$. For $\gamma \in {\rm SL}_3( F[ [ t]] ) $ we put $L_{\gamma }( F)
=L_{0}( F) \gamma ^{-1} $. Then we note the following:
\smallskip

1)  over $F$, we have the identifications
\begin{equation}\label{eq57}
\frac{L_{1}( F) }{ t^{N} L_{1}( F) }=\frac{L_{d}(
F) }{  t^{N} L_{1}(
F)  },\qquad \frac{L_{e_{m}d}( F) }{
t^{N}L_{1}( F) }=\frac{L_{de_{m}}( F) }{ t^{N} L_{1}( F)}=\frac{%
L_{e_{m}}( F) }{ t^{N}L_{1}( F)  };  
\end{equation}

2) because $e_{m}\equiv 1\mod \pi^{m}$ we have a canonical
isomorphism
\begin{equation*}
\frac{L_{e_{m}}( \overline{R}) }{ t^{N} L_{1}( \overline{R}) }=\frac{%
L_{1}( \overline{R}) }{ t^{N}L_{1}( \overline{R}) }.
\end{equation*}

Let $\{x_{i,j}\}$ for $1\leq i\leq 3$, $ 0\leq j<N$, be the natural $R$-basis
of $L_{1}( R) /  t^{N}L_{1}( R)$; let $\{\overline{x}_{i,j}\}$
for $1\leq i\leq 3$, $ 0\leq j<N$, be the images of the basis elements $%
\{x_{i,j}\}$ in $L_{1}( \overline{R}) /  t^{N}L_{1}( \overline{R}) $;
then $\{y_{i,j}=x_{ij}e_{m}^{-1}\}$ are elements in $L_{e_{m}}(
R) /  t^{N}L_{1}(R) $.  We consider the $R$-linear map $\phi:L_{1}( R) / t^{N}L_{1}( R)  \rightarrow L_{e_{m}}(
R) / t^{N}L_{1}(
R)$ given by $\phi ( x_{ij}) =y_{ij}$. By construction we have the commutative diagram:
\begin{equation}\label{eq58}
\begin{array}{ccc}
L_{1}( R) / t^{N}L_{1}( R) & \overset{\phi }{\rightarrow } &
L_{e_{m}}( R) / t^{N} L_{1}( R) \\
\downarrow &  & \downarrow \\
L_{1}( \overline{R}) / t^{N} L_{1}( \overline{R})& = & L_{1}(
\overline{R}) / t^{N}L_{1}(\overline{R}) .
\end{array}
\end{equation}%
Since there is an $\overline{R}$-isomorphism between $L_{1}( \overline{R})/ t^{N}L_{1}( \overline{R})$ and $L_{e_{m}}( \overline{R})/t^{N} L_{1}( \overline{R})$ they have the same $\overline{R}$-rank.  Therefore, by Nakayama's lemma,
$\phi$ is an isomorphism.

To conclude we evaluate $\phi ^{-1}\circ \phi ^{d}( x_{ij})$.
First observe that $
x_{ij}d=\pi ^{\varepsilon _{i}}x_{ij}$, where $\pi ^{\varepsilon
_{i}}=\pi ,\pi ^{-1},1$ if $i=1,2,3$. Therefore,
\begin{equation*}
y_{ij}d^{-1}=x_{ij}e_{m}^{-1}d^{-1}=x_{ij}d^{-1}e_{m}^{-1}=\pi
^{-\varepsilon _{i}}x_{ij}e_{m}^{-1}=\pi ^{-\varepsilon _{i}}y_{ij}.
\end{equation*}
We obtain
\begin{eqnarray*}
(\phi ^{-1}\circ \phi ^{d})( x_{ij}) &=&\phi^{-1}\circ  r( d^{-1}) \circ
\phi \circ r( d)  ( x_{ij}) \\
&=& \phi^{-1}\circ  r( d^{-1}) \circ
\phi (\pi^{\epsilon_i}x_{ij}) \\
&=& \phi^{-1}\circ  r( d^{-1}) (
 \pi^{\epsilon_i}y_{ij})\\
&=& \phi^{-1}(\pi^{\epsilon_i}y_{ij}d^{-1})\\
&=&x_{ij}.\ \ 
\end{eqnarray*}
which proves (\ref{eq56}) as desired.
\end{proof}

This now also concludes the proof of Proposition \ref{prop13}.
\end{proof}
\bigskip

\section{Pushdown maps and reciprocity laws.}

\setcounter{equation}{0}

In what follows, we will assume  that the group algebra $\Q[G]=\prod_{i}{\rm M}_{n_{i}}( Z_{i})$ splits as in
(\ref{splitgroupalgebraIntro}). The extension $Z_i/\Q$ is unramified at all 
places that do not divide the order of the group $G$.
By Morita equivalence  we have an isomorphism
\begin{equation*}
\rK_{1}( \Q[ G] ) \xrightarrow{\sim}\prod_{i}\rK_{1}(
Z_{i})  
=\prod_{i}Z_i^{\times }.
\end{equation*}

\subsection{The definition of pushdown}
\label{sec4.1}

Under the above assumptions we give one of the main constructions of this paper.

\subsubsection{} Assume that the projective regular arithmetic surface $Y\to \Spec(\Z)$
satisfies hypothesis (H) of the introduction.  
We fix a Parshin triple $\{\eta _{0},\eta _{1},\eta _{2}\} $ on $Y$. As seen previously in Section
\ref{section1}, $\hat{\O}_{Y,\eta _{1}\eta _{2}}$ is a finite product of discrete
valuation rings
\begin{equation}\label{eq59}
\what{\mathcal{O}}_{Y,\eta _{1}\eta _{2}}= \left\{
\begin{array}{c}
\prod_{\alpha }\Q_{p}( \eta _{1\alpha }) \lps t_{\alpha }\rps, \ \ \hbox{\rm if $\eta _{1}$ is horizontal}
\\ \prod_{\beta }W(k(\eta _{2\beta }))
\{\!\{t_{\beta }\}\!\},\ \hbox{\rm if $\eta _{1}$ is vertical.}
\end{array}
\right.
\end{equation}
and $\what{\mathcal{O}}_{Y,\eta _{0}\eta _{1}\eta _{2}}$ is a finite product of
fields
\begin{equation}\label{eq60}
\what{\mathcal{O}}_{Y,\eta _{0}\eta _{1}\eta _{2}}=\left\{
\begin{array}{c}
\prod\nolimits_{\alpha }\Q_{p}( \eta _{1\alpha }) (\!( t_{\alpha })\!),\ \hbox{\rm if $\eta _{1}$ is horizontal}
\\
\prod\nolimits_{\beta }\Q_{p}( \eta _{2\beta })
\{\!\{t_{\beta }\}\!\},\ \text{if }\eta _{1}\ \text{is\ vertical.}%
\end{array}%
\right.   
\end{equation}%
 
\subsubsection{}\label{ss:4a2} We define the push down 
\begin{equation}\label{eq63}
f_{\ast
\eta _{0}\eta _{1}\eta _{2}}:  {\rK}_{2}( \hat{\mathcal{O}}_{Y,\eta _{0}\eta _{1}\eta
_{2}}[G])\to   \rK_{1}( \Q_{p}[ G])
\end{equation}
as follows:

\begin{enumerate}

\item[i)]

  If $\eta _{1}$ is horizontal, we consider  the  composite
of  
$$
 {\rK}_{2}( \what{\mathcal{O}}_{Y,\eta _{0}\eta _{1}\eta
_{2}}[G]) =\prod\nolimits_{\alpha }\rK_{2}( \Q_{p}(
\eta _{1\alpha }) (\!( t_{\alpha })\!) [ G] )  \xrightarrow{\partial^{-1} }\prod\nolimits_{\alpha
}\rK_{1}( \Q_{p}( \eta _{1\alpha }) [ G]) 
$$
with 
\begin{equation*}
{\rm res}: \prod\nolimits_{\alpha }\rK_{1}( \Q_{p}( \eta _{1\alpha }) [ G]) \xrightarrow{\ \ } \rK_{1}( \Q_{p}[
G] )
\end{equation*}
where $\mathrm{res}$ is the restriction (norm) map given by viewing $
\prod_{\alpha }\Q_{p}( \eta _{1\alpha })[ G] $ as a finite dimensional $\Q_{p} [G] $-algebra.

\item[ii)] If $\eta _{1}$ is vertical,  then $f_{\ast
\eta _{0}\eta _{1}\eta _{2}}$ is the  composite
of  
$$
 {\rK}_{2}( \hat{\mathcal{O}}_{Y,\eta _{0}\eta _{1}\eta
_{2}}[G]) =\prod\nolimits_{\beta } {\rK}_{2}( \Q
_{p}( \eta _{2\beta }) \{\!\{t_{\beta }\}\!\} [G] )
\xrightarrow{\hat\partial^{-1}}\prod\nolimits_{\beta } {\rK}%
_{1}( \Q_{p}( \eta _{2\beta }) [ G])
$$
with
$$
\prod\nolimits_{\beta } {\rK}%
_{1}( \Q_{p}( \eta _{2\beta }) [ G]) \overset{\mathrm{res}}{\rightarrow } \rK_{1}( \Q_{p}[ G])
$$
where $\mathrm{res}$ is again the restriction map as above.
\end{enumerate}

Notice that we are using the {\sl inverses} $\partial^{-1}$, $\hat\partial^{-1}$ which by Proposition \ref{tameprop}
and Proposition \ref{prop13} agree with the tame symbol, resp. Kato's ${\rm Res}$ map.
Let us remark here that $f_{*\eta_0\eta_1\eta_2}$ is independent of the choice of uniformizers $t_\alpha$, $t_\beta$.
In the case that $\eta_1$ is horizontal, this follows from Proposition \ref{tameprop} and   the corresponding independence of the tame symbol, 
and in the case that $\eta_1$ is vertical from Proposition \ref{prop13} and \cite[\S 2]{KatoRes}.

We now consider the restriction of  $f_{\ast
\eta _{0}\eta _{1}\eta _{2}}$ to the image  of ${\rK}_{2}( \widehat{\mathcal{O}}_{Y,\eta
_{1}\eta _{2}}[G])$:

\begin{proposition}\label{thm15}
Let $\eta _{1}$ denote a codimension one point on $Y$, and let $\eta _{2}<\eta_1$
be a closed point of $Y$, with residue characteristic $p$. Then for $x$ in $ {\rK}_{2}( \hat{\mathcal{O}}_{Y,\eta _{1}\eta
_{2}}[G])^\flat$:
\begin{enumerate}
\item[(i)] $f_{\ast \eta _{0}\eta _{1}\eta _{2}}( x)=1$, if $\eta_1$ is horizontal,
\item[(ii)]  $f_{\ast \eta _{0}\eta _{1}\eta _{2}}( x)$ is in $ \rK_{1}( \Z_{p}[G] )^\flat$,
if $\eta _{1}$ is vertical. 
\end{enumerate}
\end{proposition}

\begin{proof} (i) In the horizontal case $\hat{\mathcal{O}}_{Y,\eta
_{1}\eta _{2}}=\prod_{\alpha }\Q_{p}( \eta
_{1\alpha }) [\![ t_{\alpha }]\!] $, where $\alpha
$ runs over the branches of the completion $\eta _{1}$ in a formal
neighborhood of $\Spec( \hat{\mathcal{O}}_{Y,\eta
_{2}})$. Then this restriction is induced by the composite
\begin{equation*}
{\rK}_{2}( \hat{\mathcal{O}}_{Y,\eta _{1}\eta
_{2}}[G]) = \prod\nolimits_{\alpha }\rK_{2}( \Q_{p}(
\eta _{1\alpha }) [\![ t_{\alpha }]\!] [ G] )  {\xrightarrow{\ \partial^{-1}\ } }\prod\nolimits_{\alpha}\rK_{1}( \Q_{p}( \eta _{1\alpha }) [ G]) {\xrightarrow{\rm res } }\rK_{1}( \Q_{p}[
G] ) ;
\end{equation*}
and so it is trivial since by (\ref{split}) 
the $\mathcal{H}$-sequence for each  $\Q_p(\eta_{1\alpha})((t_{\alpha}))[G])$
splits over  $\Q_p(\eta_{1\alpha})[\![t_{\alpha}]\!][G])$, so that $\partial $ is trivial on each 
factor $\rK_{2}(\Q_p(\eta_{1\alpha})[\![t_{\alpha}]\!][G])$.

(ii) In the vertical case we get  
\begin{eqnarray*}
 {\rK}_{2}( \widehat{\mathcal{O}}_{Y,\eta _{1}\eta
_{2}}[G]) =\prod\nolimits_{\beta } {\rK}_{2}( W(k(\eta _{2\beta })) \{\!\{t_{\beta }\}\!\}[ G] )\xrightarrow{\hat\partial^{-1}}
 \ \ \ \ \ \ \ \ \ \ \ \ \ \ \ \ \\
\ \ \ \ \ \ \ \ \ \ \ \ \ \ \  \xrightarrow{\hat\partial} \prod\nolimits_{\beta } {\rK}
_{1}(  W(k(\eta _{2\beta })) [ G]) \overset{\mathrm{res}}{\rightarrow } {\rK}_{1}(
\Z_{p}[ G] )^\flat \subset \rK_{1}( \Q_{p}[ G] ) . 
\end{eqnarray*}\end{proof}

\begin{remark}\label{NotMoritaYouIdiot}
{\rm Let us observe here that, under the assumption that the group algebra $\Q[G]$ splits,  the push down 
$f_{\ast
\eta _{0}\eta _{1}\eta _{2}}:  {\rK}_{2}( \hat{\mathcal{O}}_{Y,\eta _{0}\eta _{1}\eta
_{2}}[G])\to   \rK_{1}( \Q_{p}[ G])$ can be defined directly 
using the symbols of Kato, Liu \cite{LiuKato}, etc., by using Morita equivalence 
to reduce to the standard case of a trivial group.  
However, it is important to emphasize that 
the vertical case (ii) of  Proposition \ref{thm15}, 
which is absolutely essential for 
our construction of a well defined push down map $
f_{*}: \mathrm{CH}_{\mathbb{A}}^{2} (Y[G]) \rightarrow \Cl( \Z[ G] )
$, cannot be deduced by appealing to Morita equivalence.
Indeed,  when $p$ divides the order of $G$ the group ring $\Z_p[G]$ does not split. Instead, the proof of Proposition \ref{thm15} (ii) uses the full force of the 
construction of the symbol $\hat\partial^{-1}$ via the central extension $\mathcal H$ of \S 3;
the need to show this fact essentially dictates the complicated approach
we take in this paper.
In the classical Fr\"ohlich theory, showing 
that various constructions produce ``determinants", \emph{i.e.} elements of ${\rK}_{1}(
\Z_{p}[ G] )^\flat={\rm Det}(\Z_p[G])$, give the thorniest technical problems. Such problems are often handled either by
proving various congruences or by using logarithmic methods. In our context, we could not make any of these approaches to work. Instead, we found that using the central extension
$\mathcal H$ provides a powerful method for establishing this non-trivial fact.

}
\end{remark}
 
 \subsubsection{} Let $\eta _{1}$ be a horizontal codimension one point on $Y$. Then from \S \ref{s:parsh0}, $\hat{\mathcal{O}}_{Y,\eta _{0}\eta _{1}}= \Q( \eta _{1}) (\!( t)\!)$ is a complete
 discrete valued field with residue field $\Q(\eta_1)$.
We can therefore form the pushdown $f_{\ast \eta _{0}\eta _{1}}$
\begin{equation*}
f_{\ast \eta _{0}\eta _{1}}:\rK_{2}( \hat{\mathcal{O}}_{Y,\eta
_{0}\eta _{1}}[ G] ) = \rK_{2}( 
\Q ( \eta _{1 }) (\!( t )\!)) [ G] ) \xrightarrow{\partial^{-1}} \rK_{1}(
\Q( \eta _{1 }) [ G] ) \xrightarrow{\rm res} \rK_{1}( \Q[ G])
\end{equation*}
where $\rm res$ is obtained as above.  
By the functoriality of $\mathcal{H}$-sequences and of the map $\partial$ we have a commutative diagram:
\begin{equation*}
\begin{matrix}
\rK_{2}( \hat{\mathcal{O}}_{Y,\eta _{0}\eta _{1}}[ G]) & \rightarrow  & \prod_{\eta _{2}}\rK_{2}( \hat{%
\mathcal{O}}_{Y,\eta _{0}\eta _{1}\eta _{2}}[ G] ) \\
&&\\
f_{\ast \eta _{0}\eta _{1}}\downarrow\ \ \ &  & \ \ \ \downarrow \prod f_{\ast \eta
_{0}\eta _{1}\eta _{2}} \\
&&\\
\rK_{1}( \Q[ G] ) & \rightarrow & \rK_{1}(
\Q_{p}[ G] ).\\
\end{matrix}
\end{equation*}
(This also follows from \cite[Cor. 5.5]{LiuKato}.)
Here the product in the upper right-hand term extends over all $\eta_2$, $\eta_2<\eta_1$, with residue characteristic $p$. We have
therefore shown:

\begin{theorem}\label{thm17}
 Let $\eta _{1}$ be a horizontal codimension one point of $Y$
and suppose $x$ is an element in $\rK_{2}( \hat{\mathcal{O}}_{Y,\eta _{0}\eta
_{1}}[G]) $. Then the product
\begin{equation*}
\prod_{\eta _{2}, \eta_2< \eta_1}f_{\ast \eta
_{0}\eta _{1}\eta _{2}}( x) 
\end{equation*}
lies in the diagonal image $   \rK_{1}( \Q[G])^\flat\subset  \prod_{p}\rK_{1}( \Q_{p}[G])$.
\end{theorem}

\subsection{Reciprocity on a vertical fiber}\label{4.2}

Let $\eta _{1}$ be a vertical codimension one point of $Y$. Suppose   that the
closure $\overline{\eta }_{1}$ lies in the special fiber of $Y $ over a prime number $p$.

\begin{theorem}\label{thm18}
For $x\in  {\rK}_{2}( \hat{\mathcal{O}}_{Y,\eta
_{0}\eta _{1}}[ G] ) $ the product
$
\prod_{\eta _{2}<\eta_{1}}f_{\ast \eta _{0}\eta
_{1}\eta _{2}}( x)$ converges to  $1$ in $\rK_{1}( \Q_{p}[G])$;
here the product extends over all closed points $\eta_2$ on $\overline{\eta }_{1}$
and can be taken in any order.
\end{theorem}

\begin{proof} Since $\Q[G]$ splits, we can also write
$
\Q_{p}[ G] =\prod\nolimits_{i}{\rm M}_{n_{i}}( L_{i}).
$
 We will use the subscript $L$ to denote base change of $\Q$-schemes to the field
$L$. Similarly, we will use the subscript $L$ to denote tensor product
of $\Q$-algebras with $L$ over $\Q$, $A_L=A\otimes_\Q L$.
With this
convention we have a decomposition
\begin{equation*}
\hat{\mathcal{O}}_{Y,\eta _{0}\eta _{1}\eta _{2}}[ G]
=\prod\nolimits_{i}{\rm M}_{n_{i}}( \hat{\mathcal{O}}_{Y,\eta _{0}\eta
_{1}\eta _{2},L_i})
\end{equation*}
and so by Morita equivalence we obtain decompositions
\begin{equation}\label{eq62}
\rK_{2}( \hat{\mathcal{O}}_{Y,\eta _{0}\eta _{1}\eta _{2}}[ G] ) =
\prod\nolimits_{i}\rK_{2}( \hat{\mathcal{O}}_{Y,\eta
_{0}\eta _{1}\eta _{2},L_i}) ,\quad \ \rK_{2}( \hat{\mathcal{O}}
_{Y,\eta _{0}\eta _{1}}[ G] )
=\prod\nolimits_{i}\rK_{2}( \hat{\mathcal{O}}_{Y,\eta _{0}\eta
_{1}, L_i}). 
\end{equation}
Notice that  since $Y$ is smooth over $|G|$,  the base change $Y\otimes_\Z\O_{L_i}$ is
also regular and the morphism $Y\otimes_\Z\O_{L_i}\to \Spec(\O_{L_i})$
is also smooth over $|G|$,  for all $i$. The above shows that we can reduce to showing the vanishing statement to the case 
when $G=\{1\}$ and the base is the ring of integers of a finite extension $L$ of $\Q_p$. 
This result follows directly from \cite[Theorem 5.1]{LiuKato}.
\end{proof}

\smallskip

For future reference, we record the following:

 \begin{corollary}\label{correc}
The composition $  {\rK}_{2} (\Q_p\{t^{-1}\}[G]  )\xrightarrow{ }\rK_2(\Q_p{\ldb t\rdb}[G]) \xrightarrow{\hat\partial^{-1}} \rK_{1}( \Q_{p} [ G ]
 ) $, where the first map is induced by the inclusion $\Q_p\{t^{-1}\}\subset \Q_p{\ldb t\rdb}$, is trivial.
\end{corollary}

\begin{proof}  Recall $\Q_p\{t^{-1}\}:=\Q\otimes_\Z \Z_p\langle\!\langle t^{-1} \rangle \!\rangle$ is the free Tate algebra. We  apply Theorem \ref{thm18} above to $\mathbb{P}^1$ over $\Spec(\Z)$ and $\eta_1$ 
the generic point of the special fiber at $p$, which we denote $1_{p}$. Denote by $2_p$ the 
closed point given by $(p, t )$.
For
\begin{equation*}
x\in  {\Kr}_{2} ( \Q_p\{t^{-1}\}[ G]  )^\flat\subset  {\rK}_{2}( \hat{\mathcal{O}}_{
\mathbb{P}^{1},\eta _{0}1_{p}} [ G ]  ) \
\end{equation*}
we have
\begin{equation*}
\prod_{\eta _{2}}f_{\ast \eta _{0}1_{p}\eta _{2}} ( x ) =1
\end{equation*}%
where the product extends over all closed points $\eta _{2}$ on $\overline{1}%
_{p}$. We claim that
\begin{equation}
f_{\ast \eta _{0}1_{p}\eta _{2}} ( x ) =1\text{\ \ for \ }\eta
_{2}\neq 2_{p}.
\end{equation}%
This will then show that $f_{\ast \eta _{0}1_{p}2_{p}} ( x ) =\hat\partial^{-1}(x)=1$.
Suppose that with our usual notation $\hat{\mathcal{O}}_{\mathbb{P}%
^{1},\eta _{2}}=W( k( \eta _{2}) ) \lps
t_{\eta _{2}}\rps$. Then, since $\eta _{2}\neq 2_{p}$,  we
know that the image of $t^{-1}\in \hat{\mathcal{O}}_{\mathbb{P}%
^{1},01_{p}\eta _{2}}$ actually lies in $W( k( \eta _{2})
 \lps t_{\eta _{2}}\rps$; thus the map $\Q\otimes_{\Z}\Z_{p}\langle\!\langle
t^{-1}\rangle\!\rangle \rightarrow \Q_{p}( \eta
_{2}) \{\!\{t_{\eta _{2}}\}\!\}=\hat{\mathcal{O}}_{\mathbb{P}%
^{1},01_{p}\eta _{2}}$ factors through $\Q_{p}(
\eta _{2}) \lps  t_{\eta _{2}}\rps$. The result
then follows since the pushdown map $f_{\ast  {0}1_{p}\eta _{2}}$ is trivial on $\rK_{2} ( \Q_{p}( \eta
_{2}) \lps t_{\eta _{2}}\rps[ G])$. 
\end{proof}

\subsection{Reciprocity for codimension one points through a given closed
point. }\label{sec4.3}

 In this subsection we fix  a closed point $\eta_2$ of $Y$ with residue field $k( \eta
_{2}) $ which we suppose to have characteristic $p$.  
 
\begin{theorem}\label{thm20}
For $x\in \rK_{2}( \hat{\mathcal{O}}_{Y,\eta _{0}\eta _{2}}[G]) $, we have
\begin{equation*}
\prod_{\eta_1, \eta_1>\eta_2 }f_{\ast \eta _{0}\eta
_{1}\eta _{2}}( x) =1 
\end{equation*}
in $\rK_{1}( \Q_{p}[ G])$.
\end{theorem}

\begin{proof} As in the proof of Theorem \ref{thm18}, we
can use Morita equivalence to reduce to showing that for $x\in {\rK}_{2}(
\hat{\mathcal{O}}_{Y,\eta _{0}\eta _{2},L}) $ we have
\begin{equation*}
\prod_{\eta _{2}\in \overline{\eta }_{1}}f_{\ast \eta _{0}\eta
_{1}\eta _{2}}( x) =1\hbox{\rm \ in\ }\rK_{1}( L) =L^\times
\end{equation*}
where $L$ is a finite extension of $\Q_p$. This essentially follows from \cite[Theorem 4.2]{LiuKato};
however, since \emph{loc. cit.} gives a somewhat different statement we give the proof.
We let $\hat{\eta }$ denote a height one prime ideal of $\hat{%
\mathcal{O}}_{Y,\eta _{2}}\otimes_{\Z_p}\O_L$. Then, as in the previous construction in
Section \ref{sec4.1}, we can form push down maps:
\begin{equation*}
f_{\ast \eta _{0}\hat{\eta }\eta _{2}}:{\rK}_{2} (
\hat{\mathcal{O}}_{Y,\eta _{0}\hat{\eta }\eta _{2},L})
\rightarrow \rK_{1}( L).
\end{equation*}
For $\kappa\in {\rK}_{2} ( \hat{\mathcal{O}}
_{Y,\eta _{0}\eta _{2}, L}) $ we consider the product
\begin{equation*}
\prod_{\hat{\eta }\text{ horizontal}\, |\, \eta _{2}< \hat{\eta } }f_{\ast \eta
_{0}\widehat{\eta }\eta _{2}}( \kappa ) \cdot \prod_ {\hat{\eta }\text{ vertical}\, |\, \eta _{2}<\hat\eta}f_{\ast \eta _{0}\widehat{
\eta }\eta _{2}}( \kappa ).
\end{equation*}
From \cite[Proposition 7]{KatoRes} we know that this product  converges to one in $\rK_{1}(
L)$. In order to complete the proof it therefore remains to show that
if $\hat{\eta }$ does not arise from a codimension one point of $Y\otimes_{\Z_p}\O_L$
(\emph{i.e.} if $\hat{\eta }$ is not globally defined), then $f_{\ast \eta _{0}
\hat{\eta }\eta _{2}}( \kappa ) =1$. For simplicity, we will omit the subscript $L$.
Recall from \S \ref{s:parsh1}
that $\hat{\mathcal{O}}_{Y,\eta _{0}\eta _{2}}$ is obtained by
localizing the complete local ring $\hat{\mathcal{O}}_{Y,\eta _{2}}$
with respect to the multiplicative set of elements $K( Y_L )^{\times }$ of non-zero elements in the function field of $Y $. Thus for such $\hat{\eta }$,
 which do not arise as
codimension one points on $Y$, we deduce that in fact $\kappa \in {\rK}_{2}( \hat{\mathcal{O}}_{Y,\hat{\eta }\eta _{2}})^\flat $ and, as $\hat{\eta }$ is necessarily
horizontal, we know from Proposition \ref{thm15} above that for such $\hat{\eta }$
we have $f_{\ast \eta _{0}\hat{\eta }\eta _{2}} ( \kappa )
=1$.  
\end{proof}

\subsection{Adelic push down}\label{4.4}

Our main aim here is to show that the pushdown maps associated to Parshin
triples induce a map on the adelic restricted product group ${\rK}^\prime_2({\mathbb A}_{Y,012}[G])=
\prod^{\prime }_{(\eta_0,\eta_1,\eta_2)}{\rK}_{2}( \hat\O_{Y,\eta _{0}\eta
_{1}\eta _{2}}[G])\subset \prod_{(\eta_0,\eta_1,\eta_2)}{\rK}_{2}( \hat\O_{Y,\eta _{0}\eta
_{1}\eta _{2}}[G])$  (see Definition \ref{defrestrictedproducts}). To be
more precise, the above considerations show that we have a map on each 
${\rK}_{2}( \hat\O_{Y,\eta _{0}\eta _{1}\eta _{2}}[G])$. 
We wish to show that this extends, in a natural
convergent manner, to a pushdown map
\begin{equation}\label{pushdown}
f_*:=\sideset{}{^\prime}\prod f_{\ast \eta _{0}\eta _{1}\eta
_{2}}:\sideset{}{^\prime}\prod_{(\eta_0,\eta_1,\eta_2)}\ {\rK}_{2} ( \hat\O_{Y,\eta _{0}\eta _{1}\eta _{2}}[G])\rightarrow
\sideset{}{^\prime}\prod_p\rK_{1}( \Q_{p}[ G] ) .
\end{equation}
For this we first need: 

\begin{proposition}\label{prop21} 
Let $\eta _{1}$ denote a vertical codimension one point of $Y$ and let $\eta_{2}<\eta_1$ denote a closed point which is not contained
in the closure of any other vertical codimension $1$ point. If $x$ is in $\rK_{2}( \hat{\mathcal{O}}_{Y,\eta _{2}}[ \eta
_{1}^{-1}] [ G] )^\flat$, then $f_{\ast \eta _{0}\eta
_{1}\eta _{2}}( x) =1$.
\end{proposition}

\begin{proof} From Theorem \ref{thm20} we know that
\begin{equation*}
\prod_{\zeta_1, \eta _{2}<\zeta_{1}}f_{\ast \eta _{0}\zeta
_{1}\eta _{2}}( x) =1.
\end{equation*}
 Let $\xi_{1}$ denote   a
codimension $1$ point with $\eta_2<\xi_1$ which is different from $\eta _{1}$. It will suffice
to show that $f_{\ast \eta _{0}\xi _{1}\eta _{2}}(x)=1$. This
 follows from the inclusion $\hat{\mathcal{O}}_{Y,\eta _{2}}[
\eta _{1}^{-1}] \subset \hat{\mathcal{O}}_{Y,\xi _{1}\eta _{2}}$, our assumption, 
and Proposition \ref{thm15}. 
\end{proof}

\begin{proposition}\label{prop22} 
For any  
$
x=(x_{\eta _{0}\eta _{1}\eta _{2}})$ in  $
  {\rK}_{2}^{\prime }( \mathbb{A}_{Y,012}[ G] ) ,
$
the infinite product  of push-downs $\prod f_{\ast \eta _{0}\eta _{1}\eta
_{2}}( x_{\eta _{0}\eta _{1}\eta _{2}}) $ converges to an element
of the restricted product $\prod^{\prime }_p\rK_{1}( \Q_{p}[G]) $.
\end{proposition}

\begin{proof}  We must show that the product  $\prod_{(\eta _{1},\eta
_{2})}f_{\ast \eta _{0}\eta _{1}\eta _{2}}( x_{\eta _{0}\eta _{1}\eta
_{2}})$, where we consider  closed points $\eta _{2}$ of residue
characteristic $p$, is $p$-adically convergent in $\rK_{1}( \Q_{p}[ G] ) $, and that for 
almost all $p$, it converges to an
element of $\rK_{1} ( \Z_{p}[ G] )^\flat $.
We write this product as
\begin{equation}\label{eq68}
\prod_{(\eta _{1},\eta _{2})}f_{\ast \eta _{0}\eta _{1}\eta _{2}}(
x_{\eta _{0}\eta _{1}\eta _{2}}) =\prod_{\eta _{1}\text{\
horiz.,\ }\eta _{2}< \eta _{1}}f_{\ast \eta _{0}\eta _{1}\eta _{2}}(
x_{\eta _{0}\eta _{1}\eta _{2}}) \cdot \prod_{\eta _{1}\text{\
vert.,\ }\eta _{2}< {\eta }_{1}}f_{\ast \eta _{0}\eta _{1}\eta
_{2}}( x_{\eta _{0}\eta _{1}\eta _{2}}) .   
\end{equation}
We start by considering the first product. By (PK1) together with Proposition \ref{thm15} only a
finite number of  $\eta _{1}$ will contribute non-trivial terms; moreover
each $\overline{\eta }_{1}$ meets the special fiber $Y_p$ of $Y$ at $p$ at a
finite number of closed points,  and so we see that the first product affords
only a finite number of non-trivial terms for each prime $p$.  Moreover, applying (PK2) (with $k=0$) 
to such an $\eta _{1}$, we get that for almost all $\eta _{2}<\eta_1$
\begin{equation*}
x_{\eta _{0}\eta _{1}\eta _{2}}  \in    {\rK}_{2}( \hat\O_{Y,\eta _{1}\eta _{2}}[ G] )^\flat \cdot  
 {\rK}_{2}( \hat\O_{Y,\eta _{2}}[ \eta _{1}^{-1}] [ G])^\flat  .
 \end{equation*}
By Proposition \ref{thm15}, the first term has pushdown equal to $1$.
Notice that by \S \ref{section1}, for almost all $\eta_2<\eta_1$, $\hat\O_{Y, \eta_2}[\eta_1^{-1}]\simeq 
  W(k(\eta_2)){\lps t\rps}[1/g_1]$ where $g_1$ gives a local equation for $\bar\eta_1$. Using this and the construction of $\partial$ we can now see that for almost all $\eta_2<\eta_1$
 (in characteristic $p$), the second term  also has pushdown that lies in $\rK_{1}( \Z_{p}[G])^\flat$.

To conclude we consider the contributions to the second product for the given
prime number $p$. So we assume that $ {\eta }_{1}\in Y_{p}$ and
suppose  $\eta _{2}< {\eta }_{1}$. Given
a positive integer $k$, by (PK2) we know that the product
\begin{equation*}
\prod_{ \eta_2, \eta _{2}<  {\eta }_{1}}f_{\ast \eta
_{0}\eta _{1}\eta _{2}}( x_{\eta _{0}\eta _{1}\eta _{2}})
\end{equation*}%
may be written as a finite product multiplied by a product of terms pushed down from the groups
$ {\rK}_{2}( \hat{\mathcal{O}}_{Y,\eta _{1}\eta _{2}}[ G] , ( p^{k}))^\flat$ and $ \rK_{2}( \widehat{\mathcal{O}}_{Y,\eta _{2}}[ \eta _{1}^{-1}] [ G])^\flat $. From 
\S \ref{3.5.2} and the construction of the push-down, we know that
\begin{equation*}
f_{\ast \eta _{0}\eta _{1}\eta _{2}} ( {\rK}_{2} (
\hat{\mathcal{O}}_{Y,\eta _{1}\eta _{2}}[ G] ,(p^{k}))^\flat) \subset \rK_{1}( \Z_{p}[ G]
,( p^{k}) )^\flat .
\end{equation*}
By Proposition \ref{prop21} we know that if $\eta _{2}$ is a smooth
point of the reduced special fiber of $Y$ at $p$, then 
$$
f_{\ast \eta _{0}\eta
_{1}\eta _{2}} (\rK_{2}( \widehat{\mathcal{O}}_{Y,\eta _{2}}[ \eta
_{1}^{-1}] [ G] ) ) =1.
$$
 Since there are only a finite
number of non-smooth points on the reduced special fiber, we conclude that the product is
$p$-adically convergent. Moreover, since the special fiber is smooth for
almost all $p$, we have also shown that, for almost all $p$ the
contribution from the second product lies in $\rK_{1} ( \Z_{p}[ G] )^\flat$.
\end{proof}
\smallskip

We also have:

\begin{proposition}\label{convij} 
Consider   the intersection 
$$
\left(\prod_{   0\leq i<j\leq 2} {\rK}'_{2}  ( \mathbb{A}%
_{Y,ij} [ G ]  )^\flat\right) \cap  {\rK} 
_{2}'  ( \mathbb{A}_{Y,012} [ G ] )  \subset  {\rK}_{2}' ( \mathbb{A}_{Y,012} [
G ]   )
$$
 of subgroups in the unrestricted product $ {\rK} 
_{2}  ( \mathbb{A}_{Y,012} [ G ] )  $. If $x$ lies in this intersection, 
then $f_*(x)=\prod_{(\eta_0,\eta_1,\eta_2)}f_{* \eta_0\eta_1\eta_2}(x)$
converges to an element in $ \rK_1(\Q[G])^\flat\cdot \prod_p\rK_1(\Z_p[G])^\flat$.
\end{proposition}

\begin{proof} Suppose we write $x=a_{01}\cdot a_{12}\cdot a_{02}$ with 
$a_{ij}\in  {\rK}'_{2}  ( \mathbb{A}%
_{Y,ij} [ G ]  )^\flat$. By Proposition \ref{prop22} we know that the product $f_*(x)=\prod_{(\eta_0,\eta_1,\eta_2)}f_{* \eta_0\eta_1\eta_2}(x)$ converges in 
$\prod'_p\rK_1(\Q_p[G])$. Similarly, the product $f_*(a_{01})=\prod_{(\eta_0,\eta_1,\eta_2)}f_{* \eta_0\eta_1\eta_2}(a_{\eta_0\eta_1})$
converges to an element in $\rK_1(\Q[G])^\flat$ by Theorems \ref{thm17} and \ref{thm18} (which assure  us that for 
$a_{01}$ we get contributions only for a finite number of horizontal $\eta_1$). By Remark \ref{ko2}, the element $a_{02}$ 
belongs to the restricted product $ {\rK}_{2}' ( \mathbb{A}_{Y,012} [
G ]   )$ and therefore the product 
$$
f_*(a_{02})=\prod_{(\eta_0,\eta_1,\eta_2)}f_{* \eta_0\eta_1\eta_2}(a_{\eta_0\eta_2}).
$$
converges. In fact, using the definition of $ {\rK}'_{2}  ( \mathbb{A}_{Y,02} [ G ]  )$ and Proposition \ref{prop21}
we see that, in the above product, for a given prime $p$,
 only a finite number of pairs
$\eta_2<\eta_1$ with $\eta_2$ over $p$ can give non-trivial contributions. 
Indeed, these are pairs of two types:  either $\eta_1$ is horizontal and on the support of the divisor $D$ (as per 
definition of ${\rK}'_{2}  ( \mathbb{A}_{Y,02} [ G ]  )$) and $\eta_2$ is an intersection point of $\overline\eta_1$ with the special fiber at $p$;  or  $\eta_1$ is vertical over $p$ and $\eta_2$ is a singular point of the special fiber of $Y$.
 Now by rearranging this product and using  the reciprocity law for codimension one points 
through a closed point (Theorem \ref{thm20}) we can see that it is equal to $1$. Hence
we have $f_*(a_{02})=1$.  We can conclude that the product
$$
f_*(a_{12})=\prod_{(\eta_0,\eta_1,\eta_2)}f_{* \eta_0\eta_1\eta_2}(a_{\eta_1\eta_2})
$$
also converges to an element in $\prod'_p\rK_1(\Q_p[G])$. Since  each individual term
$f_{* \eta_0\eta_1\eta_2}(a_{\eta_1\eta_2})$ is  in $\rK_1(\Z_p[G])^\flat$ (where $p$ is the characteristic of $\eta_2$),
we conclude that $f_*(a_{12})$ converges to an element in $\prod_p\rK_1(\Z_p[G])^\flat$ and the result follows.  
\end{proof}

\subsection{Push down on the equivariant second Chow group.}\label{sec4.5}

 It follows from Proposition \ref{prop22} and Proposition \ref{convij} that $f_*=\prod_{(\eta_0, \eta_1, \eta_2)}f_{*\eta_0\eta_1\eta_2}$ induces a well-defined group homomorphism
\begin{equation*}
\frac{ {\rK}_{2}^{\prime }( \mathbb{A}_{Y,012}[ G])   }{(\prod_{0\leq i<j\leq 2} {\rK}_{2}^{\prime }(
\mathbb{A}_{Y,i j}[ G] )^\flat )\cap  {\rK}_{2}^{\prime }( \mathbb{A}_{Y,012}[ G]) }\xrightarrow { \ \ \ } \frac{
\prod'_p\rK_{1}( \Q_{p}[ G] ) }{
\rK_{1} ( \Q[ G] )^\flat \prod_p \rK_{1} 
 ( \Z_{p}[ G] )^\flat }.
\end{equation*}
We notice that the source of this map is naturally identified with $\mathrm{CH}_{\mathbb{A}}^{2} (Y[G])$
while, by the Fr\"ohlich description \S \ref{Frodescription}, the target is naturally isomorphic to $\Cl( \Z[ G] )$.
Hence, we obtain a group homomorphism
$$
f_{*}: \mathrm{CH}_{\mathbb{A}}^{2} (Y[G]) \rightarrow \Cl( \Z[ G] ) .
$$
 
\bigskip

\section{Transitions matrices and the first Chern class}\label{s:Chern}

\setcounter{equation}{0}

We return to the assumptions and notations of \S \ref{s:Kadeles}.
Suppose now that $\E$ is a $\O_{Y}[G]$-bundle;
that is to say $\E$ is a locally free coherent $\O_{Y}[G]$-module of a given rank, 
which we denote by $n$. For each point $\eta$ of $
Y $ we choose an $\hat\O_{Y,\eta }[G]$-basis $e_{\eta
}= \{ e_{\eta }^{h} \} _{h=1}^{n}$ of the completed stalk $\hat\E_{\eta}=\E_\eta\otimes_{\O_{Y,\eta}}\hat\O_{Y,\eta}$. For a given Parshin triple $
\{ \eta _{0},\eta _{1},\eta _{2}\} $ we then have transition maps
$\lambda _{\eta _{i}\eta _{j}}\in \GL_{n}( \hat{\mathcal{O}}_{Y,\eta
_{i}\eta _{j}}[G]) $ with
\begin{equation}\label{eq77}
( e_{\eta _{i}}^{h})_{h}=\lambda _{\eta _{i}\eta _{j}}\cdot (
e_{\eta _{j}}^{k}) _{k}
\end{equation}
for $0\leq i,j\leq 2$. Note that we have the obvious relation $\lambda _{\eta _{i}\eta
_{k}}=\lambda _{\eta _{i}\eta _{j}}\cdot \lambda _{\eta _{j}\eta _{k}}.\bigskip $

\subsection{Construction of the first Chern class. }\label{sec5.2.1}

\begin{theorem}\label{thm24}
With the above notation:

(a) $\prod_{\eta _{1}}\mathrm{Det} ( \lambda _{\eta _{0}\eta
_{1}} ) $ lies in the restricted product $\rK_1'({\mathbb A}_{Y, 01}[G])\subset \prod_{\eta_1} \rK_1(\hat\O_{Y, \eta_0\eta_1}[G])$; that is to say 
all but a finite number of terms are in $\rK_{1}^{\prime } ( \hat{\mathcal{%
O}}_{Y,\eta _{1}}[G])^\flat$.

(b)  The class of $\prod_{\eta _{1}}\mathrm{Det} ( \lambda
_{\eta _{0}\eta _{1}} ) $ in the first adelic equivariant Chow group $ 
{\rm CH}_{\mathbb{A}}^{1} ( Y[G]) $ is independent of the
choice of bases.
\end{theorem}

Before proving the theorem we use it to make the following
definition:

\begin{definition} The first adelic equivariant Chern class of $\E$ in ${\rm CH}_{\mathbb{A}
}^{1}( Y[ G] ) $, denoted $c_{1}( \E)$, is
the class represented by $\prod_{\eta _{1}}\mathrm{Det}(
\lambda _{\eta _{0}\eta _{1}}) $.
\end{definition}

\begin{proof} (a) The generic basis $\{ e_{\eta_{0}}^{i}\}
_{i=1}^{n}$ is an $\mathcal{O}_{Y}( U)[G]$-basis
of $\E$ for some non-empty open Zariski set in $Y$; this therefore gives an
$\hat{\mathcal{O}}_{Y, \eta _{1}}[G]$-basis of $\E$ for
all but a finite number of codimension one points $\eta _{1}$; and so 
$\lambda _{\eta _{0}\eta _{1}}\in \GL_{n}( \widehat{\mathcal{O}}
_{Y, \eta _{1}}[ G]) $ for all but a finite number of
codimension one points $\eta _{1}$. This   shows   (a).

To prove (b), let $\{ d_{\eta _{i}}^{h}\} _{h=1}^{n}$ denote a
further system of bases for $\E$. Then
\begin{eqnarray*}
 ( d_{\eta _{0}}^{h} ) _{h} &=&\gamma_{\eta _{0}}\cdot ( e_{\eta
_{0}}^{k} ) _{k}\ \ \text{for }\  \gamma_{\eta _{0}}\in \GL_{n} (
\hat{\mathcal{O}}_{Y, \eta _{0}}[G]) \\
( d_{\eta _{1}}^{h}) _{h} &=&\gamma_{\eta _{1}}\cdot ( e_{\eta_{1}}^{k})_k \ \ \text{for }\ \gamma_{\eta_{1}}\in \GL_{n}(\hat{\mathcal{O}}_{Y, \eta _{1}}[G])
\end{eqnarray*}
and so for each codimension one point $\eta _{1}$ of $Y$ we have the
equality
\begin{equation*}
 ( d_{\eta _{0}}^{h}) _{h}=\gamma _{\eta _{0}}\cdot ( e_{\eta
_{0}}^{k}) _{k}=\gamma _{\eta _{0}} \lambda _{\eta _{0}\eta _{1}} \cdot (
e_{\eta _{1}}^{j}) _{j}=\gamma _{\eta _{0}} \lambda _{\eta _{0}\eta
_{1}}\gamma_{\eta _{1}}^{-1}\cdot  ( d_{\eta _{1}}^{h}) _{h};
\end{equation*}
therefore working with the $d$-bases we get the new product of determinants
\begin{equation*}
\prod_{\eta _{1}}\mathrm{Det}( \gamma_{\eta _{0}}\lambda
_{\eta _{0}\eta _{1}}\gamma _{\eta _{1}}^{-1} ) =\mathrm{Det} (
\gamma _{\eta _{0}} ) \prod_{\eta _{1}}\mathrm{Det} (
\lambda _{\eta _{0}\eta _{1}} ) \mathrm{Det} ( \gamma _{\eta
_{1}}^{-1} )
\end{equation*}
and the result follows since
\begin{equation*}
\mathrm{Det}( \gamma _{\eta _{0}}) \in \rK_{1}( \hat{
\mathcal{O}}_{Y, \eta _{0}}[G])^\flat,\quad  \mathrm{Det}( \gamma _{\eta _{1}}) \in \rK_{1}( \hat{\mathcal{O}}_{Y,\eta _{1}}[G]))^\flat .
\end{equation*} 
\end{proof}

\begin{remark}\label{ChernFrohlich}
{\rm   Let $M$ be a locally free finitely generated $\Z[G]$-module which
defines an $\O_{S}[G]$-bundle $\E=\tilde M$ on $S=\Spec(\Z)$. The above
construction gives a Chern class $c_1(\E)$ in 
\begin{equation}
{\rm CH}^1_{\mathbb A}(S[G])=
\frac{ \prod^\prime_p\rK_{1} (
\Q_{p}[G])) }{\rK_{1} ( \Q
 [ G ]  )^{\flat } \prod_p \rK_{1} ( \Z_{p} [ G]  )^{\flat } }.
\end{equation}
In this case, we can see \cite{FrohlichBook} that the map $M\mapsto c_1(\tilde M)$ gives 
an isomorphism ${\rm Cl}(\Z[G])\xrightarrow{\sim} {\rm CH}^1_{\mathbb A}(S[G])$.
In what follows, we  use this isomorphism to identify ${\rm Cl}(\Z[G])$ with ${\rm CH}^1_{\mathbb A}(S[G])$; this is the negative of the   isomorphism given by the classical Fr\"ohlich description of
the class group ${\rm Cl}(\Z[G])$ (cf. \S \ref{Frodescription}).}
\end{remark}

 \section{Elementary structures and the second Chern class }\label{sec5.2.2}
 
 \setcounter{equation}{0}

The definition of the second adelic Chern class is   more involved. 
We first need some prerequisites.

\subsection{The Steinberg extension}\label{sec5.1} 
\hfill

Let $R$ denote a commutative ring
and  $G$   a finite group. Recall from \S \ref{equivchow-gen} that $\GL( R[ G]) $ 
denotes the full general linear group over the group ring $
R[G] $ and $ \rE( R[ G]) $ is the subgroup
of elementary matrices with entries in $R[G]$. Recall from \cite[Chapter 4.2]{RosenbergBook}  and 
\cite[Section 5]{MilnorK} that we have the Steinberg group ${\rm St}(R[G]) $
which sits in the central exact sequence
\begin{equation}\label{eq69}
1\rightarrow \rK_{2}( R[ G] ) \rightarrow {\rm St}( R
[ G] ) \rightarrow \rE( R[ G] )
\rightarrow 1.  
\end{equation}

\subsubsection{} \label{ss:cocycles}

 Suppose $G$ is a central group extension 
$$1 \to A \to G \to H \to 1$$
of a group $H$ by a  group $A$. 
If $c$, $d$ are two elements of $H$, we may
choose lifts (preimages) $\ti c$, $\ti d$, $\widetilde {cd}$, 
of the elements $c$, $d$, $cd$ of $H$ in $G$.
We can then define
$$
z:=\widetilde {cd}\cdot (\ti d)^{-1}\cdot (\ti c)^{-1}\in A.
$$
Although $z$ depends on the choice of lifts of these three elements, we will
sometimes abuse notation and write the element $z$ as $z(c,d)$.
We will refer to these elements as ``cocycles" since if $s: H\to G$ is a set-theoretic section 
of $G\to H$, then the map 
$
z: H\times H\to A
$
given by $z(c,d)=s(cd)\cdot s(d)^{-1}\cdot s(c)^{-1}$ gives a $2$-cocycle with values in $A$.

The following lemma will be applied repeatedly to the Steinberg central exact sequence
(\ref{eq69}) and its variants.

\begin{lemma}\label{cocyclele}
Suppose $b, c, d \in H$.  Fix a choice of a pre-image 
$\ti h \in G$ of each element $h$ in the sequence
$ c, d, b, cd, db, cdb $.
 These choices allow us to define as above the elements
$z( cd,b)$, $ z( c ,d)$, $z( c,db)$,  $z(d,b)$ of $A$ and we have
\begin{equation}
z( cd,b) z( c ,d)  
= z( c,db) z(d,b).
\end{equation}
\end{lemma}

\begin{proof}
This is just the two-cocycle relation of \cite[Chapter VII.3, p. 121]{SerreLocalFields} since $A$ is central in $G$.
Notice here that we do not require that the lifts of two elements in the sequence
that are equal are also equal.
\end{proof}

\subsection{Elementary structure}\label{sss: proel}

We continue
with the assumptions and notations of \S \ref{s:Kadeles}.

\begin{definition}\label{defadproel} Suppose $\mathcal{E}$  is a  $\mathcal{O}_{Y}[ G]$-bundle.
An (adelic) elementary structure $\epsilon$ on $\E$ is an equivalence class of choices of $\hat{\mathcal{O}}_{Y,\eta
_{i}}[ G] $-bases $\{e_{\eta _{i}}^{h}\}_{h}$ of $\hat\E_{\eta_i}=\mathcal{E}\otimes_{\O_{Y}}\hat{\mathcal{O}}_{Y,\eta _{i}}$ with the property that
for each Parshin pair $\{\eta _{i},\eta _{j}\}$ on $Y$ the image of the
transition map $\lambda _{\eta _{i}\eta _{j}}$ in ${\rm GL} ( \hat{\O}_{Y,\eta _{i}\eta _{j}}[G]) $   lies in the
subgroup of  elementary matrices $ {\rE} ( \hat{\O}_{Y,\eta _{i}\eta _{j}}[G])$. Here, another set  
of $\what{\mathcal{O}}_{Y,\eta
_{i}}[ G]$-bases  $\{d_{\eta _{i}}^{j}\}_{j}$ of $\E$ is called equivalent to $\{e_{\eta _{i}}^{h}\}_{h}$ if
it is  related to $\{e_{\eta
_{i}}^{h}\}_{h}$ by an elementary base change \emph{i.e.}, when we write $d_{\eta
_{i}}=\mu _{\eta _{i}}e_{\eta _{i}}$, then $\mu _{\eta _{i}}$ is
 elementary in the sense that $\mu _{\eta _{i}}\in  {\rE} (
\hat{\O}_{Y,\eta _{i}}[G]) $.
\end{definition}

\begin{remark}\label{adestable} {\rm a) In what follows, we will omit the adjective ``adelic" and talk simply of  elementary structures.

b) Having an  elementary structure is a stable property, \emph{i.e.} an $\O_Y[G]$-bundle 
$\E$ has  elementary structure
if and only if the bundle $\E\oplus \O_Y[G]^n$ does for some $n\geq 0$.

 c) If $\E$ supports an elementary structure as above, then its first adelic Chern class $c_1(\E)$  is trivial. Indeed,  we then have ${\rm Det}(\lambda_{\eta_0\eta_1})=1$. The converse is not necessarily always true; roughly speaking, the reason is the non-vanishing 
of various $\rSK_1$ terms which are not detected by $c_1(\E)$. It is however, true when $G$ is trivial, see below. (Recall that we are assuming that the arithmetic surface $Y$ is regular.)  In general, a $\O_Y[G]$-bundle $\E$ can support more than one inequivalent elementary structures.

d) If the group $G$ is trivial, then one can easily see that  an $\O_Y$-bundle has an elementary structure,
if and only if the determinant ${\rm det}(\E)$ is a trivial line bundle. In that case, $\E$ supports a unique
elementary structure.}
\end{remark}

\subsubsection{Examples} \label{exampleperf}

We will say that  the group $G$ has  trivial local $\rSK_1$, if we have
 ${\rm SK}_1(\Z_p[G])=\{1\}$, for all primes $p$. In \cite[Prop. 9]{OliverLNM}, Oliver
gives various conditions that imply that a group has  trivial  local $\rSK_1$.
He also shows  that this is the case for 
$G=A_n$, $S_n$, ${\rm SL}(2, {\Bbb F}_p)$, ${\rm PSL}(2, {\Bbb F}_p)$ ($p$ any prime), ${\rm SL}(2, 2^n)$, the dihedral group $D_{2n}$, and the generalized quaternion group $Q_{4n}$. 
We see that torsors for groups $G$ with trivial local $\rSK_1$ often provide with examples
of bundles with elementary structure:

\begin{proposition}\label{perfectProp}
Suppose that $Y$ is regular, $Y\to \Spec(\Z)$ is projective and flat of relative dimension $1$ 
and has smooth fibers at primes that divide the order of $G$. Suppose also  
that $G$ is a group with trivial local $\rSK_1$ 
and which satisfies  hypothesis (\ref{splitgroupalgebraIntro}). Let $m=|G^{\rm ab}|$
be the order of the maximal abelian quotient of $G$. 
If $q: X\to Y$ is a $G$-torsor, then the sheaf $\E=q_*\O_X^{\oplus m}$ of $\O_Y[G]$-modules
on $Y$ is an $\O_Y[G]$-bundle with elementary structure. 
\end{proposition}

\begin{proof}
Using \cite[Theorems 1.2,  1.5]{CPTDet} we see that we can choose local 
 $\hat{\mathcal{O}}_{Y,\eta
_{i}}[ G] $-bases $\{e_{\eta _{i}}^{h}\}_{h}$ of $\hat\E_{\eta_i}$ such that the transition matrices have trivial ${\rm Det}$ and hence belong to
$\SL(\hat\O_{Y, \eta_i\eta_j}[G])$. (This follows 
by an easy extension of the argument in the introduction  of \cite{CPTDet};  there, we discuss the case that $G$ is perfect, \emph{i.e.} when $m=1$.) Assume that $p$ is a prime that divides the order of the group $G$.
Since $Y\to\Spec(\Z)$ is smooth at primes $p$ that divide $|G|$,
we have $\hat\O_{Y, \eta_1\eta_2}\simeq W(k(\eta_2))\{\!\{t\}\!\}$ for  $\eta_1=(p)$. Hence,  the main result 
of \cite{CPTSK1b} gives 
$$
{\rm SK}_1(\hat\O_{Y, \eta_1\eta_2}[G])\simeq {\rm SK}_1(\Z_p[G])\otimes_{\Z} (\hat\O_{Y, \eta_1\eta_2}/(1-F)\,\hat\O_{Y, \eta_1\eta_2})
$$
where $F$ is a Frobenius lift on $\hat\O_{Y, \eta_1\eta_2}$. By our assumption on $G$, we obtain ${\rm SK}_1(\hat\O_{Y, \eta_1\eta_2}[G])=\{1\}$, for $\eta_1=(p)$.
 Using Morita equivalence and (\ref{splitgroupalgebraIntro}), we also see that ${\rm SK}_1(\hat\O_{Y, \eta_1\eta_2}[G])=\{1\}$ at all other $\eta_1$.
Also by Corollary \ref{cor:Blochcor} and Morita equivalence, ${\rm SK}_1(\hat\O_{Y, \eta_0\eta_i}[G])=\{1\}$, for $i=1$, $2$.
Hence,  we have ${\rm E}(\hat\O_{Y, \eta_i\eta_j}[G])=\SL(\hat\O_{Y, \eta_i\eta_j}[G])$  for all Parshin pairs
$\{\eta_i, \eta_j\}$.
Therefore, the  bases $\{e_{\eta _{i}}^{h}\}_{h}$ above
give an elementary structure as desired.
\end{proof}
\smallskip
 
 We can now  apply
Proposition \ref{perfectProp} to torsors for
the perfect groups $G=A_n$, $(n\geq 5)$, ${\rm PSL}(2, {\Bbb F}_p)$ ($p>3$  prime), ${\rm SL}(2, 2^n)$ ($n\geq 2$),
and to $S_n$-torsors with $m=2$. Indeed, by Oliver's work, these groups have trivial local $\rSK_1$. 
Also, all these groups, have all their Schur indices equal to $1$ (see \cite{Turull}, \cite{Feit}, \cite{Janusz}), and hence 
they also satisfy our hypothesis (\ref{splitgroupalgebraIntro}). Then, if $q: X\to Y$ is a $G$-torsor,  $Y$ satisfies (H), and $G=A_n$ ($n\geq 5$), ${\rm PSL}(2, {\Bbb F}_p)$ ($p>3$  prime), or ${\rm SL}(2, 2^n)$ ($n\geq 2$),
Theorem \ref{ARRintro} applies to $\E=q_*\O_X$.  If $G=S_n$, we can apply Theorem \ref{ARRintro} to $\E =q_*(\O_X\oplus\O_X)$. (Note that $G$-torsors $X\to Y$ 
with $Y$ as in Prop. \ref{perfectProp} can be constructed 
 for any group $G$ as in \cite[Appendix]{PaUnr} by using the theorems of Moret-Bailly
and Rumely  on the existence
of integral points.)

 \subsection{The second adelic Chern class.}\label{sec5.2.3}
\hfill
\smallskip

The second adelic Chern class will only be defined for 
bundles with  elementary structure on suitable arithmetic surfaces. 
We suppose assumption (H) on $Y$ is satisfied;  we
also continue to suppose that the group algebra $\Q[G]$ splits.

\subsubsection{}
We start by showing the following properties of  elementary transitions.
If $\lambda_{\eta_i\eta_j}$ is an element of $ {\rE}(\hat{\mathcal{O}}_{Y, \eta_i\eta_j}[ G] ) $
we will denote by $\ti\lambda_{\eta_i\eta_j}$ an element of 
$ {\rm St}( \hat{\mathcal{O}}_{Y,\eta _{i}\eta _{j}}[ G] )$ that projects to $\lambda_{\eta_i\eta_j}$. 
 For simplicity,  we will sometimes omit the subscript $Y$ from the notation.

\begin{proposition}\label{propae}
Suppose that $\E$ is an $\O_Y[G]$-bundle with  elementary structure $\epsilon$ given by the bases 
$\{e_{\eta _{i}}^{h}\}_{h}$
with 
corresponding transition matrices
  $\lambda_{\eta_i\eta_j}$   in $ {\rE}(\widehat\O_{\eta_i\eta_j}[G])$.
 Then  we also have:

\begin{enumerate}

\item[1)] There is an effective divisor $D$ on $Y$ containing all vertical fibers over primes that divide the order of $G$ such that
for every Parshin triple $(\eta_0, \eta_1, \eta_2)$ on $Y$ we have: 

a) $\lambda_{\eta_0\eta_1}$ belongs to   $  {\rE}(\hat\O_{\eta_1}[D^{-1}][G])\subset 
 {\rE}(\hat\O_{\eta_0\eta_1}[G])$,

b) $\lambda_{\eta_0\eta_2}$ belongs to $ \rE(\hat\O_{\eta_2}[D^{-1}][G])\subset  {\rE}(\hat\O_{\eta_0\eta_2}[G])$.
\smallskip

\item[2)] For all lifts $\tilde\lambda_{\eta_i\eta_j}$ of $\lambda_{\eta_i\eta_j}$ with  $\tilde\lambda_{\eta_0\eta_1}$ in $\St(\hat\O_{\eta_1}[D^{-1}][G])$, $\tilde\lambda_{\eta_1\eta_2}$ in ${\rm St}(\what\O_{\eta_1\eta_2}[G])$, and
$\tilde\lambda_{\eta_0\eta_2}$ in $\St(\what\O_{\eta_2}[D^{-1}][G])$, the element
$$
z:=\prod_{(\eta_0,\eta_1,\eta_2)}\tilde\lambda_{\eta_0\eta_2}\cdot (\tilde \lambda_{\eta_1\eta_2})^{-1}\cdot (\tilde \lambda_{\eta_0\eta_1})^{-1}
$$
lies in the group 
\begin{equation*}
 {\rK}_{2}^{\prime } ( \mathbb{A}_{Y,012}[G]) \cdot   {\rK}_{2}( \mathbb{A}_{Y,12}[G])^\flat
\cdot  {\rK}_{2}^{\prime } ( \mathbb{A}_{Y,01}[G])^\flat. 
\end{equation*} 

\end{enumerate}
\end{proposition}

\begin{proof}
Recall $Y$ is integral so there is only one $\eta_0$ which we denote by $0$.  We will first show   (1).
Since $\hat\O_{\eta_0}$ is the function field $K(Y)$ of $Y$,
there is a divisor $D\subset Y$ such that the $e^h_0$ give an $\O_U[G]$-basis of $\E_{|U}$, where $U$
is the open complement of $D$ in $Y$. For simplicity, we will omit the superscript $h$.
By increasing $D$, we may assume that: (i) $D$ contains all
the vertical fibers of $Y$ over primes that divide the order of the group $G$,
and, (ii)  for every closed point $\eta_2$ of $Y$ on the support of $D$, there is an irreducible component of $D$ which is regular at $\eta_2$. (This second condition is needed to apply Corollary  \ref{cor:Blochcor}.)
Now notice that if $\eta_i$, $i=1$, $2$,
lie in $U$, then both $e_0$ and $e_{\eta_i}$ are bases of $\E$ over 
$\what\O_{\eta_i}$ and so the transition $\lambda_{\eta_0\eta_i}$ lies in the intersection $\rE(\hat\O_{\eta_0\eta_i}[G])\cap \GL(\hat\O_{\eta_i}[G])$.
Since the order of the group $G$ is invertible here, we have by Morita equivalence  and Corollary \ref{cor:Blochcor},  $\rE(\hat\O_{\eta_0\eta_i}[G])={\rm SL}(\what\O_{\eta_0\eta_i}[G])$,
$\rE(\what\O_{\eta_i}[G])={\rm SL}(\what\O_{\eta_i}[G])$. Hence, $\lambda_{0\eta_i}$ lies in $\rE(\what\O_{\eta_i}[G])=\rE(\what\O_{\eta_i}[D^{-1}][G])$
as required. Now suppose that $\eta_i$, $i=1$, $2$ is on $D$. If $i=1$, there is nothing to show since $\what\O_{\eta_1}[D^{-1}]=\what\O_{0\eta_1}$.
Consider $\eta_2$ on $D$. Both $e_0$ and $e_{\eta_2}$ are bases of $\E$ over $\what\O_{\eta_2}[D^{-1}]$ and as above 
we have $\lambda_{0\eta_2}\in \rE(\what\O_{0\eta_2}[G])\cap \GL(\what\O_{\eta_2}[D^{-1}][G])$. The order of $G$ is invertible in the rings
$\what\O_{0\eta_2}$, $\what\O_{\eta_2}[D^{-1}]$ and we are assuming that the group algebra
  $\Q[G]$ splits.  Hence, we can apply Morita equivalence and Corollary \ref{cor:Blochcor}  to show that
the group rings $\what\O_{0\eta_2}[G]$, $\what\O_{\eta_2}[D^{-1}][G]$, have trivial $\rSK_1$. Hence, as above, the intersection
$\rE(\what\O_{0\eta_2}[G])\cap \GL(\what\O_{\eta_2}[D^{-1}][G])$
is equal to $\rE(\what\O_{\eta_2}[D^{-1}][G])$ and we have just shown (1).

We will now show (2). First, we will show that we can adjust the bases at codimension $1$ points $\eta_1$ to make sure that
the new transition matrices $\theta_{\eta_i\eta_j}$ have lifts $\ti\theta_{\eta_i\eta_j}$ for which   (2)
is satisfied. We will not change the bases
for $\eta_1$ off $D$. If $\eta_1$ is a component of $D$ we have
$$
\lambda_{0\eta_1}\in \rE(\hat\O_{0\eta_1}[G])=\SL(\hat\O_{0\eta_1}[G]).
$$
By Lemma \ref{lemmaApprox} we have
\begin{equation}\label{hope1}
{\rm SL}(\what\O_{0\eta_1}[G])= {\rm SL}(\O_{0\eta_1}[G])\cdot {\rm SL}(\what\O_{\eta_1}[G]).
\end{equation}
Hence, we can write
$
\lambda_{0\eta_1}=  \nu_{0\eta_1}'\cdot \mu'_{\eta_1}
$
with $\mu'_{\eta_1}\in \SL(\what\O_{\eta_1}[G])$, $\nu'_{0\eta_1}\in \SL(\O_{0\eta_1}[G])$.
Now use that by Corollary \ref{SK1density} the natural map
\begin{equation}\label{hope2}
\rSK_1(\O_{\eta_1}[G])\to  {\rSK}_1(\what\O_{\eta_1}[G])
\end{equation}
is surjective to find a matrix $g\in \SL(\O_{\eta_1}[G])$ such that $[g]=[\mu'_{\eta_1}]$ in 
${\rSK}_1(\what\O_{\eta_1}[G])$. Our new basis at $\eta_1$ is $f_{\eta_1}=\mu_{\eta_1}\cdot e_{\eta_1} $ where
 $\mu_{\eta_1}:=g^{-1}\cdot \mu'_{\eta_1}  $ is in $ {\rE}(\what\O_{\eta_1}[G])$.
Now we can write
$$
\lambda_{0\eta_1}=( \nu_{0\eta_1}'\cdot g)\cdot (g^{-1}\cdot \mu'_{\eta_1} )   = \theta_{0\eta_1}\cdot \mu_{\eta_1}
$$
and observe that $\theta_{0\eta_1}:=\nu_{0\eta_1}'\cdot g$
is  actually in $\SL(\O_{0\eta_1}[G])=\rE(\O_{0\eta_1}[G])=\rE(K(Y)[G])$.
We leave all the other bases $e_{\eta_i}$ unchanged, \emph{i.e.} we set
$f_{0}=e_0$, $f_{\eta_2}=e_{\eta_2}$, and $f_{\eta_1}=e_{\eta_1}$ if $\eta_1$ is not on $D$.
Hence, if $\eta_1$ is on $D$, we have:
\begin{equation}\label{transD}
\theta_{0\eta_1}= \lambda_{0\eta_1}\cdot \mu^{-1}_{\eta_1},\quad \theta_{0\eta_2}= \lambda_{0\eta_2},\quad 
\theta_{\eta_1\eta_2}= \mu_{\eta_1}\cdot \lambda_{ \eta_1\eta_2}.
\end{equation}
On the other hand, if $\eta_1$ is not on $D$,  the transitions do not change:
\begin{equation}
\theta_{0\eta_1}= \lambda_{0\eta_1} ,\quad \theta_{0\eta_2}= \lambda_{0\eta_2},\quad 
\theta_{\eta_1\eta_2}=  \lambda_{ \eta_1\eta_2}.
\end{equation}
The new bases $f_{\eta_i}$ give an equivalent  elementary structure; we can also see that they satisfy (1)
for the same divisor $D$.
We will now explain how to pick lifts $\tilde\theta_{\eta_i\eta_j}$ of $\theta_{\eta_i\eta_j}$ so that the corresponding cocycle $z(\tilde\theta)
=\prod_{(0,\eta_1,\eta_2)}\tilde\theta_{0\eta_2}\cdot (\tilde\theta_{\eta_1\eta_2})^{-1}\cdot (\tilde\theta_{0\eta_1})^{-1} $   lies in $ {\rK}'_2({\mathbb A}_{012}[G])
 \cdot  {\rK}_{2}( \mathbb{A}_{Y,12}[G])^\flat $.  

a) Suppose $\eta_1$ is not on $D$.
Since $\hat\O_{\eta_2}[D^{-1}]\subset \hat\O_{\eta_1\eta_2}$, and $\hat\O_{\eta_1}\subset 
\hat\O_{\eta_1\eta_2}$ we can view all transitions $\theta_{0\eta_1} $,
$\theta_{0\eta_2} $, $\theta_{\eta_1\eta_2} $
as elements of $ {\rE}(\what\O_{\eta_1\eta_2}[G])$;  we pick lifts $\tilde\theta_{\eta_i\eta_j}$ in $ {\St}(\what\O_{\eta_1\eta_2}[G])$
and we can see that at such triples $(0,\eta_1, \eta_2)$ we have
\begin{equation}\label{ztheta1}
z(\tilde\theta)_{(0,\eta_1,\eta_2)} \in  {\rK}_2(\hat\O_{\eta_1\eta_2}[G]).
\end{equation}

b) If $\eta_1$ is on $D$, then by the above the transition $\theta_{0\eta_1}$ is in $\rE(K(Y)[G])$.
Therefore, we can find  a divisor $D'\subset Y$ with $\eta_1$ not in $D'$ such that $V=Y-(D\cup D')$
is affine and $\theta_{0\eta_1}$ comes from an element of  $\rE(\O_Y(V)[G])$. 
We pick a lift $\tilde\theta_{0\eta_1}\in \St(\O_Y(V)[G])$. Notice that it follows that  $f_{\eta_1}=\theta_{0\eta_1}^{-1}\cdot f_0$
is a basis of $\E$ over $V$. Since $f_{\eta_1}$ is also a basis over the completion of the local ring $\what\O_{\eta_1}$ of $Y$ at $\eta_1$,
we conclude that $f_{\eta_1}$ is in fact a basis on a Zariski open $W\subset Y$ that contains both $V$ and $\eta_1$.

Now suppose that $\eta_2<\eta_1$
is away from the finite set of points of $\bar\eta_1$ that do not belong to $W$
 and the
singular points of $D$ on all fibers of $Y$. Then $\O_Y(V)\subset \hat\O_{\eta_2}[D^{-1}]=\what\O_{\eta_2}[\eta_1^{-1}]$
and so both $\theta_{0\eta_1}$ and $\theta_{0\eta_2} $ can be viewed as elements of $\rE(\what\O_{\eta_2}[\eta_1^{-1}][G])$.

Let us consider  $\theta_{\eta_1\eta_2}$ for such $\eta_2<\eta_1$. Since $\eta_2$ is in $W$, both $f_{\eta_1}$ and $f_{\eta_2}$ are bases at $\eta_2$ and therefore
$$
\theta_{\eta_1\eta_2}\in \GL(\what\O_{\eta_2}[G])\cap  {\rE}(\what\O_{\eta_1\eta_2}[G]).
$$

(i) If $\eta_1$ is horizontal, since $\eta_2$ is on $W$, $\eta_2$ does not lie
 on any vertical component of $D$. Therefore,   the order of $G$ is invertible in $\what\O_{\eta_2}$
 and $ {\rSK}_1(\hat\O_{\eta_2}[G])=(1)$ which gives $\SL(\what\O_{\eta_2}[G])=\rE(\what\O_{\eta_2}[G])$. 
 Hence, in this case $\theta_{\eta_1\eta_2}$ is in $\rE(\what\O_{\eta_2}[G])\subset \rE(\what\O_{\eta_2}[\eta_1^{-1}][G])$. 

(ii) We obtain the same conclusion, \emph{i.e.}  $\theta_{\eta_1\eta_2}$ is in $\rE(\what\O_{\eta_2}[G])$, if 
 $\eta_1$ is vertical but is away from the prime divisors of the order  of $G$.

 (iii) Suppose now that $\eta_1$ is a vertical fiber over such a prime divisor $p$ of $\#G$.
We will then check that $\theta_{\eta_1\eta_2}$, which by the above is in $\SL(\what\O_{\eta_2}[G])$, 
is actually in  ${\rE}(\what\O_{\eta_2}[G])$.
For this, it is enough to check
 that the class $[\theta_{\eta_1\eta_2}]$ in $ {\rSK}_1(\what\O_{\eta_2}[G])$
 is trivial. Since $\theta_{\eta_1\eta_2}$ is in  $ {\rE}(\what\O_{\eta_1\eta_2}[G])$,
 the image of $[\theta_{\eta_1\eta_2}]$ under
 \begin{equation}\label{sk1inj}
 {\rSK}_1(\hat\O_{\eta_2}[G])\to  {\rSK}_1(\hat\O_{\eta_1\eta_2}[G])
 \end{equation}
 is trivial. Since  $\hat\O_{\eta_2}\simeq W(k)\lps T\rps$, $\hat\O_{\eta_1\eta_2}\simeq W(k)\{\!\{T\}\!\}$
with $k$ the residue field of $\eta_2$,
by Corollary \ref{SK1cor2} (b), the map (\ref{sk1inj}) is injective. Hence, $[\theta_{\eta_1\eta_2}]=1$,
and therefore $\theta_{\eta_1\eta_2}$ is in ${\rE}(\what\O_{\eta_2}[G])$.

 To recap, we have that for $\eta_1$ on $D$ and for almost all $\eta_2<\eta_1$, 
$$
\theta_{0\eta_2} \in \rE(\hat\O_{\eta_2}[\eta_1^{-1}][G]),\quad \theta_{\eta_1\eta_2}\in {\rE}(\hat\O_{\eta_2}[G])
$$
(with $ \theta_{\eta_1\eta_2}\in {\rE}(\hat\O_{\eta_2}[G])$ if $\eta_1$ is horizontal or vertical away from $\#G$).

For these (almost all) $\eta_2<\eta_1$, pick lifts   $\tilde\theta_{0\eta_2} \in \St(\what\O_{\eta_2}[\eta_1^{-1}][G])$,
and $\tilde\theta_{\eta_1\eta_2}\in  {\St}(\what\O_{\eta_2}[G])$, in addition to
our lift $\tilde\theta_{0\eta_1}\in \St(\O_Y(V)[G])$. All these lifts map  to elements of $\St(\what\O_{\eta_2}[\eta_1^{-1}][G])$.
 Indeed, when $\eta_1$ is vertical over $p|\#G$,  ${\rE}(\hat\O_{\eta_2}[G])\subset \SL(\what\O_{\eta_2}[\eta_1^{-1}][G])=\rE(\what\O_{\eta_2}[\eta_1^{-1}][G])$
with the last equality following from Proposition \ref{sk_1vanish} and Morita equivalence.
For all the other finite set of  $\eta_2<\eta_1$  pick lifts $\tilde\theta_{\eta_i\eta_j}$  such that $\tilde\theta_{0\eta_2} \in \St(\what\O_{\eta_2}[D^{-1}][G])$, 
$\tilde \theta_{\eta_1\eta_2}\in  {\St}(\what\O_{\eta_1\eta_2}[G])$.
Notice that for almost all $\eta_2<\eta_1$, we have
\begin{equation}\label{ztheta2}
z(\tilde\theta)_{0,\eta_1,\eta_2}=\tilde\theta_{0\eta_2}\cdot (\tilde\theta_{\eta_1\eta_2})^{-1}\cdot (\tilde\theta_{0\eta_1})^{-1} \in \rK_2(\what\O_{\eta_2}[\eta_1^{-1}][G]).
\end{equation}
Now in view of (PK1) and (PK2) of  Definition \ref{defrestrictedproducts}, (\ref{ztheta1}) and (\ref{ztheta2}) implies that
 the cocycle $z(\ti\theta)$ given using these lifts $\tilde\theta_{\eta_i\eta_j}$ of $\theta_{\eta_i\eta_j}$
lies in the group
$$
 {\rK}_{2}^{\prime } ( \mathbb{A}_{Y,012}[G]) \cdot   {\rK}_{2}( \mathbb{A}_{Y,12}[G])^\flat .
$$

Finally, we want to compare $z(\ti\theta)$ with $z(\ti\lambda)$ which is given by   lifts $\ti\lambda_{\eta_i\eta_j}$ as in the statement
of (2). If $\eta_1$ is not on $D$
then $z(\tilde\theta)_{0,\eta_1,\eta_2}=z(\tilde\lambda)_{0,\eta_1,\eta_2}$.
Suppose that $\eta_1$ is on $D$. Pick  a lift  $\ti\mu_1$ of $\mu_1\in  {\rE}(\widehat{\O}_{\eta_1}[G])$ to $ {\St}(\widehat{\O}_{\eta_1}[G])$.
By Lemma \ref{cocyclele} we have the cocycle identity,  where for simplicity, we replace the subscripts $\eta_1$, $\eta_2$ by $1$ and $2$:
\begin{equation}
z(\tilde\theta)_{0,1,2}=z(\mu_{ 1}, \lambda_{12})^{-1}\cdot z(\tilde\lambda)_{0,1,2}\cdot z(\lambda_{0 1}\mu_{ 1}^{-1}, \mu_{1})
\end{equation}
where we set (recall $\theta_{01}=\lambda_{0 1}\cdot \mu_{ 1}^{-1}$, $\theta_{12}=\mu_1\cdot \lambda_{12}$)
\begin{eqnarray*}
z ( \mu _{ 1},\lambda _{12} ) &=&\ti\theta_{12 }\cdot ( \ti\lambda _{12} ) ^{-1} \cdot
 ( \ti\mu_{1} ) ^{-1} \  \\
z ( \lambda _{01}\mu _{1}^{-1},\mu _{1} ) &=&   \ti \lambda _{01} \cdot    ( \ti\mu _{1} ) ^{-1} \cdot
( \ti\theta _{01} ) ^{-1}.
\end{eqnarray*}
The first expression is in $ {\rK}_2(\widehat\O_{\eta_1\eta_2}[G])$  and the second in $ {\rK}_2(\widehat\O_{0\eta_1}[G])$.
Therefore, by the above, for almost $\eta_2<\eta_1$, $\eta_1$ on $D$, the cocycle $z(\ti\lambda)_{0,\eta_1,\eta_2}$ lies
in
$$
 \rK_2(\what\O_{\eta_2}[\eta_1^{-1}][G])\cdot  {\rK}_2(\widehat\O_{\eta_1\eta_2}[G] )\cdot \rK_2(\what\O_{0\eta_1}[G]).
$$
We can conclude that $z(\ti\lambda)$ lies in the group 
\begin{equation*}
{\rK}_{2}^{\prime } ( \mathbb{A}_{Y,012}[G]) \cdot   {\rK}_{2}( \mathbb{A}_{Y,12}[G])^\flat
\cdot  {\rK}_{2}^{\prime } ( \mathbb{A}_{Y,01}[G])^\flat\ .
\end{equation*}%
as desired. 
 \end{proof}

\subsubsection{}
 Assume that $(\E, \epsilon)$ is a $\O_Y[G]$-bundle with elementary structure
given by $\{e_{\eta_{i}}^{h}\}_{h}$ with  transition matrices $\lambda _{\eta _{i}\eta _{j}}$.
Let the element
$$
z(\ti\lambda):=\prod_{(\eta _{0},\eta _{1},\eta
_{2})}\tilde\lambda_{\eta_0\eta_2}\cdot (\tilde \lambda_{\eta_1\eta_2})^{-1}\cdot (\tilde \lambda_{\eta_0\eta_1})^{-1}\ 
\in \
 {\rK}_{2}^{\prime } ( \mathbb{A}_{Y,012}[G]) \cdot   {\rK}_{2}( \mathbb{A}_{Y,12}[G])^\flat
\cdot  {\rK}_{2}^{\prime } ( \mathbb{A}_{Y,01}[G])^\flat.  
$$
be as in (2) of Proposition \ref{propae} above.

\begin{definition}\label{defc2} 
Assume that $(\E, \epsilon)$ is a $\O_Y[G]$-bundle with elementary structure.
 We define the adelic  second Chern class $ c_2(\E, \epsilon)$  to be the class of $ z(\ti\lambda)$ in $\mathrm{CH}_{\mathbb{A}}^{2}(Y[G])$. We will often abuse notation and write simply $c_2(\E)$ instead of $c_2(\E, \epsilon)$.
\end{definition}  

The following result   implies that $c_2(\E, \epsilon)$ only depends on $(\E, \epsilon)$  
and 
is independent of   the other choices involved in the definition.

\begin{theorem}\label{thm25}
a) The class $z( \ti\lambda)$ in $\mathrm{CH}_{\mathbb{A}}^{2}( Y[ G] ) $ is independent of the  
choice  of lifts $\ti\lambda_{\eta_i\eta_j}$ used in the definition of $z$.

b)  Let $\{\mu _{\eta _{i}}\}$ be an elementary base change and put $%
\theta _{\eta _{i}\eta _{j}}=\mu _{\eta _{i}}\lambda _{\eta _{i}\eta
_{j}}\mu _{\eta _{j}}^{-1}$ for the  elementary transitions in the new basis. For any  choice of lifts
  $\ti\theta_{\eta_i\eta_j}$ that satisfy the requirements of Proposition \ref{propae} (2)
we have
\begin{equation*}
z ( \ti\lambda )\cdot 
z ( \ti \theta   )
^{-1}\in \prod\nolimits_{0\leq i<j\leq 2} \rK_{2}^{\prime } ( \mathbb{A}%
_{Y,ij} [ G]  )^\flat .
\end{equation*}
As a result, the class of $z(\ti\lambda)$ in $\mathrm{CH}_{\mathbb{A}}^{2}( Y[ G] ) $ is unchanged by an elementary base change.
\end{theorem}

 {\it Proof of Theorem \ref{thm25} (a).}
We suppose that we have two choices of lifts $\ti\lambda_{\eta_i\eta_j}$, $\ti\lambda_{\eta_i\eta_j}'$;
we then have
\begin{eqnarray*}
z ( \ti\lambda )_{0,1,2} &=& \ti
\lambda _{02}  (\ti \lambda _{12}) ^{-1} 
( \ti \lambda _{01})^{-1} \\
z ( \ti\lambda' )_{0,1,2} &=& \ti
\lambda'_{02}  (\ti \lambda'_{12}) ^{-1} 
( \ti \lambda' _{01})^{-1}.
\end{eqnarray*}%
(For simplicity, here we drop the symbol $\eta$ from the subscripts.)
It then follows that 
$$
z(\ti\lambda) \cdot  z(\ti\lambda') ^{-1}=\ti
\lambda _{02}  (\ti \lambda _{12}) ^{-1} 
( \ti \lambda _{01})^{-1} \cdot    
  \ti \lambda' _{01}  \ti \lambda'_{12}  (\ti
\lambda'_{02})^{-1}.
$$
It will suffice to prove that for each $0\leq i<j\leq 2$  we have:
\begin{equation}\label{equlifts12}
\prod\nolimits_{\eta _{i}\eta _{j}}  (\ti \lambda_{ij} ) ^{-1} 
 \ti \lambda' _{ij}  \in  {\rK}_{2}^{\prime } ( \mathbb{A}_{Y,ij} [ G ] )^\flat .
\end{equation}
This easily follows   from the definitions (see Definition
\ref{defrestrictedproducts}) and Proposition \ref{propae}.
For example, consider $i=0$, $j=1$.
In this case,  we have $\lambda _{\eta _{0}\eta _{1}}\in
 {\rE} ( \what{\mathcal{O}}_{Y,\eta _{1}}[G]
 ) $ for almost all $\eta _{1}$; hence for such $\eta _{1}$ we have 
chosen $\ti\lambda _{\eta _{0}\eta _{1}}  $, $\ti\lambda'_{\eta _{0}\eta
_{1}}   \in  {\St} ( \what{\mathcal{O}}_{Y,\eta _{1}}[G])$   and therefore%
\begin{equation*}
\ti\lambda _{\eta _{0}\eta _{1}} \cdot (\ti\lambda'_{\eta _{0}\eta
_{1}} )^{-1}\in {\rK}_{2}(\what\O_{Y,\eta_1}[G]).
\end{equation*}

 
 {\it Proof of Theorem \ref{thm25} (b).}   By definition for each $i$
\begin{equation}\label{6.11}
\mu _{\eta _{i}}\in  {\rm E} ( \what{\mathcal{O}}_{Y,\eta _{i}}[ G ]  )
\end{equation}%
and so it certainly follows that $\theta _{\eta _{i}\eta _{j}}\in  {\rE%
}( \hat{\mathcal{O}}_{Y,\eta _{i}\eta _{j}}[ G] ) $, \emph{i.e.}    $\{\theta _{\eta _{i}\eta _{j}}\}$ 
are  elementary.
\smallskip

  Our base change can be performed in three steps: in each step
we alter the bases only at points in codimension $0$, $1$ and $2$,
by each of $\mu _{\eta _{0}},\mu _{\eta _{2}}$ or $\mu _{\eta
_{1}} $  respectively. For ease of notation we again write $\mu _{i}$, resp. $\lambda
_{ij}$, for $\mu _{\eta _{i}}$, resp. $\lambda _{\eta _{i}\eta _{j}}$ etc.
 
\smallskip

{\sl  Step 1.}    Here we just change the base at $\eta_0$ by $\mu_{\eta_0}$. 
The new transition matrices are $\theta_{01}=\mu_0\lambda_{01}$, $\theta_{02}=\mu_0\lambda_{02}$,
$\theta_{12}=\lambda_{12}$. We choose lifts $\ti\lambda_{ij}$, $\ti\theta_{ij}$ as in Proposition \ref{propae}.
We also choose a lifting 
$\ti\mu_0$ of $\mu_0$ as follows: there is a divisor $D_\mu\subset Y$ 
that contains all the vertical fibers at primes that divide the order of the group such that $U=Y-D_\mu$ is affine and
 $\mu_0$ lies in $ {\rE}(\what\O_{Y}(U)[G])$;
we  pick a lift $\ti\mu_0$ in $\St(\what\O_{Y}(U)[G])$.  As a result, 
for almost all $\eta_1$,  $\ti\mu_0$ maps to $ {\St}(\what\O_{Y,\eta_1}[G])$,
and for all $\eta_2$,  $\ti\mu_0$ maps to $ {\St}(\what\O_{Y,\eta_2}[D^{-1}][G])$.
By Lemma \ref{cocyclele} applied to the subset $\{\lambda_{01},\lambda_{12}=\theta_{12}, \mu_0, \lambda_{02}=\lambda_{01}\lambda_{12}, \theta_{01}=
\mu_0\lambda_{01} ,\theta_{02}=\mu_0\lambda_{02}=\mu_0\lambda_{01}\lambda_{12}\}$ of $ {\St}(\hat\O_{Y,0\eta_1\eta_2}[G])$ we have
\begin{eqnarray}\label{6.12}
z(\theta_{01}, \theta_{12})\ =\ z ( \mu _{0}\lambda _{01},\lambda _{12} ) &=&z ( \mu
_{0},\lambda _{01} ) ^{-1}z ( \mu _{0},\lambda _{01}\lambda
_{12} ) z ( \lambda _{01},\lambda _{12} ) \\
&=&z ( \mu _{0},\lambda _{01} ) ^{-1}z ( \mu _{0},\lambda
_{02} ) z ( \lambda _{01},\lambda _{12} )  \notag
\end{eqnarray}%
where
\begin{eqnarray*}
z ( \mu _{0},\lambda _{01} ) &=&\widetilde { \mu
_{0}\lambda _{01}}\cdot (\ti \lambda _{01} ) ^{-1} 
(\ti \mu _{0} )^{-1}\in  \rK_{2} ( \hat{%
\mathcal{O}}_{Y,\eta _{1}} [ G ]  ), \\
z ( \mu _{0},\lambda _{02} ) &=&\widetilde{\mu
_{0}\lambda _{02}}\cdot  ( \ti\lambda _{02} ) ^{-1} 
 (\ti \mu _{0} ) ^{-1}\in  \rK_{2} ( \hat{\mathcal{O}}_{Y,\eta _{0}\eta _{2}} [ G ]  ) .\
\end{eqnarray*}%
 
Notice that from Proposition \ref{propae} (a) and our choice of $\ti\mu_0$,
we can see that $\prod_{\eta _{2}}z (
\mu _{0},\lambda _{02} )$ is in $ \rK_{2}^{\prime } ( \mathbb{A}_{Y,02}%
 [ G ]  )$. Similarly,   $\prod_{\eta _{1}}z ( \mu
_{0}, \lambda _{01} )$ is in  $\ \rK_{2}^{\prime } ( \mathbb{A}_{Y,01} [ G 
 ] ) $.  The result then follows
since we have shown 
\begin{equation*}
z (\ti \lambda )\cdot  z (\ti\theta ) ^{-1}\in \prod\nolimits_{0\leq i<j\leq
2} \rK_{2}^{\prime } ( \mathbb{A}_{Y,ij} [ G ]  ) .
\end{equation*}

{\sl Step 2.}   We now change the bases at $\eta_1$ by $\mu_{\eta_1}$. 
The new transition matrices are given by $\theta_{01}= \lambda_{01}  \mu_1^{-1}$, $\theta_{02}= \lambda_{02}$,
$\theta_{12}=\mu_1\lambda_{12}$. We choose a lift $\ti\mu_{1}\in {\St}(\hat\O_{ Y, \eta_1}[G])$. 
 By Lemma \ref{cocyclele} we have 
\begin{equation}
z(\theta_{01}, \theta_{12})\ =\ z ( \lambda _{01}\mu _{1}^{-1},\mu _{1}\lambda _{12} ) =z ( \mu
_{1},\lambda _{12} ) ^{-1}z ( \lambda _{01},\lambda _{12} )
z ( \lambda _{01}\mu _{1}^{-1},\mu _{1} )
\end{equation} 
where  
\begin{eqnarray*}
z ( \mu _{1},\lambda _{12} ) &=&\widetilde{ \mu
_{1}\lambda _{12} } ( \ti\lambda _{12} ) ^{-1} 
 ( \ti\mu _{1} ) ^{-1}\in  \rK_{2} ( \what{%
\mathcal{O}}_{Y,\eta _{1}\eta _{2}} [ G ]  ) \ \text{\ } \\
z ( \lambda _{01}\mu _{1}^{-1},\mu _{1} ) &=&   \ti \lambda _{01}     ( \ti\mu _{1} ) ^{-1} 
( \widetilde{\lambda _{01}\mu _{1}^{-1}} ) ^{-1}\in
\rK_{2} ( \what{\mathcal{O}}_{Y,\eta _{0}\eta _{1}} [ G ]
 ) .
\end{eqnarray*}%
 By hypothesis,
  $\lambda _{01}$ is in  $ {\rE} ( \hat{\mathcal{O}}_{Y,\eta _{1}}[D^{-1}][ G ]  ) $, and so $
\prod_{\eta _{1}}z( \lambda _{01}\mu _{1}^{-1},\mu
_{1})$ is in  $  \rK_{2}^{\prime } ( \mathbb{A}_{Y,01} [ G ]
 )$.  The result then follows.
  
{\sl Step 3.}    We now change the bases at $\eta_2$ by $\mu^{-1}_{\eta_2}$
(the inverse is for ease in the notation below). 
The new transition matrices are then given by $\theta_{01}=\lambda_{01}$, 
$\theta_{12}=\lambda_{12}\mu_2$, $\theta_{02}=\lambda_{02}\mu_2$.
In addition to the lifts $\ti\lambda_{ij}$, $\ti\theta_{ij}$ we choose lifts $\ti\mu_{\eta_2}\in  {\St}(\hat\O_{ Y, \eta_2}[G])$. 
 By Lemma \ref{cocyclele} we have 
\begin{eqnarray}\label{6.13}
z(\theta_{01}, \theta_{12})\ =\ z ( \lambda _{01},\lambda _{12}\mu _{2} ) &=&z ( \lambda
_{12},\mu _{2} ) ^{-1}z ( \lambda _{01}\lambda _{12},\mu
_{2} ) z ( \lambda _{01},\lambda _{12} ) \\
&=&z ( \lambda _{12},\mu _{2} ) ^{-1}z ( \lambda _{02},\mu
_{2} ) z ( \lambda _{01},\lambda _{12} )  \notag
\end{eqnarray} 
where 
\begin{eqnarray*}
z ( \lambda _{12},\mu _{2} ) &=&\widetilde{  \lambda
_{12}\mu _{2}  }\  ( \ti\mu _{2} ) ^{-1}  ( \ti \lambda _{12} ) ^{-1}\in \rK_{2} ( \hat{\mathcal{O}}%
_{Y,\eta _{1}\eta _{2}} [ G ]  ) \ \text{\ } \\
z ( \lambda _{02},\mu _{2} ) &=&\widetilde{\lambda
_{02}\mu _{2} }\   ( \ti \mu _{2} ) ^{-1} (\ti\lambda _{02} ) ^{-1}\in \rK_{2} ( \hat{\mathcal{O}}%
_{Y,\eta _{0}\eta _{2}} [ G  ) .
\end{eqnarray*}%
 Since $\ti\lambda_{02}$ and $\ti\theta_{02}=\widetilde{\lambda
_{02}\mu _{2} }$ are in ${\St}(\what\O_{Y, \eta_2}[D^{-1}][G])$, for some divisor $D$,
the last expression   contributes
 a term $\prod_{\eta _{2}}z( \lambda _{02}, \mu_{2})$  in     $\rK_{2}^{\prime
} ( \mathbb{A}_{Y,02} [ G ] )$ and the result follows.
 This completes the proof of Theorem \ref{thm25}. \endproof

\bigskip

\section{Equivariant Euler characteristics and the Riemann-Roch theorem}

\setcounter{equation}{0}

We can now state our main result concerning equivariant coherent Euler characteristics.
We refer the reader to \cite{ChinburgTameAnnals} or \cite{ChinburgErez} for the construction of the projective equivariant  Euler characteristic. See also the beginning of the introduction.  Recall that we can identify the locally free class group ${\rm Cl}(\Z[G])$ with both the kernel $\rK^{\rm red}_0(\Z[G])$ of the rank map and with the quotient
$\rK_0(\Z[G])/\langle \Z[G]\rangle$. We will denote by $\bar\chi^P(Y, \E)$ the image 
of the projective equivariant  Euler characteristic $\chi^P(Y, \E)$ in ${\rm Cl}(\Z[G])=\rK_0(\Z[G])/\langle \Z[G]\rangle$.

\begin{theorem}
\label{thm:trystate}  Let $Y$ be a regular flat projective scheme over $\Spec(\Z)$
of dimension $2$, with structure morphism $h:Y \to S=\mathrm{Spec}(\Z)$. 
Assume in addition that $Y$ satisfies assumption (H) and that $\Q[G]$ splits as in (\ref{splitgroupalgebraIntro}).
Let $\E$  be an $\O_Y[G]$-bundle with an  elementary structure $\epsilon$
in the sense of Definition \ref{defadproel}.    
Then
\begin{equation}
\label{eq:theresult}
\bar\chi^P(Y, \E)  = -h_{*}( c_2(\E, \epsilon)   )
\end{equation}
 in  ${\rm Cl}(\Z[G])=\Kr_0^{\rm red}(\Z[G])={\rm CH}^1_{\Bbb A}(S[G])$.
 \end{theorem}
 
Suppose that $\E$ has rank $n$.
Since $\E$ has elementary structure,  $c_1(\E)$ is trivial and the usual Riemann-Roch
theorem for the generic fiber $Y_\Q$ shows that the rank of  $\chi^P(Y, \E)$ is equal to that
of $\chi^P(Y, \O_Y[G]^n)$.  Since $\chi^P(Y, \O_Y[G]^n)$ is the class of a free $\Z[G]$-module, the
above formulation of the main result is equivalent to the one in the introduction. 
In \S \ref{s:redP1} we reduce the proof of this result to the case $Y =\mathbb{P}^1= \mathbb{P}^1_{\Z}$.
 When $Y = \mathbb{P}^1$, a stronger result is proved in \S \ref{s:overp1}.

\bigskip

\section{The proof of the theorem; reduction to the case of ${\Bbb P}^1_\Z$.}
\label{s:redP1}

\setcounter{equation}{0}

Throughout this section we suppose that $h:Y\to \mathrm{Spec}(\mathbb{Z})$, $G$  and $\E$ are as in Theorem \ref{thm:trystate}.
By a result of B. Green (see \cite{GreenSkolem} and \cite{GreenP1} but also \cite{FFcapacity}), there is a finite flat morphism $\pi:Y \to \mathbb{P}^1=\mathbb{P}^1_{\mathbb{Z}}$. 
Let $f:\mathbb{P}^1 \to \mathrm{Spec}(\Z)$ be the structure morphism, so that $h = f \circ \pi$.  
Let $d$ be the degree of $\pi$.   We can view $\V = {\cal Hom}_{\O_{\mathbb{P}^1}}(\pi_* \O_Y[G]^n,\O_{\mathbb{P}^1})$ 
and $\V' = \pi_* \O_Y[G]^n$ as locally free $\O_{\mathbb{P}^1}[G]$-modules of rank $nd$. Parts (iii) and (iv) of
the following result imply Theorem \ref{thm:trystate}. In what follows, for the sake of simplicity, we will omit the 
notation of the elementary structure and simply write $c_2(\E)$ instead of $c_2(\E, \epsilon)$.

\begin{theorem}
\label{thm:trystate2} Let $\E$ be an  $\O_Y[G]$-bundle with elementary structure.   
\begin{enumerate}
\item[i)] The bundles $\pi_* (\E) \oplus \V$  and $\V' \oplus \V$ are locally free and have  elementary structures as $\O_{\mathbb{P}^1}[G]$-bundles.  They therefore have well defined second Chern classes $c_2(\pi_* \E \oplus \V)$ and $c_2(\V' \oplus \V)$ in 
${\rm CH}^2_{\Bbb A}(\mathbb{P}^1[G])$.

\item[ii)]  There is a push down map $\pi_*:{\rm CH}^2_{\Bbb A}(Y[G]) \to {\rm CH}^2_{\Bbb A}(\mathbb{P}^1[G])$
   induced by $\pi:Y \to \mathbb{P}^1$.  One has
   \begin{equation}
   \label{eq:soso}
   c_2(\pi_* (\E) \oplus \V) = \pi_*(c_2(\E)) + c_2(\V' \oplus \V).
   \end{equation}
   
\item[iii)]  There are equalities of equivariant Euler characteristics 
 \begin{equation}
 \label{eq:eulerpush}
  \bar\chi^P(Y, \E) = \bar\chi^P(\PP^1, \pi_* \E) =  \bar\chi^P(\PP^1, \pi_* \E \oplus \V)
    \end{equation}
  in  ${\rm Cl}(\Z[G])=\Kr_0^{\rm red}(\Z[G])={\rm CH}^1_{\Bbb A}(S[G])$.
  
  \item[iv)] We have 
 \begin{eqnarray}
 \label{eq:eulerpush2}
 \bar\chi^P(\PP^1, \pi_* \E \oplus \V)  &=& -f_*(c_2(\pi_* \E \oplus \V))  =  -f_*(\pi_*(c_2(\E)) + c_2(\V' \oplus \V) )\nonumber\\
 & =&  -f_*(\pi_*(c_2(\E))  ) = -h_*(c_2(\E)).
  \end{eqnarray}
  \end{enumerate}
  \end{theorem}

Notice that, in view of (i), the first equality in part (iv) above follows from the case $Y=\mathbb{P}^1$ of Theorem 
\ref{thm:trystate}.

\subsection{Constructing bundles with  elementary structures}    

\begin{lemma}
\label{lem:invol}
Suppose $R$ is an arbitrary ring.  There is an order two automorphism $\sigma$ of  $\rK_1(R) = \GL(R)/\rE(R) = \GL(R)^{\rm ab}$
induced by the anti-involution $\sigma:A \to A^t$ on $\GL(R)$, where $A^t$ is the transpose of the
matrix $A$.  This involution is trivial if and only if for all $A \in \GL(R)$, the block matrix
\begin{equation}
\label{eq:Aist}
 \begin{pmatrix}A &0\\0&(A^t)^{-1}\end{pmatrix}  
\end{equation}
lies in $\rE(R)$.  This is the case, in particular, if $R$ is commutative and $\rSK_1(R)$ is trivial.
\end{lemma}

\begin{proof} By the Whitehead Lemma, $\rE(R)$ is the commutator subgroup of $\GL(R)$.
If $[A,B] = ABA^{-1}B^{-1}$ is a commutator, then $\sigma([A,B]) =  [(B^t)^{-1},(A^{-1})^t]$
is also a commutator.  Hence $\rE(R)$ is stable under $\sigma$.  Since $(AB)^t = B^t A^t$ for all $A$, $B \in \GL(R)$
and $\rK_1(R)$ is the maximal abelian quotient of $\GL(R)$, $\sigma$ defines a group automorphism of $\rK_1(R)$.
By   \cite[Corollary 2.1.3]{RosenbergBook}, the block matrix
$$\begin{pmatrix}A &0\\0&A^{-1}\end{pmatrix}$$
lies in $\rE(R)$ for all $A \in \GL(R)$.  Hence $\sigma(A) = A^t$ equals $A$ in $\rK_1(R)$ if and only if (\ref{eq:Aist}) lies in $\rE(R)$.
If $A$ commutative and $\rSK_1(A)$ is trivial, then $\sigma$ is trivial since $\mathrm{det}(A)= \mathrm{det}(A^t)$.
  \end{proof}
  \smallskip
  
  We do not know whether $\sigma$ is trivial for arbitrary $R$.

\begin{proposition}
\label{prop:basec}
Let  $\E$ be as in  Theorem \ref{thm:trystate}.  The direct image $\pi_* \E$ is a rank $n$ locally free sheaf
of  $\pi_* \O_{Y}[G]$-modules on $\mathbb{P}^1=\mathbb{P}^1_{\mathbb{Z}}$ as well as a locally free sheaf of 
$\O_{\mathbb{P}^1}[G]$-modules of rank $nd$.  The sheaves $\V = {\cal Hom}_{\O_{\mathbb{P}^1}}(\pi_* \O_Y[G]^n,\O_{\mathbb{P}^1})$ 
and $\V' = \pi_* \O_Y[G]^n$ are locally free $\O_{\mathbb{P}^1}[G]$-modules of rank $nd$. 
\begin{enumerate}
\item[i)] There are equalities of equivariant Euler characteristics 
 \begin{equation}
 \label{eq:eulerpush3}
  \bar\chi^P(Y, \E) = \bar\chi^P(\PP^1, \pi_* \E) \quad \mathrm{and} \quad \bar\chi^P(\PP^1,\V) = \bar\chi^P(\PP^1,\V') = 0.
  \end{equation}
  in  ${\rm Cl}(\Z[G])=\Kr_0^{\rm red}(\Z[G])={\rm CH}^1_{\Bbb A}(S[G])$.
  \item[ii)]
  The bundles $\pi_*\E \oplus \V$ and $\V' \oplus \V$ have  elementary structures on $\mathbb{P}^1$.
  The restrictions of $\pi_* \E \oplus \V$ and $\V' \oplus \V$ to the zero section of $\mathbb{P}^1$
  define locally free $\mathbb{Z}[G]$-modules which are stably free.  
  
   \end{enumerate}
\end{proposition}

\begin{proof}
 
 The first equality in (\ref{eq:eulerpush3}) is clear. Since $\V$ and $\V'$ are induced from the trivial subgroup of $G$,
they have trivial stable Euler characteristics as in (\ref{eq:eulerpush3}).  We now show  (ii) for the bundle $\pi_* \E \oplus \V$, since the case of $\V' \oplus \V$
 is similar.  

We will use $\eta'_i$ to denote 
a point of $Y$.  By assumption, there is a set  of $\what{\mathcal{O}}_{Y,\eta'
_{i}}[ G] $ bases $\{e_{\eta'_{i}}^{h}\}_{h}$ of $\mathcal{E\otimes }_{\O_{Y}}\hat{\mathcal{O}}_{Y,\eta'_{i}}$
which has the properties of Definition \ref{defadproel} when one replaces $\eta_i$ in
this definition by $\eta'_i$.  Here $h$ runs from $1$ to $n = \mathrm{rank}_{\O_Y[G]}(\E)$.  
Suppose $\eta$ is a non-degenerate Parshin chain on $\mathbb{P}^1$.  
Proposition \ref{prop:finitefree}  shows that
\begin{equation}
 \label{eq:blap3}
\widehat{\O}_{\mathbb{P}^1,\eta} \otimes_{\O_{\PP^1}} \pi_*\O_Y[G] =   \oplus_{\eta' \in \pi^{-1}(\eta)}\  \hat{\O}_{Y,\eta'}[G]
 \end{equation}
 where $\eta'$ runs over the Parshin chains on $Y$ over $\eta$.      We thus have
 \begin{equation}
 \label{eq:blap4}
 \pi_* \E \otimes_{\O_{\PP^1}} \hat{\O}_{\mathbb{P}^1,\eta} = \oplus_{\eta' \in \pi^{-1}(\eta)} \ \ \mathcal{E\otimes }_{\O_{Y}} {\hat{\O}}_{Y,\eta'}.
 \end{equation}
 The bases $\{e_{\eta'_{i}}^{h}\}_{h}$ 
 together with the isomorphisms (\ref{eq:blap3}) and (\ref{eq:blap4})  give a set of local bases $\{e_{\eta_{i}}^h\}$ 
 for $\pi_* \E \otimes_{\O_{\PP^1}}  \hat{\O}_{\mathbb{P}^1,\eta_i}$ as a $(\pi_* \O_{Y} \otimes_{\O_{\PP^1}}  \hat{\O}_{\mathbb{P}^1,\eta_i})[G]$-module, where $\eta_i$ ranges over the points of $\mathbb{P}^1 $.

We now consider transition matrices.  Let  $\eta =  (\eta_{i},\eta_{j}) $ be a non-degenerate Parshin chain of length
two on $\mathbb{P}^1 $.   We then have a transition map $\lambda _{\eta}$ in 
$\GL_n( \hat{\O}_{\mathbb{P}^1,\eta} \otimes_{\O_{\PP^1}} \pi_*\O_Y[G]) $ determined by the bases $\{e_{\eta_{i}}^h\}$  and
$\{e_{\eta_{j}}^h\}$ for the completion of $\pi_* \E$ at $\eta_i$ and $\eta_j$, respectively.  
 The isomorphisms (\ref{eq:blap3}) and (\ref{eq:blap4}) identify $\lambda_{\eta}$ with the direct sum of the transition matrices
 $\lambda_{\eta'}$ which result from taking $\{e_{\eta_{i}}^h\}$ (resp. $\{e_{\eta_{i}}^h\}$) as a  basis for the completion of $\E$  at each point $\eta'_i$ (resp. $\eta'_j$) of $Y$ over $\eta_i$ (resp. $\eta_j$).  
 
 We can choose a basis $\{w_{\eta_i}^{\ell}\}_\ell$ for $(\pi_* \O_{Y})_{\eta_i}$ as a free module for $\widehat{\O}_{\mathbb{P}^1, \eta_i}$
 at each point $\eta_i$ of $\mathbb{P}^1$ which has the following properties.  The index  $\ell$
 runs from $1$ to the degree $d$ of $\pi:Y \to \mathbb{P}^1$.  If $\eta_0$ is the generic point of $\mathbb{P}^1$, then $w_{\eta_i} = w_{\eta_0}$ for almost all codimension $1$ points $\eta_i$. For each codimension $1$ point $\eta_i$, we can
 arrange that for almost all closed points $\eta_j$ on the closure of $\eta_i$, the basis element $w_{\eta_j}^\ell$ equals
 $w_{\eta_i}^\ell$.   
 
 Returning now to our original set-up, we have obtained a basis 
$W_{\eta_i} = \{w_{\eta_i}^\ell e_{\eta_i}^h\}_{\ell,h}$  for \hbox{$\hat{\O}_{\mathbb{P}^1,\eta_i} \otimes_{\O_{\mathbb{P}^1}} \pi_* \E$} as a free module for $\hat{\O}_{\mathbb{P}^1,\eta_i}[G]$. 

Recall $\V' = \pi_* \O_Y[G]^n$, which we may consider as either a $\pi_* \O_Y[G]$-module or as a 
$\O_{\mathbb{P}^1}[G]$-module.    Let $\{ e'^{h}\}_h$  be a global basis for $\V'
$ as a free $\pi_* \O_Y[G]$-module of rank $d$.   Then $ W'_{\eta_i} = \{w_{\eta_i}^\ell  e'^h\}_{\ell,h}$ gives a basis for $\V'_{\eta_i} = (\pi_* \V')_{\eta_i}$ as locally free 
$\hat{\O}_{\mathbb{P}^1,\eta_i}[G]$-module of rank $nd$.  

We  use the bases $W_{\eta_i}$, $W_{\eta_j}$, $W'_{\eta_i}$  and $W'_{\eta_j}$ to arrive at transition matrices
$\lambda_{W,\eta}$ and $\lambda_{ W',\eta}$ in $\GL_{nd}( \hat{\O}_{\mathbb{P}^1,\eta}[G])$ for $\pi_* \E$ and
$\V'$ considered as locally free $\O_{\mathbb{P}^1}[G]$-modules.  Note that $\lambda_{ W',\eta}$ lies in
$\GL_{nd}( \hat{\O}_{\mathbb{P}^1,\eta})$, i.e. its entries have group ring elements which are in fact constants. To compare
$\lambda_{W,\eta}$ and $\lambda_{ W',\eta}$, we
will use the embedding
$$
r_j:\GL_n(\hat{\O}_{\mathbb{P}^1,\eta_j} \otimes_{\O_{\PP^1}}\pi_* \O_Y[G]) \to \GL_{nd}(\hat{\O}_{\mathbb{P}^1,\eta_j}[G])
$$
which results from the basis $\{w_{\eta_j}^{\ell}\}_\ell$  for $(\pi_* \hat{\O}_{Y})_{\eta_j}$ as a free module for $\hat{\O}_{\mathbb{P}^1_{\eta_j}}$.
This extends by tensor product with $\hat{\O}_{\mathbb{P}^1,\eta}$ to an embedding
\begin{equation}
\label{eq:rjbig}
r_j:\GL_n({\hat{\O}}_{\mathbb{P}^1,\eta} \otimes_{\O_{\PP^1}} \pi_*\O_Y[G]) \to \GL_{nd}({\hat{\O}}_{\mathbb{P}^1,\eta}[G])
\end{equation}
where $\eta = (\eta_i,\eta_j)$ as before.  Composing $r_j(\lambda _{\eta})$ with the transition matrix
$\lambda_{ W',\eta}$ associated with changing bases for $\V' $ from the $ W'_{\eta_j}$ to $W'_{\eta_i}$ gives
the transition matrix $\lambda_{W,\eta}$ associated with changing bases for the completion of $\pi_* \E$ from $W_{\eta_j}$ to $W_{\eta_i}$.
We thus have the matrix equation
\begin{equation}
\label{eq:equat}
\lambda_{W,\eta} = \lambda_{ W',\eta} \cdot r_j(\lambda_{\eta})
\end{equation}
inside 
$\GL_{nd}(\hat{\O}_{\mathbb{P}^1,\eta}[G])$.

As in the statement, let $\V$ be the locally free $\hat{\O}_{\mathbb{P}^1}[G]$-module of rank $nd$ defined by
$${\cal Hom}_{\O_{\mathbb{P}^1}}(\pi_* \O_Y[G]^n,\O_{\mathbb{P}^1}) = {\cal Hom}_{\O_{\mathbb{P}^1}}(\V',\O_{\mathbb{P}^1}).$$  Let $ W'^*_{\eta_i}$ be the basis for $\V$
which is the $\O_{\mathbb{P}^1}$ dual to the basis $W'_{\eta_i}$ for $\V'$ at $\eta_i$.
Then the transition matrix $\lambda_{ W'^*,\eta}$ associated to this choice  is 
\begin{equation}
\label{eq:zowier}
\lambda_{W'^*,\eta} = (\lambda_{ W',\eta}^t)^{-1}
\end{equation}
where the superscript $t$ on the right stands for  the transpose. 

We conclude that the transition matrix for $\pi_* \E \oplus \V'$  has the block form
\begin{equation}
\label{eq:newdp}
 \begin{pmatrix}\lambda_{ W',\eta} \cdot r_j(\lambda_{\eta})&0\\0&(\lambda_{ W',\eta}^t)^{-1}\end{pmatrix} = 
\begin{pmatrix}A &0\\0&(A^t)^{-1}\end{pmatrix}  \cdot  \begin{pmatrix}r_j(\lambda_{\eta})&0\\0&1\end{pmatrix}
\end{equation} 
where $A = \lambda_{ W',\eta}$.    
To show $\pi_* \E \oplus \V'$ has an elementary structure, it is enough by Definition \ref{defadproel} to show that each
of the two matrices on the right side of (\ref{eq:newdp}) is  elementary.  The first matrix is  elementary by 
Lemma \ref{lem:invol} since the entries of $A$ lie in the commutative ring ${\hat{\O}}_{\mathbb{P}^1,\eta}$ and $ \rSK_1({\hat{\O}}_{\mathbb{P}^1,\eta}) = \{1\}$ for all Parshin pairs $\eta$ by Corollary \ref{cor:Blochcor}. The second
matrix on the right hand side of (\ref{eq:newdp}) is elementary because $\lambda_{\eta}$ is so by assumption
and $r_j$ is a ring homomorphism.

The last statement to prove is that the restriction of $\pi_* \E \oplus \V'$ to the zero section of $\mathbb{P}^1$
is a stably free projective $\Z[G]$-module.  This restriction has an elementary structure as a projective $\Z[G]$-module. Therefore, its transition matrices have trivial determinant ${\rm Det}$ 
and hence this module is stably free by resolvent theory in dimension $1$ (e.g   \cite{FrohlichBook}). 
\end{proof}

\subsection{Steinberg extensions over $Y$ and over $\mathbb{P}^1$.}
\label{s:pushdowns}

Let $(\eta_i, \eta_j)$ stand for a Parshin chain of length two on $\mathbb{P}^1$.  For simplicity, we will denote 
by $R_{ij}$, resp. $S_{ij} $, the ring $\hat\O_{\mathbb{P}^1,\eta_i\eta_j}$, resp. 
$\hat\O_{Y, \eta_i\eta_j}=\pi_*\O_Y\otimes_{\O_{\mathbb{P}^1}}\hat\O_{\mathbb{P}^1,\eta_i\eta_j}$.
 A choice of basis of $S_{ij}$ over $%
R_{ij}\ $yields a homomorphism
\begin{equation*}
r:\GL_{n}( S_{ij}[ G] ) \rightarrow \GL_{nd}( R_{ij}[ G] )
\end{equation*}%
which induces a map on  elementary matrices%
\begin{equation*}
r_{E}: \rE_{n}( S_{ij}[ G]) \rightarrow \rE_{nd}(
R_{ij}[ G] ) 
\end{equation*}%
and we get a diagram
\begin{equation}
\begin{array}{ccccccccc}
1 & \rightarrow &  {\rK}_{2}( S_{ij}[ G] ) & \rightarrow &
 {\rm St}( S_{ij}[ G] ) & \rightarrow &  {\rE}( S_{ij}[
G] ) & \rightarrow & 1 \\
&  & \downarrow r_{K} &  & \downarrow r_{S} &  & \downarrow r_{E} &  &  \\
1 & \rightarrow &  {\rK}_{2}( R_{ij}[ G] ) & \rightarrow &
 {\rm St}( R_{ij}[ G] ) & \rightarrow &  {\rE}( R_{ij} [
G] ) & \rightarrow & 1.%
\end{array}%
\end{equation}%
Suppose now that we change bases by a matrix $c\in {\rE}( R_{ij}) $
and we then replace $r_{E}$ by $r_{E}^{c}=c^{-1}r_{E}c$; since $c$
 elementary it has a lift $s( c) \in  {\rm St}( R_{ij}) $ and
conjugation by this element in independent of the lift, as two such lifts
differ by an element of $ {\rK}_{2}( R_{ij}) $ which is central;
finally, conjugation by $s( c)$ on $ {\rK}_{2}( R_{ij}[ G] ) $ is trivial and so $r_{K}^{s( c) }=r_{K}$.

More generally, reasoning as above, if we choose $c\in  {\rE}( R_{012}[
G]) $ (in fact $c\in {\rE}( R_{012}) $ will suffice for
our purposes) and consider the effect of conjugation by $c$, denoted
when necessary $\mathrm{conj}( c)$, then we get a diagram:%
\begin{equation}
\begin{array}{ccccccccc}
1 & \rightarrow &  {\rK}_{2}( R_{ij}[ G] ) & \rightarrow &
 {\rm St}( R_{ij}[ G] ) & \underset{\pi _{ij}}{\overset{s_{ij}%
}{\leftrightarrows }} & {\rE}( R_{ij}[ G] ) & \rightarrow
& 1 \\
&  & \downarrow \mathrm{conj}( c)  &  & \downarrow \mathrm{conj}( c) &  &
\downarrow \mathrm{conj}( c) &  &  \\
1 & \rightarrow &  {\rK}_{2}( R_{ij}[ G] )^c & \rightarrow &
 {\rm St}( R_{ij}[ G] ) ^{c} & \underset{\pi _{ij,c}}{\overset%
{s_{ij,c}}{\leftrightarrows }} &   {\rE}( R_{ij}[ G] ) ^{c}
& \rightarrow & 1.%
\end{array}%
\end{equation}%
Here $s_{ij}$ and $s_{ij,c}$ denote sections and ${\rK}_{2}( R_{ij}[ G] )^c$,
 ${\rm St}( R_{ij}[ G] ) ^{c}$ are just formal copies of ${\rK}_{2}( R_{ij}[ G] )$,
 ${\rm St}( R_{ij}[ G] )$, while ${\rE}( R_{ij}[ G] ) ^{c}$ is the conjugate in 
 ${\rE}( R_{012}[
G]) $. Note here that for the images in $\rK_2(R_{012}[G])$ we have
$({\rK}_{2}( R_{ij}[ G] )^c )^\flat=({\rK}_{2}( R_{ij}[ G] )^\flat )^c=\rK_2(R_{ij}[G])^\flat$.
  Note also that the right-hand
square commutes with respect to the $\pi $ maps, but not necessarily for the section
maps: to be more precise we have:

\begin{lemma}
\label{lem:preciselem}
Given $x\in  {\rE}( R_{ij}[ G] ) $ we have $\kappa
=s_{ij,c}( x^{c}) s_{ij}( x) ^{-c}\in {\rK}_{2}(R_{ij}[ G])^c$.
\end{lemma}

\begin{proof} By definition $\kappa $ is the product of two elements in $ {\rm St}(
R_{ij}[ G] ) ^{c}$ and it will suffice to show that $\kappa
\in \ker (\pi _{ij,c})$. To this end we note:%
\begin{eqnarray*}
\pi _{ij,c}( \kappa ) &=&\pi _{ij,c}( s_{ij,c}(
x^{c}) ) \cdot \pi _{ij,c}( s_{ij}( x) ^{-c}) \\
&=&\pi _{ij,c}( s_{ij,c}( x^{c}) ) \cdot\pi _{ij}(
s_{ij}( x) ) ^{-c} \\
&=&x^{c}\cdot x^{-c}=1.
\end{eqnarray*}
\end{proof}

\subsection{Base change.}
\label{s:basechange}

Let $\{e_{i}\}$ be a basis for $\mathcal{E}\otimes_{\O_Y}S_{i}$ over $S_{i}[G] $.  
(There is an  additional implicit subscript we will suppress which runs from $1$ to the rank of $\E$.)
Let $e_{i}=\mu _{ij}e_{j}.$

Let $\{a_{in}\}_{n=1}^{d}$ be a basis for $S_{i}$ over $R_{i}$; we let $
( a_{in}) _{n}=\Lambda _{ij}( a_{jn}) _{n}$ and,
as in the proof of Proposition \ref{prop:basec}, we write
\begin{equation}
{\Lambda}^\sharp_{ij}= 
\begin{pmatrix}
\Lambda _{ij} & 0 \\
0 & (\Lambda _{ij}^t)^{-1}
\end{pmatrix} .
\end{equation}%
We then have bases $( a_{in}e_{i})_{n}$ for $\mathcal{E}\otimes_{\O_Y}
S_{i}$ over $R_{i}[G]$ and hence a further set of transition
matrices
\begin{equation}
( a_{in}e_{i}) _{n}=\lambda _{ij}\cdot ( a_{jn}e_{j})_{n}
\end{equation}%
so that  the bundle $\pi _{\ast }\mathcal{E}\oplus \mathcal{V}$
has transition matrices
\begin{equation*}
{\lambda }^\sharp_{ij}= 
\begin{pmatrix}
\lambda _{ij} & 0 \\
0 & (\Lambda_{ij}^t)^{-1}
\end{pmatrix}.
\end{equation*}
We observe that by the very definition of the map $r_{i}:\GL_{n}( S_{i}[ G]) \rightarrow \GL_{nd}( R_{i}[ G]) $ we have the further equality:
\begin{equation}
( a_{in}e_{i}) _{n}=r_{i}( \mu _{ij})\cdot  (
a_{in}e_{j}) _{n}.
\end{equation}

Next we observe that by definition of $r_{k}$
\begin{equation*}
r_{k}( \mu _{ij})\cdot ( a_{kn}e_{j}) _{n}=(
a_{kn}e_{i}) _{n}
\end{equation*}%
while
\begin{equation*}
\Lambda_{ki}\lambda_{ij}\Lambda_{jk}\cdot ( a_{kn}e_{j})_{n}
=\Lambda_{ki}\lambda_{ij}\cdot ( a_{jn}e_{j})_{n}=\Lambda_{ki}\cdot ( a_{in}e_{i})_{n}=( a_{kn}e_{i})_{n}.
\end{equation*}%
Therefore, we deduce that
\begin{eqnarray}
\label{eq:thirteen}
r_{k}( \mu _{ij})  = \Lambda _{ki}\lambda _{ij}\Lambda _{jk}, \qquad
{r}^\sharp_{k}( \mu _{ij})  = {\Lambda }^\sharp_{ki}%
{\lambda }^\sharp_{ij}{\Lambda }^\sharp_{jk}
\end{eqnarray}%
where
\begin{equation*}
{r}^\sharp_{k}( \mu_{ij}) =
\begin{pmatrix}
r_{k}( \mu_{ij}) & 0 \\
0 & 1
\end{pmatrix} ;
\end{equation*}%
and hence
\begin{equation}
\label{eq:fourteen}
\lambda _{ij} =\Lambda _{ik}r_{k}\left( \mu _{ij}\right) \Lambda _{kj}, \qquad 
{\lambda }^\sharp_{ij} ={\Lambda }^\sharp_{ik}{r}^\sharp_{k}\left( \mu _{ij}\right) {\Lambda }^\sharp_{kj}.
\end{equation}

From (\ref{eq:thirteen}) we conclude that for all  $k$, $h$ one  has
\begin{equation}
\label{eq:importanteq}
{\Lambda }^\sharp_{ik}{r}^\sharp_{k}( \mu _{ij}) {%
\Lambda }^\sharp_{kj}={\lambda }^\sharp_{ij}={\Lambda }^\sharp_{ih}{r}^\sharp_{h}( \mu _{ij}) {\Lambda }^\sharp_{hj}
\end{equation}%
and so
\begin{equation}
\label{eq:sixteen}
{r}^\sharp_{k}( \mu _{ij}) ={\Lambda }^\sharp_{ki}%
{\Lambda }^\sharp_{ih}{r}^\sharp_{h}( \mu _{ij}) {
\Lambda }^\sharp_{hj} {\Lambda }^\sharp_{jk}= {\Lambda }^\sharp_{kh} {r
}^\sharp_{h}( \mu _{ij})  {\Lambda }^\sharp_{hk}.
\end{equation}

\subsection{Reduction step outline.}

\begin{lemma} 
\label{lem:pushpat} Assume the notation and hypotheses of Theorem \ref{thm:trystate} and Proposition \ref{prop:basec}.  
The finite flat morphism
$\pi:Y \to \mathbb{P}^1=\mathbb{P}^1_{\Z}$ induces a pushdown homomorphism
\begin{equation}
\label{eq:chowpush}
\pi_{*}:{\rm CH}^2_{\Bbb A}(Y[G]) \to {\rm CH}^2_{\Bbb A}(\mathbb{P}^1[G]).
\end{equation}
The composition of the morphism $\pi:Y \to \mathbb{P}^1$
with the structure morphism $f:\mathbb{P}^1 \to \mathrm{Spec}(\Z)$
gives the structure morphism $h = f \circ \pi:Y \to \mathrm{Spec}(\Z)$.  There is an equality 
of pushdown homomorphisms
$$h_* = f_* \circ \pi_*$$
from ${\rm CH}^2_{\Bbb A}(Y[G])$ to ${\rm Cl}(\Z[G])=\Kr_0^{\rm red}(\Z[G])={\rm CH}^1_{\Bbb A}(\Spec(\Z)[G])$.
\end{lemma}

\begin{proof} Setting $i = 0$, we fix a basis for the local ring $\O_{Y,\eta'_0} = \pi_*\O_{Y} \otimes_{\O_{\PP^1}} \O_{\mathbb{P}^1,\eta_0}$ at the generic point $\eta'_0$
of $Y$ as a free rank $d$ module over $\O_{\mathbb{P}^1,\eta_0}$ when $\eta_0$ is the generic point of 
$\mathbb{P}^1$.  This basis defines an algebra map $r_0:\hat\O_{Y,\eta_0\eta_1\eta_2}[G] \to \mathrm{M}_d(\hat\O_{\mathbb{P}^1,\eta_0\eta_1\eta_2}[G])$ for every Parshin triple $\eta_0 \eta_1 \eta_2$ on $\mathbb{P}^1$, where
$\hat\O_{Y,\eta_0\eta_1\eta_2} = \pi_*\O_{Y} \otimes_{\O_{\PP^1}}  \hat\O_{\mathbb{P}^1,\eta_0\eta_1\eta_2}$ is the direct sum of
$\hat\O_{Y,\eta'_0 \eta'_1 \eta'_2}$ over all Parshin triples $\eta'_0 \eta'_1 \eta'_2$ on $Y$ lying over $\eta_0 \eta_1 \eta_2$.
This $r_0$ gives a diagram
\begin{equation}
\begin{array}{ccccccccc}
1 & \rightarrow & {\rK}_{2} ( S_{012}[G] ) & \rightarrow &
 {\rm St} ( S_{012}[G]) & \rightarrow & {\rE} ( S_{012} [G]) & \rightarrow & 1 \\
&  & \downarrow r_0 &  & \downarrow r_0 &  & \downarrow r_0 &  &  \\
1 & \rightarrow & {\rK}_{2} ( R_{012} [ G ]   & \rightarrow &
 {\rm St} ( R_{012} [ G ]   & \rightarrow & {\rE} ( R_{012}  [
G ]  ) & \rightarrow & 1%
\end{array}%
\end{equation}%
for every Parshin triple $\eta_0 \eta_1 \eta_2 = 012$ on $\mathbb{P}^1$, using a notation parallel to that in
\S \ref{s:pushdowns}.  We wish to show that we can define the map $\pi_*$ in (\ref{eq:chowpush}) by applying
$r_0$ to every local component of an element of   ${\rm CH}^2_{\Bbb A}(Y[G])$.  To show that this 
is well-defined we have to show that $r_0$ takes the numerators to numerators and denominators
to denominators in the definition of ${\rm CH}^2_{\Bbb A}(Y[G])$ and ${\rm CH}^1_{\mathbb{A}}(\mathbb{P}^1[G])$
in Definition \ref{def:ith}. 

Suppose $x$ lies in the numerator of ${\rm CH}^2_{\Bbb A}(Y[G])$ and that $\eta'$ is a Parshin triple on
$\mathbb{P}^1$ which involves $\eta_0$ and another point $\eta_k$ of $\mathbb{P}^1$. Suppose
that in computing the component $r_0(x)_\eta'$ of $r_0(x)$ we replace the basis for $\O_{Y,\eta'_0}$
as a module for $\O_{\mathbb{P}^1,\eta_0}$ by a basis for $\hat\O_{Y,\eta_k} = \pi_*\O_{Y} \otimes_{\O_{\PP^1}}  \hat\O_{\mathbb{P}^1,\eta_k}$
as a module for $\hat\O_{\mathbb{P}^1,\eta_k}$.  This changes the algebra homomorphism 
$r_0$ to the algebra homomorphism $r_k$ which results from conjugating $r_0$ by   
the  transition matrix  $\lambda_{W',\eta_0 \eta_k}$.  This matrix need not be  elementary, 
but it does have constant coefficients.  Let $\lambda^\dagger_{W',\eta_0 \eta_k}$ be the matrix having
a block in the upper left corner equal to $\lambda_{W',\eta_0 \eta_k}$ and a one-by-one matrix block
in the lower right corner with entry $\mathrm{det}(\lambda_{W',\eta_0 \eta_k})^{-1}$.  Then
$\lambda^\dagger_{W',\eta_0 \eta_k}$ also conjugates $r_0$ to $r_k$.  Since $\lambda^\dagger_{W',\eta_0 \eta_k}$
has determinant $1$ and has coefficients in $R_{0k}$, $\lambda^\dagger_{W',\eta_0 \eta_k}$ is elementary
by   Corollary \ref{cor:Blochcor}.  
Since $ {\rm St}( R_{012}[ G] )$
is a central extension of $ {\rE}( R_{012}[
G] )$ by  $ {\rK}_{2}( R_{012}[ G] )$, the conjugation action of a lift of $\lambda^\dagger_{W',\eta_0 \eta_k}$ to  
$ {\rm St}( R_{012}[ G] )$ does not depend on the choice of this lift. Furthermore,
this conjugation action is trivial on  ${\rK}_{2}( R_{012}[ G])$.  We conclude that
in the above recipe for computing $\pi_*$, we are free to replace $r_0$ by $r_k$ when computing
components at Parshin chains $\eta$ of $\mathbb{P}^1$ which involve $\eta_k$.

The first step in showing that $x \to r_0(x)$ gives a well-defined homomorphism $\pi_*$ on second adelic Chow groups is to show that
if $x \in  {\rK}_{2}^{\prime
}( \mathbb{A}_{Y,012}[ G] ) $, then $r_0(x) \in  {\rK}_{2}^{\prime
}( \mathbb{A}_{\mathbb{P}^1,012}[ G] )$. From Definition \ref{defrestrictedproducts},
we see that this assertion amounts to saying that if $x$ satisfies 
conditions (PK1) and (PK2) then $r_0(x)$ satisfies these conditions when $Y$ is replaced by $\mathbb{P}^1$.
Consider condition (PK1).  For all but finitely many codimension $1$ points $\eta_1$ on $\mathbb{P}^1$,
the component $x_{\eta'_0 \eta'_1\eta'_2}$ of $x$ at each Parshin triple of $Y$ for which $\eta'_1$ lies above $\eta_1$ will satisfy the condition 
in (PK1), namely 
\begin{equation}
\label{eq:PK1x}
x_{\eta'_0\eta'_{1}\eta'_{2}}\in  {\rK}_{2}  ( \hat{
\mathcal{O}}_{Y,\eta'_{1}\eta'_{2}} [ G ]  )^\flat
\end{equation}
for all $\eta'_{2}\in \overline{\eta}'_{1}$.  In determining the component $r_0(x)_{\eta_0 \eta_1 \eta_2}$
we are free to replace $r_0$ by the homomorphism $r_1$ defined above using local bases at the point  $\eta_1$.  Since $r_1$
comes from an algebra homomorphism 
$\what\O_{Y,\eta_1\eta_2}[G] \to \mathrm{M}_d(\what\O_{\mathbb{P}^1,\eta_1\eta_2}[G])$,
we see that (\ref{eq:PK1x}) implies $r_0(x)$ satisfies (PK1) for $\mathbb{P}_1$. The
remaining assertions one must prove in order to show $\pi_*$ is a well defined
homomorphism from ${\rm CH}^2_{\Bbb A}(Y[G])$ to $ {\rm CH}^2_{\Bbb A}(\mathbb{P}^1[G])$
can be proved in a similar way.

Let $\eta = (\eta_0,\eta_1,\eta_2)$ be a (non-degenerate) Parshin triple on $\mathbb{P}^1 $.
Then $L = \hat{\O}_{\mathbb{P}^1, \eta}$ is the total  fractions of the product of discrete valuation rings $R = \hat{\O}_{\mathbb{P}^1,\eta_1\eta_2}$.
We suppose that $\eta_2$ has residue characteristic $p$.  Then $N = L \otimes_{\O_{\mathbb{P}}^1} \pi_* \O_Y$ is a product of the
fields given by the multicompletions of $\O_Y$ at the Parshin triples of $Y$ lying over $\eta$.    We must
show that  
there is a commutative diagram  
\begin{equation}
\label{eq:pushtri2}
 { \xymatrix {
 {\rK}_2(N[G])\ar[r]^(.51){\pi_*}\ar[dr]_(.41){(f \circ \pi)_*}& {\rK}_2(L[G])
\ar[d]^(.48){f_*}\\
&
 {\rK}_1(\mathbb{Q}_p[G]).
}}
\end{equation}
Since $N$ and $L$ are products of fields of characteristic $0$, we can reduce to
the case in which $G$ is the trivial group by Morita equivalence.  There are now
two cases to consider, both of which are dealt with by \cite{KatoRes}.  If
$\eta_1$ is horizontal, then $f_*$ and $(f \circ \pi)_*$ are tame symbols,
and (\ref{eq:pushtri2}) is commutative by  \cite[Lemma 3]{KatoRes}.  If $\eta_1$ is vertical,
then $f_*$ and $(f \circ \pi)_*$ are Kato's residue maps, and the result we need
is shown on page 160 of \cite{KatoRes}, four lines above Proposition 3.  In fact, in this
case, the res map for $N$ is constructed from the res map on $L$ via
the norm map $\pi_*$.
\end{proof}

\begin{proposition} 
\label{prop:state} Assume the notation and hypotheses of Proposition \ref{prop:basec} and Lemma \ref{lem:pushpat} above.   Then we have 
\begin{equation}
\label{eq:preciser} \pi_*(c_2(\E)) + c_2(\V' \oplus \V) =  c_2(\pi_* (\E) \oplus \V).
\end{equation}
\end{proposition}

The proof of this result will be completed in \S \ref{s:cocyclecal}.

\subsubsection{}
We now summarise how
the results proved thus far will reduce the proof of Theorem  \ref{thm:trystate} to the case of $Y = {\Bbb P}^1_\Z$.

Part (i) of Theorem \ref{thm:trystate2} follows from Proposition  \ref{prop:basec} (ii).  Part (ii) of Theorem 
\ref{thm:trystate2} is shown by Lemma \ref{lem:pushpat} and Proposition \ref{prop:state}.  The 
equalities in part (iii) of Theorem \ref{thm:trystate2} follow from Proposition  \ref{prop:basec} (i).
The equalities in part (iv) will be shown in Theorem \ref{proHorr} of the next section using Proposition \ref{prop:basec} (ii)
to show that hypothesis (b) of Theorem \ref{proHorr} can be stably satisfied for $\F = \pi_* \E \oplus \V$
and $\F = \V' \oplus \V$.

\subsection{Proof of Proposition \ref{prop:state}.}
\label{s:cocyclecal}

In this section we will prove Proposition \ref{prop:state} 
via cocycle calculations.  These calculations will require
repeated use of Lemma \ref{cocyclele}, and in this Lemma
one must make choices of various lifts in order for the
identity in the Lemma to apply.  We   start by forming adelic cocycles that give representatives for the classes $c_2(\pi_*\E\oplus \V)$, $\pi_*(c_2(\E))$ and $c_2(\V'\oplus \V)$ in ${\rm CH}^{2}_{\mathbb{A}}(\mathbb{P}^1[G])$.
These will be given by choosing lifts to suitable Steinberg groups of  elementary transition
matrices which are as in Proposition \ref{propae}. These lifts have to satisfy the conditions of this
 Proposition \ref{propae} (2) with respect to some divisor $\Delta$ in $\PP^1$ which contains the fibers over
the primes dividing the order of $G$.  
We will call such lifts {\sl acceptable}. By Theorem \ref{thm25},
we can calculate using any set of acceptable lifts. 

To make the notation more clear we will use $s_{i_1\ldots i_k}(\lambda)$ instead of $\tilde \lambda$
for a lift to the  Steinberg group associated to a Parshin chain $\eta_{i_1}\ldots\, \eta_{i_k}$ 
of an  elementary matrix $\lambda$. (Although the notation might be suggesting this, 
we are not choosing sections of the Steinberg sequence.) 
Recall that we have  elementary transition matrices $\lambda^\sharp_{ij}$ for the $\O_{{\mathbb P}^1}[G]$-bundle
$\pi_*\E\oplus\V$.  By Proposition \ref{propae} we can choose acceptable lifts:
\begin{itemize}

\item[A)] $s_{01}(\lambda_{01}^\sharp)$, $s_{02}(\lambda^\sharp_{02})$ of $\lambda_{01}^\sharp$, $\lambda_{02}^\sharp$.

\end{itemize}
Similarly, for $\V'\oplus\V$ and its transitions we can choose acceptable lifts 
\begin{itemize}

\item[B)] $s_{ij}(\Lambda^\sharp_{ij})$ of $\Lambda_{ij}^\sharp$.

\end{itemize}
We can also choose acceptable lifts 
\begin{itemize}

\item[C)]  $s_{01}(r_0^\sharp(\mu_{01}))$, $s_{12}(r_2^\sharp(\mu_{12}))$, $s_{02}(r_0^\sharp(\mu_{02}))$
of the matrices $r_0^\sharp(\mu_{01})$, $r_2^\sharp(\mu_{12})$, $r_0^\sharp(\mu_{02})$.

\end{itemize} 

(These last three matrices in (C) are integral in the sense of Proposition \ref{propae} (1) with respect to some divisor $\Delta$
 in $\PP^1$ which contains the fibers over the primes dividing the order of $G$.  This follows since
$\mu_{ij}$ are  elementary transition matrices for the $\O_Y[G]$-bundle and so they satisfy the
conclusion of Proposition \ref{propae} (1) over $Y$ for a divisor $D$ on $Y$. Indeed, it is now enough 
to take any  $\Delta$ that contains   the image of $D$ under $\pi: Y\to \PP^1$ together with
the complement of the open of $\PP^1$ where the generic basis of $\pi_*\O_Y$ involved in the choice of $r_0$ 
is actually a basis.)

Starting from these lifts we now also consider lifts of some additional elements as follows:
\begin{itemize}

\item[D)] We lift $\lambda_{12}^\sharp=\Lambda^\sharp_{12}\cdot r_2^\sharp(\mu_{12})$ by setting 
\begin{equation}\label{eq:sharp11p}
s_{12}(\lambda_{12}^\sharp):=s_{12}(\Lambda^\sharp_{12})\cdot s_{12}(r_2^\sharp(\mu_{12})).
\end{equation}

\item[E)] We lift $r_0^\sharp(\mu_{12})=\Lambda_{02}^\sharp\cdot r_2^\sharp(\mu_{12})\cdot (\Lambda_{02}^\sharp)^{-1}$ by setting
\begin{equation}\label{eq:sharp12}
s_{012}(r_0^\sharp(\mu_{12})):=s_{02}(\Lambda_{02}^\sharp)\cdot s_{12}(r_2^\sharp(\mu_{12}))\cdot s_{02}(\Lambda_{02}^\sharp)^{-1}.
\end{equation}

\item[F)] We lift $\Lambda_{01}^\sharp\cdot \lambda_{12}^\sharp=r_0^\sharp(\mu_{12})\cdot \Lambda_{02}^\sharp=\Lambda_{02}^\sharp\cdot r_2^\sharp(\mu_{12})$ (see   (\ref{eq:fourteen})) by setting
 \begin{equation}\label{eq:s012} 
s_{012}(\Lambda_{01}^\sharp\cdot \lambda_{12}^\sharp)=s_{012}(r_0^\sharp(\mu_{12})\cdot \Lambda_{02}^\sharp)=s_{012}(\Lambda_{02}^\sharp\cdot r_2^\sharp(\mu_{12})):=s_{02}(\Lambda_{02}^\sharp)\cdot s_{12}(r_2^\sharp(\mu_{12}).
\end{equation}

\end{itemize}

Using these lifts, we can now calculate our various cocycles.
We denote by $z(\lambda^\sharp)$  the adelic element 
with components
 \begin{equation}
z(\lambda^\sharp)_{0,1,2}=z( \lambda^\sharp_{01},\lambda^\sharp_{12}) :=s_{02}(\lambda_{02}^\sharp)\cdot s_{12}(\lambda_{12}^\sharp)^{-1}\cdot s_{01}(\lambda_{01}^\sharp)^{-1}.
\end{equation}
Similarly consider 
$z(\Lambda^\sharp)$ to be the adelic element 
with components
 \begin{equation}
z(\Lambda^\sharp)_{0,1,2}=z( \Lambda^\sharp_{01},\Lambda^\sharp_{12}) :=s_{02}(\Lambda_{02}^\sharp)\cdot s_{12}(\Lambda_{12}^\sharp)^{-1}\cdot s_{01}(\Lambda_{01}^\sharp)^{-1}.
\end{equation}
Also consider the adelic element $z(r^\sharp_0(\mu) )$ with components
 \begin{equation}
z(r_0^\sharp(\mu) )_{0,1,2}= z(r_0^\sharp(\mu_{01}), r_0^\sharp(\mu_{12})):=s_{02}(r_0^\sharp(\mu_{02}))\cdot s_{012}(r_0^\sharp(\mu_{12}))^{-1}\cdot 
s_{01}(r_0^\sharp(\mu_{01}))^{-1}.
\end{equation}
The class of $z(\lambda^\sharp)$  in ${\rm CH}^{2}_{\mathbb{A}}(\mathbb{P}^1[G])$
corresponds to $c_{2}( \pi _{\ast }\mathcal{E\oplus V})$. By the construction in 
Lemma \ref{lem:pushpat}, the class of $z( {r}^\sharp_{0}( \mu )) $ corresponds to $\pi _{\ast
}c_{2}( \mathcal{E})$.  Finally, the class of 
$z({\Lambda}^\sharp )$ corresponds to 
$c_2(\mathcal{V}' \oplus \mathcal{V})$. Thus Proposition \ref{prop:state}
will follow if we can show that  
\begin{equation}
\label{eq:thepoint}
z( {\lambda }^\sharp )  \cdot z(  {r}^\sharp_{0}( \mu  ) )^{-1} \cdot  z({\Lambda}^\sharp )^{-1} \in 
\prod_{   0\leq i<j\leq 2} {\rK}'_{2}  ( \mathbb{A}%
_{\mathbb{P}^1,ij} [ G ]  )^\flat.
\end{equation}

\begin{lemma} We have an equality 
\begin{eqnarray}
\label{eq:luminy1}
z( \lambda^\sharp_{01},\lambda^\sharp_{12}) z(
r^\sharp_{0}( \mu _{01}) ,\Lambda^\sharp_{01})
&=&z( r^\sharp_{0}( \mu _{01}) \Lambda^\sharp_{01},\lambda^\sharp_{12}) z( r^\sharp_{0}( \mu
_{01}) ,\Lambda^\sharp_{01})  \notag \\
&=&z( r^\sharp_{0}( \mu _{01}) ,\Lambda^\sharp_{01} \lambda^\sharp_{12}) z( \Lambda^\sharp_{01},%
\lambda^\sharp_{12})
\end{eqnarray}%
where%
\begin{eqnarray*}
z( \lambda^\sharp_{01},\lambda^\sharp_{12})
&=&z( r^\sharp_{0}( \mu _{01}) \Lambda^\sharp
_{01},\lambda^\sharp_{12}) =s_{02}( \lambda^\sharp%
_{02}) s_{12}( \lambda^\sharp_{12}) ^{-1}s_{01}(
\lambda^\sharp_{01}) ^{-1} \\
z( r^\sharp_{0}( \mu _{01}) ,\Lambda^\sharp_{01}) &=&s_{01}( \lambda^\sharp_{01}) s_{01}(
\Lambda^\sharp_{01})^{-1}s_{01}( r^\sharp_{0}(
\mu _{01}) ) ^{-1} \\
z( r^\sharp_{0}( \mu _{01}) ,\Lambda^\sharp_{01}%
\lambda^\sharp_{12}) &=&s_{02}( \lambda^\sharp
_{02}) s_{012}( \Lambda^\sharp_{01} \lambda^\sharp
_{12}) ^{-1}s_{01}( r^\sharp_{0}( \mu _{01})
) ^{-1} \\
z( \Lambda^\sharp_{01},\lambda^\sharp_{12})
&=&s_{012}( \Lambda^\sharp_{01} \lambda^\sharp_{12})
s_{12}( \lambda^\sharp_{12}) ^{-1}s_{01}( 
\Lambda^\sharp_{01}) ^{-1}
\end{eqnarray*}%
with all lifts as defined in (A)-(F) above.  \end{lemma}

\begin{proof} We just need to show the second equality; this is an application of Lemma \ref{cocyclele}
for $c=r_0^\sharp(\mu_{01})$, $d=\Lambda_{01}^\sharp$, $b=\lambda_{12}^\sharp$, $cd=\lambda_{01}^\sharp$,
$db=\Lambda_{01}^\sharp\lambda_{12}^\sharp$, $cdb=\lambda_{02}^\sharp$ with their lifts chosen as in (A)-(F). 
\end{proof}
\smallskip

We note that $z( r^\sharp_{0}( \mu _{01}) , 
\Lambda^\sharp_{01}) \in \rK_{2}( R_{01}[G])$; in
fact, for almost all $\eta _{1}$, we have $r^\sharp_{0}( \mu
_{01}) \in {\rm GL}(  R_1[G])$, $\Lambda^\sharp_{01}\in {\rm GL}( R_1[ G ] )$, and so
\begin{equation}\label{eq:zip}
\prod\nolimits_{\eta _{1}}z( r^\sharp_{0}( \mu _{0\eta
_{1}}) ,\Lambda^\sharp_{0\eta _{1}}) \in  {\rK}_{2}^{\prime
}( \mathbb{A}_{\PP^1,01} [ G] ) 
\end{equation}%
which lies in the denominator of ${\rm CH}^{2}_{{\mathbb A}}(\PP^1[G])$.

\begin{lemma} We have an identity 
\begin{align}\label{old8.43}
z( r^\sharp_{0}( \mu _{01}) , r^\sharp_{0}(
\mu _{12}) \Lambda^\sharp_{02}) z( r^\sharp_{0}( \mu _{12}) ,\Lambda^\sharp_{02}) &=z(
r^\sharp_{0}( \mu _{01}) r^\sharp_{0}( \mu_{12}), \Lambda^\sharp_{02}) z( r^\sharp_{0}( \mu _{01}) ,r^\sharp_{0}( \mu _{12}) ) \\
&=z( r^\sharp_{0}( \mu _{02}) ,\Lambda^\sharp_{02}) z( r^\sharp_{0}( \mu _{01}) ,r^\sharp_{0}( \mu _{12}) )\notag
\end{align}%
where%
\begin{eqnarray*}
z( r^\sharp_{0}( \mu _{01}) ,r^\sharp_{0}(
\mu _{12}) \Lambda^\sharp_{02}) &=&s_{02}( 
r^\sharp_{0}( \mu _{02}) \Lambda^\sharp_{02}) s_{012}(
r^\sharp_{0}( \mu _{12}) \Lambda^\sharp_{02})
^{-1}s_{01}( r^\sharp_{0}( \mu _{01}) ) ^{-1} \\
z( r^\sharp_{0}( \mu _{12}) ,\Lambda^\sharp_{02}) &=&s_{012}( r^\sharp_{0}( \mu _{12})
\Lambda^\sharp_{02}) s_{02}( \Lambda^\sharp_{02})
^{-1} s_{012}( r^\sharp_{0}( \mu _{12}) ) ^{-1} \\
z( r^\sharp_{0}( \mu _{01}) r^\sharp_{0}( \mu
_{12}) ,\Lambda^\sharp_{02}) &=&z( r^\sharp_{0}( \mu _{02}) ,\Lambda^\sharp_{02}) =s_{02}(
r^\sharp_{0}( \mu _{02}) \Lambda^\sharp_{02})
s_{02}( \Lambda^\sharp_{02}) ^{-1}s_{02}( r^\sharp_{0}( \mu _{02}) ) ^{-1} \\
z( r^\sharp_{0}( \mu _{01}) ,r^\sharp_{0}(
\mu _{12}) ) &=&s_{02}( r^\sharp_{0}( \mu
_{02}) ) s_{012}( r^\sharp_{0}( \mu _{12})
) ^{-1}s_{01}( r^\sharp_{0}( \mu _{01}) )
^{-1}
\end{eqnarray*}%
 with the lifts  defined as in (A)-(F). Note here that 
 $r_0^\sharp(\mu _{02}) \Lambda^\sharp_{02}=\lambda^\sharp_{02}$
 and therefore we take $s_{02}(r_0^\sharp(\mu _{02}) \Lambda^\sharp_{02})=s_{02}(\lambda^\sharp_{02})$.
 Recall that we also have $r^\sharp_{0}( \mu _{01}) r^\sharp_{0}( \mu
_{12})=r^\sharp_{0}(\mu_{02})$.
 \end{lemma}

\begin{proof} We just need to show the first identity. This follows
from Lemma \ref{cocyclele} applied to $c=r_0^\sharp(\mu_{01})$, $d=r_0^\sharp(\mu_{12})$,
$b=\Lambda^\sharp_{02}$, $cd=r_0^\sharp(\mu_{01})r_0^\sharp(\mu_{12})=r_0^\sharp(\mu_{02})$, $db=r_0^\sharp(\mu_{12})\Lambda^\sharp_{02}$,
$cdb=r_0^\sharp(\mu_{01})r_0^\sharp(\mu_{12})\Lambda_{02}^\sharp=r_0^\sharp(\mu_{02})\Lambda_{02}^\sharp=\lambda_{02}^\sharp$
and their lifts as specified in (A)-(F).  
 \end{proof}
 \smallskip

Note  that  $z( r^\sharp_{0}( \mu _{02}) , 
\Lambda^\sharp_{02}) \in \rK_{2}(R_{02}[G])$. 
In fact, the following stronger statement is true.  We chose lifts which are acceptable relative to some effective divisor $\Delta$ on $\mathbb{P}^1$ which contains
all the vertical fibers over primes which divide the order of $G$.  Therefore the lifts $s_{02}(r^\sharp_{0}( \mu _{02})$,
$s_{02}(\Lambda^\sharp_{02})$ and 
$s_{02}(r^\sharp_{0}( \mu _{02}) \cdot 
\Lambda^\sharp_{02})$ lie in ${\rm St}(R_2[\Delta^{-1}][G])$.
This implies that 
$z( r^\sharp_{0}( \mu _{02}) , 
\Lambda^\sharp_{02}) \in \rK_{2}(R_2[\Delta^{-1}][G])$
for all choices of $\eta_2$.
We conclude from this and Definition \ref{defrestrictedproducts}(b3) that 
\begin{equation}\label{eq:zip2}
\prod\nolimits_{\eta _{2}}z( r^\sharp_{0}( \mu _{0\eta
_{2}}) ,\Lambda^\sharp_{0\eta _{2}}) \in {\rK}_{2}^{\prime
}( \mathbb{A}_{\PP^1,02} [ G] ) 
\end{equation}%
which lies in the denominator of ${\rm CH}^{2}_{{\mathbb A}}( \PP^1[G])$.  

\begin{corollary}
\label{cor:hopeit}
There is a  congruence
\begin{equation}
\label{eq:rhsgood}
z(\lambda^\sharp) \cdot z(r^\sharp_0(\mu))^{-1} 
\equiv \prod_{(\eta_1,\eta_2)} z(r^\sharp_0(\mu_{12}),\Lambda^\sharp_{02})^{-1} \cdot z(\Lambda^\sharp_{01},\lambda^\sharp_{12})
\end{equation}
in $\mathrm{CH}_{\mathbb{A}}^2(\mathbb{P}^1[G])$.
\end{corollary}

\begin{proof} Because of (\ref{eq:zip}), (\ref{eq:luminy1}) gives the congruence
\begin{equation}
\label{eq:step1}
z(\lambda^\sharp)
\equiv \prod_{(\eta_1,\eta_2)} z( r^\sharp_{0}( \mu _{01}) \Lambda^\sharp_{01},\lambda^\sharp_{12})\equiv 
\prod_{(\eta_1,\eta_2)}z( r^\sharp_{0}( \mu _{01}) ,\Lambda^\sharp_{01} \lambda^\sharp_{12}) \cdot z( \Lambda^\sharp_{01},%
\lambda^\sharp_{12})  
\end{equation}
in $\mathrm{CH}_{\mathbb{A}}^2(\mathbb{P}^1[G])$. Because of (\ref{eq:zip2}), (\ref{old8.43}) gives the congruence
\begin{eqnarray}
\label{eq:step2}
 \prod_{(\eta_1,\eta_2)} z( r^\sharp_{0}( \mu _{01}) , r^\sharp_{0}(
\mu _{12}) \Lambda^\sharp_{02}) \cdot z( r^\sharp_{0}( \mu _{12}) ,\Lambda^\sharp_{02})  
 &\equiv& \prod_{(\eta_1,\eta_2)}  z(
r^\sharp_{0}( \mu _{01}) r^\sharp_{0}( \mu_{12}), \Lambda^\sharp_{02}) \cdot z( r^\sharp_{0}( \mu _{01}) ,r^\sharp_{0}( \mu _{12}) ) \nonumber \\
&\equiv& z( r^\sharp_{0}( \mu) ) 
\quad \mathrm{in}  \quad \mathrm{CH}_{\mathbb{A}}^2(\mathbb{P}^1[G]).
\end{eqnarray}
Multiply (\ref{eq:step1}) by the inverse of  (\ref{eq:step2}).  We conclude that 
$ z( \lambda^\sharp) \cdot z( r^\sharp_{0}( \mu))^{-1}  $ is equal to 
\begin{equation}
\label{eq:step3}
\prod_{(\eta_1,\eta_2)}z( r^\sharp_{0}( \mu _{01}) ,\Lambda^\sharp_{01} \lambda^\sharp_{12}) \cdot z( \Lambda^\sharp_{01},%
\lambda^\sharp_{12}) \cdot 
z( r^\sharp_{0}( \mu _{01}) , r^\sharp_{0}(
\mu _{12}) \Lambda^\sharp_{02})^{-1} \cdot z( r^\sharp_{0}( \mu _{12}) ,\Lambda^\sharp_{02})^{-1}.
\end{equation}
in $ \mathrm{CH}_{\mathbb{A}}^2(\mathbb{P}^1[G])$.
We have $\Lambda^\sharp_{01} \lambda^\sharp_{12} = r_0^\sharp(\mu_{12})\Lambda^\sharp_{02}$ by (\ref{eq:fourteen})
and by our choices two terms on the right side
of  (\ref{eq:step3}) 
cancel to give (\ref{eq:rhsgood}). 
\end{proof}
\smallskip

We now expand the right hand side of (\ref{eq:rhsgood}).  By definition we have
\begin{eqnarray}
\label{eq:uglier}
z(r^\sharp_0(\mu_{12}),\Lambda^\sharp_{02})^{-1} \cdot z(\Lambda^\sharp_{01},\lambda^\sharp_{12}) = \\
\left( s_{012}(r_0^\sharp(\mu_{12})) \cdot s_{02}(\Lambda^\sharp_{02}) \cdot s_{012}(r^\sharp_0(\mu_{12}) \cdot \Lambda^\sharp_{02})^{-1} \right ) \cdot 
(s_{012}(\Lambda^\sharp_{01} \lambda^\sharp_{12}) \cdot s_{12}(\lambda^\sharp_{12})^{-1} 
\cdot s_{01}(\Lambda^\sharp_{01})^{-1}).\nonumber
\end{eqnarray}
Notice that $\Lambda^\sharp_{01} \lambda^\sharp_{12} = r^\sharp_0(\mu_{12}) \Lambda^\sharp_{02}$ and $s_{012}(r^\sharp_0(\mu_{12}) \cdot \Lambda^\sharp_{02}) = s_{012}(\Lambda^\sharp_{01} \lambda^\sharp_{12})$
and so the middle terms on the right in (\ref{eq:uglier}) cancel.  This and
the expression for $s_{012}(r_0^\sharp(\mu_{12}))$ in (\ref{eq:sharp12}) show 
\begin{eqnarray}
\label{eq:uglier2}
z(r^\sharp_0(\mu_{12}),\Lambda^\sharp_{02})^{-1} \cdot z(\Lambda^\sharp_{01},\lambda^\sharp_{12}) = \\
 s_{02}(\Lambda^\sharp_{02}) \cdot s_{12}(r_2^\sharp(\mu_{12}))\cdot 
s_{02}(\Lambda^\sharp_{02})^{-1} \cdot s_{02}(\Lambda^\sharp_{02})
\cdot s_{12}(\lambda^\sharp_{12})^{-1} 
\cdot s_{01}(\Lambda^\sharp_{01})^{-1} = \nonumber\\
 s_{02}(\Lambda^\sharp_{02}) \cdot s_{12}(r_2^\sharp(\mu_{12}))
\cdot s_{12}(\lambda^\sharp_{12})^{-1} 
\cdot s_{01}(\Lambda^\sharp_{01})^{-1}
.\nonumber
\end{eqnarray}
We now use the expression for 
$s_{12}(\lambda^\sharp_{12})$ in (\ref{eq:sharp11p}) to have  
\begin{eqnarray}
\label{eq:uglier3}
z(r^\sharp_0(\mu_{12}),\Lambda^\sharp_{02})^{-1} \cdot z(\Lambda^\sharp_{01},\lambda^\sharp_{12}) = \\
 s_{02}(\Lambda^\sharp_{02}) \cdot s_{12}(r_2^\sharp(\mu_{12}))
\cdot\left ( s_{12}(\Lambda^\sharp_{12}) \cdot s_{12}(r_2^\sharp(\mu_{12})) \right )^{-1} 
\cdot s_{01}(\Lambda^\sharp_{01})^{-1}
 = \nonumber\\
  s_{02}(\Lambda^\sharp_{02}) 
\cdot  s_{12}(\Lambda^\sharp_{12})^{-1} \cdot s_{01}(\Lambda^\sharp_{01})^{-1} = \nonumber\\
z(\Lambda^\sharp_{02},\Lambda^\sharp_{01}).\nonumber
\end{eqnarray}
Plugging this into the right hand side of (\ref{eq:rhsgood}) shows 
(\ref{eq:thepoint}), and this completes the proof of Proposition
 \ref{prop:state}.  \endproof

\section{The proof of the theorem; bundles over ${\Bbb P}^1_\Z$}
\label{s:overp1}

\setcounter{equation}{0}

\subsection{Bundles over the projective line $\PP^1$}

Let $R$ be a commutative ring and let $G$ be a finite group.  Let $\PP^1 = \PP^1_R$  be the projective line over $\Spec(R)$.
Thus $\PP^1$ is covered by two affine patches  ${\Bbb A}^1_0 = \mathrm{Spec}(R[t])$ and ${\Bbb A}^1_\infty = \mathrm{Spec}(R[t^{-1}])$ glued along $\mathrm{Spec}(R[t,t^{-1}])$. If $L$ is a module for $ R[G]\otimes_R R[t]  = R[G][t]$ (resp. $R[G] \otimes_R R[t^{-1}]$), let
$\tilde L $   be the corresponding sheaf of $R[G]$-modules on ${\Bbb A}^1_0$ (resp. ${\Bbb A}^1_\infty$).
The following is a variation of a result of Horrocks.

\begin{theorem}
\label{thm:Horrocks}
Let $R$ be a finite field or a Dedekind ring with finite residue fields.   Let $\E$ be an $\O_{\PP^1_R}[G]$-bundle, 
i.e. a finitely generated locally free $\O_{\PP^1_R}[G]$-module.
Then there is a finitely generated locally free $R[G]$-module $\M$ such that
$$
 \E_{|{\Bbb A}^1_0}\simeq \widetilde{(\M\otimes_R R[t]) } \quad and \quad  \E_{|{\Bbb A}^1_\infty}\simeq \widetilde{ (\M\otimes_R R[t^{-1}])}  
.$$
In particular, if $R$ is a field or a local Dedekind ring then $\M$ is a free $R[G]$-module.
\end{theorem}

\begin{proof}  We first show that it will suffice to prove there is an isomorphism of $R[G][t]$-modules
\begin{equation}
\label{eq:enough}
\Gamma({\Bbb A}^1_0,\E) \simeq \M\otimes_R R[t]
\end{equation}
for some $\M$ as in the Theorem.  For then,  on replacing ${\Bbb A}^1_0$ by ${\Bbb A}^1_\infty$,
we will have shown there is a finitely generated locally free
free $R[G]$-module $\M'$ such that $\Gamma({\Bbb A}^1_\infty,\E)\simeq \M' \otimes_R R[t]$.  This will imply
there are $R[G][t,t^{-1}]$-module isomorphisms
$$
\M\otimes_R R[t,t^{-1}] \simeq \Gamma({\Bbb A}^1_0\cap {\Bbb A}^1_1, \E) \simeq \M' \otimes_R R[t,t^{-1}].
$$
By tensoring these isomorphisms with the $R$-algebra surjection $R[t,t^{-1}] \to R$ which sends $t$ to $1$  we find that $\M \simeq \M'$ as $R[G]$-modules,
so Theorem \ref{thm:Horrocks} will follow.

Suppose now that  $R$ is a finite field.  To prove (\ref{eq:enough}) it will suffice to show
that $\Gamma({\Bbb A}^1_0,\E)$ is a free $R[G][t]$-module.
Let $r(R[G])$ be the quotient of $R[G]$ by
its maximal two-sided nilpotent ideal $n(R[G])$.
Because
$\E$ is a locally free $\O_{\PP^1}[G]$-module, we have an exact
sequence of sheaves
\begin{equation}
\label{eq:exactnil}
0 \to n(R[G]) \otimes_{R[G]} \E \to \E \to r(R[G]) \otimes_{R[G]} \E \to 0.
\end{equation}
Since ${\Bbb A}^1_0$ is affine, this gives a surjection
\begin{equation}
\label{eq:surjit}
\Gamma({\Bbb A}^1_0,\E) \to \Gamma({\Bbb A}^1_0, r(R[G]) \otimes_{R[G]} \E).
\end{equation}
The stalk $\E_P$ of $\E$ at each $P \in {\Bbb A}^1_0$ is the localization $\Gamma({\Bbb A}^1_0,\E)_{P}$
of $\Gamma({\Bbb A}^1_0,\E)$ at $P$ since ${\Bbb A}^1_0$ is affine.

Let $m$ be the  rank of the locally free $\E$.  Suppose we prove there
is an isomorphism
\begin{equation}
\label{eq:radm}
(r(R[G])[t])^m \simeq \Gamma({\Bbb A}^1_0, r(R[G]) \otimes_{R[G]} \E)
\end{equation}
of modules for  $r(R[G]) \otimes_R R[t] = r(R[G])[t]$.
Lift a set of  $m$ generators for the  $r(R[G])[t]$-module
$ \Gamma({\Bbb A}^1_0, r(R[G]) \otimes_{R[G]} \E)$ 
via the surjection (\ref{eq:surjit}).  Because
$n(R[G])$ is nilpotent in (\ref{eq:exactnil}), this
produces $m$ elements of $\Gamma({\Bbb A}^1_0,\E)$
which generate the stalk $\E_P = \Gamma({\Bbb A}^1_0,\E)_{P}$ at each point $P \in {\Bbb A}^1_0$.
This gives a homomorphism $\psi: (R[G][t])^m \to \E$ which localizes at each $P$
to an isomorphism of locally free  $\O_{\PP^1, P}[G]$-modules.  Thus $\psi$
is an isomorphism, so  when $R$ is finite we are reduced to showing (\ref{eq:radm}).

The ring $r(R[G])$ is semi-simple and is thus isomorphic
to a finite direct sum $\oplus_i R_i$ of simple $R$-algebras $R_i$. Since $R$ is finite, $R_i$ is isomorphic
to a matrix algebra $\mathrm{Mat}_{n_i}(k_i)$ for some finite extension $k_i$ of $R$ and some integer $n_i \ge 1$.  Thus
$r(R[G]) \otimes_{R[G]} \E$ is isomorphic to $\oplus_i E_i$ where $E_i$ is a rank $m$ locally
free $R_i$-module on $\PP^1 = \PP^1_R$.  Therefore to show (\ref{eq:radm}), it will suffice to show that $\Gamma({\Bbb A}^1_0,E_i)$ is a free
rank $m$ module for $R_i \otimes_R R[t] = R_i[t]$.  There is a Morita equivalence between the
category of modules for $R_i = \mathrm{Mat}_{n_i}(k_i)$ and the category of vector spaces over $k_i$.
This implies that it will suffice to show that a locally free rank $m$
sheaf $T_i$ of $k_i$-modules on $\PP^1_R$ has the property that
\begin{equation}
\label{eq:Teq}
\Gamma({\Bbb A}^1_0,T_i) \simeq (k_i \otimes_R R[t])^m = (k_i[t])^m
\end{equation}
as  $k_i[t]$-modules.  Here $T_i$ corresponds to a rank $m$ vector bundle on
$k_i \otimes_R \PP^1_R = \PP^1_{k_i}$, so the isomorphism (\ref{eq:Teq}) follows
from   \cite[Theorem 1]{Horrocks}. This completes the proof when $R$ is a finite field.

Suppose now that $R$ is a discrete valuation ring with finite residue field $k$
and uniformizer $\pi$.  Since ${\Bbb A}^1_0$ is affine and $\E$ is a locally free $\O_{\PP^1_R}[G]$-module,
we have an exact sequence
$$
0 \to \pi \cdot \Gamma({\Bbb A}^1_0,\E) \to \Gamma({\Bbb A}^1_0,\E) \to \Gamma({\Bbb A}^1_0,k \otimes_R \E) \to 0
$$
where $\Gamma({\Bbb A}^1_0,k \otimes_R \E) \simeq \Gamma(k \otimes {\Bbb A}^1_0, k \otimes_R \E)$
is a free $k[G][t]$-module by what has already been shown for finite fields.  On lifting generators
and using Nakayama's Lemma we see that $\Gamma({\Bbb A}^1_0,\E)$ is a free $R[G][t]$-module.

Finally, suppose $R$ is a Dedekind ring with finite residue fields.  Following Quillen \cite{QuillenSerre} we will
call an $R[G][t]$-module $M$ {\sl extended} if it is isomorphic to $N \otimes_R R[t]$ for some
locally free $R[G]$-module $N$.  This implies $N$ is isomorphic to $M/tM$.  By what has already
been shown for discrete valuation rings, for each maximal ideal $\mathfrak{m}$ of $R$, the localization
$$
\Gamma({\Bbb A}^1_0,\E)_{\mathfrak{m}} = \Gamma(R_{\mathfrak{m}} \otimes_R {\Bbb A}^1_0, R_{\mathfrak{m}} \otimes_R \E)
$$
is an extended $R_{\mathfrak{m}}[G][t]$-module.  To complete the proof it will suffice to show
that $M = \Gamma({\Bbb A}^1_0,\E)$ is an extended $R[G][t]$-module. We briefly sketch how this
follows from Quillen's patching Lemma (\cite[Theorem 1]{QuillenSerre}).

First observe that since we do know that $M_{\mathfrak{m}}$ is extended, when we let
$N = M/tM$, the localization $N_{\mathfrak{m}}$ is a locally free $R_{\mathfrak{m}}[G]$-module
of rank equal to the locally free rank $m$ of $\E$.  Since $R$ is a Dedekind ring and $\mathfrak{m}$
ranges over all maximal ideals of $R$, this
implies $N$ is a locally free $R[G]$-module of rank $m$.

As in \cite{QuillenSerre}, let $S$ be the set of $f \in R$ such that $M_f$ is  extended as a module for $R_f[G][t]$.
It will suffice to show that $1 \in S$.  The argument in the first part of the proof of Theorem 1
of \cite{QuillenSerre} shows that it will suffice to show that if $f_0, f_1 \in S$ and $Rf_0 + R f_1 = R$
then $1 \in S$.  Suppose $f \in S$.  We have
\begin{eqnarray}
\label{eq:quileq}
\mathrm{Hom}_{R_f[G][t]}(N \otimes_R R_f[t], N \otimes_R R_f[t]) &=& \mathrm{Hom}_{R[G]}(N, N\otimes_R R_f[t])\nonumber\\
&=& R_f[t] \otimes_R {\mathcal A} = {\mathcal A}_f[t]
\end{eqnarray}
when ${\mathcal A} = \mathrm{End}_{R[G]}(N)$. The remainder of the proof of Theorem 1 in \cite{QuillenSerre}
now applies because ${\mathcal A}$ is allowed to be non-commutative in \cite[Lemma 1]{QuillenSerre}.
\end{proof}

\subsubsection{} Suppose that $M_0$ is a finitely generated locally free 
$R[G]$-module and $\gamma$ an element of the group ${\rm Aut}(M_0\otimes_RR[t, t^{-1}])$
of the $R[G][t,t^{-1}]$-linear automorphisms of $M_0\otimes_R R[t, t^{-1}]$. Then, glueing
$M_0\otimes_R R[t]$ and $M_0\otimes_R R[t^{-1}]$ by using $\gamma$ provides a
 a finitely generated locally free $\O_{\PP^1_R}[G]$-module $\E=\E(M_0, \gamma)$.
By Theorem \ref{thm:Horrocks}, if $R$ is a finite field or a Dedekind ring,
 every finitely generated locally free $\O_{\PP^1_R}[G]$-module $\E$ is of this form.
 We then call $(M_0, \gamma)$ ``Horrocks data" associated to $\E$.
 When $M_0\simeq R[G]^n$ is free, we can identify  $\GL_n(R[G][t,t^{-1}])$ 
 with the group ${\rm Aut}(M_0\otimes_RR[t, t^{-1}])$
by sending $g$ to $\gamma_g$ given by $m\mapsto m\cdot g^{-1}$.
Then we write the Horrocks data $(M_0, \gamma_g)$ simply as $(M_0, g)$.

\subsection{The adelic Riemann-Roch theorem over $\PP^1_\Z$}

\subsubsection{} We first show a special case of our main result over $\PP^1=\PP^1_\Z$.
Write $f: \PP^1_\Z\to S=\Spec(\Z)$
for the structure morphism. We continue to assume that the group algebra $\Q[G]$ splits
in the sense of Definition \ref{splitgroupalgebra}.

 \begin{theorem}\label{proHorr}
 Suppose $\F$ is an  $\O_{\PP^1}[G]$-bundle of rank $m$
  on $\PP^1=\PP^1_\Z$  which satisfies:

(a) The reduced Euler characteristics $\bar\chi(\PP^1, \F)_\Q\in \Kr_0(\Q[G])$, and $\bar\chi(\PP^1, \F)_{\Z_p}\in \Kr_0(\Z_p[G])$,
for all primes $p$, are trivial;

(b) The $\Z[G]$-module obtained by pulling back $\F$ along $\Spec(\Z)\to \PP^1$ given by $t=1$  is stably free.

Then

i) The sheaf $\F$ has an (adelic)  elementary structure $\epsilon$. Therefore, the first Chern class $c_1(\F)$
  is trivial in  ${\rm CH}^1_{\Bbb A}(\PP^1[G])$ and the second Chern class
  $c_2(\F, \epsilon)$ is defined in ${\rm CH}^2_{\Bbb A}(\PP^1[G])$.

ii) We have the Riemann-Roch identity
  \begin{equation}\label{RRhorror}
\bar\chi^P(\PP^1, \F)=-f_*(c_2(\F, \epsilon) )
\end{equation}
in ${\rm Cl}(\Z[G])=\Kr_0^{\rm red}(\Z[G])={\rm CH}^1_{\Bbb A}(\Spec(\Z)[G])$.
  \end{theorem}

Note that the left hand side of (\ref{RRhorror}) is independent of the choice of $\epsilon$,
and hence, a posteriori, so is the right hand side.

\begin{proof} Recall that, by definition $\bar\chi(\PP^1,\F)_\Q=\chi(\PP^1, \F)_\Q-\chi(\PP^1, \O_{\PP^1}[G]^m)_\Q$
and similarly for $\bar\chi(\PP^1, \F)_{\Z_p}$. Notice that by our constructions, the statements (i) and (ii) are true for $\F$ if and only if they are true
for the bundle   $\F\oplus \O_{\PP^1}[G]^n$ for some $n\geq 0$. Also $\bar\chi(\PP^1, \F)_\Q= \bar\chi(\PP^1, \F\oplus \O_{\PP^1}[G]^n)_\Q$,
$\bar\chi(\PP^1, \F)_{\Z_p}= \bar\chi(\PP^1, \F\oplus \O_{\PP^1}[G]^n)_{\Z_p}$.  By Theorem \ref{thm:Horrocks}, assumption (b) and these observations, we may assume
 that the sheaf $\F$ is given by Horrocks data $(\Z[G]^m, \gamma_g)$ where $g\in \GL_m(\Z[G][t, t^{-1}])$.
Hence, we
 can write $\F=\E(L_0\cdot g^{-1})$  where $L_0=\Z[G][t]^m$ and
 in our notation (see \S \ref{3c2}, \S \ref{3d}),
 $$
 \V_{g}=\delta(L_0\cdot g^{-1})-\delta(L_0)=\det(R\Gamma(\PP^1, \F))-\det(R\Gamma(\PP^1, \O_{\PP^1}[G]^m)).
 $$
This shows that   assumption
(a)   implies that the matrix $g\in \GL_m(\Z[G][t, t^{-1}])$
actually belongs to $\GL'_m(\Q[G][t, t^{-1}])$ and to $\GL'_m(\Z_p[G][t, t^{-1}])$
for all primes $p$. Denote by $[g]$ the class of $g$ in $\Kr_1(\Z[G][t,t^{-1}])$.
In the next paragraph, we will denote $g$ by $g_\Q$ when we consider it as an element of
$\GL'_m(\Q[G][t, t^{-1}])$ and by $g_p$ when we consider it as an element of
$\GL'_m(\Z_p[G][t, t^{-1}])$.
Recall now that
$$
{\rm Cl}(\Z[G])=\Kr_0^{\rm red}(\Z[G])={\rm CH}^1_{\Bbb A}(\Spec(\Z)[G])=\frac{{\prod'_p} \Kr_1(\Q_p[G])}{\Kr_1(\Q[G])\cdot\prod_p  {\rm K}_1(\Z_p[G])^\flat}.
$$
To calculate a $\Kr_1$-idele   in ${\prod'_p} \Kr_1(\Q_p[G])$ which maps
to $\bar\chi^P(\PP^1,\F)$ under this map we argue as follows.
Choose trivializations
$$
\alpha_\Q: [0]\xrightarrow{\sim} (\V_{g})_\Q=(\det(R\Gamma(\PP^1, \F))-\det(R\Gamma(\PP^1, \O_{\PP^1}[G]^m)))_\Q
$$
$$
\alpha_{\Z_p}: [0]\xrightarrow{\sim} (\V_g)_{\Z_p}=(\det(R\Gamma(\PP^1, \F))-\det(R\Gamma(\PP^1, \O_{\PP^1}[G]^m)))_{\Z_p}
$$
and consider, for each prime $p$, the element $ \alpha_p^{-1}\cdot \alpha_\Q$
 in the group ${\rm Aut}_{V(\Q_p[G])}([0])=\Kr_1(\Q_p[G])$. The $\Kr_1$-idele
 $( \alpha_p^{-1}\cdot \alpha_\Q)_p$ represents $\bar\chi^P(\PP^1,\F)$.
 By the definition of the central extensions, the elements
 $\alpha_\Q$,  $\alpha_p$  correspond to lifts $\ti g_\Q=w_{\Q} ( g_{\Q
} )  $, $\ti g_p=w_{p} ( g_{p} )  $
of $g_\Q$, $g_p$ in
$$
1\to \Kr_1(\Q[G])\to \Hh(\Q[G][t, t^{-1}]^m)\to \GL'_m(\Q[G][t, t^{-1}])\to 1,
$$
$$
1\to \Kr_1(\Z_p[G])\to \Hh(\Z_p[G][t, t^{-1}]^m)\to \GL'_m(\Z_p[G][t, t^{-1}])\to 1.
$$
We can now write (recall $g_p=g_\Q=g$ in $\GL'_m(\Q_p[G][t, t^{-1}])$)
\begin{eqnarray}\label{calc}
 \ti g_\Q\cdot \ti g^{-1}_p &=& ( g_{\Q},\alpha _{\Q} )
 ( g_{p},\alpha _{p} ) ^{-1}= ( g_{\Q},\alpha _{\Q} )  ( g_{p}^{-1},\alpha _{p}^{-g_{p}^{-1}} ) \\
&=& ( 1,   (      \alpha _{p}^{-g_{p}^{-1}})^{g_\Q}\cdot \alpha_\Q              ) =  ( 1, \alpha_p^{-1}\cdot \alpha_\Q) =\alpha_p^{-1}\cdot \alpha_\Q.  \notag
\end{eqnarray}%

Hence, $\alpha_p^{-1}\cdot \alpha_\Q=\ti g_\Q\cdot \ti g^{-1}_p$ with the product calculated
in $\Hh(\Q_p[G][t, t^{-1}])$ and we conclude that
\begin{equation}\label{FroRR}
\bar\chi^P(\PP^1, \F)=\prod_p (\ti g_\Q\cdot \ti g^{-1}_p)
\end{equation}
in $\Kr_0^{\rm red}(\Z[G])= ({\prod'_p} \Kr_1(\Q_pG))/ \Kr_1(\Q[G])\cdot\prod_p \rK_1(\Z_p[G])^\flat$. 

To show the Riemann-Roch identity, we will express the element in the right hand side
of (\ref{FroRR}) as the negative of the push-down of the second Chern class of $\F$.
We continue by giving first some preliminaries.
\smallskip

\subsubsection{} \label{9b2} Recall the set-up and definitions of  \S \ref{s:SK_1}.
Let $R$ be an integral domain with fraction field $N$ of characteristic $0$.
Let $N^c$ be an algebraic closure of $N$.
Consider the base change  
$$
  \Kr_{1}(R[t,t^{-1}][G]) \to  \Kr_{1}( N^c[t,t^{-1}][G]).
$$
Using Lemma \ref{sk_1lemma1} we can see that the kernel of this is equal to ${\rm SK}_1(R[t,t^{-1}][G])$.
Recall that the Bass-Heller-Swan theorem gives a homomorphism
$$
b_R: \Kr_{1}(R[t,t^{-1}][G]) \to \Kr_0(R[G])\times \Kr_1(R[G]).
$$
The base change $\Kr_1(R[G])\to \Kr_{1}(R[t,t^{-1}][G]) $
splits the second projection.
The map $b$ is an isomorphism when $R=N$ is a field of
characteristic zero. An explicit description of $b$ for the algebraically closed
$N^c$ is as follows:
Using Morita equivalence and by taking determinants we
obtain an isomorphism
\begin{equation}
 \Kr_{1} (  N^c[G][t,t^{-1}]) \xrightarrow{ \sim } {\rm Hom}(R_G,   N^c[t,t^{-1}]^{\times})
\end{equation}
where $R_G$ is the
group of $N^c$-valued characters
of $G$.
Since $ (N^c[t,t^{-1}])^{\times}=t^{\Z}\cdot  (N^c)^\times$
the target can be written
$$
{\rm Hom}(R_G,  t^\Z)\times {\rm Hom}(R_G,  (N^c)^\times )\xrightarrow{\sim}
\Kr_0(N^c[G])\times \Kr_1(N^c[G])
$$
 and $b_{N^c}$ is the resulting composition.

\begin{lemma}\label{bhs}
Suppose $(N[G]^m, g)$ are Horrocks data
for an $\O_{\PP^1_N}[G]$-bundle $\E$ on $\PP^1_N$.
Denote by $[g]$ the class of $g$ in $\rK_1(N[t,t^{-1}][G])$.
Then the  component
of $b_{ N^c}([g])$ in ${\rm Hom}(R_G,  t^\Z)=\Kr_0( N^c[G])$
is given by the character function
$$
\chi\mapsto t^{{\rm deg}((\E\otimes_{  N^c}V_{\bar\chi})^G)},
$$
where $V_\psi$ is a $N^c[G]$-module with character $\psi$. As a result, the  component
of $b_{ N}([g])$ in $ \Kr_0( N[G])$ is equal to the reduced
Euler characteristic $\bar\chi(\PP^1_{N}, \E)= \chi(\PP^1_{N}, \E)- \chi(\PP^1_{N}, \O_{\PP^1_{N}}[G]^m)$
in $\Kr_0( N[G])$.
\end{lemma}

\begin{proof}
The second part of the statement follows from the first part and the (usual) Riemann-Roch theorem
on the curve $\PP^1_{N^c}$. To show the first part is enough to
observe  that the degree
of a vector bundle obtained by gluing as above is given by the
valuation of the determinant of the transition (gluing) matrix at
$t=0$. 
\end{proof}
\smallskip

\subsubsection{} We now continue with the proof of Theorem \ref{proHorr}. Since $g$ is in $\GL'_m(\Q[G][t, t^{-1}])$
Lemma \ref{bhs} and the above discussion implies that there is  $\kappa_\Q\in \Kr_1(\Q[G])$
with ${\rm Det}([g])^{-1}={\rm Det}(\kappa_\Q)$.
Similarly,  since  $g$ is in $\GL'_m(\Z_p[G][t, t^{-1}])$
and $\Kr_0(\Z_p[G])\subset \Kr_0(\Q_p[G])$ we obtain that there
is $\kappa_p\in \Kr_1(\Z_p[G])$ such that
${\rm Det}([g])^{-1}={\rm Det}(\kappa_p)$. Lift $\kappa_\Q$, $\kappa_p$ to $ z_\Q\in \Q[G]^\times$, $ z_p\in \Z_p[G]^\times$, and consider the elements $g'_\Q=z_\Q\cdot  g \in
\GL(\Q[G][t, t^{-1}])$, $g'_p=z_p\cdot g \in
\GL(\Z_p[G][t, t^{-1}])$. For these elements we have
$$
[g'_\Q]\in {\rm SK}_1(\Q[G][t, t^{-1}]),\quad [g'_p]\in {\rm SK}_1(\Z_p[G][ t, t^{-1}]).
$$
Since by Morita equivalence and Lemma \ref{sk_1lemma1}, ${\rm SK}_1(\Q[G][t, t^{-1}])=(0)$, this shows that $g'_\Q$ is in $\rE(\Q[G][t, t^{-1}])$.
Consider the image of $[g'_p]$ in  ${\rm SK}_1(\Z_p[G]{\ldb t\rdb})$.
By Corollary \ref{SK1cor2}, the natural homomorphism
$$
  {\rm SK}_1(\Z_p[G]\langle\!\langle t^{-1}\rangle\!\rangle)\to  {\rm SK}_1(\Z_p[G]{\ldb t\rdb}),
$$
where $\Z_p\langle\!\langle t^{-1}\rangle\!\rangle$ is the $p$-adic completion of $\Z_p[t^{-1}]$,
is surjective. 
Therefore, for each $p$, we can find an element
$h_p\in {\rm GL}'(\Z_p[G]\langle\!\langle t^{-1}\rangle\!\rangle)$ with $[h_p]=[g'_p]^{-1}$
in  $ {\rm SK}_1(\Z_p[G]{\ldb t\rdb})\hookrightarrow 
 {\rK}_1(\Z_p[G]{\ldb t\rdb})$.  
Notice that for those $p$ which do not divide the order of $G$ and where $z_{\Q}$ is 
a unit,
 we can take $z_{p}=z_{\Q}$ and $h_p=1$.

We will now show how to choose (stably) Parshin bases
$f_{\eta_i}$ for $\F$  which provide us with an (adelic)  elementary structure.
For simplicity, we will write $\hat\O_{\eta_i}$, $\hat\O_{\eta_i\eta_j}$, instead of $\hat\O_{\PP^1,\eta_i}$,
$\hat\O_{\PP^1,\eta_i\eta_j}$, etc.
The Horrocks description for $\F$ provide us bases $e_{\eta_i}$ for $\hat \F_{\eta_i}$
which are determined by fixing a basis of $M_0\simeq \Z[G]^m$. 
Denote by $0$ the (unique) generic point of $\PP^1$.
Denote by $1_H$ the generic point of the divisor $H=\{t=0\}$, and for each prime $p$, denote
by $1_p$ the generic point of the fiber of $p$. We also denote by
$2_p$ the unique closed point $t=0$ in characteristic $p$ which
is the intersection of $1_H$ and $1_p$. The Horrocks gluing description implies
that there is a basis $e_0=\{e^h_0\}_{h=1}^m$ over the generic point such that
that $e_\eta=g^{-1}e_0$ if $\eta$ is on $1_H$ and $e_\eta=e_0$ otherwise.
This implies the following values for the transition matrices $\lambda_{\eta_i\eta_j}$
(recall $e_{\eta_i}=\lambda_{\eta_i\eta_j}e_{\eta_j}$):
\begin{equation}
\lambda_{0\eta_1}= 
\begin{cases} 
g, \ \ \hbox{\rm if $\eta_1=1_H$}, \\
 1, \ \  \hbox{\rm if $\eta_1\neq 1_H$}.
\end{cases}
 \end{equation}
\begin{equation}
\lambda_{\eta_1\eta_2}= 
\begin{cases}
1, \ \ \ \hbox{\rm if    $\eta_2\neq 2_p$},\\
1, \ \ \ \hbox{\rm if   $\eta_2=2_p$, $\eta_1=1_H$},\\
g, \ \ \  \hbox{\rm if   $\eta_2=2_p$, $\eta_1\neq 1_H$},
\end{cases}
\end{equation}
All the other values are determined from these and the cocycle condition.
We now give different bases $f_{\eta_i}$ by $f_{0}=z_\Q\cdot e_0$ and 
\begin{equation*}
 \ f_{\eta_1}=\begin{cases} e_{1_H},\ \ \ \ \ \ \  \hbox{\rm if $\eta_1=1_H$}\\
z_ph_p\cdot e_{1_p},     \hbox{\rm if $\eta_1=1_p$}\\
z_\Q\cdot e_{\eta_1}, \  \hbox{\rm if $\eta_1\not\in\{ H, 1_p\ \hbox{\rm for all $p$}\}$}
\end{cases}, \quad 
 f_{\eta_2}=\begin{cases} e_{2_p}, \ \ \ \ \ \ \ \ \hbox{\rm if $\eta_2=2_p$}\\
z_ph_p\cdot e_{\eta_2}, \ \hbox{\rm if $\eta_2\neq 2_p$ \hbox{\rm in char. $p$}}.
\end{cases}
\end{equation*}
 These give the following values for the transition matrices $\theta_{\eta_i\eta_j}$ with
respect to   $f_{\eta_i}$:
\begin{equation}
\theta_{0\eta_1}= 
\begin{cases}
1,   \ \ \ \ \ \ \ \ \ \ \ \hbox{\rm if $\eta_1\not\in\{ 1_H, 1_p\ \hbox{\rm for all $p$}\}$}\\
 z_\Q\cdot  g ,\ \ \ \ \ \ \hbox{\rm if $\eta_1= 1_H$}\\
 z_\Q h_p^{-1}z_p^{-1}, \ \hbox{\rm if $\eta_1= 1_p$}
\end{cases}
\end{equation}
 \begin{equation}
\theta_{\eta_1\eta_2}= 
\begin{cases}
1, \ \ \ \ \ \ \ \ \ \ \hbox{\rm if $\eta_1=1_H$}, \\
1, \ \ \ \  \ \ \ \ \ \ \hbox{\rm if $\eta_1=1_p$, $\eta_2\neq 2_p$}\\
 z_\Q\cdot  g, \ \ \ \ \ \hbox{\rm if $\eta_2=2_p$, $\eta_1\neq 1_H, 1_p\ \hbox{\rm for all $p$}$}\\
z_ph_p\cdot g, \ \ \ \hbox{\rm if $\eta_1=1_p$, $\eta_2=2_p$}\\
z_\Q h_p^{-1}z_p^{-1}, \hbox{\rm if $\eta_1$ is horizontal, $\eta_1\neq 1_H$,
$\eta_2\neq 2_p$ in characteristic $p$.}
\end{cases}
\end{equation}
 We now verify that the matrices $\theta_{\eta_i\eta_j}$
are   elementary for all pairs $\eta_i$, $\eta_j$.
By our construction,  ${\rm Det}(\theta_{\eta_i\eta_j})=1$. (Notice for example that
${\rm Det}( z_\Q)={\rm Det}( z_p)={\rm Det}(g)^{-1}$.)
Observe that: by Morita equivalence and the fact that ${\rm SK}_1$ 
is trivial for commutative local rings, we know that   ${\rm SK}_1(\hat\O_{0\eta_1}[G])=(0)$,
and that ${\rm SK}_1(\hat\O_{\eta_1\eta_2}[G])=(0)$
if $\eta_1$ is horizontal; and by  Proposition \ref{sk_1vanish}
and Morita equivalence we know that ${\rm SK}_1(\hat\O_{0\eta_2}[G])=(0)$.
The only thing left to check
is that $\theta_{1_p2_p}=z_ph_p  g$
has trivial image in $ {\rm SK}_1(\hat\O_{1_p2_p}[G])={\rm SK}_1(\Z_p[G]{\ldb t\rdb})$.
This follows from our choice of $ z_p$, $h_p$ above.

Hence, the above completes
the proof of  part (i) of the statement
of Theorem \ref{proHorr}.

 It now remains to show part (ii) which is the Riemann-Roch identity.

Notice that the elements $\theta_{\eta_i\eta_j}$ satisfy the conclusion of 
part (a) of Proposition \ref{propae} for the divisor $D$ which is the union of $1_H$ with the
fibers $1_p$ over the finite list $T$ of primes $p$ which either divide the order of $G$ or are 
such that  $z_\Q\in \Q[G]^\times$ does not belongs to $\Z_p[G]^\times$. Set $Q=\prod_{p\in T}p$. 
Enlarge $T$ and the corresponding divisor $D$ to ensure that the group $\rSK_1(\Z[Q^{-1}][t, t^{-1}][G])$ is trivial. 
(We can do this since, by our assumption, for sufficiently large $Q$, $\Z[Q^{-1}][G]$ is a product
of matrix rings with entries in principal ideal domains; we can then follow the same arguments as in \S \ref{9b2}.)

We will now show how to choose lifts  $\ti\theta_{\eta_i\eta_j}$ as in Proposition 
\ref{propae}  that can be used to calculate the adelic second Chern class 
according to the recipe in Definition \ref{defc2}.

\begin{proposition}\label{adelictheta} There   are choices of lifts $\ti \theta_{\eta_i\eta_j}$ 
of $\theta_{\eta_i\eta_j}$ so that $\ti\theta_{\eta_0\eta_1}\in \St(\hat\O_{ \eta_1}[D^{-1}][G])$,
$\ti\theta_{\eta_1\eta_2}\in \St(\hat\O_{ \eta_1\eta_2}[G])$, and $\ti\theta_{\eta_0\eta_2}
\in \St(\hat\O_{ \eta_2}[D^{-1}][G])$, so that for the element $z(\ti\theta)$ of Proposition \ref{propae} (b) we have:

(i)  $z( \ti\theta)_{(\eta_0,\eta_1,\eta_2)} =1$, unless $\eta _{1}=1_{p}$ and $\eta _{2}=2_{p}$,

(ii) $z (\ti \theta )_{(0,1_p,2_p)}=1 $, if $p$ does not divide the order of the group $G$ and is such that $z_\Q\in \Q[G]^\times$ belongs to $\Z_p[G]^\times$.
\end{proposition}

\begin{proof} Let us prove (i) first.   We will consider $z ( \ti\theta)_{(\eta _{0},\eta _{1}, \eta_{2})}$ for $\eta_2$ in characteristic $p$.  We suppose we do not have $\eta_1=1_p$, $\eta_2=2_p$.
Then   all cases of such triples $(\eta_0, \eta_1, \eta_2)$ have
 similar structure: namely, one of the $\theta _{\eta _{i}\eta _{j}}=1  $ and so the remaining two transition maps $ 
\theta _{\eta _{a}\eta _{b}}$ are equal (up to inversion) to a value that we denote $\theta $; there
are three relevant rings $\hat\O_{\eta_1}[D^{-1}]$, $\hat\O_{\eta_1\eta_2}$, and $\hat\O_{\eta_2}[D^{-1}]$
and the common value $\theta$ belongs to the intersection of two of them.
 In all cases we shall identify a subring $R$ of this intersection with the
property that $\theta \in  {\rE}( R[ G])$; we
can then use a lift $\ti \theta   \in {\St}(
R[ G] ) $ twice in computing $z( \ti \theta)_{(\eta_0, \eta_1,\eta_2)}$ to get the value 1. 
\smallskip

{\it Case 1:  Horizontal case $\eta_{1}=1_{H}$.}

 Here there is only one situation to consider; namely,   $\eta _{2}=2_p$; then $\theta _{1_{H}\eta _{2}}=1$, $\theta _{0\eta _{2}}=z_{\Q}g_{\Q}=\theta _{01_{H}}$.
We know that $z_{\Q}g_{\Q}\in {\rm SL}( \Z[Q^{-1}]
[t,t^{-1}] [ G] ) $. Since $\rSK_{1}( \Z[Q^{-1}][t,t^{-1}] [G]) =\{1\}$, we have 
$z_{\Q}g_{\Q}\in \rE ( \Z[Q^{-1}] [ t,t^{-1} ]  [ G ])$. 
Note that
\begin{equation}
\Z[Q^{-1}][ t,t^{-1}] \rightarrow \hat{\mathcal{O}}_{ 1_{H}}[D^{-1}],
\ \ \Z[Q^{-1}][ t,t^{-1}] \rightarrow \hat{\mathcal{O}}_{ \eta _{2}}[D^{-1}].
\end{equation}%
  We can therefore use the surjection $\St ( \Z[Q^{-1}] [
t,t^{-1} ]  [ G ]  ) \rightarrow \rE ( \Z[Q^{-1}] [
t,t^{-1} ][G])$ to find a common lift of both 
$\theta _{1_{H}\eta _{2}}$ and $\theta _{0\eta _{2}}$;
the corresponding $z(\ti\theta)_{(0,1_H,\eta_2)}$ is then trivial.
In other words, here we take in the above sketch $R=\Z[Q^{-1}][t, t^{-1}]$.
\smallskip

{\it Case 2: Horizontal case  $\eta_{1}\neq 1_{H} $. }

Here $\theta_{\eta _{0}\eta _{1}}=1$ and there are two subcases to
consider.

Subcase (a):  $\eta _{2}\neq 2_{p}$. 
Then $\theta _{\eta _{1}\eta _{2}}=z_{\Q}h_{p}^{-1}z_{p}^{-1}=\theta
_{\eta _{0}\eta _{2}}$. We have  $z_{\Q}h_{p}^{-1}z_{p}^{-1}=1$, if $p\not\in T$, and $z_{\Q}h_{p}^{-1}z_{p}^{-1}\in  {\rm SL} ( \Q\otimes_{\Z}\Z_{p} \langle\!\langle t^{-1} \rangle\!\rangle  [ G])=\rE(\Q\otimes_{\Z}\Z_{p} \langle\!\langle t^{-1} \rangle\!\rangle  [ G])$  if $p\in T$ (cf. Lemma \ref{sk_1lemma2}). 
The situation is trivial if $p\not\in T$. If $p\in T$, $\hat\O_{\eta_2}[D^{-1}]$ and $\hat\O_{\eta_1\eta_2}$ are the two relevant rings; we can then take
$R$ to be $\Q\otimes_{\Z}\Z_{p} \langle  \langle t^{-1} \rangle \rangle$ and proceed as before
 to get $z(\ti \theta)_{(\eta _{0}, \eta _{1},\eta _{2})} =1$.

Subcase (b):  $\eta _{2}=2_{p} $. Then $\theta _{\eta _{1}\eta _{2}}=z_{\Q}g_{\Q}=\theta _{\eta
_{0}\eta _{2}}$, $\hat\O_{\eta_2}[D^{-1}]=\Z_p\llps t\lrps[Q^{-1}]$ and $\hat\O_{\eta_1\eta_2}$ are the two relevant rings; 
here we can work with $R=\Z_p[Q^{-1}][t,t^{-1}]$. Notice here that since $\eta_1\neq 1_H$, $t$ is invertible in $\hat\O_{\eta_1\eta_2}$.
 \smallskip

{\it Case 3: Vertical case when $\mathbf{\eta }_{1}={1}_{p}$.}

Here 
there is only the case $\eta _{2}$ $\neq 2_{p}$ (since the case $\eta _{2} 
=2_{p}$ is excluded).  Then $\theta _{1_{p}\eta _{2}}=1$ and $\theta _{\eta_01_{p}}=z_{\Q}h_{p}^{-1}z_{p}^{-1}=\theta _{\eta _{0}\eta _{2}} $. The two relevant
rings are $\hat\O_{\eta_1}[D^{-1}]=\hat\O_{\eta_1}[t^{-1}, Q^{-1}]$, $\hat\O_{\eta_2}[D^{-1}]=\Z_p[[t-\eta_2]][t^{-1}, Q^{-1}]$; Since $t$ is a unit in $\Z_p[[t-\eta_2]]$,
here we can work with $R=\Z_p\langle
\!\langle t^{-1} \rangle\!\rangle[Q^{-1}]$.
 \smallskip
 
To prove (ii) we observe that for a prime $p$ that does not divide the 
order of the group and with $z_\Q$ a unit at $p$, we have taken $z_p=z_\Q$, $h_p=1$.
Then $\theta_{01_p}=1$, $\theta_{1_p2_p}=\theta_{02_p}=z_\Q\cdot g=z_p\cdot g$.
Part (ii) then follows by a similar argument as above by using the ring $R=\Z_p[Q^{-1}][t, t^{-1}]$. 

This completes the proof of Proposition \ref{adelictheta}.
\end{proof} 
\smallskip

By Proposition \ref{adelictheta} a representative for the pushdown $f_*(c_2(\F))$ 
of the second adelic Chern class $c_2(\F)$ of $\F$ is given by the idele in $\prod_p'\rK_1(\Q_p[G])$ whose
component at $p$ is given by the push-down $f_{\ast}$ of 
\begin{equation}\label{eq15}
z (\ti \theta )_{(0,1_p,2_p)}= {s}_{02} ( z_{\Q}g_{\Q} )  {s}_{12} (
z_{p}h_{p}g ) ^{-1} {s}_{01} ( z_{\Q}h_{p}^{-1}z_{p}^{-1} ) ^{-1}.
\end{equation}   
Here, for clarity, in the right hand side,  we use the symbol $s_{12}$ to denote the 
lift of an element in the Steinberg group 
$ {\St}(\hat\O_{ \eta_1\eta_2}[G])={\St}(\Z_p{\ldb t\rdb}[G])$, the symbol $s_{02}$ to denote a lift in the Steinberg group 
$ {\St}(\hat\O_{\eta_2}[D^{-1}][G])={\St}(\Z_p\llps t\lrps[Q^{-1}][G])$, and the symbol $s_{01}$ to denote a lift in the Steinberg group 
$ {\St}(\hat\O_{\eta_1}[D^{-1}][G])$. 
The product is taken in $ {\St}(\hat\O_{\eta_0\eta_1\eta_2}[G])=\St(\Q_p{\ldb t\rdb}[G])$.

Recall that with the above notation we write 
$ 
 w_{\Q
} ( g_{\Q} ) = ( g_{\Q},\alpha _{\Q} )$, $ w_{p} ( g_{p} ) = ( g_{p},\alpha
_{p} ) $.
The desired result will follow if we show
$$
f_{\ast } ( z (\ti \theta )_{(0,1_p,2_p)} )
 ) =\alpha^{-1}_{\Q}  \alpha _{p}\kappa^{-1}_{\Q}\kappa
_{p}
$$
 with $\kappa _{\Q}\in \rK_{1} ( \Q[G]
 )^\flat $, $\kappa _{p}\in \rK_{1} ( \Z_p[G])^\flat$. We now  evaluate the pushdown $f_{\ast } ( z (\ti \theta )_{(0,1_p,2_p)} )$  by working with each of the three
right hand terms in equation (\ref{eq15}). Recall that the pushdown is defined 
via the inverse of the homomorphisms $\partial$, $\hat\partial$ and in particular,
$f_{\ast } ( z (\ti \theta )_{(0,1_p,2_p)} )=\hat\partial  ( z (\ti \theta )_{(0,1_p,2_p)} )^{-1}$.

 \medskip

1) $   {s}_{02}( z_{\Q}g_{\Q}) $: By the above we know that
\begin{equation*}
z_{\Q}g_{\Q}\in {\rm SL}( \Z[Q^{-1}][ t,t^{-1}]
[ G] ) \subset \rE( \Z[Q^{-1}]\otimes\Z\llps t\lrps [ G] ) \subset \rE( \hat{\mathcal{O}%
}_{02_{p}}[ G]) .
\end{equation*}
We let $\mathcal{H}( \Z[Q^{-1}]\otimes \Z\llps t\lrps
[ G])  $ denote the pullback of $\mathcal{H}( \Q(( t))[ G] ) $ along
\begin{equation*}
{\rm GL}^{\prime }( \Z[Q^{-1}]\otimes \Z\llps t\lrps [ G
] ) \subset {\rm GL}^{\prime }( \Q\llps t\lrps[ G] ).
\end{equation*}
We may now compute   using the following diagram
\begin{equation}\label{stsections}
\begin{array}{ccc}
 {\rE}(  \Z[Q^{-1}]\otimes\Z\llps t\lrps[ G] ) & \overset{ {s}_Q}{\rightarrow } &
 {\St}( \Z[Q^{-1}]\otimes\Z\llps t\lrps [ G] ) \\
\downarrow \text{inclusion} &  & \downarrow \partial \\
{\rm GL}^{\prime } ( \Z[Q^{-1}]\otimes\Z\llps t\lrps[G]) & \overset{ w_Q}{\rightarrow } &
{\mathcal{H}}( \Z[Q^{-1}]\otimes\Z\llps t\lrps[G])
\end{array}
\end{equation}
which commutes up to an element of $\rK_{1}( \Q[ G]
 )$ (see (\ref{eq:StoHmap})). Here $ w_Q$ is a set-theoretic section of the $\Hh$-sequence
that is compatible with the natural splitting of   $\mathcal{H}( \Q( (
t) ) [ G])$   over $\GL(\Q [ G ] )$.  We then get the equality
\begin{equation}\label{eq9.18}
\hat\partial(  {s}_{02}( z_{\Q}g_{\Q})
) =\partial( s_{Q}( z_{\Q}g_{\Q
}) ) =\kappa _{\Q} w_{Q}( z_{\Q}g_{\Q}) =\kappa_{\Q}z_{\Q}  w_{Q}(
g_{\Q})
\end{equation}%
for some $\kappa _{\Q}\in \rK_{1}( \Q[ G]) $.
\medskip

2)  $  {s}_{12}(
z_{p}h_{p}g_{p})$: From the above we know that
\begin{equation*}
z_{p}\in   \Z_{p}[ G]^\times ,\quad  h_{p}\in \SL(
\Z_{p}\langle\!\langle t^{-1}\rangle\!\rangle [ G]) ,\quad  g_{p}\in \GL^{\prime }( \Z_{p}[
t,t^{-1}] [ G] ) .
\end{equation*}
Thus all terms lie in ${\rm GL}^{\prime }( \Z_{p}{\ldb t\rdb}[ G] ) =\GL^{\prime }( 
\hat\O_{1_p2_p}[ G] ) $ and by construction the product of
these three elements lies in $\rE( \Z_{p}{\ldb t\rdb}[ G] ) $. We now compute $\hat\partial( s_{12}( z_{p}h_{p}g_{p}) ) $ using the analogous diagram to
(\ref{stsections}) above for the ring $\Z_{p}{\ldb t\rdb}[ G] $, which
commutes up to an element of $\rK_{1}( \Z_{p}[ G])$. Using the fact that $ {\mathcal{H}}({\Z}_{p}{\ldb t\rdb}[ G] ) $ splits naturally over ${\rm GL}^{\prime
}( \Z_{p}\langle\!\langle t^{-1}\rangle\!\rangle [ G] )$, we obtain the equality
\begin{equation}\label{eq9.19}
\hat\partial(  {s}_{12}( z_{p}h_{p}g_{p}) )
=\kappa _{p} {w}_{p}( z_{p}h_{p}g_{p}) =\kappa
_{p}z_{p}h_{p} {w}_{p}( g_{p})
\end{equation}%
for some $\kappa _{p}\in \rK_{1}( \Z_{p}[ G] )$.
\medskip

3) $  {s}_{01}( z_{\Q
}h_{p}^{-1}z_{p}^{-1})  $: If $p\not\in T$ then this term is trivial. We assume that $p$ is in $T$
so that $\Z_p[Q^{-1}]=\Q_p$.  
From the above work we know that
\begin{equation*}
z_{\Q}\in   \Z[Q^{-1}][ G]^\times ,\quad  h_{p}\in
{\rm SL}( \Z_{p}\langle\!\langle t^{-1}\rangle\!
\rangle [ G]) ,\quad  z_{p}\in   \Z_{p}[ G]^\times .
\end{equation*}
Thus all terms lie in ${\rm GL}( \Q_p\{t^{-1}\}[ G]) $. By construction their product lies in $
{\rm SL}( \Q_p\{t^{-1}\} [ G])$.  Using Lemma \ref{sk_1lemma2} we deduce that $z_{\Q}h_{p}^{-1}z_{p}^{-1}\in \rE( 
\Q_p\{t^{-1}\} [ G])$.
Define 
$\hat{\mathcal{H}}(\Q_p\{t^{-1}\} [ G] ) $ to be the pullback of $\hat\Hh( \Q_{p}{\ldb t\rdb}[ G] ) $ along
the inclusion
\begin{equation*}
 \GL( 
\Q_p\{t^{-1}\} [ G]) 
\hookrightarrow \GL^* ( \Q_{p}{\ldb t\rdb}[ G]).
\end{equation*}
By \S \ref{3e3} we have a natural splitting $w$  of the resulting
central extension of $\GL(\Q_p\{t^{-1}\}[G])$ which agrees with the 
natural splittings over $\Q[G]^\times$, $\Z_p[G]^\times$,
${\rm SL}(\Z_p\langle\!\langle t^{-1}\rangle\!\rangle[G])$
that we have used in cases (1) and (2) above. We can also restrict this extension along $\rE(\Q_p\{t^{-1}\}[G])$.
The fact that this last extension of  $\rE(\Q_p\{t^{-1}\}[G])$ splits also follows from Corollary \ref{correc}, since the Steinberg sequence for 
$ \Q_p\{t^{-1}\}[G]$
is the universal central extension of the perfect group $\rE(\Q_p\{t^{-1}\}[G])$.
The universality gives  
$$
\hat\partial: \St(\Q_p  \{ t^{-1}\} [ G]) \to \hat{\mathcal{H}}( \Q_{p}{\ldb t\rdb}[ G] )
$$
which then has to be equal to the composition $w\circ \pi$ where $\pi$ is the natural homomorphism from the   Steinberg group to the
 elementary group of $\Q_p\{t^{-1}\}[G]$.
Thus we get
\begin{equation}\label{eq9.20}
\hat\partial({s}_{01}( z_{\Q}h_{p}^{-1}z_{p}^{-1}
) ) =w( z_{\Q}h_{p}^{-1}z_{p}^{-1}) =z_{\Q}h_{p}^{-1}z_{p}^{-1}.
\end{equation}

In summary from (\ref{eq9.18}), (\ref{eq9.19}) and (\ref{eq9.20}), we see that the expression
\begin{equation*}
\hat\partial  ( z (\ti \theta )_{(0,1_p,2_p)} )=\hat\partial( {s}_{02}( z_{\Q}g_{\Q})) \hat\partial( {s}_{12}( z_{p}h_{p}g_{p})) ^{-1}
\hat\partial(  {s}_{01}( z_{\Q
}h_{p}^{-1}z_{p}^{-1})) ^{-1}
\end{equation*}%
is equal to
\begin{eqnarray*}
&&z_{\Q}w_{Q}( g_{\Q}) ( z_{p}h_{p}
{w}_{p}( g_{p}) ) ^{-1}( z_{\Q}h_{p}^{-1}z_{p}^{-1}) ^{-1}\kappa_{\Q}\kappa _{p}^{-1} \\
&=&z_{\Q}w_{Q}( g_{\Q})  {w}_{p}( g_{p}) ^{-1}h_{p}^{-1}z_{p}^{-1}z_{p}h_{p}z_{\Q}^{-1}\kappa_{\Q}\kappa _{p}^{-1} \\
&=&w_{Q}( g_{\Q}) {w}_{p}(g_{p}) ^{-1}\kappa_{\Q}\kappa_{p}^{-1}.
\end{eqnarray*}
By (\ref{calc}) above this is equal to $ \alpha _{p}^{-1}\alpha_\Q\kappa _{\Q}\kappa _{p}^{-1}$.
Since $f_{\ast } ( z (\ti \theta )_{(0,1_p,2_p)} )=\hat\partial  ( z (\ti \theta )_{(0,1_p,2_p)} )^{-1}$
we obtain $f_{\ast } ( z (\ti \theta )_{(0,1_p,2_p)} )= \kappa_p\kappa_\Q^{-1}\alpha^{-1}_\Q\alpha _{p}$
which completes our proof.
\end{proof}

\bigskip
\bigskip

\section{Appendix: Adelic Riemann-Roch for general bundles when $G=1$}\label{Appendix}

\setcounter{equation}{0}

In this appendix, we assume that $G$ is trivial. We explain how our results, when combined with the Deligne-Riemann-Roch theorem \cite{DeligneDet}, and some standard constructions (\cite{BudChern}, \cite{BeSch})
imply, in this case, a general adelic Riemann-Roch theorem for all bundles.

\subsection{More central extensions}\label{morecentral} Suppose that $A$ is a commutative ring 
with ${\rm SK}_1(A)=(1)$. Then ${\rm E}(A)=\SL(A)$ 
and we can construct a central extension
\begin{equation}\label{c2}
1\to \rK_2(A)\to {\rm GSt}(A)\to \GL(A)\to 1
\end{equation}
of $\GL(A)$ as follows: Observe that $\GL(A)=A^*\ltimes \SL(A)=A^*\ltimes {\rm E}(A)$,
where $A^*\subset \GL(A)$ by $a\mapsto d_1(a):={\rm diag}(a, 1, 1,\ldots )$ and 
 $A^*$ acts on $\SL(A)$ by conjugation. Since the Steinberg extension
\begin{equation}\label{cextSt}
1\to \rK_2(A)\to {\rm  St}(A)\to {\rm E}(A)=\SL(A)\to 1
\end{equation}
is the universal central extension of the perfect group ${\rm E}(A) $, this conjugation action 
lifts to ${\rm St}(A)$ and we can define ${\rm GSt}(A)$ as the corresponding semi-direct product
$A^*\ltimes {\rm St}(A)$; the restriction of the homomorphism  ${\rm GSt}(A)\to \GL(A)$
to $A^*$ is then given by $a\mapsto d_1(a)$
and the restriction of the extension (\ref{c2}) over $\SL(A)$ is the Steinberg extension (\ref{cextSt}).

\begin{proposition}\label{compareExt}
1) The restriction of (\ref{c2}) to  $A^*\times A^*\subset\GL(A)$,
$(a, b)\mapsto {\rm diag}(a, b, 1,   \ldots )$, is isomorphic to the extension 
\begin{equation}\label{c1c1}
1\to \rK_2(A)\to \widehat{ A^*\times A^*}\to A^*\times A^*\to 1 
\end{equation}
where $
\widehat{ A^*\times A^*}:=A^*\times A^*\times \rK_2(A)
$ 
with operation given by 
$$
(f, g; r)\cdot (f', g'; r')=(ff', gg'; rr' \{f', g\} ),
$$
where $\{ , \}: A^*\times A^*\to \rK_2(A)$ is the Steinberg symbol.

2) The restriction of (\ref{c2})  
 to 
the subgroup 
$$
A^*\times \GL(A) =\left\{\left (\begin{matrix}
a & 0\\
0& X
\end{matrix}\right )\right\} \subset  \GL(A)
$$
 is isomorphic to the sum of the pull-back extensions $p_2^*({\rm GSt}(A))$ and 
$(p_1\times \det(p_2))^*(\widehat{ A^*\times A^*})$,
where $p_1: A^*\times \GL(A)\to A^*$, $p_2: A^*\times\GL(A)\to \GL(A)$
are the two projections.
\end{proposition}
\begin{proof} Notice that (2) implies (1).
Statement (2) is stated, without proof, in \cite{BeSch}. 
Statement (1) follows easily from the definition of
the Steinberg symbol and the calculations in \cite[\S 9]{MilnorK}
(see   the last line of p. 1654 in \cite {BudChern}).
To show (2), observe that we can write $A^*\times \GL(A)$
as a semi-direct product $(A^*\times A^*)\ltimes \SL(A)$. 
By  \cite[Construction 1.7]{BrylinskiDeligne}, central extensions of a semi-direct product $H\ltimes \Gamma$
by a group $T$ are determined 
by triples consisting of
\begin{itemize}
\item[a)] a central extension $\hat \Gamma$ of $\Gamma$ by $T$,
\item[b)] a central extension $\hat H$ of $H$ by $T$,
\item[c)] an action of $H$ on $\hat \Gamma\to \Gamma$, lifting the
action of $H$ on $\Gamma$.
\end{itemize} 
Given a central extension of $H\ltimes \Gamma$ we obtain the
corresponding central extensions as in (a) and (b) by restrictions.
Apply this to $H=A^*\times A^*$, $\Gamma=\SL(A)$. 
The restriction of $(\ref{c2}) $ to $\Gamma=\SL(A)$
is isomorphic to the Steinberg extension and the action
of $A^*\times A^*$ on ${\rm St}(A)$ is uniquely determined
(the first factor acts trivially and the second by lifting 
the conjugation). It remains to consider the restriction 
of $(\ref{c2}) $ to $A^*\times A^*$ which is determined
by part (1). Hence, we see that the triple for the restriction 
of $(\ref{c2}) $ to $A^*\times \GL(A)$ agrees with that of 
$p_2^*({\rm GSt}(A))\cdot  (p_1\times \det(p_2))^*(\widehat{ A^*\times A^*})$.
By   loc. cit., this implies part (2).
\end{proof}

\subsection{Second Chern class}

We  now assume that $R$
is a Dedekind ring, that $f: Y\to S=\Spec(R)$ is a flat projective 
morphism of relative dimension $1$, and that $Y$ is regular.
We can define the $\rK_2$-adele groups  ${\rK}_{\ell}^{\prime } ( \mathbb{A}_{Y,012})$,
$ {\rK}_{\ell}^{\prime } ( \mathbb{A}_{Y,ij})$, and the adelic Chow groups ${\rm CH}^\ell_{\Bbb A}(Y)$, $\ell=1, 2$, 
${\rm CH}^1_{\Bbb A}(S)$ as in \S \ref{def:ECg}, by taking $G=\{1\}$.

Now suppose that $\E$ is a locally free coherent sheaf of 
$\O_Y$-modules on $Y$ of rank $n$. As in \S \ref{s:Chern}, choose trivializations $e_{\eta }$ of $\hat\E_{\eta }$ over $\Spec(\hat\O_{Y,\eta})$ for each $\eta\in Y$;
these give transition matrices $\lambda_{\eta_i\eta_j}\in \GL(\hat\O_{Y, \eta_i\eta_j})$.
We define the first  adelic Chern class $c_1(\E)$ of $\E$ in ${\rm CH}^1_{\Bbb A}(Y)$ as in 
\S \ref{sec5.2.1}. 
Notice that all the rings $\hat\O_{Y, \und\eta}$,
where $\und\eta$ is any Parshin chain on $Y$, have trivial $\rSK_1$ and so  the constructions of \S \ref{morecentral}
apply to all of them. To define the second adelic Chern class $c_2(\E)$ of $\E$ we  
 consider 
\begin{equation}
z(\ti \lambda):=\prod_{(\eta_0, \eta_1, \eta_2)} \tilde \lambda_{\eta_0\eta_2}\cdot (\ti \lambda_{\eta_1\eta_2})^{-1}\cdot 
(\ti \lambda_{\eta_0\eta_1})^{-1}
\end{equation}
as in \S\ref{sec5.2.3}. Here $\tilde\lambda_{\eta_i\eta_j} \in {\rm GSt}(\hat\O_{Y, \eta_i\eta_j})$
is a lift of $\lambda_{\eta_i\eta_j}\in \GL(\hat\O_{Y, \eta_i\eta_j})$.
We can easily see by picking a divisor $D$ as in Proposition \ref{propae}, that
$z(\ti \lambda)$ lies in the group 
\begin{equation*}
 {\rK}_{2}^{\prime } ( \mathbb{A}_{Y,012}) \cdot   {\rK}_{2}( \mathbb{A}_{Y,12})^\flat
\cdot  {\rK}_{2}^{\prime } ( \mathbb{A}_{Y,01})^\flat. 
\end{equation*} 
Hence, we can define the second adelic Chern class $c_2(\E)$ to be the class
of $z(\ti \lambda)$ in ${\rm CH}^2_{\Bbb A}(Y)$. We can now see, using the arguments
in the proof of Theorem \ref{thm25}, that $c_2(\E)$ is independent of 
the choices of $e_\eta$ and of the lifts $\ti\lambda_{\eta_i\eta_j}$.
(These arguments become simpler in the case at hand.)

\subsection{A pairing}
The special case that $\E=\L\oplus \M$, where $\L$ and $\M$ are of rank $1$, \emph{i.e.} invertible sheaves, is going to be useful in what follows. To calculate $c_2(\L\oplus \M)$ we pick rational sections $l$ and $m$ of the line bundles $\L$ and $\M$;
these provide bases of $\L_{\eta_0}$ and $\M_{\eta_0}$ and allow us to identify $\L$ and $\M$ with $\O_Y(L)$ and $\O_Y(M)$ where $L$ and $M$ are Cartier divisors
on $Y$. Choose further bases for $\L$ and $\M$ at all other points $\eta$ of $Y$. 
Our choice of bases provide us with transitions
$l_{\eta_i\eta_j}$, $m_{\eta_i\eta_j}$ in $\hat\O_{Y, \eta_i\eta_j}^*$.
This gives
$$
t_{\eta_i\eta_j}=(l_{\eta_i\eta_j}, m_{\eta_i\eta_j})\in \hat\O_{Y, \eta_i\eta_j}^*\times \hat\O_{Y, \eta_i\eta_j}^*
$$
as transitions for $\L\oplus\M$ and, for a Parshin triple $(\eta_0,\eta_1,\eta_2)$,
we can form
$$
z_{\eta_0\eta_1\eta_2}=\tilde t_{\eta_0\eta_2}\cdot (\ti t_{\eta_1\eta_2})^{-1}\cdot (\ti 
t_{\eta_0\eta_1})^{-1}\in \rK_2(\hat\O_{Y, \eta_0\eta_1\eta_2})
$$
using the lifts in the extension $(\ref{c1c1})$; then we have
 \begin{equation}\label{gen5}
z_{\eta_0\eta_1\eta_2}=\{l_{\eta_1\eta_2}, m_{\eta_0\eta_1}\}^{-1}.
\end{equation}
To determine a favorable expression for the representative $z_{\eta_0\eta_1\eta_2}$
we can assume that the supports of $L$ and $M$ do not share any 
irreducible components and that do not contain any irreducible
components of singular fibers of $f: Y\to S$. If $\eta$ is outside the support $|L|$, resp. $|M|$, we choose as $1$ the basis
of $\L=\O_Y(L)$, resp. $\M=\O_Y(M)$, over $\hat\O_{Y, \eta}$; we choose also bases
over $\hat\O_{Y, \eta}$ for all other points $\eta$.
There are several cases to consider:

a) $\eta_1\not\in |L|\cup |M|$ and $\eta_2\not\in |L|\cup |M|$.
Then   obviously $z_{\eta_0\eta_1\eta_2}=1$.

b) $\eta_1\not\in |L|\cup |M|$ but $\eta_2 \in |L|\cup |M|$.
Then $l_{\eta_0\eta_1}=m_{\eta_0\eta_1}=1$ and  
$z_{\eta_0\eta_1\eta_2}=1$.

c) $\eta_1\in |M|$; then $\eta_1\not\in |L|$. Suppose that $\eta_2\not\in |D_1|$.
Then $l_{\eta_i\eta_j}=1$ and $z_{\eta_0\eta_1\eta_2}=1$.

d)  $\eta_1\in |M|$; then $\eta_1\not\in |L|$. Suppose that 
$\eta_2 \in |L|$, then $\eta_2\in |L|\cap |M|$. 
In this case, we have $m_{\eta_0\eta_1}=\varpi_{\eta_1}^{n(\eta_1; M)}$, where
$\varpi_{\eta_1}$ is a uniformizer at $\eta_1$ and $n(\eta_1; M)$ is the multiplicity of $M$ along 
$\eta_1$. We also have $l_{\eta_0\eta_2}=l_{\eta_0\eta_1}l_{\eta_1\eta_2}$
and since $l_{\eta_0\eta_1}=1$,   $l_{\eta_0\eta_2}= l_{\eta_1\eta_2}=t_{L,\eta_2}$
where $t_{L,\eta_2}$ is a local equation for the Cartier divisor $L$ in the (complete)
local ring $\hat\O_{Y, \eta_2}$. We get 
$$
z_{\eta_0\eta_1\eta_2}=\{t_{L,\eta_2}, \varpi_{\eta_1}\}^{-n(\eta_1;M)}.
$$

e) $\eta_1\in |L|$; then $\eta_1\not\in |M|$. Suppose that 
$\eta_2 \in |M|$, then $\eta_2\in |L|\cap |M|$. In this case, we have
$m_{\eta_0\eta_1}=1$ and so again $z_{\eta_0\eta_1\eta_2}=1$.

We see that case (d) is the only one in which we have a non-trivial contribution; this is when 
$\eta_1\in |M|$ and $\eta_2\in |L|\cap |M|$. Since $|L|$ and $|M|$ have no 
common irreducible components, there is only a finite number 
of Parshin triples $(\eta_0,\eta_1, \eta_2)$ where this happens. 
Therefore,  the second Chern class $c_2(\L\oplus\M)$ is represented by
\begin{equation}\label{gen6}
 z(\ti t)_{\eta_0\eta_1\eta_2}=\begin{cases} \{t_{L,\eta_2}, \varpi_{\eta_1}\}^{-n(\eta_1;M)} , 
\ \hbox{\rm if $\eta_1\in |M|$  and  $\eta_2\in |L|\cap |M|$,} \\
\ \   1, \ \ \ \ \ \ \ \ \ \ \ \ \ \ \ \ \ \ \ \ \ \hbox{\rm otherwise}
\end{cases}
\end{equation}
and $z(\ti t)$  belongs to $\rK_2'({\Bbb A}_{Y, 012})$.

Notice that the pairing 
$$
\cap: {\rm Pic}(Y)\times{\rm Pic}(Y)\to {\rm CH}^2_{\Bbb A}(Y)
$$
given by $([\L], [\M])\mapsto [\L]\cap [\M]:=c_2(\L\oplus \M)$ is bilinear and symmetric.
The bilinearity follows from Proposition \ref{compareExt} (1) 
and the bilinearity of the Steinberg symbol. Symmetry follows
from $\L\oplus\M\simeq \M\oplus\L$.

\subsection{The adelic Riemann-Roch theorem}
Assume now that, in addition to the above hypotheses, all the irreducible components of all the fibers
of $f: Y\to S=\Spec(R)$ are smooth and that the residue fields of $R$ at closed points are finite. Then we can also define a push-down homomorphism
$$
f_*:  {\rm CH}^2_{\Bbb A}(Y)\to {\rm CH}^1_{\Bbb A}(S)
$$ 
exactly as in \S \ref{sec4.1}. Recall that $\rK_0(R)\cong \Z\oplus {\rm Pic}(R)$ given by $[P]\to ({\rm rank}(P), \det(P))$;
the usual theory of ideles of $R$ gives a canonical isomorphism ${\rm Pic}(R)\cong {\rm CH}^1_{\Bbb A}(S)$
given by the first Chern class $c_1$.

Suppose that 
$\L$ and $\M$ are invertible sheaves on $Y$; then Deligne \cite[\S 6]{DeligneDet}
associates to the pair $(\L, \M)$ an invertible sheaf $\langle \L, \M\rangle$
on $S$; $(\L, \M)\mapsto  \langle \L, \M\rangle$ is the ``Deligne pairing".

\begin{proposition}\label{DP}
Under the above assumptions, we have
$ 
f_*([\L]\cap [\M])=[\langle \L, \M\rangle]
$  
in ${\rm CH}^1_{\Bbb A}(S)\cong {\rm Pic}(R)$.
\end{proposition}
\begin{proof} 
Using the bilinearity of the Deligne pairing and the above,
we see that it is enough to assume that $\L=\O_Y(L)$, $\M=\O_Y(M)$, where $L$ and $M$ are irreducible with $|L|\neq |M|$, and are either horizontal, 
in which case the morphism to $S$ is finite and flat, or are equal to a smooth special fiber of $Y\to S$. If they are both vertical, we obviously have 
$\langle \L, \M\rangle\simeq \O_S$ and also $[\L]\cap [\M]=0$.
By symmetry, we can assume that $M$ is horizontal.  Then by \cite{DeligneDet}
$$
\langle \L, \M\rangle=\langle \L, \O_Y(M)\rangle\cong {\rm Norm}_{M/S}(\L_{|M}).
$$
To compare this with $f_*([\L]\cap[\M])$ we consider the push-down
$f_*(\{t_{L, \eta_2}, \varpi_{\eta_1}\}^{-1})\in k(S)_{\xi_1}^*$. Here $\xi_1=f(\eta_2)$
is a closed point of $S$ and $k(S)_{\xi_1}=\hat\O_{S,\xi_0\xi_1}$ the completion of the fraction field
$k(S)$ at $\xi_1$. By our construction, since now $ \eta_1 \in |M|$
is horizontal, this is
$$
f_*(\{t_{L, \eta_2}, \varpi_{\eta_1}\}^{-1})={\rm Norm}_{k(\eta_1)_{\eta_2}/k(S)_{\xi_1}}(
\partial  (\{t_{L, \eta_2}, \varpi_{\eta_1}\} ))
$$
where 
$$
\partial (\{t_{L, \eta_2}, \varpi_{\eta_1}\})=(-1)^{1\cdot 0} \frac{t^1_{L, \eta_2}}
{\varpi^0_{\eta_1}} \mod (\varpi_{\eta_1})=t_{L, \eta_2}\mod (\varpi_{\eta_1}).
$$
Since $M$ is locally at $\eta_2$ given by $\varpi_{\eta_1}=0$, 
and $L$ by $t_{L, \eta_2}=0$,
the norm of $t^{-1}_{L, \eta_2}\mod (\varpi_{\eta_1})$ is a generator of 
${\rm Norm}_{M/S}(\L_{|M})$ at the point $\xi_1$ of $S$ and the result
follows.  \end{proof}

\smallskip

Suppose $\E$ is a locally free coherent sheaf of 
$\O_Y$-modules on $Y$ of rank $n$. We denote by $\bar\chi(Y, \E)$ the stable Euler characteristic of $\E$
in ${\rm Pic}(R)\cong{\rm CH}^1_{\Bbb A}(S)$.
We also denote by $\omega_{Y/S}$ the (invertible) dualizing sheaf of 
the morphism $f:Y\to S$.

\begin{theorem}\label{genARR} (Adelic-Riemann-Roch theorem) Under the above assumptions,
\begin{equation}\label{generalARReq}
2\cdot (\bar\chi(Y, \E)-n\cdot \bar\chi(Y,\O_Y))= f_*(\det(\E)\cap \det(\E)-2c_2(\E) +\det(\E)\cap \omega^{-1}_{Y/S})
\end{equation}
in ${\rm Pic}(R)\cong{\rm CH}^1_{\Bbb A}(S)$.
\end{theorem}

\begin{proof} To ease notation, we will write $\bar\chi(\E)$
instead of $\bar\chi(Y, \E)$, etc. Notice that by our definitions, if $\E=\LL$ has rank $1$, then 
$c_2(\E)=0$ since the extension (\ref{c2}) splits over $A^*$. In this case,
(\ref{generalARReq}) amounts to
\begin{equation}\label{genDRR}
2\cdot (\bar\chi( \LL)-  \bar\chi( \O_Y))= f_*(   \LL\cap\LL +\LL\cap \omega^{-1}_{Y/S} ).
\end{equation}
In view of Proposition \ref{DP}, this follows
directly from one of the forms of Deligne's Riemann-Roch theorem \cite[(7.5.1)]{DeligneDet}. 
Another special case is when $\det(\E)$ is trivial; then $\E$ has an elementary structure 
 (\ref{adestable} (d)) and we can argue as in the body of the paper. 
For clarity, we explain the steps of the argument: 
We choose a finite flat morphism $\pi: Y\to {\Bbb P}^1_S$, set $\V=\pi_*\O_Y^n$
and denote by $\V^\vee $ the dual of $\V$. Denote by $h: {\Bbb P}^1_S\to S$ the structure morphism so that $f=h\cdot \pi$. We have
$$
\bar\chi(Y, \E-\O_Y^n)=\bar\chi({\Bbb P}^1_S, \pi_*\E-\V)
$$
Then the vector bundles $\pi_*\E\oplus\V^\vee$ and $\V\oplus \V^\vee$ 
on ${\Bbb P}^1_S$ have elementary structure (cf. Proposition \ref{prop:basec} (ii)). The adelic Riemann-Roch theorem for ${\Bbb P}^1_S$ (Theorem 
\ref{proHorr} for $G=\{1\}$ extended in the obvious manner to the base $R$) gives
$$
\bar\chi({\Bbb P}^1_S, \pi_*\E\oplus \V^\vee )=-h_* c_2(\pi_*\E\oplus\V^\vee),\quad 
\bar\chi({\Bbb P}^1_S, \V\oplus \V^\vee )=-h_* c_2(\V\oplus\V^\vee).
$$
Subtracting these gives
$$
\bar\chi(Y, \E-\O_Y^n)=\bar\chi({\Bbb P}^1_S, \pi_*\E-\V )=-h_*(c_2(\pi_*\E\oplus\V^\vee)-c_2(\V\oplus\V^\vee)).
$$
By  Proposition \ref{prop:state} for $G=\{1\}$,
$$
c_2(\pi_*\E\oplus\V^\vee)-c_2(\V\oplus\V^\vee)=\pi_*c_2(\E).
$$ 
Hence, the right hand side above is equal to 
$-h_*\pi_*c_2(\E)$ and the equation
\begin{equation}\label{gen9}
\bar\chi( \E)-n\cdot \bar\chi( \O_Y)=-f_*(c_2(\E))
\end{equation}
then follows from 
Lemma \ref{lem:pushpat}.

In general, 
for simplicity, set $\DD=\det(\E) $ and $\E_0= \DD^{-1}\oplus \E$. 
 The vector bundle
$\E_0$ has trivial determinant, hence, as above, we have
\begin{equation*}\label{gen1}
\bar\chi( \E_0)-(n+1)\cdot \bar\chi( \O_Y)=-f_*(c_2(\E_0)).
\end{equation*}
Proposition \ref{compareExt}   together with the definition of the
second Chern class implies
\begin{equation*}\label{gen2}
c_2(\E_0)=c_2(\E)+c_2(\DD^{-1}\oplus \DD)=c_2(\E)-\DD\cap\DD
\end{equation*}
in ${\rm CH}^2_{\Bbb A}(Y)$. 
Since $\bar\chi( \E_0)=\bar\chi( \E)+\bar\chi( \DD^{-1})$ these combine to
give us the identity
\begin{equation*}\label{gen3}
(\bar\chi( \E)-n\cdot \bar\chi( \O_Y))+(\bar\chi(\DD^{-1})-\bar\chi(\O_Y))=-f_* (c_2(\E)) +f_*(\DD\cap\DD).
\end{equation*}
We can easily see that the result  follows by combining the above with the identity
\begin{equation*}\label{gen4}
2\cdot (\bar\chi(\DD^{-1})-\bar\chi(\O_Y))=f_*(\DD^{-1}\cap\DD^{-1}+ \DD^{-1}\cap \omega^{-1}_{Y/S}))
\end{equation*}
which is (\ref{genDRR}) for $\LL=\DD^{-1}$.  
\end{proof}

\smallskip

 \bibliographystyle{plain}
 
\bibliography{Parshin1}

\end{document}